\documentclass[11pt]{amsart}
\theoremstyle{plain}
\usepackage{amsmath, amssymb,amscd, graphicx, enumitem, cancel, setspace}
\usepackage{accents}
\usepackage{nccmath}
\usepackage
[
a4paper,% other options: a3paper, a5paper, etc
textwidth=15.6cm,
textheight=22.1cm,
centering,
]
{geometry}
\newtheorem{lemma}{Lemma}[section]
\newtheorem{prop}[lemma]{Proposition}%[section]
\newtheorem{theorem}[lemma]{Theorem}%[section]
\newtheorem{cor}[lemma]{Corollary}%[section]
\newtheorem{definition}[lemma] {Definition}%[section]
\newtheorem{notation}[lemma] {Notation}%[section]

\newtheorem*{theorem*}{Theorem}
\newtheorem*{conj*}{Conjecture}
\newtheorem*{claim*}{Claim}

\newlength{\dhatheight}

\newcommand{\lstar}{{{{\mathcal{L}_{\highsthree}}^{\mkern-20.5mu\highstar\mkern15mu}}}}

\newcommand{\highstar}{\mathchoice
  {\mbox{\smaller$pq$}}
  {\mbox{\smaller$pq$}}
  {\raisebox{.5pt}{\larger[-3]$*$}}
  {\raisebox{.5pt}{\larger[-3]$*$}}
}

\newcommand{\highsthree}{\mathchoice
  {\mbox{\smaller$pq$}}
  {\mbox{\smaller$pq$}}
  {\raisebox{1.5pt}{\larger[-3]$S^3$}}
  {\raisebox{1.5pt}{\larger[-3]$S^3$}}
}

\newcommand{\highpq}{\mathchoice
  {\mbox{\smaller$pq$}}
  {\mbox{\smaller$pq$}}
  {\raisebox{1.5pt}{\larger[-3]$pq$}}
  {\raisebox{1.5pt}{\larger[-4]$pq$}}
}

\newcommand{\lowR}{\mathchoice
  {\mbox{\smaller$\mkern.2mu\mathbb{R}$}}
  {\mbox{\smaller$\mkern.2mu\mathbb{R}$}}
  {\raisebox{-.8pt}{\larger[-3]$\mkern.1mu\mathbb{R}$}}
  {\raisebox{-.8pt}{\larger[-4]$\mkern.1mu\mathbb{R}$}}
}

\newcommand{\C}{{{\mathbb C}}}
\renewcommand{\P}{{{\mathbb P}}}
\newcommand{\Q}{{{\mathbb Q}}}
\newcommand{\R}{{{\mathbb R}}}
\newcommand{\Z}{{{\mathbb Z}}}

\renewcommand{\mod}{{{\mathrm{mod\;}}}}

\renewcommand{\ker}{{{\text {ker} \ }}}

\newcommand{\into}{{\hookrightarrow}}

\renewcommand{\epsilon}{{\varepsilon}}

\title[Rational L-space surgeries 
on satellites by algebraic links]{Rational L-space surgeries\\
on satellites by algebraic links}
\author{Sarah Dean Rasmussen}
\email{S.Rasmussen@dpmms.cam.ac.uk}
\address{DPMMS\\ University of Cambridge\\ UK}

\begin{document}

\begin{abstract}
Given an $n$-component link $L$ in any 3-manifold $M$, the space 
$\mathcal{L} \subset (\mathbb{Q}\cup \mkern-1.5mu\{\infty\})^n$ 
of rational surgery slopes yielding L-spaces is already fully characterized
(in joint work by the author \cite{lslope}) 
when $n\!=\!1$ and $\mathcal{L}$ is nontrivial.
For $n\mkern-2mu>\mkern-3mu1$, however,
there are no previous results for 
$\mathcal{L}$ as a rational subspace, and only limited results for
integer surgeries $\mathcal{L}\cap\mathbb{Z}^n$ on $S^3\mkern-2mu$.
Herein, we provide the first nontrivial explicit descriptions of $\mathcal{L}$
for rational surgeries on multi-component links.
Generalizing Hedden's and Hom's L-space result for cables,
we compute both $\mathcal{L}$, and its topology,
for all satellites by torus-links in $S^3\mkern-2mu$. 
For fractal-boundaried $\mathcal{L}$ resulting from satellites by algebraic links
or iterated torus links, we develop arbitrarily precise approximation tools.
We also extend the provisional validity of the L-space conjecture for rational surgeries on a knot $K \subset S^3$ to rational surgeries on such satellite-links of $K$.
These results exploit the author's generalized Jankins-Neumann formula
for graph~manifolds~\cite{lgraph}.
\end{abstract}

\vspace*{-.0cm}

\maketitle

\vspace{-.4cm} 

\section{Introduction}

A connected, closed, oriented 3-manifold is called an {\textit{L-space}}
if its reduced Heegaard Floer homology vanishes.
The present work focuses on the following relative notion of L-space.

\begin{definition}
For a compact oriented 3-manifold $\mkern1muY\mkern-3mu$ with
boundary a disjoint union of $n$ tori,
%boundary $\partial Y \mkern-2mu= \coprod_{i=1}^n \partial_i Y$
%a disjoint union of tori,
the
{\em{L-space region}} 
$\mathcal{L}(Y) \subset \prod_{i=1}^n\P(H_1(\partial_i Y; \Z)) \cong (\Q\cup\{\infty\})^n$,
with complement $\mathcal{N}\mkern-2mu\mathcal{L}(Y)$,
is the space of (rational) Dehn-filling slopes of $\mkern2muY\mkern-3.5mu$ which yield
L-space Dehn-fillings.
%Its complement, the
%{\em{non-L-space region}} 
%$\mathcal{NL}(Y) := \prod_{i=1}^n\P(H_1(\partial_i Y; \Z)) \setminus
%\mathcal{L}(Y)$ is the space of non-L-space slopes of $Y$.
\end{definition}

%As subsets of $\Q \mkern-0mu\cup \mkern-1.5mu 
%\{\infty\}\mkern-2.5mu,\mkern4mu$non-L-space regions 
%chart the full silhouette of Heegaard$\mkern4mu$Floer complexity 

\vspace{.03cm}

%\subsubsection{\em{\textbf{Prior Results.}}}Until
\noindent {\textbf{Prior Results.$\;$}}Until
now, studies of multi-component L-space surgery slopes have been confined to {\em{integer}} surgeries on links in $S^3\mkern-2mu$. These primarily include numerical methods~of~Liu to plot individual points in $\mathcal{L} \cap \Z^2$ for 2-component links \cite{liulspace}, Gorsky and Hom's identification of torus-link satellites with integer L-space surgery slopes in the
positive orthant \cite{gorskyhom},
and Gorsky and N{\'e}methi's work on integer torus-link surgeries \cite{GNalglink}
and on a partial characterization (complete for algebraic links) of which 2-component links have
$\mathcal{L}\mkern.5mu\cap\mkern.5mu \Z^2$ bounded~from~below~\cite{GNLbounded}.

\vspace{.18cm}

\noindent {\textbf{Present Motivation.$\mkern5mu$}}As
subsets of $(\Q \cup \mkern-1.5mu\{\infty\})^n\mkern-2mu,$
L-space regions
exhibit qualitative features invisible to 
the set of integer L-space surgery slopes, such as nontrivial topological properties, fractal behaviors at the boundary of $\mathcal{L}$, and symmetries such as the action of 
$\Lambda$ in~Theorem~\ref{thm: torus link satellite in s3, lambda action}.

Since non-L-space regions
%$\mkern3mu\mathcal{N}\mkern-3mu\mathcal{L}$
chart the silhouette of Heegaard Floer complexity as a function of
varying surgery slope, 
this creates a rich template to compare against 
the~surgery~regions supporting
any candidate
%any$\mkern3mu${\em{candidate}} 
geometric structure potentially
responsible for nontrivial HF classes.
%that the torsion $\text{spin}^c$ component of HF might be detecting.
Such comparisons for Seifert fibered spaces led to 
the$\mkern4.5mu${\em{L-space$\mkern4.5mu$conjecture}}
that non-L-spaces are characterized by the existence of left orders on fundamental groups and/or
co-oriented taut~foliations~\cite{BGW, JuhaszConj}.
Both L-spaces and $\mathcal{L}$ also constrain complex singularities:
see Section~\ref{ss: intro sats by alg links}.
%also impact algebraic geometry \cite{NemethiLO}. 

%L-space regions also inform bordered Floer homology
%and the Heegaard Floer gluing behavior 
%of manifolds with toroidal boundary. 
%For instance, 
The author's joint result with J.~Rasmussen \cite{lslope}
characterizing nontrivial $\mathcal{L}$ for knot exteriors in 3-manifolds
led both to our toroidal gluing theorem for L-spaces~\cite{lslope}
and to the author's independent~proof~of~the 
L-space conjecture for graph manifolds \cite{lgraph}.
Our joint work on $\mathcal{L}$ combined
with Hanselman and 
Watson's studies of combinatorial
properties of certain bordered Floer algebras \cite{HanWat}
%led to an additional, joint graph-manifold L-space-conjecture proof
%\cite{HRRW}, but it more importantly 
gave rise to a topological
realization of bordered Floer homology for 
single-torus~boundaries~\cite{HRW}. 
A multiple-boundary-component version 
%of~this~is~yet~unknown.
of~this~should~also~exist.

\noindent {\textbf{Methods: Classification formula.}}
Despite reliance on an enhanced L-space gluing tool proved in 
Theorem~\ref{thm: knot exterior gluing theorem},\footnote{Seven months after the current article's appearance on the arXiv, Hanselman, Rasmussen, and Watson posted a revised version of \cite{HRW} with a new L-space gluing theorem subsuming the current paper's
Theorem~\ref{thm: knot exterior gluing theorem}.}
this paper was primarily made possible by
the author's classification of graph manifolds admitting co-oriented taut foliations, with proof of the
graph-manifold L-space conjecture as by-product \cite{lgraph}. (This is not to be confused with the author's joint work with
Hanselman {\em{et al}} \cite{HRRW}.)\footnote{The author conceived this foliation-classification project \cite{lgraph} shortly before her summons to collaborate with 
Hanselman$\mkern5mu${\em{et$\mkern5mu$al}} \cite{HRRW}.  These two proofs of the 
graph-manifold L-space conjecture make contact with foliations via disparate mechanisms. 
The classification result itself is exclusive to~the~author's~independent~work.}
 This classification 
combines a new {\em{classification formula}} 
(Theorem~\ref{thm: l-space interval for seifert jsj}),
generalizing that of Jankins and Neumann for Seifert~fibered~spaces~\cite{JankinsNeumann},
with
a {\em{structure$\mkern4.5mu$theorem}} 
(Theorem~\ref{thm: gluing structure theorem})
prescribing the interpretation of outputs of this formula.

%For $\mkern-.8mu\mathcal{L}\mkern-.5mu$ nontrivial,
%this classification tool also governs For L nontrivial
%
This classification tool also governs L-space regions for
unions of graph manifolds with
single-torus-boundary manifolds.
In particular it gives a complete abstract characterization of $\mathcal{L}$
for any graph-manifold-exterior satellite
of any knot in any 3-manifold.
The classification formula alternately composes a 
linear-fractional transformation $\phi_{e*}^{\P}$,
induced by a gluing map $\phi_e$ for each edge $e$,
with a pair $y^v_{\pm}$, for each vertex $v$,
of extremizations of locally-finite collections of piecewise-constant 
functions of slopes in a certain Seifert-data-compatible basis.

%as a function of slopes in the Seifert-data basis ``\textsc{sf}'' 
%(see Section~\ref{ss: seifert fibered and special subsets}),

\vspace{.2cm}

\noindent {\textbf{Results.}}
Herein, we analyze the intricate
%\noindent {\textbf{Results.}}$\mkern4mu$The
%following paper analyzes the
behavior of solutions $\mathcal{L}$ to the 
classification formulae for exteriors of such satellites.
The bounded-chaotic behavior of these $y^v_{\pm}$
generically leads to fractal-boundaried $\mathcal{L}$,
but we develop precise tools for local approximation and
topological characterization.
As sample applications of these tools, 
Theorems 
\ref{thm: algebraic link satellites}
and~\ref{thm: iterated torus links satellites}~construct~global 
inner approximations of $\mathcal{L}$
for satellites by
algebraic links and iterated-torus-links, respectively.
%These approximation tools
%are more general

Moreover, for a satellite in $S^3$ by an $n$-component torus link, the
%For the torus-link satellite~of~a~{\em{knot}}~in $S^3$, the
chaotic behavior of
$\mkern.5mu y^v_{\pm}\mkern-2.5mu$ generically degenerates, and we provide an
{\em{exact}}$\mkern1.5mu$ explicit description of $\mathcal{L}$
and its various possible topologies, in
Theorems~\ref{thm: torus link satellite in s3, lambda action}
and
\ref{thm: intro topological theorem}, respectively.
Lastly, in
Theorem~\ref{thm: bgw and juhasz} and
Corollary~\ref{cor: bgw and juhasz}, we
promote L-space conjecture results for knot surgeries to results for satellite surgeries.

\vspace{.08cm}

\subsection{Torus-link satellites}

%\noindent 1.1 {\textbf{Torus-link satellites.}}
%\subsection{Torus-link satellites}
The $T(np,nq)$-torus-link satellite
$K^{(np,nq)} \mkern-4mu\subset\mkern-3mu M$
of a knot $K \mkern-4mu\subset\mkern-4mu M$ in a 3-manifold $M$
embeds the torus link $T(np,nq)$ in the boundary of a neighborhood $\nu(K)$ of 
$K\mkern-3.8mu\subset\mkern-3.7mu M$.
The exterior 
$\mkern-.1mu\boldsymbol{Y^{\mkern-1mu(np,nq\mkern-.5mu)}}\mkern-.2mu$
of $K^{(np,nq)}\mkern-1mu$ splices $K \mkern-3.6mu\subset\mkern-3.7mu M$ to the 
multiplicity-$q$ fiber of the Seifert fibered exterior of 
%of the Seifert-fibered exterior of 
$T(np,nq)$.
This Seifert structure also
%$S^3 \setminus \overset{\circ}{\nu}(T(np,nq))$
prescribes  
3 distinguished subsets
$\Lambda, \mathcal{R}, \mathcal{Z} \subset \prod_{i=1}^n\P(H_1(\partial_i Y^{(np,nq)}))$
of slopes.
The lattice $\Lambda$ {\em{acts}} on slopes by reparametrization of Seifert data,
and $\mathcal{R} \setminus \mathcal{Z}$ catalogs reducible surgeries
with no $S^1 \mkern-1.2mu\times\mkern-.7mu S^2$ connected summand.

%Note that $(\Q \cup \{\infty\})^n$ also carries the natural permutation action of 
%the symmetric group $\mathfrak{S}_n$.
%, with
%``$\mathfrak{S}_n \mkern1mu\cdot$''
%taking the union of $\mathfrak{S}_n$-orbits.
%, $\sigma \cdot \boldsymbol{\alpha} = 
%(\alpha_{\sigma(1)}, \ldots, \alpha_{\sigma(n)})$,
%We provide explicit descriptions of $\Lambda$, $\mathcal{R}$, and $\mathcal{Z}$
%on the following page.

\begin{theorem}
\label{thm: torus link satellite in s3, lambda action}
Suppose that $K \mkern-3mu\subset\mkern-3mu S^3\mkern-1.5mu$ is a positive L-space knot of genus 
$g(K)$, and that $\mkern1mun, p,q \in \Z,$
with $n, p\mkern-2mu >\mkern-2mu 0$ and
$\gcd(p,q) \mkern-2mu=\mkern-2mu 1$.
Then the $\mkern1muT\mkern-1mu(np,nq)$ torus-link satellite 
$K^{(np,nq)}\mkern-2mu\subset\mkern-2mu S^3$ of $K$ has L-space surgery region
given by the union of $\Lambda$-orbits
$\;\mathcal{L}_{S^3} = \Lambda \cdot
\mathcal{L}^*_{S^3}\;,$
where

\begin{itemize}
\item[$(i)$]
If 
$N:= 2g(K)-1 > \frac{q}{p}$ and
$K\subset S^3$ is nontrivial, then 

\vskip-.3cm
$$\mathcal{L}^*_{S^3}
:= 
\begin{cases}
%\{\boldsymbol{\infty}\} := 
\mkern-1mu\{(\infty, \ldots, \infty)\}
   &
p>1
   \\
\bigcup_{i=1}^{\mkern1mu n}\left(\{\infty\}^{i-1} \times [N, +\infty] \times \{\infty\}^{n-i}\right)
%\mkern.5mu\mathfrak{S}_n\cdot\left([N, +\infty] \times \{\infty\}^{n-1}\right)
   &
p=1
\end{cases}.
$$

\vspace{-.02cm}

\item[$(ii)$]
If $2g(K)-1$
%\le \frac{q}{p}$ 
\raisebox{.1pt}{$\le$} 
$\mkern-8mu$ \raisebox{.09pt}{$\frac{q}{p}$},
and $K \subset S^3$ is nontrivial,
or if $p, q > 1$ and $K \subset S^3$ is the unknot (so that $K^{(np,nq)} = T(np,nq)$),
then for $N_{pq}:= pq-p-q+2g(K)p$, we have
\begin{align*}
& \mkern62mu
\mathcal{L}^*_{S^3}
\mkern2mu:=\mkern15mu
%\mkern2mu
%&=
\mathcal{L}^-_{S^3} 
\,\;\;\;\cup\;\;\; \,
(\mathcal{R}_{S^3}\mkern-2.5mu\setminus\mkern-1.5mu\mathcal{Z}_{S^3})
\,\,\;\;\;\cup\;\;\;\,
\mathcal{L}^+_{S^3};
    \\
\mathcal{R}_{S^3} \mkern-2.5mu\setminus\mkern-1.5mu
\mathcal{Z}_{S^3}
&=
{\textstyle{\coprod_{i=1}^{\mkern1mu n}}}
%\{ \boldsymbol{y} \in (\Q \cup \mkern-2mu\{\infty\})^n|
%\mkern2mu y_j = pq \Leftrightarrow j=i \}
\left( \vphantom{A^a} 
\left[-\infty, pq\right> \mkern-1mu\cup\mkern-2mu \left< pq, +\infty\right])^{i-1}
\mkern2mu\times\mkern2mu \{pq\} \mkern2mu\times\mkern2mu
(\left[-\infty, pq\right> \mkern-1mu\cup\mkern-2mu  \left< pq, +\infty\right])^{n-i} \right),
  \\
& \mkern30mu
\mathcal{L}^-_{S^3} 
=
%&=
\left[-\infty, pq\right>^{\mkern-1mu n}
\setminus \left[-\infty,\right. \left.N_{pq}\right>^{\mkern-1mu n}\mkern-2mu,
\mkern35mu
\mathcal{L}^+_{S^3}
=\left<pq, +\infty\right]^n \mkern-2mu.
\end{align*}

%\mathfrak{S}_n \cdot\left(\{pq\} \times
%(\left[-\infty, pq\right> \cup  \left< pq, +\infty\right])^{n-1} \right).

\end{itemize}
\end{theorem}

\begin{figure}
\includegraphics[scale=.65]{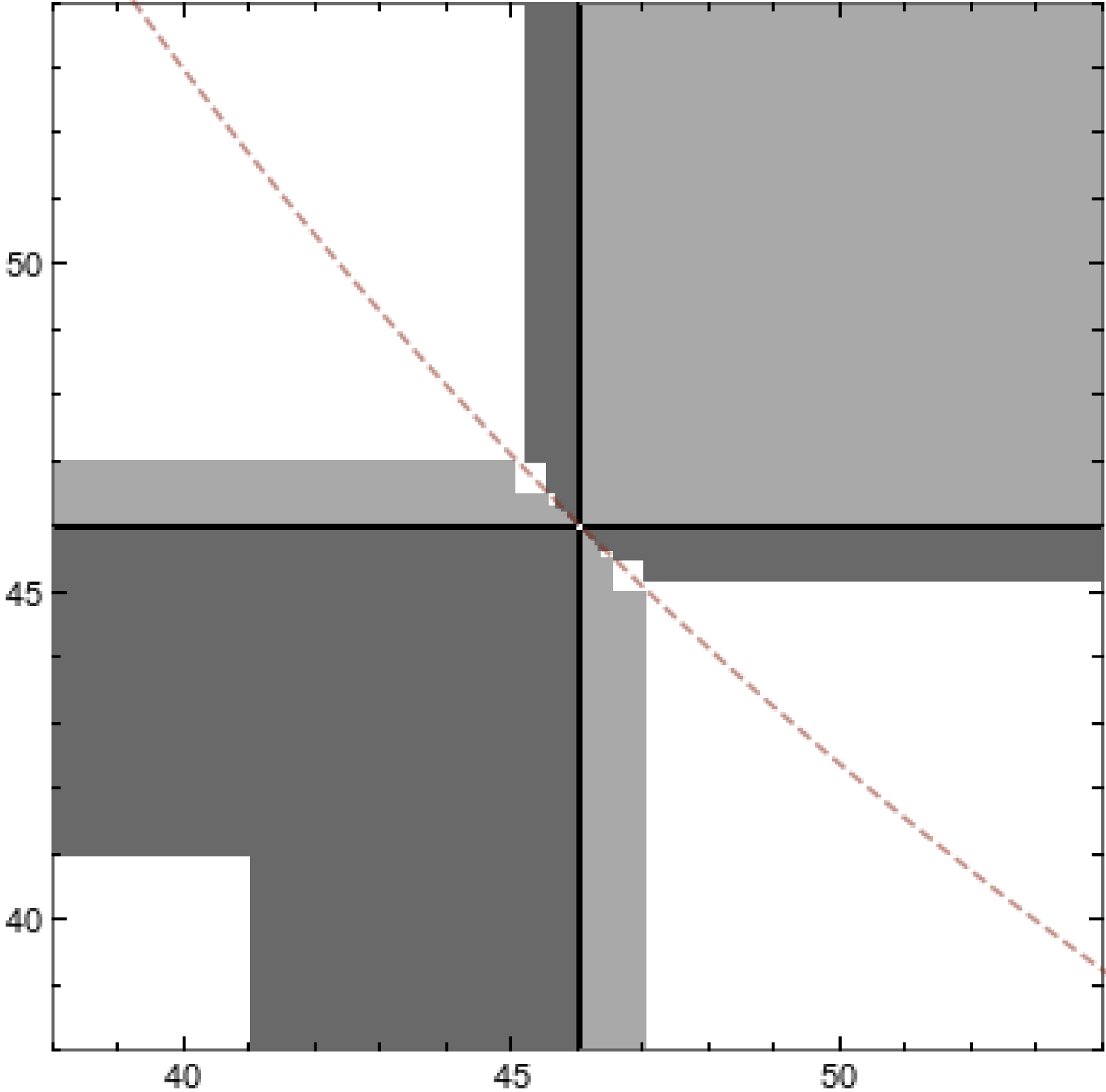}
\caption{
The L-space surgery region $\mathcal{L}_{S^3}$ for the $T(4,46)$~satellite~of~the
% genus =  5 
$P(-2,3,7)$
pretzel knot, with $pq \mkern-2mu=\mkern-2mu 46$ and $N_{pq} \mkern-3mu=\mkern-2mu 41$.
%$\mathcal{L}^-_{S^3} = \left(\left[-\infty, pq\right>^{\mkern-1.7mu n} 
%\setminus \left[-\infty, N_{pq}\right>^{\mkern-1.7mu n} \right)$ 
Here, $\Lambda \cdot \mathcal{L}^-_{S^3}$
is dark grey, 
$\Lambda \cdot \mathcal{L}^+_{S^3}$
%$\mathcal{L}^+_{S^3}=\left<pq, +\infty\right]^n$ 
is light grey,
%the reducible L-space region 
$\mathcal{R}_{S^3}$
%\mkern-2.5mu\setminus\mkern-1mu \mathcal{Z}_{S^3}$ 
is black except for
$\mkern1mu\mathcal{Z}_{S^3} \mkern-2.5mu=\mkern-2.5mu \{(pq,pq)\}
\mkern-2.5mu\subset\mkern-1.5mu
\mathcal{N}\mkern-3mu\mathcal{L}_{S^3}$, and
$\mathcal{N\mkern-4muL}_{S^3}$ is white
except for the (dotted) conic
$\mathcal{B}_{S^3} \mkern-1.7mu=\mkern-1.7mu \{\alpha_1\alpha_2 \mkern-2.2mu=\mkern-2mu pq\}
\subset \mathcal{N}\mkern-2.5mu\mathcal{L}_{S^3}$ 
of rational longitudes of the satellite exterior.
As usual,  
$\mathcal{L}_{S^3} \mkern-3mu\setminus\mkern-1.8mu \mathcal{L}_{S^3}^*$
has radius $1$ about $\mathcal{R}_{S^3}$.
\vspace{-.18cm}
}
\label{fig:T(4,46) pretzel}
\end{figure}

\noindent {\textbf{Remarks.}}$\;$Positive L-space knots $K \mkern-4mu\subset\mkern-3mu S^3$ have
$\mathcal{L}_{S^3} 
%(S^3\setminus\overset{\circ}{\nu}(K))
\mkern-3.5mu=\mkern-1.8mu\mkern2mu[2g(K)\!-\!1,+\infty]$
\cite{OSRat}.
Theorem~\ref{thm: s3 satellite in sf basis}~and its remark cover the remaining (redundant or less interesting) cases of negative L-space or non-L-space knots $K$, and the fractal-boundaried case of $K^{(np,nq)} \mkern-2mu=\mkern-3mu T^{(n,nq)}\mkern-1mu$ (for $p\mkern-2.5mu=\mkern-2.5mu1$ and $K$ the unknot).
We use ``$\left<\right.\mkern-.1mu$''$\mkern-2mu$
%or ``$\mkern-.1mu\left.\right>$''
for open endpoints and
implicitly intersect intervals
with $\Q\mkern0mu \cup \mkern-1.5mu\{\infty\}$.

%$\Q\mkern1.5mu \cup \mkern-.8mu\{\infty\}$.

\vspace{.29cm}
%\vspace{.32cm}

\noindent {\textbf{Example:}}$\;${\textbf{{\textit{Cables.}}}}$\,$ 
For $n\mkern-1mu=\mkern-1mu1$, $\Lambda$ is trivial and 
$\mathcal{R} \setminus \mathcal{Z} \mkern-1.3mu=\mkern-1.3mu \{pq\}$.
Thus
Theorem~\ref{thm: torus link satellite in s3, lambda action} yields

\vspace{-.48cm}

$$
(ii)
\mkern20mu
\mathcal{L}_{S^3\mkern-.8mu}
(\mkern.1mu Y^{\mkern-.2mu\text{\raisebox{-.6pt}{(}}p,q\text{\raisebox{-.6pt}{)}}} \mkern-1mu) 
\mkern2mu
\mkern5mu=\mkern5mu
\left[N_{pq}, pq\right> 
\mkern3mu
\cup\mkern3mu
\{pq\}
\mkern5mu\cup\mkern5mu
\left<pq, +\infty\right] 
\mkern5mu=\mkern5mu
[N_{pq}, +\infty]
\mkern5mu=\mkern5mu
[2g
(\mkern.1mu K^{\mkern-.2mu{\text{\raisebox{-.6pt}{(}}}p,\mkern1mu q\text{\raisebox{-.6pt}{)}}} \mkern-1mu) 
-1, +\infty]
\hphantom{(ii)}
\mkern20mu
$$

\vspace{-.01cm}

\noindent for 
$K \mkern-3.8mu\subset\mkern-4.3mu\mkern2mu S^3\mkern-.5mu$
a nontrivial positive L-space knot
with
$2g(K)\mkern-.7mu-\mkern-.7mu1 \mkern-1.3mu$
\raisebox{.2pt}{$\le$} 
$\mkern-8mu$ \raisebox{1pt}{$\frac{q}{p}$},
and $(i)$
$\mathcal{L}_{S^3\mkern-.8mu}
(Y^{\mkern-.1mu\text{\raisebox{-.8pt}{(}}p,q\text{\raisebox{-.8pt}{)}}}\mkern-1mu) 
\mkern-2mu=\mkern-2mu \{\infty\}$
for 
$2g(K)\mkern-1mu-\mkern-1mu 1 \mkern-2mu>\mkern-2mu \frac{q}{p}$,
%---recovering
recovering
well-known results of
Hedden \cite{Heddencableii} and Hom \cite{Homcable} for cables.

%\vspace{.35cm}
\vspace{.31cm}

\noindent {\textbf{Topology of $\mkern1mu\boldsymbol{\mathcal{L}(Y^{(np,nq)})}.$}}$\mkern3mu$
For any subset $A \mkern-1mu\subset\mkern-2mu (\Q\cup\{\infty\})^n \mkern-.7mu\hookrightarrow \mkern-.5mu
(\R \cup \{\infty\})^n$ with complement 
$A^c := (\Q\cup\{\infty\})^n \setminus A$
and real closure $\overline{A} \subset (\R \cup \{\infty\})^n$,
we define
$A^{\lowR} := \overline{A} \setminus A^c \subset (\lowR \cup \{\infty\})^n$.

\vspace{-.05cm}

\begin{theorem}
\label{thm: intro topological theorem}
Take $K,\mkern1mu K^{(np,nq)} \mkern-3mu\subset\mkern-3mu S^3\mkern-2mu,\mkern1mu$
$\mathcal{L},\mkern2mu \mathcal{N}\mkern-2mu\mathcal{L}, N,$ 
%and $\Lambda$
and $\Lambda$
as in Theorem~\ref{thm: torus link satellite in s3, lambda action},
with $q\!>\!0$, and let 
$\mathcal{B}$ be the set of rational longitudes,
as described in 
{\em{(\ref{eq: stuff with B})}},
of the exterior $Y^{(np,nq)}\mkern-3mu$ of $K^{(np,nq)}\mkern-3mu.$
%by Section~\ref{ss: rational longitudes B},
%$\mkern1mu\mathcal{B}_{S^3}^{\mkern.6mu\lowR}
%\mkern-1.5mu\cong\mkern-1.5mu 
%\mathbb{T}^{n-1} \mkern2.5mu\into\mkern2.7mu (\R \cup\{\infty\})^n_{S^3}
%\mkern-1.5mu\cong\mkern-1.5mu\mathbb{T}^n\mkern-1mu.$

\begin{itemize}
\item[$(i.a)$]
If $N \mkern-3mu$
\raisebox{.2pt}{$>$} 
$\mkern-10mu$ \raisebox{1pt}{$\frac{q}{p}$}
%>\mkern-1.5mu \frac{q}{p}$
and $p\mkern-2mu>\mkern-2mu 1$, 
or if 
$N \mkern-3mu$
\raisebox{.2pt}{$>$} 
$\mkern-10mu$ \raisebox{1pt}{$\frac{q}{p}$}$\mkern2mu+\mkern1mu 1$
%$N \mkern-2mu>\mkern-1.5mu \frac{q}{p} + 1$
and $p\mkern-2mu=\mkern-2mu 1$, 
then $\mathcal{L}^{\lowR}\!$ deformation retracts onto $\Lambda$.

\item[$(i.b)$]
%If $N \mkern-2mu=\mkern-1.5mu \frac{q}{p}+1$ 
If $N \mkern-3mu$
\raisebox{.2pt}{$=$} 
$\mkern-10mu$ \raisebox{1pt}{$\frac{q}{p}$}$\mkern2mu+\mkern1mu1$
%$(p=1)$
and $n \mkern-2mu>\mkern-2mu 2$, 
then 
%$\mkern1mu\pi_0(\mathcal{L}^{\lowR})\mkern-1.5mu=\mkern-1.5mu 0$,
$\mkern1mu\mathrm{rank} \,H_1(\mathcal{L}^{\lowR}) 
\mkern-2mu=\mkern-2mu \textstyle{\binom{n}{2}}\mkern-1mu-\mkern-1mu1$ and
$\mkern1mu\pi_1(\mathcal{L}^{\lowR}) 
\mkern-1.5mu\simeq\mkern-1.5mu 
\mathrm{ker}(\delta)$,
for the map

\vspace{-.05cm}

\noindent $\mkern1mu\delta\mkern-1mu:\mkern-1mu 
\mathfrak{F}^{\binom{n}{2}} \mkern-2mu\to\mkern-1mu 
\Lambda_{\textsc{sf}}\mkern-1mu$,
$(x_{ij})^{\mkern1mu e} \mapsto (\epsilon_i-\epsilon_j)e$
as in {\em{(\ref{eq: def of delta})}},
with $\mathfrak{F}^{\binom{n}{2}} := \left<x_{ij}\right>_{i<j}$ free and
$\epsilon_i \in \Z^n_{\textsc{sf}} $.

\item[($i.c)$]
If $N \mkern-2mu=\mkern-1.5mu \frac{q}{p}\mkern-1mu+\mkern-2mu 1$ and
$n \in \{1,2\}$, 
or if $N \mkern-2mu=\mkern-1.5mu \frac{q}{p}$,
then $\mathcal{L}^{\lowR}$ is contractible, of
dimension 1 or $n$.

\item[$(ii)$]
If $N \mkern-3mu<\mkern-2.5mu \frac{q}{p}$ or
$K^{(np,nq)} \mkern-1mu=\mkern-.5mu T(np,nq)$, then	
$\mathcal{N}\mkern-2.5mu\mathcal{L}^{\lowR}\mkern-2mu$ and $\mathcal{L}^{\lowR}\mkern-2mu$
deformation retract, respectively, onto 
$\mkern1mu\mathcal{B}^{\mkern.6mu\lowR}\mkern-2mu\cong\mkern-1mu \mathbb{T}^{n-1}\mkern-2mu$
and onto a $\mkern2mu\mathbb{T}^{n-1}\mkern-2mu$ parallel to 
$\mathcal{B}^{\mkern.6mu\lowR}\mkern-1mu$
in $\mkern1mu\prod_{i=1}^n \mkern-2mu\P(H_1(\partial_i Y^{(np,nq)};\R))
\mkern-2mu\cong\mkern-1mu \mathbb{T}^{n}\mkern-1mu$.
\end{itemize}
\end{theorem}

\vspace{.20cm}
%\vspace{.22cm}

\noindent {\textbf{$\boldsymbol{S^3}$ and $\boldsymbol{\textsc{SF}}$ slope bases.}}
The $S^3\mkern-1mu$ subscript
on, {\em{e.g.}}, $\mathcal{L}_{S^3}$, specifies the
conventional $S^3$ surgery basis for slopes. On
Seifert fibered $M$, the
``$\textsc{sf}$-slopes'' 
$(\Q\cup\mkern-2mu\{\infty\})^n_{\textsc{sf}} 
\mkern-2mu:=\mkern-2mu \prod_{i=1}^n \mkern-2mu\P(H_1(\partial_i M;\Z))_{\textsc{sf}}$
(see Section~\ref{ss: seifert fibered and special subsets})
mimic the Seifert-data fractions
$\frac{\beta_i}{\alpha_i}\mkern-2mu.\mkern1mu$
For $Y^{\mkern-.9mu(\mkern-.2munp,nq\mkern-.3mu)}\!,\mkern.5mu$ 
we change slope basis via

\vspace{-.21cm}

\begin{equation}
\psi: 
(\Q \cup \{\infty\})^n_{\textsc{sf}} \longrightarrow (\Q \cup \{\infty\})^n_{S^3},\;\;
\boldsymbol{y} \mapsto
\left(pq + {\textstyle{\frac{1}{y_1}}}, \ldots, pq + {\textstyle{\frac{1}{y_n}}}\right).
\end{equation}

\vspace{0cm}

We briefly pause here to elaborate on the distinguished subsets 
$\mathcal{R}$,
$\mathcal{Z}$,
$\mathcal{B}$,
and $\Lambda$ of $\textsc{sf}$-slopes.

\vspace{.2cm}

\noindent {\textbf{Reducible Slopes $\boldsymbol{\mathcal{R}}$ and $\boldsymbol{\mathcal{Z}}$.}}
Dehn filling along the smooth-fiber slope
$\infty \mkern-2.5mu\in\mkern-2.5mu (\Q \cup \mkern-2mu\{\infty\})_{\textsc{sf}}$
decomposes a Seifert fibered space as a connected sum, with one summand for each exceptional fiber  or boundary component.
After this connected-sum decomposition has occurred,
any additional $\infty$-fillings create $S^1\mkern-2mu \times\mkern-1mu S^2$ summands.
Thus, since $\mkern1mu\psi\mkern-2.5mu: \infty \mapsto pq + \frac{1}{\infty} = pq$,
our {\em{reducible slopes}} $\mathcal{R}$ and their
{\em{exceptional subset}} $\mathcal{Z}$
(see Definitions~\ref{def: R and Z} and \ref{def: false reducible}) satisfy
\vspace{-.03cm}
\begin{align}
\nonumber
\mathcal{R}_{S^3}(Y^{(np,nq)})
&=
{\textstyle{\bigcup_{i=1}^{\mkern1mu n}}} \mkern0mu
\{{\boldsymbol{\alpha}}\in (\Q\cup \{\infty\})^n_{S^3} |
\,\alpha_i=pq\},
\\ 
\mathcal{Z}_{S^3}(Y^{(np,nq)})
&=
{\textstyle{\bigcup_{i<j}}} \{{\boldsymbol{\alpha}}\in 
(\Q\cup \{\infty\})^n_{S^3} |
\,\alpha_i = \alpha_j =pq\},
\\ \nonumber
\mathcal{R}_{S^3}(Y^{(np,nq)}) \setminus
\mathcal{Z}_{S^3}(Y^{(np,nq)})
&=
\mathfrak{S}_n \cdot\left(\{pq\} 
\mkern1mu\times\mkern1mu
(\left[-\infty, pq\right> 
\mkern-1mu\cup\mkern-1mu  
\left< pq, +\infty\right])^{n-1}
\right),
\end{align}

\vspace{-.06cm}

\noindent where $\mathfrak{S}_n$ acts
on $S^3$-slopes
%$(\Q \cup \{\infty\})^n_{S^3}$
by reordering boundary components: $\sigma \cdot \boldsymbol{\alpha} = 
%(\alpha_{\sigma(1)}, \ldots, \alpha_{\sigma(n)})$,~and
%where ``$\mathfrak{S}_n \mkern1mu\cdot$'' takes the union of $\mathfrak{S}_n$-orbits.
%Note that $\mathcal{Z} \subset \mathcal{N}\mkern-2mu\mathcal{L}$,
%due to the $S^1 \mkern-1mu\times S^2$ summands.
%
(\alpha_{\sigma(1)}, \ldots, \alpha_{\sigma(n)})$.
%Note that if
%$\mkern1mu\mathcal{B}(Y) \mkern-2mu\subset\mkern-1.5mu 
% \prod_{i=1}^n\mkern-3mu \P(H_1(\partial_i Y; \Z)\mkern-1mu)$
%denotes the set of slopes with $b_1 \mkern-3.5mu>\mkern-2.5mu 0$ Dehn
% fillings,~then

\vspace{.45cm}

\noindent {\textbf{Rational longitudes $\boldsymbol{\mathcal{B}}$.}}
Let $\mathcal{B}(M)$ denote the set of {\em{rational longitudes}} of $M$, that is, the
slopes yielding Dehn fillings with $b_1 \mkern-2mu>\mkern-2mu 0$.
In particular, for $\partial M$ with $n$ components, we have
\begin{equation}
\mathcal{Z} \mkern1mu=\mkern1mu \mathcal{R} \cap \mathcal{B},
\;\;\;\text{and}\;\;\;
\mathcal{B}^{\mkern.6mu\lowR} \cong \mathbb{T}^{n-1}
\hookrightarrow\mkern.5mu
{\textstyle{\coprod_{i=1}^n\mkern-2mu \P(H_1(\partial_i M ;\R))}}
\simeq
(\R\cup\{\infty\})^n 
\cong \mathbb{T}^n.
\end{equation}

For the exterior $Y^{(np,nq)}$ of $K^{(np,nq)}\mkern-2mu\subset\mkern-2mu S^3$,
Proposition~\ref{prop: basics for B}
tells us $\mathcal{B}_{\textsc{sf}}$
%$(Y^{(np,nq)})$
is the closure
\begin{equation}
\label{eq: stuff with B}
\mathcal{B}_{\textsc{sf}}(Y^{(np,nq)}) 
\mkern2mu
=
\left\{ \vphantom{\textstyle{\frac{1}{pq}}}\mkern-1mu
\boldsymbol{y} \in (\Q \cup \{\infty\})^n_{\textsc{sf}} \right|\,\left. 
{\textstyle{\frac{1}{pq} 
		\mkern-2mu+\mkern-2mu 
		\sum_{i=1}^n\mkern-1mu y_i}} = 0\right\}
\end{equation}

\vspace{-.04cm}

\noindent 
of the linear subspace
$\left\{ \vphantom{A^A} \mkern-4.5mu \right.
%\vphantom{\textstyle{\frac{1}{pq}}} \mkern-1mu
\boldsymbol{y} \in \Q^n_{\textsc{sf}} 
\mkern-1.5mu\left| \vphantom{A^A} \mkern-2mu \right.
%\right|\,
%\left. 
{\textstyle{\frac{1}{pq} 
		\mkern-2mu+\mkern-2mu 
		\sum_{i=1}^n\mkern-1mu y_i}} = 0
\left. \vphantom{A^A} \mkern-4.5mu \right\}$
in $(\Q \mkern.3mu\cup \mkern-2mu\{\infty\})^n_{\textsc{sf}}\mkern1mu$.
The change of slope-basis $\psi$ transforms
$\mathcal{B}_{\textsc{sf}}(Y^{(np,nq)})$ into a
degree-$n$ hypersurface in $(\Q\cup\mkern-2mu\{\infty\})^n_{S^3}$.
%$\mathcal{B}_{S^3}(Y^{(np,nq)})$.
For example, the $n\mkern-1.5mu=\mkern-1.5mu2$ case yields the conic
$\mathcal{B}_{S^3}(Y^{(2p,2q)})
\mkern-1.5mu=\mkern-2.5mu 
\{\boldsymbol{\alpha} \mkern-3mu\in\mkern-3mu \left.(\Q\cup\mkern-2.5mu\{\infty\})^n_{S^3} \mkern-1mu\right|\alpha_1 \alpha_2 = (pq)^2\}$,
as in Figure~\ref{fig:T(4,46) pretzel}.

%\vspace{-.05cm}
%\begin{equation}
%\mathcal{B}_{S^3}(Y^{(2p,2q)}) 
%= \{\boldsymbol{\alpha}\in (\Q \cup \{\infty\})^2|\mkern4mu \alpha_1 \alpha_2 = (pq)^2\}.
%\end{equation}

% Section~\ref{ss: R and Z} for more on

\vspace{.45cm}

\noindent {\textbf{Symmetry by $\mkern-1mu\boldsymbol{\Lambda}$.}}
The lattice 
$\Lambda_{\textsc{sf}}(Y^{(np,nq)}) \mkern-3mu\subset\mkern-3mu 
(\Q \cup \mkern-.7mu\{\infty\})^n_{\textsc{sf}}$ of Seifert-data reparametrizations,
\begin{equation}
\Lambda_{\textsc{sf}}(Y^{(np,nq)}) 
:=\{ \boldsymbol{l} \in \Z^n \mkern-.5mu\subset 
(\Q \cup \mkern-.7mu\{\infty\})^n_{\textsc{sf}} |\;
{\textstyle{\sum_{i=1}^n l_i}} = 0\},
\end{equation}
acts on $\textsc{sf}$-slopes
%the slopes $(\Q \cup \{\infty\})^n_{\textsc{sf}}$
by addition, $\boldsymbol{y} \mapsto \boldsymbol{l} + \boldsymbol{y}$,
thereby determining an action of $\Lambda$ on slopes in any basis.
%hence acts on slopes in any basis.
In particular, in Theorem~\ref{thm: torus link satellite in s3, lambda action},
$\psi$ induces an action of
$\Lambda_{\textsc{sf}}(Y^{(np,nq)}\mkern-1mu)$ on
$(\Q \cup \mkern-2mu\{\infty\})^n_{S^3}$, via
\begin{equation}
\mkern100mu
\boldsymbol{l} \cdot \boldsymbol{\alpha}
= \psi(\boldsymbol{l} + \psi^{-1}(\boldsymbol{\alpha})),
\;\;\;\;\;\;\;\;
\boldsymbol{l}\in
\Lambda_{\textsc{sf}}(Y^{(np,nq)}),\,
\boldsymbol{\alpha}\in
(\Q \cup \{\infty\})^n_{S^3}.
\end{equation}

\vspace{-.02cm}

The action of $\Lambda$ induces homeomorphism on Dehn surgeries:
$\mkern-1muS^3_{\boldsymbol{l\mkern1mu\cdot\mkern1mu\alpha}}(K^{(np,nq)}\mkern-1.5mu)
\mkern-2mu\cong\mkern-2mu
S^3_{\boldsymbol{\alpha}}(K^{(np,nq)}\mkern-1.5mu)$
for
$\boldsymbol{l} \mkern-3mu\in\mkern-3mu \Lambda(Y^{(np,nq)})$
and
$\boldsymbol{\alpha}\mkern-3mu\in\mkern-3mu(\Q \mkern.5mu\cup\mkern-1.5mu \{\infty\}\mkern-1mu)^n_{S^3}$.
Thus $\Lambda$ preserves $\mathcal{R}$, $\mathcal{Z}$, $\mathcal{B}$, $\mathcal{L}$,
and $\mathcal{N}$\!$\mathcal{L}$ as sets, and since
$S^3_{\mkern-2mu\mathcal{L}^*}\mkern-1mu(K^{(np,nq)}\mkern-1.5mu)
\mkern-2mu\cong\mkern-2mu S^3_{\mkern-2mu\mathcal{L}}(K^{(np,nq)}\mkern-1.5mu)$,
$\mathcal{L}_{S^3}^*\mkern-1.5mu$ completely catalogs
the L-space {\em{surgeries}} on $K^{(np,nq)}\mkern-1.5mu.$

The L-space surgery {\em{slopes}} for $K^{(np,nq)}\mkern-1.5mu$
must retain their full $\Lambda$-orbits, but
$\textsc{sf}$-slopes do this automatically:
the expression of 
$\mathcal{L}_{\textsc{sf}}(Y^{(np,nq)})$ in
Theorem~\ref{thm: s3 satellite in sf basis}
is naturally $\Lambda$-invariant.
Moreover, 
%Nevertheless, the $S^3$-slope subset 
$\mathcal{L}^*_{S^3}(Y^{(np,nq)})$ 
% in Theorem~\ref{thm: torus link satellite in s3, lambda action}
is {\em{almost}} $\Lambda$-invariant.
When $\mathcal{L}^*_{S^3} = \{\boldsymbol{\infty}\}$,
$\mathcal{L}_{S^3} \mkern-3mu\setminus\mkern-1.5mu \mathcal{L}^*_{S^3}$
is given by
\begin{equation}
(\Lambda_{S^3} \cdot \{\boldsymbol{\infty}\} ) \setminus
\{\boldsymbol{\infty}\}
=
\Lambda_{S^3} \mkern-1mu
\cap \mkern1mu \Q^n_{S^3}
\;\;\subset\;\;
(\left[pq-1, pq\right> \cup  \left<pq,pq+1\right])^n.
\end{equation}
For any $K^{(np,nq)}$ in 
Theorem~\ref{thm: torus link satellite in s3, lambda action},
the $\Lambda$-mismatch
$\mathcal{L}_{S^3} \mkern-2mu\setminus\mkern-.5mu \mathcal{L}^*_{S^3}$
lies inside a radius-1 neighborhood,
\begin{equation}
\label{eq: nbd of R of radius 1}
\mathcal{L}_{S^3} \mkern-3mu\setminus\mkern-1.5mu \mathcal{L}^*_{S^3}
\;\;\subset\;\;
{\textstyle{\bigcup_{i=1}^{\mkern1mu n} }}
\{{\boldsymbol{\alpha}} |\,\alpha_i\in
\left[pq-1, pq\right> \cup  \left<pq,pq+1\right]\}
=: U_{pq}^n(1),
\end{equation}
of $\mathcal{R}_{S^3}
\mkern-3mu=\mkern-1mu
\bigcup_{i=1}^{\mkern1mu n} \mkern-1mu 
\{{\boldsymbol{\alpha}} |\,\alpha_i \mkern-3mu=\mkern-2mu pq\},$
and $\mathcal{L}_{S^3}$ consists of a finite union of rectangles outside
any positive-radius neighborhood 
of $\mathcal{R}_{S^3}\mkern-1.5mu$.
Moreover, (\ref{eq: nbd of R of radius 1}) implies
the {\em{integer}} slopes in $\mathcal{L}_{\highsthree}(Y^{(np,nq)})$ 
satisfy

\vspace{-.12cm}

\begin{equation}
\mkern-3mu
\mathcal{L}_{\highsthree\mkern-2mu}
%\mathcal{L}_{S^3\mkern-2mu}
%(Y^{(np,nq)}\mkern-.7mu) 
\mkern0mu\cap\mkern0mu
(\Z \mkern-.3mu\cup \mkern-2mu\{\infty\}\mkern-1.1mu)^n
\mkern4mu\subset\mkern4.5mu
\lstar
%\mathcal{L}^*_{S^3\mkern-2mu}
%(Y^{(np,nq)}\mkern-.5mu)
\mkern1.6mu\cup\mkern2mu 
\mathfrak{S}_n \mkern-5mu
\left(\vphantom{\mfrac{A}{B}}\mkern-3mu
\{pq\mkern-2mu-\mkern-4mu 1 \mkern-.8mu,\mkern-.5mu pq\mkern-2mu+\mkern-4mu 1\mkern-1mu\} 
\mkern-3mu\times\mkern-3mu 
(\Z \mkern-1.3mu\cup\mkern-2.2mu \{\infty\}\mkern-2mu)^{n-1} \mkern-3mu\right)\mkern-4.5mu,
\mkern-3mu
\end{equation}

\vspace{-.11cm}

\noindent and it is a simple exercise 
to determine
%$\boldsymbol{\alpha} \in (\mathcal{L}_{S^3} \setminus \mathcal{L}^*_{S^3}) 
%\cap (\Z \cup \{\infty\})^n$
$\boldsymbol{\alpha} \in (\mathcal{L}_{\highsthree} \setminus \lstar)\cap (\Z \cup \{\infty\})^n$
with $\alpha_i \in \{pq-1, pq+1\}$.

\vspace{0cm}

\noindent {\textbf{New features:
Torus-link satellites vs One-strand Cables.}}
L-space regions for satellites by $n\mkern-2mu > \mkern-2mu 1$ torus links
introduce qualitatively new phenomena not present for $n\mkern-1mu=\mkern-1.5mu1$ cables.

\vspace{.05cm}

(a)
For $n\mkern-1mu>\mkern-1mu 1$, the action of $\Lambda$ becomes nontrivial,
although as discussed before, this action does not impact 
actual L-spaces resulting from surgery.

(b)
For $n\mkern-1mu>\mkern-1mu 1$, 
the codimension-1 subspace
$\mathcal{R}_{S^3} \mkern-2mu\subset\mkern-2mu (\Q \cup \mkern-.5mu\{\infty\})^n_{S^3}$ 
acquires positive dimension and 
the codimension-2 subspace
$\mathcal{Z}_{S^3} \mkern-2mu\subset\mkern-2mu (\Q \cup \mkern-.5mu\{\infty\})^n_{S^3}$ 
becomes nonempty,
although the set of reducible L-space slopes
$\mathcal{R}_{S^3} \mkern-1mu\setminus\mkern-1mu \mathcal{Z}_{S^3}$
remains a disjoint union of hyperplanes $\cong \Q^{n-1}$ for all $n$.

(c) For $n \mkern-2mu=\mkern-2mu p \mkern-2mu=\mkern-2mu 1$, both the L-space regions 
$\mathcal{L}_{S^3}(S^3\mkern0mu\setminus\mkern0mu\overset{\circ}{\nu}(K))
\mkern-2mu=\mkern-2mu[N,\infty] 
\mkern-2mu=\mkern-2mu [N_{1q},\infty]
\mkern-2mu=\mkern-2mu\mathcal{L}_{S^3}(Y^{(1,q)})$
and the spaces of resulting L-space surgeries 
$S^3_{\mathcal{L}}(K) \mkern-1mu=\mkern-1mu S^3_{\mathcal{L}}(K^{(1,q)})$
for $K$ and $K^{(1,q)}$
are identical, since the $p\mkern-2mu=\mkern-2mu1$ cable affects
framing without changing the knot. For $n\mkern-2mu>\mkern-2mu 1$, however,
the relationship between $S^3_{\mathcal{L}}(K)$ and $S^3_{\mathcal{L}}(K^{(n,nq)})$
depends on the difference $(2g(K)\mkern-1mu-\mkern-1mu 1)\mkern-.5mu-\mkern-.5mu \frac{q}{p}$:

%$N \mkern-3mu>\mkern-2.5mu \frac{q}{p} \Rightarrow
%S^3_{\mathcal{L}}(K^{(np,nq)}) 
%\mkern-2mu=\mkern-2mu 
%S^3_{\mathcal{L}}(K)\mkern1mu,

\begin{center}
$(=)$\;\;
$S^3_{\mathcal{L}}(K^{(n,nq)}) 
= S^3_{\mathcal{L}}(K)\mkern1mu$ when 
$\mkern1mu2g(K)\mkern-1mu-\mkern-1mu 1 \mkern-1mu>\mkern-.5mu q$, 
%$N \mkern-2mu>\mkern-1mu \frac{q}{p}$, 
\\
$(\subsetneq)$\;\;
$S^3_{\mathcal{L}}(K^{(n,nq)}) 
\subsetneq S^3_{\mathcal{L}}(K)\mkern1mu$ when 
$\mkern1mu2g(K)\mkern-1mu-\mkern-1mu 1 \mkern-1mu=\mkern-.5mu q$, 
%$N \mkern-2mu=\mkern-1mu \frac{q}{p}$, 
\\
$(\supsetneq)$\;\;
$S^3_{\mathcal{L}}(K^{(n,nq)}) 
\supsetneq S^3_{\mathcal{L}}(K)\mkern1mu$ when 
$\mkern1mu2g(K)\mkern-1mu-\mkern-1mu 1 \mkern-1mu<\mkern-.5mu q$. 
%$N \mkern-2mu<\mkern-1mu \frac{q}{p}$.
\end{center}

(d)
For $n\mkern-2mu=\mkern-2mu 1$, 
$\mathcal{L}
(Y^{\mkern-.1mu\text{\raisebox{-.6pt}{(}}p,q\text{\raisebox{-.6pt}{)}}}\mkern-1.3mu)^{\lowR}$
is contractible and of dimension 0 or 1.
For $n\mkern-2mu>\mkern-2mu1$, however, Theorem \ref{thm: intro topological theorem}
catalogs 6 distinct topologies that occur for 
$\mathcal{L}
(Y^{\mkern-.1mu\text{\raisebox{-.6pt}{(}}p,q\text{\raisebox{-.6pt}{)}}}\mkern-1.3mu)^{\lowR}$,
including

$\mkern38mu$1. an infinite disjoint union of points
--- $(i.a)$, $p\neq 1$;

$\mkern38mu$2. an infinite disjoint union of contractible 1-dimensional spaces 
--- $(i.a)$, $p= 1$;

$\mkern38mu$3. a connected 1-dimensional space with $\,\text{rank}\,H_1 \mkern-2.5mu=\mkern-2.5mu 
\binom{n}{2}\mkern-2.5mu-\mkern-2.5mu 1$
--- $(i.b)$;
%$\pi_1 = \ker(\delta)$
%an infinite subgroup of a free group 

$\mkern38mu$4. a contractible, 1-dimensional space with $S^1\mkern-2mu$ closure 
%in $(\R\cup\mkern-2mu\{\infty\})^n$
--- $(i.c)$, $2g(K)\mkern-1mu-\mkern-1mu 1 
\mkern-3mu=\mkern-2.5mu \frac{q}{p}\mkern-1mu+\mkern-2mu 1$, $n \mkern-3mu=\mkern-3mu2$;

$\mkern38mu$5. a contractible $n$-dimensional space
--- $(i.c)$, $2g(K)\mkern-1mu-\mkern-1mu 1 \mkern-2mu=\mkern-1.5mu \frac{q}{p}$;

$\mkern38mu$6. an $n$-dimensional space that deformation retracts onto $\mathbb{T}^{n-1}$
--- $(ii)$.

\vspace{.2cm}

\subsection{The L-space Conjecture}
The L-space conjectures,
stated formally by Boyer-Gordon-Watson
\cite{BGW}
and 
Juh{\'a}sz 
\cite{JuhaszConj},
posit the existence of left-invariant orders on
fundamental groups and of co-oriented taut foliations, respectively, 
for all prime, compact, oriented non-L-spaces.

For $Y$ a compact oriented 3-manifold with torus boundary,
let $\mathcal{F}(Y) \subset \P(H_1(\partial Y;\Z))$ denote the space
of slopes $\alpha \in \P(H_1(\partial Y;\Z))$ for which $Y$ admits a co-oriented
taut foliation (CTF) restricting
to a product foliation of slope $\alpha$ on $\partial Y$.
Along a similar vein, we define
$\mathcal{LO}(Y) := \{\alpha\in \P(H_1(\partial Y;\Z)) | \mkern2.5mu \pi_1(Y(\alpha)) \text{ is LO}\}$, where LO stands for left-orderable.

\begin{theorem}
\label{thm: bgw and juhasz}
Take $K, K^{(np,nq)} \subset S^3$ as in 
Theorem \ref{thm: torus link satellite in s3, lambda action}, 
with $K$ nontrivial and $p\mkern-2mu>\mkern-2mu 1$.
%with exteriors $Y$ and $Y^{(np,nq)}$, for 
%with $K$ nontrivial and $Y:=S^3\mkern0mu\setminus\mkern0mu\overset{\circ}{\nu}(K)$.
%$Y:= K\subset S^3$ nontrivial and $p>1$.
%Take nontrivial $K \subset S^3$ and its $T(np,nq)$-satellite $K^{(np,nq)} \subset S^3$,

\noindent $\;\;(\textsc{LO})$
Suppose
$\mathcal{LO}(Y)
\supset \mathcal{NL}(Y)$, for
$Y:= S^3\mkern0mu\setminus\mkern0mu\overset{\circ}{\nu}(K)$.

$\;(\textsc{lo}.i)$
If $2g(K)-1 > \frac{q+1}{p}$, then 
$\mathcal{LO}(Y^{(np,nq)}) = 
\mathcal{NL}(Y^{(np,nq)}).$

$\;(\textsc{lo}.ii)$
If $2g(K)-1 < \frac{q}{p}$, then
$\mathcal{LO}(Y^{(np,nq)}) \supset$
\vspace{-.10cm}
$$\mathcal{NL}(Y^{(np,nq)})
\;\setminus\; \Lambda(Y^{(np,nq)}) \cdot 
(\left[-\infty, N_{pq}\right>^n \mkern-1mu\setminus\mkern-1mu \left[-\infty, N_{pq}-p\right>^n).$$

\noindent $\;\;(\textsc{CTF})$
Suppose $\mathcal{F}(Y) = \mathcal{NL}(Y)$.

$\;(\textsc{ctf}.i)$
If $2g(K)-1 > \frac{q+1}{p}$, then 
$\mathcal{F}(Y^{(np,nq)}) = 
\mathcal{NL}(Y^{(np,nq)}) \setminus \mathcal{R}(Y^{(np,nq)}).$

$\;(\textsc{ctf}.ii)$
If $2g(K)-1 < \frac{q}{p}$, then
$\mathcal{F}(Y^{(np,nq)}) \supset$
$$\;(\mathcal{NL}(Y^{(np,nq)}) \mkern-1mu\setminus\mkern-1mu \mathcal{R}(Y^{(np,nq)}))
\;\setminus\; \Lambda(Y^{(np,nq)}) \mkern-2mu\cdot\mkern-2mu 
(\left[-\infty, N_{pq}\right>^n \mkern-1mu\setminus\mkern-1mu \left[-\infty, N_{pq}-p\right>^n).$$

\end{theorem}
One might notice that our sharper results here lie in
case $(i)$ $2g(K)\mkern-1mu-\mkern-1mu 1 \mkern-3mu>\mkern-2.5mu \frac{q}{p}$ of
Theorem~\ref{thm: torus link satellite in s3, lambda action}.
While this case is the less interesting one from the standpoint of L-space production, 
it is the more nontrivial one from the standpoint of the L-space conjecture,
since in this case every 
non-$S^3$ surgery on $K^{(np,nq)}$
has non-trivial reduced Heegaard Floer homology.

%as it takes a knot with with infinitely many
%trivial-$\widehat{HF}_{\mathrm{red}}$ surgeries,
%and produces a link for which all non-$S^3$ surgeries have non-trivial
%reduced Heegaard Floer homology.

\vspace{-.05cm}

For the 
$2g(K)\mkern-1mu-\mkern-1mu 1$
\raisebox{.3pt}{$<$} 
$\mkern-8mu$ \raisebox{1.4pt}{$\frac{q}{p}$}
case, 
the difficulty with slopes $\alpha \in \left(\left[-\infty, N_{\highpq}\right>^{\mkern-1.2mu n} 
\mkern-1mu\setminus\mkern-.7mu \left[-\infty, N_{\highpq}-p\right>^{\mkern-1.2mu n}\right)$ 
is that the existence of a CTF on 
$Y^{\mkern-.2mu\text{\raisebox{-.5pt}{(}}np,nq\text{\raisebox{-.5pt}{)}}} \mkern-1mu(\alpha)$ 
depends on the family of suspension foliations on 
$\partial Y$---necessarily of nontrivial holonomy---that arise from 
taking a CTF $F$ of slope $2g(K)\mkern-2mu-\mkern-2mu1$ on $Y$ and restricting
$F$ to $\partial Y$.  Such $F$ can only be extended over the union $Y^{(np,nq)}$
if it matches with the boundary restriction of some CTF
of $\textsc{sf}$-slope $\frac{p^*-Nq^*}{q-Np}$ on the Seifert fibered space
glued to $Y$ to form the satellite.
A similar phenomenon occurs for LOs on the fundamental group of
$Y^{(np,nq)}(\alpha)$.
See Boyer and Clay
\cite{BoyerClay} for more on this subtlety in gluing behavior.

In Theorem~\ref{thm: bgw and juhasz for general satellites}
of Section~\ref{s: bgw conjectures}, we prove a result analogous to the one above, but for
satellites by algebraic links or iterated torus-links.
Instead of restating this theorem here, we state a

\vspace{-.05cm}

\begin{cor}
\label{cor: bgw and juhasz}
For $K \mkern-2mu\subset\mkern-2mu S^3$ a positive L-space knot
with exterior $Y$,
suppose $K^{\Gamma} \mkern-3mu\subset\mkern-2mu S^3$ is
an algebraic link satellite or iterated torus-link satellite of
$K \mkern-2mu\subset\mkern-2mu S^3$,
with $2g(K)\mkern-1mu-\mkern-1mu 1 \neq \frac{q_r+1}{p_r}$
at the

\vspace{-.08cm}

\noindent root torus-link-satellite operation of $\Gamma$,
such that
$K^{\Gamma} \mkern-3mu\subset\mkern-2mu S^3$ has no L-space surgeries besides~$S^3$.

\vspace{.03cm}

\noindent $\;\;(\textsc{lo})$\;If $\mathcal{LO}(Y) \mkern-2mu=\mkern-2mu \mathcal{NL}(Y)$,
then every non-$S^3$ surgery on $K^{\Gamma} \mkern-2mu\subset\mkern-2mu S^3$
has LO fundamental group.

\noindent $\;\;(\textsc{ctf})$\;If $\mathcal{F}(Y) \mkern-2mu=\mkern-2mu \mathcal{NL}(Y)$,
then every irreducible non-$S^3$ surgery on $K^{\Gamma} \mkern-2mu\subset\mkern-2mu S^3$
admits a CTF.

\end{cor}

\vspace{-.05cm}

{\noindent{In \cite{lslope}, J.~Rasmussen and the author conjectured that
our L-space gluing theorem (see Theorem~\ref{thm: gluing theorem for floer simple or graph manifold} below)
also holds without the hypothesis of admitting more than one L-space Dehn filling.
Hanselman, Rasmussen, and Watson recently announced a proof of this conjecture in \cite{HRW},
implying that
%Subject to this conjecture,
the above corollary also holds for any non-L-space knot $K \subset S^3$.
%$2g(K)\mkern-1mu-\mkern-1mu 1 \neq \frac{q_r+1}{p_r}$.
}}

\vspace{.2cm}

\subsection{Satellites by algebraic links}
\label{ss: intro sats by alg links}

In the context of negative definite graph manifolds,
the distinction between L-space and non-L-space has consequences
for algebraic geometry.

N{\'e}methi recently showed that the
unique negative-definite graph manifold $\text{Link}(X,\circ)$
bounding the germ of a normal complex surface singularity $(X,\circ)$
is an L-space~if and only if~$(X,\circ)$~is rational~$\!$\cite{NemethiLO}. Due to results
of the author in \cite{lgraph}, we can promote this statement to a relative version:~the
subregion
$\mathcal{L}^{\textsc{nd}} \mkern-2mu\subset\mkern-1.5mu \mathcal{L}(Y^{\Gamma})$
of negative-definite L-space Dehn filling slopes for a graph manifold $\mkern1muY^{\Gamma}\mkern-2mu$
parameterizes,
up to equisingular deformation,
the rational surface singularities $(X,\circ)$ admitting 
``end curves''
$(C, \circ) \mkern-2.5mu\subset\mkern-2.5mu (X, \circ)$
%in the sense of \cite{NeumannEnd}, 
(see \cite{NeumannEnd}), 
such that $\text{Link}(X \mkern-1.5mu\setminus\mkern-1.5mu C) \mkern-2mu=\mkern-2mu Y^{\Gamma}\!$.
If~one~such $(X,\circ)$
is $(\C^2, 0)$, then $Y^{\Gamma}$ is the
exterior of an algebraic link, motivating the following study.

\vspace{.23cm}

\noindent{\textbf{Setup.}} Whereas a $T(np,nq)$-satellite operation is specified by
(an unknot complement in) 
the Seifert fibered exterior of $T(np,nq)$ determined by the triple $(p,q,n)$,
a {\em{sequence}} of torus-link-satellite operations
is specified by a rooted tree $\Gamma$
determining the graph manifold exterior of the pattern link,
where each vertex $v \mkern-2.5mu\in\mkern-2.5mu \mathrm{Vert}(\Gamma)$ specifies 
the Seifert fibered $T_v\mkern-3mu:=\mkern-3mu T(n_v p_v, n_v q_v)$-exterior in $S^3_v$,
determined by the triple
$(p_v, q_v, n_v)$, so that $S^3_v$ has 2 exceptional fibers $\lambda_{-1}^v$ and
$\lambda_0^v$ of respective multiplicities $p_v$ and $q_v$, and components of
$T_v$ are regular fibers in $S^3_v$.

Since we direct the edges of $\Gamma$ {\em{rootward}}, each vertex $w$
has a unique outgoing edge $e_w$, corresponding to the 
incompressible torus in whose neighborhood
$T_w$ is embedded, or equivalently, to a gluing map $\phi_{e_w}\mkern-2mu$ splicing
the multiplicity-$q_w$ fiber $\lambda_0^w \in S^3_w$ to a Seifert fiber in 
$S^3_u$, for $u:=v(e_w)$, where we write $v(e)$ 
to denote the vertex on which an edge $e\in \mathrm{Edge}(\Gamma)$ terminates.
(A ``splice'' is a type of toroidal connected sum 
exchanging meridians with longitudes.)

%sending meridians to longitudes in order to maximize preservation
%of homology. See Section \ref{s: algebraic link satellites} 
%for a definition and sign conventions.)

There are only two types of fiber in $S^3_u$ available for splicing:
a regular fiber, which we then regard as one of the $n_v$ components 
$f_j^u \in S^3_u$ of $T_v$, or the multiplicity-$p_v$ fiber $\lambda_{-1}^u \in S^3_u$.
When $\phi_{e_w}$ splices $\lambda_0^w \in S^3_w$ to some
$j^{\mathrm{th}}$ component $f^u_j \in S^3_u$ of $T_u$, we call
$\phi_{e_w}$ a {\em{smooth splice}}, set $j(e_w) := j \in\{1,\ldots, n_v\}$,
and declare the JSJ component $Y_u$ at $u$ to be the 
exterior of $\lambda_0^u \amalg T_u$ in $S^3_u$.
When $\phi_{e_w}$ splices $\lambda_0^w \in S^3_w$ to $\lambda_{-1}^u \in S^3_u$,
we call $\phi_{e_w}$ an {\em{exceptional splice}}, set $j(e_w) = -1$, and
define $Y_u$ to be the exterior of
$\lambda_{-1}^u \amalg \lambda_0^u \amalg T_u$ in $S^3_u$.
Since this latter splice could be redefined as a smooth one if $p_u=1$, we
demand $p_u > 1$ without loss of generality.

If we define $J_v \subset j(E_{\mathrm{in}}(v))$ and its complement $I_v$ by
\vspace{-.07cm}
\begin{equation}
J_v := 
j(E_{\mathrm{in}}(v))
%\{j(e)|e\in E_{\mathrm{in}}(v)\} 
\cap \{1, \ldots, n_v\},
\;\;\;\;\;
I_v := \{1, \ldots, n_v\} \setminus J_v,
\end{equation}

\vspace{-.07cm}

\noindent then $I_v$
catalogs the boundary components of $Y_v$ left unfilled, forming
the exteriors of link components,
so that the total satellite $K^{\Gamma} \mkern-4mu\subset\mkern-2mu S^3$
of $K \mkern-2mu\subset\mkern-2mu S^3$
has 
%$|\Gamma| \mkern-3mu:=
$\mkern-2mu \sum_{v \in \text{Vert}(\Gamma)}\mkern-2mu |I_v|$ components.
The pattern link specified by $\Gamma$
is then an {\em{algebraic}} link if and only if 
its graph manifold exterior is negative definite,
which, by straight-forward calculations as appear, for example, 
in Eisenbud and Neumann's book \cite{EisenbudNeumann},
is equivalent to the condition that $\Gamma$ is a tree, and that
\vspace{-.08cm}
\begin{align*}
(i)&\;\; p_v, q_v, n_v > 0 \;\;\text{for all } v\in \text{Vert}(\Gamma),
    \\
(ii)&\;\;\mkern46mu \Delta_e > 0 \;\;\text{for all } e\in \text{Edge}(\Gamma),
    \\
\Delta_e :=\mkern20mu&\mkern-20mu
\begin{cases}
\;p_{v(e)}q_{v(-e)} - p_{v(-e)}q_{v(e)} 
 & \;\;j(e) = -1
   \\
\;q_{v(-e)} - p_{v(e)} p_{v(-e)}q_{v(e)} 
 & \;\;j(e) \neq -1
\end{cases}.
\end{align*}

\vspace{-.08cm}

\noindent Conversely, given an isolated planar complex curve singularity 
$(\circ, C) \hookrightarrow (0, \C^2)$,
one can obtain such a tree $\Gamma$ from
Newton-Puisseux expansions for the defining equations of $C$,
or alternatively from the amputated
splice diagram of the dual plumbing graph of $\widetilde{X}$ for a good embedded resolution 
$(\widetilde{X}, \widetilde{C}) \to (\C^2, C)$.
Again, see  \cite{EisenbudNeumann} for details.

If $Y^{\Gamma}$
%:= S^3 \setminus \overset{\mkern2mu\circ}{\nu}(K^{\Gamma})$
denotes the exterior of the $\Gamma$-satellite $K^{\Gamma} \subset S^3$
of $K \subset S^3$, then
for each JSJ component $Y_v$ of $Y^{\Gamma}$,
we again have reducible and exceptional subsets
$\mathcal{R}_v, \mathcal{Z}_v \subset \prod_{i\in I_v}\P(H_1(\partial_i Y_v; \Z))$,
along with a lattice $\Lambda_v$ acting on 
$\prod_{i\in I_v}\P(H_1(\partial_i Y_v; \Z))$
by addition of $\textsc{sf}_v$-slopes.

\begin{theorem}
\label{thm: algebraic link satellites}
Suppose $Y^{\Gamma} := S^3 \setminus \overset{\mkern2mu \circ}{\nu}(K^{\Gamma})$
is the exterior of an algebraic-link satellite $K^{\Gamma} \subset S^3$ of
a (possibly trivial) positive L-space knot $K \subset S^3$,
and suppose the triple $(p_r, q_r, n_r)$ specifies the initial torus-link satellite operation,
occuring at the root vertex $r \in \mathrm{Vert}(\Gamma)$.

$(i.a)$ If $K$ is nontrivial, $\frac{q_r}{p_r} < 2g(K) -1$, $p_r > 1$,
and $-1 \notin j(E_{\mathrm{in}}(r))$, then
$$
\mathcal{L}(Y^{\Gamma}) = \Lambda_{\Gamma};
\;\;\;
Y^{\Gamma}(\boldsymbol{y}^{\Gamma}) = S^3
\;\text{ for all }\; \boldsymbol{y}^{\Gamma} \in \mathcal{L}(Y^{\Gamma}).
$$

$(i.b)$ If $K$ is nontrivial, $\frac{q_r}{p_r} < 2g(K) -1 =: N$,
and $p_r = 1$, then

\vspace{-.4cm}

$$
\mathcal{L}_{S^3}(Y^{\Gamma})
=\mkern-3mu
\left(\mkern-3mu
\Lambda_r \mkern-3.7mu\cdot\mkern-2.5mu
\mathfrak{S}_{|I_r|}( [N,+\infty] \mkern-4mu\times\mkern-4mu \{\infty\}^{\mkern-1mu|I_r|-1})
\times
\mkern-15mu
{{\prod\limits_{e\in E_{\mathrm{in}}(r)}}} \mkern-12mu\Lambda_{\Gamma_{v(-e)}}
\mkern-5mu
\right)
%   \\
\mkern1mu
\amalg
\mkern-8mu
\coprod_{e \in E_{\mathrm{in}}(r)}\mkern-20mu
\left(\mkern-1mu
\mathcal{L}_{\mkern-1muS^3}\mkern-1mu(Y^{\Gamma_{v(-e)}}) \times
\Lambda_{\Gamma  \mkern2mu\setminus\mkern2mu \Gamma_{\mkern-1muv(-e)}}
\mkern-1mu
\right)\mkern-2mu,
$$
\vspace{-.2cm}

{{where $Y^{\Gamma_{v(-e)}}$ denotes the exterior of the $\Gamma_{v(-e)}$-satellite
of $K \subset S^3$.}}

\vspace{.1cm}
$(ii)$ If $K$ is trivial, or if 
$K$ is nontrivial with
$\frac{q_r}{p_r} \ge \mkern2mu 2g(K) -1$ and  $q_r > \mkern2mu 2g(K) -1$,
then
$$
\;\;\;\;\;
\mathcal{L}_{\textsc{sf}_{\Gamma}}(Y^\Gamma) \supset
\prod_{v \in \mathrm{Vert}(\Gamma)}
\left(
\mathcal{L}^{\min -}_{\textsc{sf}_v}
\;\cup\;
\mathcal{R}_v \setminus \mathcal{Z}_v
\;\cup\;
\mathcal{L}^{\min +}_{\textsc{sf}_v}
\right),
\;\;\;
\text{where}$$
\vspace{-.4cm}
\begin{align*}
\mathcal{L}^{\min +}_{\textsc{sf}_v}
&:=
\left\{\vphantom{\textstyle{\sum\limits_{i\in I_v}}}
\boldsymbol{y}^v \mkern-4mu\in\mkern-2.5mu \Q^{|I_v|}
\right. 
\mkern-4mu
\left|\vphantom{\frac{A}{B}}\right.
\mkern-3mu
\left.
\textstyle{\sum\limits_{i\in I_v}} \lfloor y_i^v \rfloor \ge 0 \right\},
    \\
\mathcal{L}^{\min -}_{\textsc{sf}_v}
&:=
\left\{\vphantom{\textstyle{\sum\limits_{i\in I_v}}}
\boldsymbol{y}^v \mkern-4mu\in\mkern-2.5mu \Q^{|I_v|}
\right. 
\mkern-4mu
\left|\vphantom{\frac{A}{B}}\right.
\mkern-3mu
\left.
\textstyle{\sum\limits_{i\in I_v}} \lceil y_i^v \rceil \le m^-_v \right\}
\setminus 
\begin{cases}
\left\{
\mkern-4mu
\begin{array}{c}
\sum \lceil y_i^v \rceil = 0,
\\
\sum \mkern-2mu\left\lfloor [-y_i^v](q_v\mkern-2.7mu-\mkern-2.4mu p_v) \right\rfloor = 0
\end{array}
\mkern-4mu
\right\}
&
\mkern-6.5mu\begin{array}{l}
J_v = \emptyset; 
  \\
j(e_v) \neq -1
\end{array}
  \\
\vphantom{\mfrac{A^A}{B^B}^B}
\emptyset
& 
\text{otherwise}\,,
\end{cases}
\end{align*}
\vspace{-.3cm}
\begin{align*}
m_v^-
\mkern1mu
&:=
\mkern4mu-\mkern-6mu
{\sum\limits_{e\mkern1mu\in\mkern1mu 
E_{\mathrm{in}}\mkern-1mu(v)}}
\mkern-8mu
\left\{
\mkern-7mu
\begin{array}{ll}
1
& 
\mkern-3mu
j(e) \mkern-3mu\neq\mkern-3mu -1
  \\
\left\lceil \mkern-2mu\frac{p_{v(-e)}}{p_v\Delta_e}\mkern-3mu \right\rceil 
\mkern-3mu+\mkern-2mu 1
&
\mkern-3mu
j(e) \mkern-3mu=\mkern-3mu -1
\end{array}
\mkern-7mu
\right\}
\;-\;
\begin{cases}
1
& J_v \neq \emptyset;\, j(e_v) \neq -1
    \\
\left\lceil \mkern-2mu\frac{p_{v(e_v)}}{p_v\Delta_{e_v}}\mkern-3mu \right\rceil 
\mkern-3mu+\mkern-2mu 1
&
j(e_v) \mkern-3mu=\mkern-3mu -1
    \\
0
& \text{otherwise}\,.
\end{cases}\mkern2mu
\end{align*}
\end{theorem}

This is not as horrible as it looks.
Part $(i.a)$ describes the case in which all L-space surgeries yield $S^3$.
For part $(i.b)$, either 
we trivially refill all
components of $Y^{\Gamma}$ except for one exterior component
in the root, effectively replacing $K^{\Gamma}$ with $K$;
or, we trivially refill all exterior components but
those of $Y_{\Gamma_{v(-e)}}$ for some incoming edge $e$, replacing $K^{\Gamma}$
with some $K^{\Gamma_{v(-e)}}$.

The notation $\Lambda_{\Gamma_v}$
and the term ``trivially refill'' hide a subtlety, however.
For both the above theorem and for
Theorem~\ref{thm: iterated torus links satellites} for iterated torus-link satellites,
we {\em{define}}
$\Lambda_{\Gamma_v}$ to mean
\begin{equation}
\label{eq: lambda_gamma def}
\Lambda_{\Gamma_v} = 
\left\{ \boldsymbol{y}^{\Gamma_v} \mkern-2mu\left|\mkern4mu
\vphantom{Y_0^{\Gamma_v}(\boldsymbol{y}^{\Gamma_v})}\right.
Y_0^{\Gamma_v}(\boldsymbol{y}^{\Gamma_v}) = S^3 \right\},
\;\;\;
Y_0^{\Gamma_v}
\mkern-4mu:=\mkern-4mu
\;\text{exterior of the}\;
\Gamma_v\text{-satellite of the unknot}.
\end{equation}
While 
$\Lambda_{\Gamma_v} \supset \prod_{u \in \text{Vert}(\Gamma_v)}\Lambda_u$,
the two sets need not be equal.
Similarly, if $N > \frac{q_r}{p_r}$ and
$j(e') = -1$ for some $e' \in E_{\mathrm{in}}(r)$,
then 
$\mathcal{L}(Y^{\Gamma}) \supset
\mathcal{L}\mkern-1mu(Y^{\Gamma_{v(-e')}}) \times
\Lambda_{\Gamma  \mkern2mu\setminus\mkern2mu \Gamma_{\mkern-1muv(-e')}}$,
but the containment can be proper.
Section~\ref{ss: exceptional symmetries}
provides a more explicit characterization of
$\mathcal{L}(Y^{\Gamma})$ in this case, along with
a concrete description of $\Lambda_{\Gamma}$
in the case of iterated torus-link satellites.

Part $(ii)$
of Theorem~\ref{thm: algebraic link satellites}
is analogous to
Theorem~\ref{thm: torus link satellite in s3, lambda action}.ii
for torus-link satellites,
but this similarity is masked by our transition from
$S^3$-slopes to $\textsc{sf}$-slopes.
For example, if $v$ is a leaf and its emanating edge $e_v$
does not correspond to an exceptional splice, then we have
\begin{equation*}
\mathcal{L}^{\min+}_{S^3_v} \mkern-2mu=
\Lambda_v \mkern-1mu\cdot\mkern-.5mu \left<p_v q_v , +\infty\right]^{|I_v|},
\;\;\;
\mathcal{L}^{\min-}_{S^3_v} =
\Lambda_v \cdot 
\left(\left[ -\infty , p_v q_v\right>^{|I_v|} 
\setminus 
\left[ -\infty , 
p_v q_v \mkern-2mu-\mkern-2mu q_v \mkern-2mu+\mkern-2mu p_v
\right>^{|I_v|}\right).
\end{equation*}

\vspace{-.2cm}

{\noindent{In all other cases, we still have
$\mathcal{L}^{\min+}_{S^3_v} \mkern-2mu=
\Lambda_v \mkern-1mu\cdot\mkern-1mu \left<p_v q_v , +\infty\right]^{|I_v|}$, but
$\mathcal{L}^{\min-}_{S^3_v}$ now sits inside the
unit-radius neighborhood of $\mathcal{R}_v$,
with $m_v^-$ providing a measure of how
deeply inside it sits.}}

In the set 
$\psi_v^{-1}\mkern-2mu\left(\left[ -\infty , 
p_v q_v \mkern-2mu-\mkern-2mu q_v \mkern-2mu+\mkern-2mu p_v
\right>^{|I_v|}\right)$
removed from
$\mathcal{L}^{\min -}_{\textsc{sf}_v}$
at a smoothly-spliced leaf $v$
(in part ($ii$) of the theorem),
the notation 
$[\cdot]: \Q \to \Z$, $[x] := x - \lfloor x \rfloor$
gives the fractional part of a rational number $x$,
whereas the notation 
$[a]_b := a - \left\lfloor\mkern-2mu\frac{a}{|b\mkern.2mu|}\mkern-2mu\right\rfloor \mkern-2mu |b|$
picks out the smallest nonnegative representative of  $a\; \mod b$
for any integers $a,b \in \Z$. For a suitable L-space region approximation
when $q_r = 2(K) -1 $ and $p_r = 1$,
see line (\ref{eq: special case of p=1 and q=N}) and the associated remark.

\subsection{Iterated torus-link-satellites}
For the case of iterated torus-link satellites,
we only allow ``smooth splice'' edges, corresponding
to the original type of torus-link satellite operation.
We also drop the algebraicity condition that $\Delta_e >0$,
and while we keep all $p_v, n_v >0$ without loss of generality,
we allow $q_v <0$ but demand $q_v \neq 0$, for each $v \in \text{Vert}(\Gamma)$.

\begin{theorem}
\label{thm: iterated torus links satellites}
Suppose $Y^{\Gamma} := S^3 \setminus \overset{\mkern2mu \circ}{\nu}(K^{\Gamma})$
is the exterior of an iterated torus-link satellite $K^{\Gamma} \subset S^3$ of
a (possibly trivial) positive L-space knot $K \subset S^3$,
and suppose the triple $(p_r, q_r, n_r)$ specifies the initial torus-link satellite operation
occuring at the root vertex $r \in \mathrm{Vert}(\Gamma)$.

$(i.a)$ If $K$ is nontrivial, $\frac{q_r}{p_r} < 2g(K) -1$, $p_r > 1$, then
$\mathcal{L}(Y^{\Gamma}) = \Lambda_{\Gamma}$.

$(i.b)$ If $K$ is nontrivial, $\frac{q_r}{p_r} < 2g(K) -1 =: N$,
and $p_r = 1$, then
\vspace{-.4cm}
$$
\mathcal{L}_{S^3}(Y^{\Gamma})
=\mkern-3mu
\left(\mkern-3mu
\Lambda_r \mkern-3.7mu\cdot\mkern-2.5mu
\mathfrak{S}_{|I_r|}( [N,+\infty] \mkern-4mu\times\mkern-4mu \{\infty\}^{\mkern-1mu|I_r|-1})
\times
\mkern-15mu
{{\prod\limits_{e\in E_{\mathrm{in}}(r)}}} \mkern-12mu\Lambda_{\Gamma_{v(-e)}}
\mkern-5mu
\right)
%   \\
\mkern1mu
\amalg
\mkern-8mu
\coprod_{e \in E_{\mathrm{in}}(r)}\mkern-20mu
\left(\mkern-1mu
\mathcal{L}_{\mkern-1muS^3}\mkern-1mu(Y^{\Gamma_{v(-e)}}) \times
\Lambda_{\Gamma  \mkern2mu\setminus\mkern2mu \Gamma_{\mkern-1muv(-e)}}
\mkern-1mu
\right)\mkern-2mu,
$$

\vspace{-.08cm}

where $Y^{\Gamma_{v(-e)}}$ denotes the exterior of the $\Gamma_{v(-e)}$-satellite
of $K \subset S^3$.

\vspace{.15cm}

$(ii)$ If $K$ is trivial, or if 
$K$ is nontrivial with
$\frac{q_r}{p_r} \ge \mkern2mu 2g(K) -1$ and  $q_r > \mkern2mu 2g(K) -1$,
then
$$
\;\;\;\;\;
\mathcal{L}_{\textsc{sf}_{\Gamma}}(Y^\Gamma) \supset
\prod_{v \in \mathrm{Vert}(\Gamma)}
\left(
\mathcal{L}^{\min -}_{\textsc{sf}_v}
\;\cup\;
\mathcal{R}_v \setminus \mathcal{Z}_v
\;\cup\;
\mathcal{L}^{\min +}_{\textsc{sf}_v}
\right),
\;\;\;
\text{where}
$$

\vspace{-.4cm}

\begin{align*}
\mathcal{L}^{\min +}_{\textsc{sf}_v}
&:=
\left\{\vphantom{\textstyle{\sum\limits_{i\in I_v}}}
\boldsymbol{y}^v \mkern-4mu\in\mkern-2.5mu \Q^{|I_v|}
\right. 
\mkern-4mu
\left|\vphantom{\frac{A}{B}}\right.
\mkern-3mu
\left.
\textstyle{\sum\limits_{i\in I_v}} \lfloor y_i^v \rfloor \ge m^+_v \right\}
\setminus 
\begin{cases}
\left\{ \sum \lfloor y_i^v\rfloor = \sum \lfloor [y_i^v][p_v]_{q_v} \rfloor = 0 \right\}
&
J_v = \emptyset; 
\frac{p_v}{q_v}>+1
  \\
\left\{ \sum \lfloor y_i^v\rfloor = \sum \lfloor [y_i^v][-p_v]_{q_v} \rfloor = 0 \right\}
&
J_v = \emptyset; 
q_v <\mkern-2mu -1
  \\
\emptyset
&
\text{otherwise}
\end{cases},
    \\
\mathcal{L}^{\min -}_{\textsc{sf}_v}
&:=
\left\{\vphantom{\textstyle{\sum\limits_{i\in I_v}}}
\boldsymbol{y}^v \mkern-4mu\in\mkern-2.5mu \Q^{|I_v|}
\right. 
\mkern-4mu
\left|\vphantom{\frac{A}{B}}\right.
\mkern-3mu
\left.
\textstyle{\sum\limits_{i\in I_v}} \lceil y_i^v \rceil \le m^-_v \right\}
\setminus 
\begin{cases}
\left\{ \sum \lceil y_i^v \rceil = \sum \lfloor [-y_i^v][p_v]_{q_v} \rfloor = 0 \right\}
&
J_v = \emptyset; 
\frac{p_v}{q_v}
\mkern-1mu<\mkern-2mu -1
  \\
\left\{ \sum \lceil y_i^v \rceil = \sum \lfloor [-y_i^v][-p_v]_{q_v} \rfloor = 0 \right\}
&
J_v = \emptyset; 
q_v > +1
  \\
\emptyset
& 
\text{otherwise}
\end{cases},
\end{align*}
\begin{align*}
\text{with}
\;\;\;
m_v^+
&:=
\mkern4mu-\mkern-6mu
{\sum\limits_{e\mkern1mu\in\mkern1mu 
E_{\mathrm{in}}\mkern-1mu(v)}}
\mkern-12mu
\left(\left\lceil
\mfrac{p_{v(-e)}}{q_{v(-e)}}
\right\rceil - 1\right)
\;+\;
\begin{cases}
2 
&
q_v = -1
    \\
1
& J_v \neq \emptyset;\, q_v < -1 \text{ or } 
\frac{p_v}{q_v} > +1
    \\
0
& \text{otherwise}
\end{cases},
     \\
m_v^-
&:=
\mkern4mu-\mkern-6mu
{\sum\limits_{e\mkern1mu\in\mkern1mu 
E_{\mathrm{in}}\mkern-1mu(v)}}
\mkern-12mu
\left(\left\lfloor
\mfrac{p_{v(-e)}}{q_{v(-e)}}
\right\rfloor + 1\right)
\;-\;
\begin{cases}
2 
&
q_v = +1
    \\
1
& J_v \neq \emptyset;\, q_v > +1 \text{ or } 
\frac{p_v}{q_v} < -1
    \\
0
& \text{otherwise}
\end{cases}.
\hphantom{\text{with}
\;\;\;}
\end{align*}
\end{theorem}

\noindent {\textbf{Monotonicity.}}
In both of the above theorems,
%the product 
$\prod_{v \in \mathrm{Vert}(\Gamma)}
\left(
\mathcal{L}^{\min -}_{\textsc{sf}_v}
\;\cup\;
\mathcal{R}_v \setminus \mathcal{Z}_v
\;\cup\;
\mathcal{L}^{\min +}_{\textsc{sf}_v}
\right)$
is a component of
what we call the {\em{monotone stratum}}
$\,\mathcal{L}_{\textsc{sf}_{\Gamma}}^{\text{mono}}(Y^\Gamma)$
of $\mathcal{L}_{\textsc{sf}_{\Gamma}}(Y^\Gamma)$,
as discussed in
Section~\ref{ss: monotone strata}.
We say that
$\boldsymbol{y}^{\Gamma} \in \mathcal{L}_{\textsc{sf}_{\Gamma}}(Y^\Gamma)$
is {\em{monotone at $v \in \mathrm{Vert(\Gamma)}$}} if
\begin{equation}
\label{eq: monotonicity at v intro}
\infty \in 
\phi_{\mkern-1mue*\mkern1mu}^{\P}
\mathcal{L}_{\textsc{sf}_{v(-e)}\mkern-4mu}^{\circ}
(Y_{\Gamma_{v(-e)}}\mkern-1.5mu(\boldsymbol{y}^{\Gamma}|_{\Gamma_{v(-e)}}))
\;\forall\,e \in E_{\mathrm{in}}(v)
\;\;\;\text{and}\;\;\;
\infty\in
\phi_{\mkern-1mu{e_v}*\mkern1mu}^{\P}
\mathcal{L}_{\textsc{sf}_v\mkern-4mu}^{\circ}
(Y_{\Gamma_v}\mkern-1.5mu(\boldsymbol{y}^{\Gamma}|_{\Gamma_v})),
\end{equation}
or more prosaically (when the above interval interiors are nonempty) is monotone at $v$ if
\begin{equation*}
y^v_{j(e)+} \le y^v_{j(e)-} \;\forall\;e \in E_{\mathrm{in}}(v),
\;\;\;
y^{v(e_v)}_{j(e_v)+} \le y^{v(e_v)}_{j(e_v)-},
\end{equation*}
as these are the respective endpoints of the above intervals.
The {\em{monotone stratum}}
$\mathcal{L}^{\mathrm{mono}}_{\textsc{sf}_{\Gamma}}(Y^\Gamma)$
of $\mathcal{L}_{\textsc{sf}_{\Gamma}}(Y^\Gamma)$ is the set of slopes 
$\boldsymbol{y}^{\Gamma} \in \mathcal{L}_{\textsc{sf}_{\Gamma}}(Y^\Gamma)$
such that $\boldsymbol{y}^{\Gamma}$ is monotone at all $v \in \mathrm{Vert}(\Gamma)$.

Specifying different collections of local monotonicity conditions allows one to
decompose an L-space region into strata of disparate topologies.
For example, for the (globally) monotone stratum, we have the following topological result,
proved in Section~\ref{ss: monotone strata}.

\begin{theorem}
\label{thm: intro topology of monotone stratum}
Suppose that $K^{\Gamma} \subset S^3$ is an algebraic link satellite,
specified by $\Gamma$, of a positive L-space knot $K\subset S^3$,
where either $K$ is trivial, or
$K$ is nontrivial with
$\frac{q_r}{p_r} > \mkern2mu 2g(K) -1$.
Let $V \subset \mathrm{Vert}(\Gamma)$ denote the subset of vertices $v \in V$
for which $|I_v| > 0$.

Then the $\Q$-corrected $\R$-closure
$\mathcal{L}^{\mathrm{mono}}_{\textsc{sf}_{\Gamma}}(Y^\Gamma)^{\R}$
of the monotone stratum of 
$\mathcal{L}_{\textsc{sf}_{\Gamma}}(Y^\Gamma)$
is of dimension $|I_{\Gamma}|$ and 
%in $(\R \cup \{\infty\})^{|I_{\Gamma}|}_{\textsc{sf}_{\Gamma}}$,
deformation retracts onto an $(|I_{\Gamma}| - |V|)$-dimensional embedded torus,
\begin{equation}
\label{eq: torus}
\mathcal{L}^{\mathrm{mono}}_{\textsc{sf}_{\Gamma}}(Y^\Gamma)^{\R}
\;\;\leadsto\;\;
%_{\text{def. retracts to}}\;
{\textstyle{\prod_{v\in V}}} \mathbb{T}^{|I_v|-1}
\;\into\; 
{\textstyle{\prod_{v\in V}}} 
(\R\cup\{\infty\})^{|I_v|}_{\textsc{sf}_v},
\end{equation}
projecting to embedded tori
$\mathbb{T}^{|I_v|-1}
\;\into\; 
(\R\cup\{\infty\})^{|I_v|}_{\textsc{sf}_v}$
parallel to 
$\mathcal{B}_{\textsc{sf}_v} \subset (\R\cup\{\infty\})^{|I_v|}_{\textsc{sf}_v}$.
\end{theorem}

Non-monotone strata, when they exist, change the topology of the
total L-space region and have implications for ``boundedness from below''
in the sense of N{\'e}methi and Gorsky \cite{GNLbounded},
but we defer the study of non-monotone regions to later work, 
whether by this author or others.

%See Section \ref{ss: monotone strata}

%********************** hamburger ******* monotone monotonicity

\vspace{.28cm}

\noindent {\textbf{New tools.}}
In fact, the propositions in
Sections~\ref{s: iterated torus-link-satellites}
and \ref{s: algebraic link satellites}
provide many tools for analyzing questions not addressed in this paper.
For example, in the absence of an exceptional splice at $v$,
Proposition~\ref{prop: main bounds for y0- and y0+}$(+)$.$iii$
precisely characterizes when 
$y_{0+}^{v(e)} \mkern-2mu=\mkern-2mu \frac{p_v^*}{q_v}$,
defining the right-hand boundary of the non-monotone stratum
at that component.
These tools can also be used to characterize the non-product components
of the monotone stratum more explicitly.

\vspace{.28cm}

\noindent {\textbf{New features.}}
Even so, Theorems 
\ref{thm: algebraic link satellites} and
\ref{thm: iterated torus links satellites}
already 
reveal more interesting behavior than appears for
nondegenerate torus-link satellites.  In particular, 
{\em{the boundary of the $S^3$-slope L-space region need not occur at infinity}}.
For example, if
\begin{itemize}
\item[(a)]
$\Gamma = v$ specifies a single torus-link satellite of the unknot, and $p_v =q_v = 1$,
\item[(b)]
$\Gamma$ specifies an iterated torus-link satellite, and $m^+_v < 0$, or
\item[(c)]
$\Gamma$ specifies an iterated or algebraic-link satellite,
and we restrict to an appropriate piece of 
$\mathcal{L}^{\text{mono}}_{\textsc{sf}_{\Gamma}}(Y^{\Gamma})|_{\textsc{sf}_v}$
outside the product region,
\end{itemize}
then we encounter regions of the form
\begin{align*}
\psi\left(\left\{\vphantom{\textstyle{\sum\limits_{i\in I_v}}}
\boldsymbol{y}^v \mkern-4mu\in\mkern-2.5mu \Q^{|I_v|}
\right.\right. 
\mkern-4mu
\mkern-10mu&\mkern10mu
\left|\vphantom{\frac{A}{B}}\right.
\mkern-3mu
\left.\left.
\textstyle{\sum\limits_{i\in I_v}} \lfloor y_i^v \rfloor \ge m^+_v \right\}
\right)
   \\
\;\,=\,\;& \Lambda_v \mkern-2mu\cdot\mkern-1.5mu 
\left<p_v q_v , +\infty\right]^{|I_v|}
\;\,\cup\;\,
\Lambda_v \mkern-2.5mu\cdot\mkern-1mu \mathfrak{S}_{|I_v|}\mkern-5mu
\left(
\left<-\infty, p_vq_v - 1\right]^{|m_v^+|} \times \left<p_v q_v , +\infty\right]^{|I_v|-|m_v^+|}
\right)
\end{align*}
for $-|I_v| \le m_v^+ < 0$, with additional components
added onto the unit-radius neighborhood of $\mathcal{R}_v$
if $m_v < -|I_v|$.  (An analogous phenomenon occurs in the
negative direction when $m_v^- > 0$.)
Thus, in such cases, the L-space region ``wraps around'' infinity.
In fact, given any $n-, n_+ \in \Z_{\ge 0}$,
it is possible to
construct an iterated satellite by torus links
for which the $S^3_v$ component of some stratum of the L-space region
fills up the quadrant
\begin{equation*}
\left< -\infty, p_v q_v - 1 \right>^{n-} \times
\left< p_v q_v + 1, +\infty \right>^{n+},
\;\;\;n_- + n_+ = |I_v|.
\end{equation*}
There likewise exist iterated torus-link satellite exteriors $Y^{\Gamma}$
with $u, v \in \text{Vert}(\Gamma)$ for which 
the projections of $\mathcal{L}_{S^3}(Y^{\Gamma})$
to the positive quadrant
$\left< p_u q_u \mkern-2mu+\mkern-2mu 1, +\infty \right>^{|I_u|}$
and to the negative quadrant 
$\left< -\infty, p_v q_v \mkern-2mu-\mkern-2mu 1\right>^{|I_v|}$
are both empty.

We therefore feel that the notion of ``L-space link''
should be broadened to encompass any link 
whose L-space surgery region contains an open neighborhood
in the space of slopes, rather than
defining this notion in terms of large 
positive slopes in the L-space surgery region.

%**** Revise "Organization" section.

\subsection{Organization}
Section \ref{s: conventions} establishes basic Seifert fibered space conventions and 
elaborates on the distinguished slope subsets $\Lambda$, $\mathcal{R}$,
and $\mathcal{Z}$.
Section \ref{s: L-space intervals and gluing}
introduces notation for L-space intervals
and proves a new gluing theorem for knot exteriors with graph manifolds.
Section~\ref{s: torus-link satellites}
introduces machinery
developed by the author in \cite{lgraph}
to compute L-space intervals for fiber
exteriors in graph manifolds, and applies this
to prove 
Theorem 
\ref{thm: s3 satellite in sf basis},
an $\textsc{sf}$-slope version
of the torus-link satellite results in 
Theorem~\ref{thm: torus link satellite in s3, lambda action}.
Section \ref{s: L-space region topology} addresses the topology of L-space regions
and proves Theorem \ref{thm: intro topological theorem}.
Section~\ref{s: iterated torus-link-satellites}
describes the graph $\Gamma$ associated
to an iterated torus-link satellite, computes
various estimates useful for bounding L-space surgery regions
for iterated-torus-link and algebraic link satellites, and proves 
Theorem~\ref{thm: iterated torus links satellites}
for iterated-torus-link satellites.
Section~\ref{s: algebraic link satellites}
describes adaptations of this graph $\Gamma$ to accomodate
algebraic link exteriors, discusses monotonicity, and proves
Theorems~\ref{thm: algebraic link satellites} and
\ref{thm: intro topology of monotone stratum}.
Section~\ref{s: bgw conjectures}
proves results related to
conjectures of Boyer-Gordon-Watson and Juh{\'a}sz,
including generalizations of Theorem \ref{thm: bgw and juhasz}.

Readers interested in constructing their own L-space regions
for algebraic link satellites or iterated torus-link satellites
should refer to the L-space interval technology introduced for graph manifolds in
Section~\ref{s: torus-link satellites},
and to the analytical tools developed in
Section~\ref{s: iterated torus-link-satellites}.

\subsection*{Acknowledgements}
It is a pleasure to thank Eugene Gorsky, Gordana Mati{\'c}, Andr{\'a}s N{\'e}methi,
and Jacob Rasmussen
and for helpful conversations.
I am especially indebted to Maciej Borodzik,
for his hospitality at the University of Warsaw,
for extensive feedback, and for his encouraging
me to write something down about L-space surgeries on algebraic links.

\section{Basis conventions and the slope subsets 
$\Lambda$, $\mathcal{R}$, and $\mathcal{Z}$}
\label{s: conventions}

Suppose $Y := M \setminus \overset{\,\circ}{\nu}(\boldsymbol{L})$,
with boundary $\partial Y = \coprod_{i=1} \partial_i Y\mkern-4mu$,
$\mkern3mu\partial_i Y = \partial \nu(L_i)$,
is the link exterior of 
an $n$-component link $\boldsymbol{L} = \coprod_{i=1}^n L_i \subset M$
in a closed oriented 3-manifold $M$.
Then, up to choices of sign, the Dehn filling $M$ of $Y$ specifies
a ({\textit{multi}}-){\textit{meridional class}}
$(\mu_1, \ldots, \mu_n) \in H_1(\partial Y; \Z) 
= \bigoplus_{i=1}^n H_1(\partial_i Y; \Z)$,
where each meridian
$\mu_i \in H_1(\partial_i Y; \Z)$
is the class of a curve bounding a compressing disk 
of the solid torus $\nu(L_i)$.
Any choice of classes $\lambda_1, \ldots, \lambda_n \in H_1(\partial Y; \Z)$
satisfying $\mu_i \cdot \lambda_i = 1$ for each $i$ then produces
a {\textit{surgery basis}} $(\mu_1, \lambda_1, \ldots, \mu_n, \lambda_n)$
for $H_1(\partial_i Y; \Z)$.
We call these $\lambda_i$ {\textit{surgery longitudes}},
or just {\textit{longitudes}}
if the context is clear.

When $M = S^3$,
$H_1(\partial Y; \Z)$
has a conventional basis
given by taking each $\lambda_i$ to be the
{\textit{rational longitude}}; that is,
each $\lambda_i$ generates the kernel of the homomorphism
$\iota_*^i : H_1(\partial_i Y; \Q) \to H_1( M \setminus \overset{\,\circ}{\nu}(L_i) ; \Q)$
induced by the inclusion 
$\iota: \partial_ i Y \hookrightarrow M \setminus \overset{\,\circ}{\nu}(L_i)$.
For $M = S^3$, the rational longitude coincides with Seifert-framed longitude.

It is important to keep in mind that for knots and links in $S^3$,
the conventional homology basis is not always the most natural surgery basis.
In particular, any cable or satellite of a knot in $S^3$ determines a
surgery basis for which the surgery longitude corresponds to the 
Seifert longitude of the associated torus knot or companion knot.
This {\textit{cable surgery basis}} or {\textit{satellite surgery basis}} does not
coincide with the conventional basis for $S^3$.

For $Y$ a compact oriented 3-manifold with
boundary $\partial Y = \coprod_{i=1}^n \partial_i Y$ a disjoint
union of tori,
any basis
$\textsc{B} = \prod_{i=1}^n (m_i, l_i)$
for $H_1(\partial Y; \Z)$
determines a map
\begin{align}
\pi_{\textsc{b}}:
\prod_{i=1}^n \P\left(H_1(\partial_i Y; \Z)\right)
&\longrightarrow 
(\Q \cup \{\infty\})^n_{\textsc{b}},
 \\
{\textstyle{\prod_{i=1}^n}} [a_i m_i + b_i l_i]
&\longmapsto 
({\textstyle{\frac{a_1}{b_1}, \ldots, \frac{a_n}{b_n}}}),
\nonumber
\end{align}
which associates $\textsc{b}$-{\textit{slopes}} $\frac{a_i}{b_i}  \in \Q \cup \{\infty\}$
to nonzero elements $a_i m_i + b_i l_i \in  H_1(\partial_i Y; \Z)$.
Each $\textsc{b}$-slope 
$({{\frac{a_1}{b_1}, \ldots, \frac{a_n}{b_n}}})\in (\Q \cup \{\infty\})^n_{\textsc{b}}$
specifies a {\textit{Dehn filling}} 
$Y_{\textsc{b}}({{\frac{a_1}{b_1}, \ldots, \frac{a_n}{b_n}}})$,
which is the closed 3-manifold given by attaching a compressing disk,
for each $i$,
to a simple closed curve in the primitive homology class corresponding to 
$[a_i m_i + b_i l_i] \in \P\left(H_1(\partial_i Y; \Z)\right)$,
and then gluing in a 3-ball to complete this solid torus filling of $\partial_i Y$.
Notationally, we write $\mathcal{A}_{\textsc{b}}(Y) := \pi_{\textsc{b}}(\mathcal{A}(Y))$
for any subset $\mathcal{A}(Y) \subset 
\prod_{i=1}^n \P\left(H_1(\partial_i Y; \Z)\right)$ of slopes for $Y$.
Thus, $\mathcal{L}_{S^3}(Y^{(np,nq)}) \subset (Q \cup \{\infty\})^n_{S^3}$
realizes $\mathcal{L}(Y^{(np,nq)})$ with respect
to the
conventional homology basis for link exteriors in $S^3\mkern-1mu$

\subsection{Seifert fibered basis}
\label{ss: seifert fibered and special subsets}
For $Y$ Seifert fibered over an $n$-times punctured $S^2$,
there is a conventional {\textit{Seifert fibered basis}}
$\textsc{sf} =(\tilde{f}_1, -\tilde{h}_1, \ldots, \tilde{f}_n, -\tilde{h}_n)$
for
$H_1(\partial Y; \Z)$
which makes slopes correspond to Seifert data 
for Dehn fillings of $Y$.
That is,
each $-\tilde{h}_i$ is the meridian of the $i^{\text{th}}$ excised regular fiber,
and each $\tilde{f}_i$ is the lift of the regular fiber class $f \in H_1(Y; \Z)$
to a class $\tilde{f}_i \in H_1(\partial_i Y; \Z)$ satisfying 
$-\tilde{h}_i \cdot {\tilde{f}}_i = 1$.  Note that this makes
$(\tilde{f}_i, -\tilde{h}_i)$ a {\textit{reverse-oriented}} basis for each 
$H_1(\partial_i Y; \Z)$, 
but this choice is made so that if $Y$ is trivially Seifert fibered,
then with respect to our Seifert fibered basis,
the Dehn filling
$Y_{\textsc{sf}}(\frac{\beta_1}{\alpha_1}, \ldots, \frac{\beta_n}{\alpha_n})$
coincides with the genus zero Seifert fibered space
$M := M_{S^2}(\frac{\beta_1}{\alpha_1}, \ldots, \frac{\beta_n}{\alpha_n})$
with (non-normalized) Seifert
invariants $(\frac{\beta_1}{\alpha_1}, \ldots, \frac{\beta_n}{\alpha_n})$
and first homology
\begin{equation}
\label{eq: presentation of H_1 of sf}
H_1(M;\Z) = \left< f, h_1 \ldots, h_n |\,
\textstyle{\sum_{i=1}^n} h_i = 0; \iota_{1*}(\mu_1) = \cdots = \iota_{n*}(\mu_n)=0 \right>,
\;\;
\text{where}
\end{equation}
\vspace{-.6cm}
\begin{equation}
\mu_i := \beta_i \tilde{f}_i - \alpha_i \tilde{h}_i \in H_1(\partial_i Y; \Z),\;\;
\iota_i : \partial_i Y \mkern3mu\into\mkern3mu Y,\;\;
h_i := \iota_{i*}(\tilde{h}_i),
\;\text{ and }\;
\iota_{i*}(\tilde{f}_i)=f.
\end{equation}

\subsection{Action of $\boldsymbol{\Lambda}$}
The relation $\sum_{i=1}^n h_i = 0$ in (\ref{eq: presentation of H_1 of sf})
comes from regarding the meridian images $-h_i \in H_1(Y; \Z)$
as living in some global section $S^2 \setminus \coprod_{i=1}^n D^2_i \,\into\, Y$
of the $S^1$ fibration, so that each $-h_i = -\partial D^2_i$ can be regarded as
$-h_i = \partial_i(S^2 \setminus \coprod_j^n D^2_j)$,
making the total class $-\sum_{i=1}^n h_i$ bound a disk in $S^2 \setminus \coprod_i^n D^2_i$.
This choice of global section is not canonical, however.
Any new choice of global section would correspond to a 
new choice of meridians,
\begin{equation}
\tilde{h}_i \,\longmapsto\, \tilde{h}_i' := \tilde{h}_i - l_i \tilde{f}_i,\;\;
i\in \{1,\ldots, n\}\;\,\text{ for some }\;
\boldsymbol{l} \in \Z^n,\;
\textstyle{\sum\limits_{i=1}^n} l_i = 0.
\end{equation}
Writing
$\mu_i = 
\beta'_i \tilde{f}_i - \alpha'_i \tilde{h}'_i$ to express $\mu_i$ in terms
of this new basis yields
\begin{equation}
\beta'_i = \beta_i + l_i \alpha_i,\;\;
\alpha'_i = \alpha_i,\;\;
{\textstyle{\frac{\beta'_i}{\alpha'_i} = \frac{\beta_i}{\alpha_i} + l_i}},\;\;\;\;\;
i\in \{1, \ldots, n\}.
\end{equation}
In other words, the lattice of global section reparameterizations
\begin{equation}
\Lambda(Y) := \{ \boldsymbol{l} \in \Z^n | \, \textstyle{\sum_{i=1}^n l_i = 0} \}
\end{equation}
acts on Seifert data, hence on $\textsc{sf}$-slopes, by addition,
without changing the underlying manifold or $S^1$ fibration.
Moreover,
for any choice of boundary-homology basis \textsc{b},
the change of basis from \textsc{sf}-slopes to \textsc{b}
slopes induces an action of $\Lambda$ on \textsc{b}-slopes.

\subsection{Action of $\boldsymbol{\Lambda}$ on torus-link-exterior slopes}
As occurs in the case when $Y=Y^n_{(p,q)}$ is the exterior
of the torus link
$T(np,nq) \subset S^3$
(see (\ref{eq: T(np,nq) as reg fibers in Seifert fibered space})),
the $\Lambda$-action on $S^3$-slopes
%(see (\ref{eq: T(np,nq) as reg fibers in Seifert fibered space})
%for a description of its Seifert structure in this case),
\begin{equation}
\label{eq: psi action of lambda defined for prop}
%\mkern100mu
\mkern30mu
\boldsymbol{\alpha} \mapsto
\boldsymbol{l} \cdot \boldsymbol{\alpha}
= \psi(\boldsymbol{l} + \psi^{-1}(\boldsymbol{\alpha})),
\;\;\;\;\;\;\;\;
\boldsymbol{l}\in
\Lambda_{\textsc{sf}}(Y^n_{(p,q)}),\,
\boldsymbol{\alpha}\in
(\Q \cup \{\infty\})^n_{S^3},
\end{equation}
induced by the transformation
\begin{equation}
\psi: 
(\Q \cup \{\infty\})^n_{\textsc{sf}} \longrightarrow (\Q \cup \{\infty\})^n_{S^3},\;\;
\boldsymbol{y} \mapsto
\left(pq + {\textstyle{\frac{1}{y_1}}}, \ldots, pq + {\textstyle{\frac{1}{y_n}}}\right),
\end{equation}
%to what, for present purposes, we call the 
%$S^3$-basis,
is of particular interest to us.
To aid in the introduction's discussion of
the role of $\Lambda$ in
Theorem~\ref{thm: torus link satellite in s3, lambda action},
we temporarily introduce the sets
$L_0, L_1, L_2 \subset (\Q \cup \{\infty\})^n_{S^3}$
of~$S^3$-slopes,~as~follows:
%For $N \in \Z_{>0}$, %with $0<N< pq$,
%define 
%$L_0, L_1, L_2 \subset (\Q \cup \{\infty\})^n_{S^3}$ to be the sets
%of $S^3$-slopes
\begin{align}
%\nonumber
L_0 
&:= \{\boldsymbol{\infty}\}
\;\subset\;
(\Q\cup\{\infty\})^n_{S^3},
\;\;\;\;\;\; \boldsymbol{\infty} := (\infty, \ldots, \infty),
  \\
L_1
&:=
\mkern2mu
\mathfrak{S}_n \mkern-1mu\cdot ([N,+\infty] \times \{\infty\}^{n-1})
\;\subset\; (\Q\cup\{\infty\})^n_{S^3},
      \\
L_2
&:= 
\left(\vphantom{\textstyle{\frac{a}{b}}}\mkern-1.5mu
\left[-\infty, pq\right>^{\mkern-2.5mu n} 
\mkern-1.5mu\setminus \mkern-.5mu
\left[-\infty, N\right>^{\mkern-2.5mu n} \right)
\mkern3mu\cup\mkern3mu
{\widehat{\mathcal{R}}}_{S^3}
\mkern3mu\cup\mkern3mu
\left<pq,+\infty\right]^n
\;\subset\; (\Q\cup\{\infty\})^n_{S^3},
\end{align}
for some $N \mkern-2.5mu\in\mkern-1.5mu \Z$,
where we have temporarily introduced the notation ${\widehat{\mathcal{R}}}_{S^3}$
to denote the union
\begin{equation}
{\widehat{\mathcal{R}}}_{S^3}
:= 
\mathfrak{S}_n \mkern-2.5mu\cdot\mkern-1.5mu
\left(
\mkern-1mu \{pq\} 
\mkern-1mu\times\mkern-1mu
(\left[-\infty, pq\right> 
\mkern-2.5mu\cup\mkern-2.5mu  
\left< pq, +\infty\right])^{n-1}
\right)
\;\subset\;
(\Q\cup\{\infty\})^n_{S^3}
\end{equation}
of sets of $S^3$-slopes. Lastly, for $\varepsilon \mkern-2mu>\mkern-2mu 0$, 
we take $U_{pq}^n(\varepsilon)$
to be the radius-$\varepsilon$~punctured~neighborhood 
\begin{equation}
U_{pq}^n(\varepsilon) := {\textstyle{\bigcup_{i=1}^{\mkern1mu n} }}
\{{\boldsymbol{\alpha}} |\,\alpha_i\in
\left[pq-\varepsilon, pq\right> \cup  \left<pq,pq+\varepsilon\right]\}
\;\subset\;
(\Q\cup\{\infty\})^n_{S^3}
\end{equation}
of the union of hyperplanes
$
\mathcal{R}_{S^3}\mkern-3mu =
\bigcup_{i=1}^{\mkern1mu n} \mkern-1mu \{{\boldsymbol{\alpha}} |\,\alpha_i \mkern-2mu=\mkern-2mu pq\}
\subset
(\Q\cup\{\infty\})^n_{S^3}$ 
discussed in Section~\ref{ss: R and Z}.

\begin{prop}
The action of $\Lambda$ on $(\Q \cup \{\infty\})^n_{S^3}$ in
(\ref{eq: psi action of lambda defined for prop})
 satisfies the following properties:

\noindent $(a)$ 
$ \mkern5mu\Lambda_{S^3}\mkern-1mu \cap \mkern2mu \Q^n_{S^3} 
\;\subset\; 
%\{(\infty, \ldots, \infty)\} 
%\mkern1mu\cup \mkern1mu
\left(\vphantom{\frac{a}{b}}\left[pq-1, pq\right> 
\cup \left<pq, pq+1\right]\right)^n$;
$\;\;\Lambda_{S^3}= \Lambda \cdot \{\boldsymbol{\infty}\}_{S^3} = \Lambda \cdot L_0$.

\noindent $(b)$ 
$(\Lambda \mkern-2mu\cdot\mkern-2mu L_i) \mkern-3mu\setminus\mkern-3mu L_i
\mkern.2mu \subset \mkern.2mu U_{pq}^n(1)\mkern-.5mu$
for $i \mkern-2mu=\mkern-2mu 0$ always, 
for $i \mkern-2mu=\mkern-2mu 1$ 
when  $N \mkern-2mu>\mkern-2mu pq$,
and for $i \mkern-2mu=\mkern-2mu 2$
when $N \mkern-2mu\le\mkern-2mu pq$.

\noindent $(c)$ 
For $\varepsilon > 0$, each of the following sets of $S^3$-slopes
can be realized as a
union of finitely many rectangles of dimensions 0, 1, and $\{n-1, n\}$,
respectively:

$(\Lambda \cdot L_0) \setminus U_{pq}^n(\varepsilon),
\;\;$
$(\Lambda \cdot L_1) \setminus U_{pq}^n(\varepsilon)$ for $N > pq,\;\;$
and
$(\Lambda \cdot L_2) \setminus U_{pq}^n(\varepsilon)$ for $N \le pq$.

\end{prop}

\begin{proof}
Part $(a)$. The first statement follows from the fact that
$\frac{1}{m} \in \left[-1, 0\right> \cup \left<0, 1\right] \cup \{\infty\}$ 
for all $m \in \Z$. The second statement is due to the fact that
$\boldsymbol{0} \mkern-1.5mu\in\mkern-1.5mu \Lambda_{\textsc{sf}}$
implies
$\boldsymbol{\infty} \mkern-1.5mu =\mkern-1.5mu \psi(\boldsymbol{0}) \mkern-1.5mu\in\mkern-1.5mu \Lambda_{S^3}$.

Part $(b)$.
First note that the action of $\Z$ on $(\Q \cup \{\infty\})$ by addition fixes both 
$\infty = \psi_j^{-1}(pq)$ as a point
and its complement 
$\Q = \psi^{-1}_j
\left(\vphantom{\frac{a}{b}} 
\left[-\infty,pq\right> \cup \left<pq, +\infty\right] \right)
$
as a set, for any $j \in \{1, \dots,n\}$.
The action of $\Lambda$ on $(\Q\cup\{\infty\})^n_{S^3}$
therefore fixes setwise
the union
${\widehat{\mathcal{R}}}_{S^3}$
of products of such sets.
Since $(\Lambda \cdot X) \setminus X  \mkern1mu\subset\mkern1mu
(\Lambda_{\textsc{sf}} \setminus \{\boldsymbol{0}\}) \cdot X$
for any subset $X \subset (\Q \cup \{\infty\})^n_{S^3}$ of $S^3$-slopes,
and since 
$L_i \subset (\left[ -\infty, pq \right>^{\mkern-1.5mu n} \cup \left< pq,+ \infty \right]^n)_{S^3}$ 
for $i \mkern-2mu=\mkern-2mu 0$ always, 
for $i \mkern-2mu=\mkern-2mu 1$ 
when  $N \mkern-2mu>\mkern-2mu pq$,
and for $i \mkern-2mu=\mkern-2mu 2$
when $N \mkern-2mu\le\mkern-2mu pq$,
it is sufficient to show that
$(\Lambda_{\textsc{sf}} \mkern-2mu\setminus \{\boldsymbol{0}\}) 
\cdot (\left[ -\infty, pq \right>^{\mkern-1.5mu n} \cup \left< pq,+ \infty \right]^n)_{S^3}
\mkern-2mu \subset \mkern-1.5mu U_{pq}^n(1)$.
To see this, we first note that
$\boldsymbol{l} \mkern-2mu\in\mkern-2mu 
(\Lambda_{\textsc{sf}}\setminus \{\boldsymbol{0}\})\mkern-1mu$
must have at least one positive and at least one negative component, 
say  $l_{i+} \mkern-3.5mu \in \left[1, +\infty \right> \cap \Z$
and $l_{i-} \mkern-3.5mu \in \left<-\infty, -1\right] \cap \Z$, implying
$(\boldsymbol{l} + \left[0,+\infty\right>^n)|_{i+} \subset \left[1,+\infty\right>$
and
$(\boldsymbol{l} + \left<-\infty, 0\right]^n)|_{i-} \subset \left<-\infty,-1\right]$.
We then have
\begin{align}
(\boldsymbol{l}\cdot \left<pq,+\infty\right]^n)|_{i+}\mkern-1.5mu
&= \psi(\boldsymbol{l}+\psi^{-1}(\left<pq,+\infty\right]^n) )|_{i+}\mkern-1.5mu
= \psi(\boldsymbol{l}+ \left[0,+\infty\right>^n)|_{i+}
\subset \left<pq, pq+1\right],
   \\
(\boldsymbol{l}\cdot \left<-\infty, pq\right]^n)|_{i-}\mkern-1.5mu
&= \psi(\boldsymbol{l}+\psi^{-1}(\left<-\infty, pq\right]^n) )|_{i-}\mkern-1.5mu
= \psi(\boldsymbol{l}+ \left<-\infty, 0\right]^n)|_{i+}
\subset \left[pq-1, pq\right>,
\end{align}
and so we conclude that $\boldsymbol{l}  \cdot 
(\left[ -\infty, pq \right>^{\mkern-1.5mu n} \cup \left< pq,+ \infty \right]^n)_{S^3}
\mkern-2mu \subset \mkern-1.5mu U_{pq}^n(1)$, completing the proof of $(b)$.

% phone number 7049369192

% Embassy Suites by Hilton Washington DC Convention Center, 900 10th Street NW, Washington.

Part $(c)$.
In the \textsc{sf} basis, the complement of $U_{pq}^n(\varepsilon)$ within
the set of $S^3$-slopes is given by
\begin{equation}
\psi^{-1}
\left(\vphantom{\textstyle{\frac{a}{b}}}
(\Q \cup \{\infty\})^n_{S^3} \setminus U^n_{pq}(\varepsilon)\right)
\mkern7mu= \mkern7mu
\left< -\textstyle{\frac{1}{\varepsilon}} , 
+\textstyle{\frac{1}{\varepsilon}}\right>^n
%_{\textsc{sf}}
\mkern3mu\cup\mkern5mu
{\textstyle{\bigcup_{i=1}^n}}\mkern-1.5mu \{\boldsymbol{y} 
|\,y_i \mkern-1mu=\mkern-1mu \infty\}
\;\subset\;
(\Q\cup \{\infty\})^n_{\textsc{sf}}.
%\left< -{\frac{1}{\varepsilon}} , +{\frac{1}{\varepsilon}}\right>
\end{equation}
Since both
${\textstyle{\bigcup_{i=1}^n}}\mkern-1.5mu \{\boldsymbol{y} 
|\,y_i \mkern-2.5mu=\mkern-2.5mu \infty\}_{\textsc{sf}}
\mkern-2mu=\mkern-1mu \psi^{-1}(\mathcal{R}_{S^3})$
and $\psi^{-1}(\widehat{\mathcal{R}}_{S^3})$
are fixed setwise by $\Lambda$,
and since
$\psi^{-1}(L_i) 
\cap 
\psi^{-1}(\mathcal{R}_{S^3})
%{\textstyle{\bigcup_{i=1}^n}}\mkern-1.5mu \{\boldsymbol{y} 
%|\,y_i \mkern-1mu=\mkern-1mu \infty\}_{\textsc{sf}} 
\mkern-2mu=\mkern-2mu \emptyset$ for $i\mkern-2mu=\mkern-2mu0$ and for $i\mkern-2mu=\mkern-2mu1$ with $N \mkern-2mu>\mkern-2mu pq$,
but
$\psi^{-1}(L_i) 
\cap 
\psi^{-1}(\mathcal{R}_{S^3})
%{\textstyle{\bigcup_{i=1}^n}}\mkern-1.5mu \{\boldsymbol{y} 
%|\,y_i \mkern-1mu=\mkern-1mu \infty\}_{\textsc{sf}} 
\mkern-2mu=\mkern-2mu \psi^{-1}({\widehat{\mathcal{R}}}_{S^3})
$
when $i\mkern-2mu=\mkern-2mu2$ and $N \mkern-2mu\le\mkern-2mu pq$,
it follows that
\begin{equation}
(\Lambda _{\textsc{sf}}+ \psi^{-1}(L_i) ) \,\cap\, 
{\textstyle{\bigcup_{i=1}^n}}\mkern-1.5mu \{\boldsymbol{y} 
|\,y_i \mkern-1mu=\mkern-1mu \infty\}_{\textsc{sf}} 
=
\begin{cases}
\emptyset
&
i=0 \text{ or } (i=1\text{ and }N> pq)
  \\
\psi^{-1}(\widehat{\mathcal{R}}_{S^3})
&
i=2 \text{ and } N \le pq
\end{cases}.
\end{equation}
Thus, since $\psi^{-1}({\widehat{\mathcal{R}}}_{S^3}) $ is already a finite union 
of $(n-1)$-dimensional rectangles, it suffices to show that
$(\Lambda_{\textsc{sf}} + \psi^{-1}(L_i)) \cap \left< -\textstyle{\frac{1}{\varepsilon}} , 
+\textstyle{\frac{1}{\varepsilon}}\right>^n_{\textsc{sf}}$
is a union of finitely many rectangles of dimensions $0$, $1$, or 
$n$ in the respective cases that $i=0$, $i=1$ with $N> pq$, or
$i=2$ with $N \le pq$. The proof of this latter statement is straightforward,
however, since $\psi^{-1}(L_i)$ is already a finite union of rectangles of dimensions 0, 1, or $n$, respectively for the three above respective cases, and only finitely many distinct rectangles can be formed by intersecting $\Z^n$ translates of these rectangles with 
$\left< -\textstyle{\frac{1}{\varepsilon}} , 
+\textstyle{\frac{1}{\varepsilon}}\right>^n_{\textsc{sf}} \subset (\Q\cup\{\infty\})^n_{\textsc{sf}}$.

\end{proof}

\subsection{Reducible and exceptional sets 
$\boldsymbol{\mathcal{R}}$ and $\boldsymbol{\mathcal{Z}}$}
\label{ss: R and Z}
Like the above action of $\Lambda$, the following facts about reducible fillings
are well known in low dimensional topology,
but for the benefit of a diverse readership we provide some details.
\begin{prop}
Let $\hat{Y}$ denote the trivial $S^1$ fibration over
$S^2 \setminus \coprod_{i=0}^n D^2_i$,
and let $Y$ denote the Dehn filling of $\hat{Y}$ along
the $S^1$-fiber lift $\tilde{f}_0 \in H_1(\partial_0 \hat{Y}; \Z)$,
{\em{i.e.}}, along the $\infty$ $\textsc{sf}$-slope of $\partial_0 \hat{Y}$.
Then $Y$ is a connected sum
%\begin{equation}
$
Y = \#_{i=1}^n (S_i^3 \setminus S^1_f \times D^2_i)
$
%\;\;\;\;\text{with}\;\;\;\;\;
%[S^1 \mkern-2mu\times \mathrm{core}(D^2_i)] = f\;\;
%\forall\;i\in\{1,\ldots,n\},
%\end{equation}
(where $S^1_f$ is the fiber),
and each exterior $S_i^3 \setminus S^1_f \times D^2_i$
has meridian $-\tilde{h}_i$ and rational longitude $\tilde{f}_i$.

\end{prop}

\begin{proof}
Choose a global section 
$S^2 \setminus \coprod_{i=0}^n D^2_i \into \hat{Y}$
which respects the $\textsc{sf}$ basis.
We shall stretch the disk $S^2 \setminus D^2_0$ into a (daisy) flower shape,
with one $D^2_i$ contained in each petal.
Embed $2n$ points $p_1^-, p_1^+, \ldots, p_n^-, p_n^+ \,\into\, \partial D^2_0$,
in that order with respect to the orientation of $-\partial D^2_0$.
For each $i \in \{1, \ldots, n\}$, let $\delta_i$ and $\varepsilon_i$ denote 
the respective arcs from
$p_i^-$ to $p_i^+$ and from 
$p_i^+$ to $p_{i+1(\mod n)}^-$ along $-\partial D_0^2$, and
properly embed an arc 
$\gamma_i \into S^2 \setminus \coprod_{i=0}^n D^2_i$
from $p_i^+$ to $p_i^-$ which winds once positively around $D^2_i$
and winds zero times around the other $D^2_j$, without intersecting
any of the other $\gamma_j$ arcs.  Holding the $p_i^{\pm}$ points
fixed while stretching the $\delta_i$ arcs outward and pulling the $\gamma_i$
arcs tight realizes our global section 
as the punctured flower shape 
\begin{equation}
S^2 \setminus  
\textstyle{\coprod\limits_{i=0}^n D_i^2} = \textstyle{\tilde{D}_0^2 \cup \coprod\limits_{i=1}^n (\tilde{D}^2_i \setminus D_i^2)},
\end{equation}
where $\tilde{D}^2_0$
denotes the central disk of the flower,
bounded by 
$\partial \tilde{D}^2_0 = (\coprod_{i=1}^n -\gamma_i) \cup (\coprod_{i=1}^n \varepsilon_i)$,
and each $\tilde{D}^2_i$ denotes the petal-shaped disk bounded by
$\partial \tilde{D}^2_i = \delta_i \cup_{p_i^{\pm}} \gamma_i$.

The Dehn filling $Y$ is formed by multiplying the above global section with the fiber
$S^1_f$ and then gluing a solid
torus $D^2_f \times \partial D_0^2$
(with $\partial D^2_f = S^1_f$)
along $S^1_f \times \partial D_0^2$.
Since 
\begin{equation}
-\partial D_0^2 
=
\textstyle{(\coprod\limits_{i=1}^n \delta_i ) \cup (\coprod\limits_{i=1}^n \varepsilon_i)},
\;\;
\partial \tilde{D}^2_0 = \textstyle{
(\coprod\limits_{i=1}^n -\gamma_i) \cup (\coprod\limits_{i=1}^n \varepsilon_i)},
\;\;\text{and}\;\;
\partial \tilde{D}^2_i = \delta_i \cup_{p_i^{\pm}} \gamma_i,
\end{equation}
we can decompose the solid torus $D^2_f \times \partial D_0^2$ 
along the disks $D^2_f \times p_i^{\pm}$,
and distribute these solid-torus components among the boundaries of $\tilde{D}_0^2$ and
the $\tilde{D}^2_i$, so that
\begin{equation}
-\partial D^2_f \times \partial D^2_0
\;\,=\;\, 
(D^2_f \times \partial \tilde{D}^2_0) \setminus 
\textstyle{\coprod_{i=1}^n} (D^2_f \times -\overset{\circ}{\gamma}_i)
\;\,
\cup
\;\,
\textstyle{\coprod_{i=1}^n}
\left((D^2_f \times \partial \tilde{D}_i) \setminus (D^2_f \times \overset{\circ}{\gamma}_i)
\right),
\end{equation}
where the union is along the boundary 2-spheres 
\begin{equation}
S^2_i \;:=\; (D^2_f \times p_i^-) \,\cup\, (S^1_{\!f} \times \gamma_i) \,\cup\, (D^2_f \times p_i^+) 
\;=\; \partial (D^2_f \times \overset{\circ}{\gamma}_i)
\end{equation}
of the balls $D^2_f \times \pm\overset{\circ}{\gamma}_i$.
Thus, if we set 
$S^3_i \,:=\, (D^2_f \mkern-0mu\times\mkern-.5mu \partial \mkern-1.5mu\tilde{D}^2_i)
\mkern2.5mu\cup\mkern2.5mu 
(S^1_f \mkern-0mu\times\mkern-.5mu \tilde{D}^2_i)$ for $i \in \{0,\ldots, n\}$, then
\begin{equation}
Y \,=\, S^3_0 \;\#\, \textstyle{\coprod\limits_{i=1}^n} (S^3_i \setminus S^1_f \times D_i^2),
\end{equation}
with the connected sum taken along the spheres $S^2_i$.
\end{proof}

\begin{cor}
If $\hat{Y}$ is as above,
and if $(y_1, \ldots, y_n) \in \{\infty\}^k \times \Q^{n-k}$
for some $k \in \{0, \ldots, n\}$, then
$\hat{Y}_{\textsc{sf}}(\infty, y_1, \ldots, y_n) \;=\; 
(\#_{i=1}^k S^1 \times S^2)
\;\; \#\;\; 
(\#_{i=k+1}^n M_{S^2}(y_i))$.
\end{cor}
This motivates the following terminology.

\begin{definition}
\label{def: R and Z}
Suppose that $\partial Y = \coprod_{i=1}^n \mathbb{T}^2_i$, and that
$Y$ has a Seifert fibered JSJ component containing $\partial Y$.
Then the {\em{reducible slopes}} $\mathcal{R}(Y)$ and
{\em{exceptional slopes}} $\mathcal{Z}(Y) \subset \mathcal{R}(Y)$
are given by
$\mathcal{R}_{\textsc{sf}}(Y) := \mathfrak{S}_n \cdot (\{\infty\} \times (\Q \cup \infty)^{n-1})$
and $\mathcal{Z}_{\textsc{sf}}(Y) := \mathfrak{S}_n \cdot 
(\{\infty\}^2 \times (\Q \cup \infty)^{n-2})$, where $\mathfrak{S}_n$ acts by permutation of slopes.
\end{definition}
Note that occasionally, slopes in $\mathcal{R}(Y)$ yield Dehn fillings
which are connected sums of a lens space with 3-spheres, hence
are not reducible.
\begin{definition}
\label{def: false reducible}
Suppose $Y$ is as above.  If the Seifert fibered component containing
$\partial Y$ has no exceptional fibers, then
the {\em{false reducible slopes}}
$\mathcal{R}^0(Y) \subset \mathcal{R}(Y)$ of $Y$ are given by
$\mathcal{R}^0_{\textsc{sf}}(Y) = 
\mathfrak{S}_n \cdot (\{\infty\} \times \Q \times \{0\}^{n-2})$.
Any reducible slopes which are not false reducible are called {\em{truly reducible}}.
Equivalently, the truly reducible slopes are those slopes which 
yield reducible Dehn fillings.
\end{definition}

\subsection{Rational longitudes $\boldsymbol{\mathcal{B}}$}
\label{ss: rational longitudes B}
Our last distinguished slope set of interest, the set $\mathcal{B}$ of rational longitude slopes, makes sense for $Y$ of any geometric type.

\begin{definition}
Suppose $Y$ is a compact oriented 3-manifold with $\partial Y$ a union of $n >0$  toroidal boundary components and with at least one rational-homology-sphere Dehn filling. The {\em{set of rational longitudes}} 
$\mathcal{B} \subset 
\coprod_{i=1}^n \P(H_1(\partial_i Y; \Z))$ 
{\em{of}} $Y$ is the set of slopes
\begin{equation}
\mathcal{B}
\,:=\,
\left\{\beta \mkern1.5mu\in\mkern1.5mu 
{\textstyle{\coprod_{i=1}^n\mkern-1mu \P(H_1(\partial_i Y; \Z))}} \mkern1mu|\;
b_1(Y(\beta))>0\right\}.
\end{equation} 
\end{definition}
Note that this implies
$\mathcal{Z}= \mathcal{R} \cap \mathcal{B}$.
Moreover, when $Y$ is Seifert fibered, $\mathcal{B}_{\textsc{sf}}(Y)$
is the closure of a linear subspace of $\textsc{sf}$-slopes.
That is, by Section 5 of \cite{lslope} (for example), we have
%$M_{S^2\mkern-2mu}(\frac{\beta_1}{\alpha_1},\ldots,
%\frac{\beta_n}{\alpha_n})$ satisfies
\begin{equation}
\label{eq: sf B in general}
b_1\mkern-3mu\left( M_{S^2\mkern-6mu}
		\left(\vphantom{A^A} \right.
		\mkern-4mu{\textstyle{\frac{\beta_1}{\alpha_1},
				\ldots,\frac{\beta_n}{\alpha_n}}}\mkern-4mu
		\left.\vphantom{A^A}\right) \right) > 0
\;\,\iff\,\; {\textstyle{\sum_{i=1}^n\frac{\beta_i}{\alpha_i}}} = 0.
\end{equation}
In particular, since the $T(np,nq)$-exterior 
$Y^n_{(p,q)} :=
M_{S^2}(\mkern1mu{\textstyle{
\mkern-2mu-\mkern-1mu\frac{q^*\mkern-6mu}{p},
\mkern-1mu\frac{p^*\mkern-7mu}{q}}}, \overset{1}{*}, \ldots, \overset{n}{*}\mkern1mu)$,
constructed in
(\ref{eq: T(np,nq) as reg fibers in Seifert fibered space}),
has $\frac{\beta_{-1}}{\alpha_{-1}} + \frac{\beta_0}{\alpha_0} = \frac{1}{pq}$, 
the next proposition follows immediately from line 
(\ref{eq: sf B in general}).
\begin{prop}
\label{prop: basics for B}
If $\mkern1.5muY \mkern-1.5mu=\mkern-.5mu S^3 \mkern-1.5mu\setminus\mkern.5mu \overset{\circ}{\nu}(T(np,nq))$ is the
exterior of the $(np,nq)$ torus link, then
$$
\mathcal{B}_{\textsc{sf}}(Y)
=
\left\{ \vphantom{\textstyle{\frac{1}{pq}}}
\boldsymbol{y} \mkern-1.5mu\in\mkern-1.5mu (\Q\cup\mkern-2mu\{\infty\})^n_{\textsc{sf}}\mkern1mu\right|
\left.\mkern-1.5mu
{\textstyle{\frac{1}{pq} + \sum_{i=1}^n y_i}}\mkern-1.5mu=\mkern-1.5mu0\right\},
$$
is the closure
in $(\Q \cup \mkern-1.5mu\{\infty\}\mkern-1mu)^n_{\textsc{sf}}\mkern-1mu$ of the hyperplane
$\left\{ \vphantom{\textstyle{\frac{1}{pq}}}\mkern-1mu
\boldsymbol{y} \mkern-1mu\in\mkern-1mu \Q^n \mkern-.5mu\right|\left. \mkern-2mu
{\textstyle{\frac{1}{pq} 
		\mkern-3mu+\mkern-2.5mu 
		\sum_{i=1}^n\mkern-1.7mu y_i}} = 0\right\}$ in
	$\Q^n \mkern-4.2mu\subset\mkern-3.2mu (\Q\cup\mkern-1mu\{\infty\}\mkern-.8mu)^n_{\textsc{sf}}$.
Moreover, if 
$\overline{\mathcal{B}}_{\textsc{sf}}$ denotes the real closure of
$\mathcal{B}_{\textsc{sf}}$ in
$\mkern2mu\coprod_{i=1}^n\P(H_1(\partial_i Y; \R))
\stackrel{\sim}{\to}
(\R\cup\{\infty\})^n_{\textsc{sf}} $, 
then
$$
\overline{\mathcal{B}}_{\textsc{sf}} \cong \mathbb{T}^{n-1}
\cong (\R \cup \{\infty\})^{n-1} \hookrightarrow (\R\cup\{\infty\})^n_{\textsc{sf}}
\cong \mathbb{T}^n.
$$
\end{prop}

\section{L-space intervals and gluing}
\label{s: L-space intervals and gluing}

\subsection{L-space interval notation}

For the following discussion,
$Y$ denotes a compact oriented 3-manifold
with torus boundary $\partial Y$,
and $\textsc{B}$ is a basis
for $H_1(\partial Y; \Z)$, inducing an
identification 
$\pi_B : \P(H_1(\partial Y; \Z)) \to (\Q \cup \{\infty\})_{\textsc{b}}
=: \P(H_1(\partial Y; \Z))_{\textsc{b}}$.

\begin{definition}
\label{def: def of [[,]] notation}
We introduce the notation $[[\cdot, \mkern-3.2mu\cdot\mkern-.3mu]]\mkern-.2mu,$ so that
for $\mkern1muy_-, y_+ \mkern-2.2mu\in\mkern-.3mu \P(H_1(\partial Y; \Z))_{\textsc{b}}$,
the subset $[[y_-, y_+]] \subset \P(H_1(\partial Y; \Z))_{\textsc{b}}$
is defined as follows.
$$
[[y_-, y_+]] := 
\begin{cases}
\P(H_1(\partial Y; \Z))_{\textsc{b}} \setminus \{y_-\} =
\P(H_1(\partial Y; \Z))_{\textsc{b}} \setminus \{y_-\}
  &\;\; y_- = y_+
    \\
\P(H_1(\partial Y; \Z))_{\textsc{b}} \cap
I_{(y_-, y_+)}
 &\;\; y_- \neq y_+,
\end{cases}
$$
where $I_{(y_-, y_+)} \subset  \P(H_1(\partial Y; \R))_{\textsc{b}}$
indicates the closed interval with
left-hand endpoint $y_-$ and right-hand endpoint $y_+$.
\end{definition}

By Proposition 1.3 and Theorem 1.6 of
J.~Rasmussen and the author's \cite{lslope},
the L-space interval $\mathcal{L}(Y) \subset \P(H_1(\partial Y ; \Z))$
of L-space Dehn filling slopes of
$Y$ can only take certain forms.

\begin{prop}[J. Rasmussen, S. Rasmussen \cite{lslope}]
\label{prop: structure of L-space intervals}

One of the following is true:
\begin{itemize}
\item[(i)] $\mathcal{L}(Y) = \emptyset$,

\item[(ii)] $\mathcal{L}(Y) = \{\eta\}$,
for some $\eta \in \P(H_1(\partial Y ; \Z))$,

\item[(iii)] $\mathcal{L}_{\textsc{b}}(Y) = [[l, l]]$, with
$l \in \P(H_1(\partial Y; \Z))_{\textsc{b}}$ the rational longitude of $Y$, or

\item[(iv)] $\mathcal{L}_{\textsc{b}}(Y) = [[y_-, y_+]]$ with $y_ \neq y_+$.
\end{itemize}
\end{prop}
It is for this reason 
that we refer to the space of L-space Dehn filling
slopes as an {\em{interval}}.
It also makes sense to speak of the
the {\em{interior}} of this interval.

\begin{definition}
\label{def: of interior of interval}
The {\em{L-space interval interior}} $\mathcal{L}^{\circ}(Y) \subset \mathcal{L}(Y)$
of $Y$ satisfies
$$
\mathcal{L}^{\circ}_{\textsc{b}}(Y)
:=
\begin{cases}
\emptyset
& \mathcal{L}(Y) = \emptyset
\text{ or }
\mathcal{L}(Y) = \{\eta\}
  \\
[[l,l]]
&
\mathcal{L}_{\textsc{b}}(Y) = [[l,l]]
  \\
\P(H_1(\partial Y; \Z))_{\textsc{b}} \cap {\overset{\circ}{I}}_{(y_-,y_+)}
& \mathcal{L}_{\textsc{b}}(Y) = [[y_-,y_+]] \text{ with } y_- \neq y_+,
\end{cases}
$$
where ${\overset{\circ}{I}}_{(y_-,y_+)} \subset \P(H_1(\partial Y; \R))_{\textsc{b}}$
indicates the {\em{open}} interval with
left-hand endpoint $y_-$ and right-hand endpoint $y_+$.
\end{definition}

This gives us a new way to characterize the property of
{\em{Floer simplicity}} for $Y$.
\begin{prop}[J. Rasmussen, S. Rasmussen \cite{lslope}]
The following are equivalent:
\begin{itemize}
\item
$Y$ has more than one L-space Dehn filling,
\item
$\mathcal{L}_{\textsc{b}}(Y) = [[y_-, y_+]]$ for some 
$y_-, y_+ \in \P(H_1(\partial Y; \Z))_{\textsc{b}}$,
\item
$\mathcal{L}^{\circ}(Y) \neq \emptyset$.
\end{itemize}
In the case that any, and hence all, of these three properties hold, 
we say that $Y\mkern-4mu$ is {\em{Floer~simple}}.
\end{prop}

Both Floer simple manifolds and graph manifolds
have predictably-behaved unions with respect to the property
of being an L-space.
\begin{theorem}[Hanselman, J.~Rasmussen, S.~Rasmussen, Watson \cite{lslope,HRRW, lgraph}]
\label{thm: gluing theorem for floer simple or graph manifold}
If the manifold
$Y_1 \cup_{\varphi}\mkern-2mu Y_2$,
with gluing map
$\varphi: \partial Y_1 \to -\partial Y_2$,
is a closed union of 3-manifolds,
each with~incompressible single-torus boundary,
and with $Y_i$ both Floer simple or both graph manifolds,~then
$$
Y_1 \cup_{\varphi}\mkern-2mu Y_2
\,\text{ is an L-space }
\;\;\iff\;\;
\varphi_*^{\P}(\mathcal{L}^{\circ}(Y_1))
\cup
\mathcal{L}^{\circ}(Y_2) = \P(H_1(\partial Y_2; \Z)).
$$
\end{theorem}

Unfortunately, this theorem fails to encompass the case
in which an L-space knot exterior is glued to a non-Floer simple graph manifold,
and our study of surgeries on iterated or algebraic satellites will certainly require
this case.  We therefore prove the following result.

\vspace{.2cm}

\begin{theorem}
\label{thm: knot exterior gluing theorem}
If $Y_1 \cup_{\varphi}\mkern-2mu Y_2$,
with gluing map
$\varphi: \partial Y_1 \to -\partial Y_2$,
is a closed union of 3-manifolds, such that
$Y_1$
%:=S^3 \setminus \overset{\mkern2mu \circ}{\nu}(K)$
is the exterior of a nontrivial L-space knot $K \subset S^3$,
and $Y_2$ is a graph manifold, or connected sum
thereof, with incompressible single torus boundary, then
\begin{equation}
\label{eq: knot exterior gluing theorem}
Y_1 \cup_{\varphi}\mkern-2mu Y_2
\,\text{ is an L-space }
\;\;\iff\;\;
\varphi_*^{\P}(\mathcal{L}^{\circ}(Y_1))
\cup
\mathcal{L}^{\circ}(Y_2) = \P(H_1(\partial Y_2; \Z)).
\end{equation}
Moreover, if $Y_2$ is not Floer simple, then 
(\ref{eq: knot exterior gluing theorem})
holds for $K \subset S^3$ an arbitrary nontrivial knot.
\end{theorem}

\begin{proof}
We first reduce to the case prime $Y_2$.
Let $Y_2^{\prime}$ denote the connected summand
of $Y_2$ containing $\partial Y_2$, and recall that hat
Heegaard Floer homology tensors over connected sums.
Thus, if $Y_2$ has any non-L-space closed connected summands,
then $Y_1 \cup_{\varphi}\mkern-2mu Y_2$ is a non-L-space and
$\mathcal{L}(Y_2) = \emptyset$, regardless of the union
$Y_1 \cup_{\varphi}\mkern-2mu Y_2^{\prime}$.
On the other hand, if all the closed connected summands of $Y_2$
are L-spaces, then
$\mathcal{L}^{\circ}(Y_2) = 
\mathcal{L}^{\circ}(Y_2^{\prime})$, and
$Y_1 \cup_{\varphi}\mkern-2mu Y_2$ is an L-space if and only if 
$Y_1 \cup_{\varphi}\mkern-2mu Y_2^{\prime}$ is an L-space.
We therefore henceforth assume $Y_2$ is prime.

If $Y_2$ is Floer simple, then when $K \subset S^3$ is an L-space knot,
$Y_1$ is Floer simple, and so the desired result
is already given by
Theorem~\ref{thm: gluing theorem for floer simple or graph manifold}.
We therefore assume $Y_2$ is not Floer simple,
implying $\mathcal{NL}(Y_2) = \P(H_1(\partial Y_2; \Z))$ or 
$\mathcal{NL}(Y_2) = \P(H_1(\partial Y_2; \Z)) \setminus \{y\}$
for some single slope $y$.
In either case, $\mathcal{L}^{\circ}\mkern-1.5mu(Y_2) = \emptyset$,
and so it remains to show that
$Y_1 \cup_{\varphi}\mkern-2mu Y_2$
is not an L-space.

In \cite{lgraph}, the author showed
for any (prime) graph manifold $Y_2$ with single torus boundary that
if $\mathcal{F}(Y_2)$ (called $\mathcal{F}^{\textsc{d}}\mkern-2mu(Y_2)$ in that
paper's notation) denotes the set of slopes $\alpha \in \P(H_1(\partial Y_2; \Z))$
for which $Y_2$ admits a co-oriented taut foliation restricting
to a product foliation of slope $\alpha$ on $\partial Y_2$, then
$\mathcal{F}(Y_2) = \mathcal{NL}(Y_2) \setminus \mathcal{R}(Y_2)$.
Since $Y_2$ is prime, 
$\mathcal{R}(Y_2) = \{\infty\} \in \P(H_1(\partial Y_2; \Z))_{\textsc{sf}}$.
In particular, $\mathcal{F}(Y_2)$ is the
complement of a finite set in $\P(H_1(\partial Y_2; \Z))$.

On the other hand, 
Li and Roberts show
in \cite{LiRoberts}
that for the exterior of an arbitrary nontrivial knot in $S^3$,
such as $Y_1$, one has $\mathcal{F}_{S^3}(Y_1) \supset \left<a,b\right>$ 
for some $a < 0 < b$.  In particular, $\mathcal{F}(Y_1)$ is infinite,
implying $\varphi_*^{\P}(\mathcal{F}(Y_1)) \cap \mathcal{F}(Y_2)$
is nonempty.  Thus we can construct a co-oriented taut foliation 
$F$ on $Y_1 \cup_{\varphi}\mkern-2mu Y_2$ by gluing together
co-oriented taut foliations restricting to a matching product foliation
of some slope $\alpha \in \varphi_*^{\P}(\mathcal{F}(Y_1)) \cap \mathcal{F}(Y_2)$
on $\partial Y_2$.

Eliashberg and Thurston showed in \cite{EliashbergThurston}
that a $C^2$ co-oriented taut foliation can be perturbed to a pair of 
oppositely oriented tight contact structures, each with a symplectic semi-filling
with $b_2^+ \mkern-4mu>\mkern-3mu 0$.
Ozsv{\'a}th and Szab{\'o} \cite{OSGen} showed that one can
associate a nonzero class in reduced Heegaard Floer homology
to such a contact structure.
This result was recently extended to $C^0$ co-oriented taut foliations
by Kazez and Roberts \cite{KazezRobertsCzero}
and independently by 
Bowden \cite{Bowden}.
Thus, our co-oriented taut foliation
$F$ on $Y_1 \cup_{\varphi}\mkern-2mu Y_2$
implies that
$Y_1 \cup_{\varphi}\mkern-2mu Y_2$
is not an L-space.

\end{proof}

\section{Torus-link satellites}
\label{s: torus-link satellites}

\subsection{$\,\boldsymbol{T(np,nq) \mkern-2mu\subset\mkern-2mu S^3}$ and
Seifert structures on $\boldsymbol{S^3}$}
\label{ss: seifert structures on exterior stuff}
Since $S^3$ is a lens space, any Seifert fibered realization of $S^3$ can 
have at most 2 exceptional fibers:
\begin{equation}
S^3 = M_{S^2}(\mkern1mu{\textstyle{
\mkern-2mu-\mkern-1mu\frac{q^*\mkern-6mu}{p},
\mkern-1mu\frac{p^*\mkern-7mu}{q}}}\mkern1mu),
\;\;\;\;
p^*p - q^*q = 1,
\end{equation}
where the right-hand constraint on $p, q, p^*\mkern-2mu, q^* \mkern-2mu\in\mkern-2mu \Z$
is necessary (and sufficient) to achieve
$H_1(M_{S^2}(\mkern1mu{\textstyle{
\mkern-2mu-\mkern-1mu\frac{q^*\mkern-6mu}{p},
\mkern-1mu\frac{p^*\mkern-7mu}{q}}}\mkern1mu)
; \Z) \mkern-2mu=\mkern-2mu 0$.
The one-exceptional-fiber Seifert structures for $S^3$ are exhausted by
the cases
$\left\{\vphantom{\frac{n}{1}}\right. \mkern-4mu
\mkern-2mu-\mkern-4mu\frac{q^*\mkern-6mu}{p}, \mkern-1mu\frac{p^*\mkern-7mu}{q}
\mkern-3mu\left.\vphantom{\frac{n}{1}}\right\}
\mkern-2mu=\mkern-2mu\left\{\frac{n}{1}, \frac{1}{0}\right\}$, $n\mkern-2mu\in\mkern-2mu\Z$.
The above Seifert structure 
%$S^3 = M_2(\mkern1mu{\textstyle{\frac{p^*\mkern-7mu}{q}, 
%\mkern-2mu-\mkern-1mu\frac{q^*\mkern-6mu}{p} }}\mkern1mu)$
exhibits $S^3$ as a union
\begin{align}
S^3 
&= \mkern18mu\nu(\lambda_{-1}) 
\mkern20mu\cup \mkern5mu 
\left(S^2 \mkern-2mu\setminus\mkern-2mu (D^2_0 \mkern1mu\amalg\mkern-1mu D^2_{-1})\vphantom{A^a}\right)
\mkern-3mu\times\mkern-2mu S^1
\mkern6mu\cup \mkern23mu 
\nu(\lambda_{0}) 
     \\
&=
(D^2_{-1} \mkern-1mu\times\mkern-2mu S^1)
\mkern5mu\cup\mkern46mu
[-\epsilon, +\epsilon] \mkern-1mu\times\mkern-1mu \mathbb{T}^2
\mkern40mu\cup\mkern8mu
(D^2_{0} \mkern-2mu\times\mkern-2mu S^1),
\end{align}
where $\lambda_{-1}$ and $\lambda_0$ are
exceptional fibers of 
meridian-slopes 
$[\mu_{-1}]_{\textsc{sf}} \mkern-2mu=\mkern-2mu -\frac{q^*\mkern-6mu}{p}$ and
$[\mu_0]_{\textsc{sf}} \mkern-2mu=\mkern-2mu \frac{p^*\mkern-7mu}{q}$,
respectively,
forming a Hopf link $\lambda_{-1}$, $\lambda_0 \mkern-2mu\subset\mkern-2mu S^3$.

Regular fibers in this Seifert fibration are confined to
some neighborhood $[-\epsilon, +\epsilon] \mkern-1mu\times\mkern-1mu \mathbb{T}^2$
of a torus $\mathbb{T}^2$, and they foliate this 
$\mathbb{T}^2$ with fibers all of the same slope.
Since $\lambda_{-1}$ and $\lambda_0$ are of multiplicities $p$ and $q$, respectively,
any regular fiber $f$ wraps 
$p$ times around the core $\lambda_{-1}$ of the solid torus neighborhood $\nu(\lambda_{-1})$,
and wraps $q$ times around the core $\lambda_0$ of 
$\nu(\lambda_{0})$, or equivalently, winds $q$ times {\em{along}} the core of $\nu(\lambda_{-1})$.
That is, any regular fiber $f$ is a $(p,q)$ curve in the boundary
$\mathbb{T}^2 = \partial\nu(\lambda_{-1})$ of the solid torus $\partial\nu(\lambda_{-1})$
of core $\lambda_{-1}$.
(See Proposition~\ref{prop: make satellite} for a more careful treatment of 
framings and orientations.)
Thus,  any collection
$f_1, \ldots, f_n$
of regular fibers in 
$M_{S^2}(\mkern1mu{\textstyle{
\mkern-2mu-\mkern-1mu\frac{q^*\mkern-6mu}{p},
\mkern-1mu\frac{p^*\mkern-7mu}{q}}}\mkern1mu)$
allows us to realize the exterior
\begin{equation}
\label{eq: T(np,nq) as reg fibers in Seifert fibered space}
Y^{\mkern1mu n}_{(p,q)}
\mkern1mu:=\mkern1.5mu
M_{S^2}(\mkern1mu{\textstyle{
\mkern-2mu-\mkern-1mu\frac{q^*\mkern-6mu}{p},
\mkern-1mu\frac{p^*\mkern-7mu}{q}}}\mkern1mu)
\mkern-1mu\setminus\mkern-1mu \overset{\mkern1mu\circ}{\nu}
\mkern-2mu\left(\vphantom{A^A}
\mkern1mu{\textstyle{\coprod_{i=1}^{\mkern1mu n} \mkern-2muf_i}}\mkern1mu\right)
\mkern2mu=:\mkern1.5mu
M_{S^2}(\mkern1mu{\textstyle{
\mkern-2mu-\mkern-1mu\frac{q^*\mkern-6mu}{p},
\mkern-1mu\frac{p^*\mkern-7mu}{q}}}, \overset{1}{\hat{*}}, \ldots, \overset{n}{\hat{*}}\mkern1mu)
\mkern3mu=\mkern3mu
S^3
\setminus \overset{\mkern1mu\circ}{\nu}
(T(np,nq))
\end{equation}
of $T(np,nq) \mkern-2mu\subset\mkern-2mu S^3$.
As a link in the solid torus, $T(np,nq) \subset \nu(\lambda_{-1})$
inhabits the exterior
\begin{equation}
\label{eq: def of solid torus Ypq}
\hat{Y}_{(p,q)} 
\mkern1mu:=\mkern1.5mu
M_{S^2}(\mkern1mu{\textstyle{
\mkern-2mu-\mkern-1mu\frac{q^*\mkern-6mu}{p},
\mkern-1mu\frac{p^*\mkern-7mu}{q}}}\mkern1mu)
\mkern-1.2mu\setminus\mkern-1mu \overset{\circ}{\nu}(\lambda_0)
\mkern2mu=:\mkern1.5mu
M_{S^2}(\mkern1mu{\textstyle{
\mkern-2mu-\mkern-1mu\frac{q^*\mkern-6mu}{p}}},
\overset{0}{\hat{*}}\mkern1mu)
\mkern3mu=\mkern3mu
\nu(\lambda_{-1})
\end{equation}
of the fiber $\lambda_0$ of meridian-slope 
%$[\mu_0]_{\textsc{sf}} \mkern-2mu=$
$\frac{p^*\mkern-7mu}{q}$.
This solid-torus link
$T(np,nq) \subset \hat{Y}_{(p,q)}$ then
has exterior
\begin{equation}
\label{eq: def of hat Y n p q}
\hat{Y}_{(p,q)}^{\mkern1mu n} 
\mkern1mu:=\mkern1.5mu
\hat{Y}_{(p,q)} 
\mkern-1mu\setminus\mkern-1mu \overset{\mkern1mu\circ}{\nu}
\mkern-2mu\left(\vphantom{A^A}
\mkern1mu{\textstyle{\coprod_{i=1}^{\mkern1mu n} \mkern-2muf_i}}\mkern1mu\right)
\mkern2mu=:\mkern.5mu
M_{S^2}(\mkern1mu{\textstyle{
\mkern-2mu-\mkern-1mu\frac{q^*\mkern-6mu}{p}}},
\overset{0}{\hat{*}},
\overset{1}{\hat{*}}, \ldots, \overset{n}{\hat{*}}\mkern1mu)
\mkern3mu=\mkern3mu
\hat{Y}_{(p,q)} 
\mkern-3mu\setminus\mkern-1mu \overset{\mkern1mu\circ}{\nu}\mkern.5mu
(T(np,nq)).
\end{equation}

To make this association $(p,q) \mapsto \hat{Y}^n_{(p,q)}$
well defined, 
we adopt the following convention.
\begin{definition}
\label{def: p*q*}
To any $(p, q) \in \Z^2$ with $\gcd(p,q) = 1$,
we associate the pair $(p^*, q^*) \in \Z^2$:
\begin{equation}
(p, q) \mapsto (p^*, q^*) \in \Z^2,\;\;\;\;
p^*p - q^*q =1 ,\;\;\; q^* \in \{0, \ldots, p-1\},
\end{equation}
where we demand $p>0$ without loss of generality
(since $p=0$-satellites are unlinks).
\end{definition}

\subsection{Construction of satellites}
For a knot $K \mkern-2mu\subset\mkern-2mu M$ 
in a closed oriented 3-manifold $M$,
%Suppose $M$ is a compact oriented 3-manifold 
%with boundary $\partial M$  a possibly-empty disjoint union of tori
we define the $T(np,nq)$-torus-link satellite 
$K^{(np,nq)}\mkern-1mu \subset M$ 
to be the image of 
the torus link $T(np,nq)$ embedded in the boundary
of $\nu(K)$, composed with the inclusion $\nu(K) \mkern3mu \into \mkern3mu M$.

Thus, if we write $Y \mkern-2mu:=\mkern-2mu M \mkern-1.5mu\setminus\mkern-1.5mu 
\overset{\circ}{\nu}(K)$
for the exterior of 
$K \mkern-2mu\subset\mkern-2mu M$ 
and take $\hat{Y}^n_{(p,q)}$
as defined in
(\ref{eq: def of hat Y n p q}),
then for an appropriate choice of gluing map
$\bar{\varphi}: \partial Y \to -\partial_0 \hat{Y}^n_{(p,q)}$,
we expect the union
\begin{equation}
\label{eq: cable union construction}
Y^{(np,nq)} := Y \cup_{{\bar{\varphi}}}
\hat{Y}^n_{(p,q)},
\;\;\;\;
\bar{\varphi}: \partial Y \to -\partial_0 \hat{Y}^n_{(p,q)}
\end{equation}
to be the exterior of 
$K^{(np,nq)}\mkern-1mu \mkern-2mu\subset\mkern-2mu M$.

\begin{prop}
\label{prop: make satellite}
Suppose $p,q,n \in \Z$ with $n,p>0$ and $\gcd(p,q) = 1$.
Choose a surgery basis $(\mu, \lambda)$ for the boundary
homology $H_1(\partial Y; \Z)$ of the knot exterior 
$Y := M \setminus \overset{\circ}{\nu}(K)$,
and take $\hat{Y}^n_{(p,q)}$ as in
(\ref{eq: def of hat Y n p q}) and
$Y^{(np,nq)}$ as in (\ref{eq: cable union construction}).
%with $\mu_{-1} := -q^* \tilde{f}_{-1} - p \tilde{h}_{-1}$.
If the gluing map
$\bar{\varphi}: \partial Y \to -\partial_0 \hat{Y}^n_{(p,q)}$
induces the homomorphism
${{\bar{\varphi}}}_* : H_1(\partial Y; \Z) \to 
H_1(\partial_0 \hat{Y}^n_{(p,q)}; \Z)$,
\begin{equation}
\label{eq: homology def of bar varphi}
{{\bar{\varphi}}}_*(\mu) := -q^* \mkern-2mu\tilde{f}_0 +\mkern.5mu p\tilde{h}_0,\;\;
{{\bar{\varphi}}}_*(\lambda):= p^*\mkern-2.2mu\tilde{f}_0 -\mkern.7mu q\tilde{h}_0,
\end{equation}
on homology, and hence
the orientation-preserving linear fractional map
\begin{align}
{{\bar{\varphi}}}_*^{{\mathbb{P}}} :&\; \P(H_1(\partial Y; \Z))_{\text{surg}} 
\to \P(H_1(\partial_0 \hat{Y}^n_{(p,q)}; \Z))_{\textsc{sf}},
\nonumber
    \\
\label{eq: phi lin fract map on slopes for cable glue}
{{\bar{\varphi}}}_*^{{\mathbb{P}}}\mkern-3mu &\left(\frac{a}{b}\right)
= \frac{a q^* - b p^*}{a p - b q}
=\mkern1mu \frac{\mkern.5mu q^*\mkern-2.5mu}{p} - 
\frac{b}{p(a p-b q)}, 
\end{align}
on slopes, then $Y^{(np,nq)}_{\textsc{sf}}(\boldsymbol{0}) = M$, 
and $Y^{(np,nq)}$ is the exterior of 
%$M \mkern-1mu\setminus\mkern-1mu \overset{\circ}{\nu}(K^{(np,nq)})$
of the $T(np,nq)$ satellite 
$K^{(np,nq)} \subset\mkern-1mu M$
of $K \mkern-1mu\subset\mkern-1.5mu M$.
\end{prop}

\begin{proof}
The Dehn filling 
$Y^{\mkern-.8mu(np,nq)}_{\textsc{sf}}\mkern-.8mu(\boldsymbol{0})$
is given by the union
$Y^{\mkern-.8mu(np,nq)}_{\textsc{sf}}\mkern-.8mu(\boldsymbol{0})
\mkern-2mu=\mkern-2mu Y \cup_{{\bar{\varphi}}} \mkern-2mu \hat{Y}_{(p,q)}$,
for the solid torus
$\hat{Y}_{(p,q)} \mkern-2mu=\mkern-2mu M_{S^2}(\mkern1mu{\textstyle{
\mkern-2mu-\mkern-1mu\frac{q^*\mkern-6mu}{p}}},
*\mkern1mu)$
defined in
(\ref{eq: def of solid torus Ypq}).
The boundary of the compressing disk of $\hat{Y}_{(p,q)}$
is given by the rational longitude
$l = -\sum_{i=-1}^{-1} \frac{\beta_i}{\alpha_i} = \frac{q^*\mkern-2mu}{p}$
of $\hat{Y}_{(p,q)}$
(see the last line of 
Theorem~\ref{thm: l-space interval for seifert jsj}).
Thus, since $[\bar{\varphi}(\mu)]_{\textsc{sf}} = \frac{q^*\mkern-2mu}{p}$,
$Y^{\mkern-.8mu(np,nq)}_{\textsc{sf}}\mkern-.8mu(\boldsymbol{0})$
is in fact the Dehn filling 
$Y^{\mkern-.8mu(np,nq)}_{\textsc{sf}}\mkern-.8mu(\boldsymbol{0}) = Y(\mu) = M$.

Since $Y^{\mkern-.8mu(np,nq)} = M \setminus \coprod_{i=1}^n
\overset{\circ}{\nu}(f_i)$
is the exterior of $n$ regular fibers from
$\hat{Y}_{(p,q)} \mkern-2mu=\mkern-2mu M_{S^2}(\mkern1mu{\textstyle{
\mkern-2mu-\mkern-1mu\frac{q^*\mkern-6mu}{p}}},
*\mkern1mu)$
we must verify that our lift 
$\tilde{f}_0 \in H_1(\partial \hat{Y}_{(p, q)};\Z)$
of a regular fiber class to~the boundary
$\partial \hat{Y}_{(p,q)}
\mkern-2mu=\mkern-2mu 
\partial_0 \hat{Y}^n_{(p,q)}$
of the solid torus $\hat{Y}_{(p,q)}$
is represented by a $(p,q)$ torus knot
on $\partial \hat{Y}_{(p,q)}$
relative to the framing specified by
$\mu$ and $\lambda$.  Indeed, from
(\ref{eq: homology def of bar varphi}), we have
$$
\tilde{f}_0 = (pp^* - qq^*) \tilde{f}_0
= p{{\bar{\varphi}}}_*(\lambda) + q{{\bar{\varphi}}}_*(\mu),
$$
as required.  The induced map ${{\bar{\varphi}}}_*^{\P}$ on slopes
preserves orientation, because the map ${{\bar{\varphi}}}$
is orientation reversing, but the surgery basis
and Seifert fibered basis are positively
oriented and negatively oriented, respectively.
\end{proof}

\subsection{Computing L-space intervals}
The primary tool we shall use
is a result of the author which computes
the L-space interval for the exterior of 
a regular fiber in a closed 3-manifold
with a Seifert fibered JSJ component.
\begin{theorem}[S. Rasmussen (Theorem 1.6) \cite{lgraph}]
\label{thm: l-space interval for seifert jsj}
Suppose $M$ is a closed oriented 3-manifold
with some JSJ component $\hat{Y}$ which is Seifert
fibered over an $n_{\textsc{bi}}$-times-punctured $S^2$,
so that we may express $M$ as a union
\begin{align*}
M = \hat{Y} 
%&
\cup_{\boldsymbol{{{\bar{\varphi}}}}} 
{\textstyle{\coprod_{j=1}^{n_{\textsc{bi}}} Y_j}},
%\;\;\;\;\hat{Y} = 
%\hat{M}_{S^2 \setminus \coprod_{j=1}^{n_{\textsc{bi}}}
%\overset{\circ}{D^2_j}}(y_1, \ldots, y_m),
%    \\
\;\;\;\;\;{{\bar{\varphi}}}_j
%&
: \partial Y_j \to -\partial_j \hat{Y},
\end{align*}
where each $Y_j$ is boundary incompressible ({\em{i.e.}} is not
a solid torus or a connected sum thereof).
Write $(y_1, \ldots, y_m)$ for the Seifert slopes
of $\hat{Y}$, so that $\hat{Y}$ is the partial Dehn filling
of $S^1 \times (S^2 \setminus \coprod_{i=1}^{m+n_{\textsc{bi}}}
D^2_i)$
by $(y_1, \ldots, y_m)$ in our Seifert fibered basis.
Further suppose that each $Y_j$ is Floer simple,
so that we may write
\begin{equation*}
{{\bar{\varphi}}}^{\P}_*(\mathcal{L}(Y_j))_{\textsc{sf}} = 
[[y_{j-}^{\textsc{bi}}, y_{j+}^{\textsc{bi}}]]
\subset \P(H_1(\partial_j \hat{Y}; \Z))_{\textsc{sf}}
\end{equation*}
for each $j \in \{1, \ldots, n_{\textsc{bi}}\}$.
Let $Y$ denote the exterior 
$Y = M \setminus \overset{\circ}{\nu}(f)$
of a regular fiber $f \subset \hat{Y}$.
If  
$\mathcal{L}(Y)$ is nonempty, then
\vspace{-.1cm}
\begin{equation*}
\hphantom{\mkern50mu\text{where}}
\mathcal{L}_{\textsc{sf}}(Y) =
\begin{cases}
\{y_-\} = \{y_+\}
& Y \text{ Floer not simple}
    \\
[[y_-, y_+]]
& Y \text{ Floer simple},
\end{cases}
\mkern50mu\text{where}
\end{equation*}
%*** squash?

\vspace{-.3cm}

\begin{align*}
y_-
&:=\,
\sup_{k > 0}
-\mfrac{1}{k}\mkern-2.5mu\left( 1
+ {{\textstyle{\sum\limits_{i=0}^m  \left\lfloor y_i k \right\rfloor}}}
+ {{\textstyle{\sum\limits_{j=1}^{n_{{\textsc{bi}}}} \left(  \lceil 
y^{\textsc{bi}}_{j+}k 
\rceil - 1 \right)}}}
\right),
\\ \nonumber
y_+ 
&:=\,
\inf_{k > 0}
-\mfrac{1}{k}\mkern-2.5mu\left(\mkern-2mu -1
+ {{\textstyle{\sum\limits_{i=0}^m \left\lceil y_i k \right\rceil}}}
+ {{\textstyle{\sum\limits_{j=1}^{n_{{\textsc{bi}}}} \left(  \lfloor y^{\textsc{bi}}_{j-} k 
\rfloor + 1 \right)}}}
\right).
\end{align*}
The above extrema are realized for finite $k$
if and only if $Y$ is boundary incompressible.
When $Y$ is boundary compressible, $y_- = y_+ = l = -\sum_{i=1}^m y_i$ is the rational longitude of $Y$.
\end{theorem}
{\noindent{{\textbf{Remarks.}} In the above, we define $y_- := \infty$ or $y_+ := \infty$, respectively,
if any infinite terms appear as summands of $y_-$ or $y_+$, respectively.
For $x \in \R$, the notations
$\lfloor x \rfloor$ and
$\lceil x \rceil$ respectively
indicate the greatest integer less than
or equal to $x$ and the least integer
greater than or equal to $x$, as usual.
In addition, we always take $k$ to be an integer.
Thus the expression ``$k>0$'' always indicates $k \in \Z_{>0}$.
}}

\smallskip
In order to use the above theorem,
we first have to know whether
$\mathcal{L}(Y)$ is nonempty and whether
$Y$ is Floer simple.
The author provides a complete (and lengthy) answer 
to this question in \cite{lgraph}.
Here, we restrict to the cases of most relevance to the current question.
\begin{theorem}[S. Rasmussen \cite{lgraph}]
\label{thm: gluing structure theorem}
Assuming the hypotheses of 
Theorem
\ref{thm: l-space interval for seifert jsj},
set
\begin{align*}
n^{\infty} 
&:= \mkern1mu|\{i \in \{1, \ldots, n\}|\; y_i = \infty\}|,
    \\
n^{<}_{\textsc{bi}} 
&:= |\{j \in \{1, \ldots, n_{\textsc{bi}}\}\,|\; 
-\infty <y_{j-}^{\textsc{bi}} \!<\, y_{j+}^{\textsc{bi}} < +\infty \}|.
\end{align*}

{\noindent{If
$\infty \notin 
\{ y_{1-}^{\textsc{bi}}, y_{1+}^{\textsc{bi}},
\ldots,
y_{{n_{\textsc{bi}}}-}^{\textsc{bi}}, y_{{n_{\textsc{bi}}}+}^{\textsc{bi}}\}
$,
then the following are true.}}
\begin{itemize}
\item[(i)]
If $n^{\infty} > 1$,
then any Dehn filling of $Y$ is a connected sum
with $S^1 \times S^2$, and
$\mathcal{L}_{\textsc{sf}}(Y) = \emptyset$.

\item[(ii)]
If $n^{\infty} = 1$, then
$\mathcal{L}_{\textsc{sf}}(Y) = 
\begin{cases}
\left< -\infty, +\infty\right>
&
n^{<}_{\textsc{bi}} = 0
  \\
\emptyset
&
n^{<}_{\textsc{bi}} > 0
\end{cases}.
$

\item[(iii)]
If $n^{\infty} = 0$, then
$\mathcal{L}_{\textsc{sf}}(Y) =
\begin{cases}
[[y_-, y_+]] \text{ with } y_- \mkern-2mu> y_+
&\mkern7mu
n^{<}_{\textsc{bi}}=0
   \\
[y_-, y_+]
&\mkern7mu
n^{<}_{\textsc{bi}}=1 \mkern-2mu\text{ and } y_-<y_+
  \\
\{y_-\} = \{y_+\}
&\mkern6mu
n^{<}_{\textsc{bi}}=1 \mkern-2mu\text{ and } y_-=y_+
  \\
\emptyset
&\mkern7mu
n^{<}_{\textsc{bi}}=1 \mkern-2mu\text{ and } y_->y_+
  \\
\emptyset
&\mkern7mu
n^{<}_{\textsc{bi}}>1
\end{cases}\mkern-4mu.
$

\end{itemize}

{\noindent{Suppose instead that
$\infty \in 
\{ y_{1-}^{\textsc{bi}}, y_{1+}^{\textsc{bi}},
\ldots,
y_{{n_{\textsc{bi}}}-}^{\textsc{bi}}, y_{{n_{\textsc{bi}}}+}^{\textsc{bi}}\}
$.}}
\begin{itemize}
\item[(iv)]
If either
$\infty \notin
\{ y_{1-}^{\textsc{bi}}, 
\ldots,
y_{{n_{\textsc{bi}}}-}^{\textsc{bi}}\}$ 
or
$\infty \notin \{
\{ y_{1+}^{\textsc{bi}}, 
\ldots,
y_{{n_{\textsc{bi}}}+}^{\textsc{bi}}\}$, then

$\mathcal{L}_{\textsc{sf}}(Y)
=
\begin{cases}
[[y_-, y_+]]
  &
\mkern1mu n^{<}_{\textsc{bi}} = 0
\text{ and }
n^{\infty} = 0
    \\
\left<-\infty, +\infty\right>
  &
\mkern1mu n^{<}_{\textsc{bi}} = 0
\text{ and }
n^{\infty} = 1
    \\
\emptyset
  &
\mkern1mu n^{<}_{\textsc{bi}} = 0
\text{ and }
n^{\infty} > 1
    \\
\emptyset
  &
\mkern1mu n^{<}_{\textsc{bi}} \neq 0
\end{cases}.
$
\end{itemize}

\end{theorem}

\begin{proof}
See Proposition 4.7 in \cite{lgraph}.
\end{proof}

To state the below theorem efficiently, we need to introduce one last notational convention.

\vskip.1cm

\noindent {\textbf{Notation.}} 
When the brackets $[\cdot]$ are applied to a real number, they always indicate the map
\begin{equation}
[\cdot]: \R \to \left[0, 1\right>,
[x]:= x - \lfloor x \rfloor.
\end{equation}
Note that the maps $\lfloor \cdot \rfloor$, $\lceil \cdot \rceil$, and $[\cdot]$
satisfy the useful identities,
\begin{equation}
-\lfloor -x \rfloor 
\mkern.5mu=\mkern.5mu 
\lceil x \rceil,\;\;\;\;
x 
\mkern.5mu=\mkern.5mu 
\lfloor x \rfloor 
\mkern-1.5mu+\mkern-1.5mu 
[x] 
\mkern.5mu=\mkern.5mu 
\lceil x \rceil 
\mkern-1.5mu-\mkern-1.5mu 
[-x]
\;\;\;\;
\text{for all }
x\in\R.
\end{equation}

We are now ready to
classify L-space surgeries on
torus-link satellites of
L-space knots.
\vspace{-.1cm}

\begin{theorem}
\label{thm: s3 satellite in sf basis}
Let $Y := S^3 \setminus 
\overset{\mkern2.5mu\circ}{\nu}(K)$
denote the exterior of a
positive L-space knot $K \subset S^3$ of genus $g(K)$,
and let 
$Y^{(np,nq)} := S^3 \setminus 
\overset{\mkern2.5mu\circ}{\nu}(K^{(np,nq)})$
denote the exterior of the
$(np,nq)$-torus-link satellite $K^{(np,nq)} \subset S^3$ of $K \subset S^3$,
for $n, p,q \in \Z$ with $n,p>0$, and
$\gcd(p,q) = 1$.

Construct
$\hat{Y}^n_{(p,n)}$,
$\hat{Y}_{(p,q)}$,
and 
$Y^{(np,nq)} := Y \cup_{{{\bar{\varphi}}}} \hat{Y}^n_{(p,q)}$
as in Proposition
\ref{prop: make satellite},
with the Seifert structure $\textsc{sf}$
on $\hat{Y}^n_{(p,q)}$
as specified by 
Proposition~\ref{prop: make satellite}.

$(\psi)$ There is a change of basis map
\begin{equation*}
\psi: 
(\Q \cup \{\infty\})^n_{\textsc{sf}} \longrightarrow (\Q \cup \{\infty\})^n_{S^3},\;\;
\boldsymbol{y} \mapsto
\left(pq + {\textstyle{\frac{1}{y_1}}}, \ldots, pq + {\textstyle{\frac{1}{y_n}}}\right),
\end{equation*}
which converts the above-specified Seifert basis slopes 
$(\Q \cup \{\infty\})^n_{\textsc{sf}}$
into conventional link exterior slopes in $S^3$, so that
$\mathcal{L}_{S^3}(Y^{(np,nq)})=
\psi(\mathcal{L}_{\textsc{sf}}(Y^{(np,nq)}))$.

$(i.a)$
If $N:= 2g(K)-1 > \frac{q}{p}$, $K \subset S^3$ is nontrivial, and $p>1$, then
$$
\mathcal{L}_{\textsc{sf}}(Y^{(np,nq)})
=
\Lambda_{\textsc{sf}}(Y^{(np,nq)})
:= \{ \boldsymbol{l} \in \Z^n \subset (\Q \cup \{\infty\})^n |\;
{\textstyle{\sum_{i=1}^n l_i}} = 0\}.
$$

$(i.b)$
If $2g(K)-1 > \frac{q}{p}$, $K \subset S^3$ is nontrivial, and $p=1$, then
$$
\mathcal{L}_{\textsc{sf}}(Y^{(np,nq)})
=
\Lambda_{\textsc{sf}}(Y^{(np,nq)}) +
\mathfrak{S}_n\left(
\left[0,{\textstyle{\frac{1}{N-q}}}\right] \times \{0\}^{n-1}
\right).
$$

$(ii)$
If $2g(K)\mkern-1mu-\mkern-1mu 1 \mkern-2mu\le\mkern-2mu \frac{q}{p}$ 
with $K \mkern-2mu\subset\mkern-4mu S^3\mkern-3mu$ nontrivial,
$\mkern-7mu$
or if $\mkern.5mup,\mkern-1muq \mkern-2.5mu>\mkern-3mu 1$ with $K \mkern-4mu\subset\mkern-4mu 
S^3\mkern-2mu$ trivial, then
$$
\mathcal{NL}_{\textsc{sf}}(Y^{(np,nq)})
= 
\mathcal{Z}_{\textsc{sf}}(Y^{(np,nq)})
\;\cup\;
\{\boldsymbol{y} \in \Q^n\,|\; y_+ < 0 < y_-\},\;\text{where}
$$
\begin{align*}
\mathcal{Z}_{\textsc{sf}}(Y^{(np,nq)})
&= \left\{ \boldsymbol{y} \in (\Q \cup \{\infty\})^n \mkern2mu \left|\mkern5.5mu
\#\{i \in \{1, \ldots, n\}| y_i = \infty\} > 1 \right.\right\},
 \\
y_- 
&= 
-{\textstyle{\sum_{i=1}^n\lfloor y_i \rfloor}},
    \\
y_+ 
&=
-{\textstyle{\sum_{i=1}^n\lceil y_i \rceil}}
-
\begin{cases}
\frac{1}{p+q-2g(K)p}
& \sum_{i=1}^n\lfloor[-y_i](p+q - 2g(K)p)\rfloor =0
    \\
0
& \text{otherwise}
\end{cases}.
\end{align*}

$(iii)$ If $K \subset S^3$ is the unknot, with $p=1$ and $q>0$, then
\begin{align}
\nonumber
\mathcal{L}_{\textsc{sf}\mkern-.2mu}(Y^{(n,nq)}) 
=
\mathcal{R}_{\textsc{sf}\mkern-.2mu}(Y^{(n,nq)}) 
\mkern-1mu\setminus\mkern-1mu 
\mathcal{Z}_{\textsc{sf}\mkern-.2mu}(\mkern1mu&\mkern-1mu Y^{(n,nq)})
 \mkern-1mu\cup\mkern-1mu
\{\boldsymbol{y} \mkern-1.5mu\in\mkern-1.5mu \Q^n \mkern.5mu|\mkern2mu 
M_{S^2\mkern-.7mu}( {\textstyle{\frac{1}{q}}}
\mkern-.2mu,\mkern-1.7mu\boldsymbol{y})
\text{ a SF L-space}
\},\mkern-4mu
    \\
\nonumber
\text{or equivalently,}
\mkern35mu
\mathcal{NL}_{\textsc{sf}}(Y^{(n,nq)}) 
&=
\mathcal{Z}_{\textsc{sf}}(Y^{(n,nq)}) \mkern7mu\cup
\hphantom{\text{or equivalently,}
\mkern35mu}
  \\
\left\{ 
\vphantom{1-\left\lceil {\textstyle{\frac{k}{q}}} \right\rceil}
\boldsymbol{y} \in \Q^n \mkern.5mu\right|\mkern1mu 
1\mkern-1mu-\mkern-3mu\left\lceil 
\mkern-1mu{\textstyle{\frac{k}{q}}}\mkern-1mu \right\rceil\mkern-2.6mu
-\mkern-3mu{\textstyle{\sum_{i=1}^n}} 
\lceil 
y_i k \rceil
&
<
0<
\left.
-1\mkern-1mu -\mkern-3.6mu\left\lfloor 
\mkern-1mu{\textstyle{\frac{k}{q}}}\mkern-1mu \right\rfloor\mkern-2.6mu
-\mkern-3mu{\textstyle{\sum_{i=1}^n}} \lfloor y_i k \rfloor
\mkern15mu\forall\mkern3mu
k \mkern-2.5mu\in\mkern-2.5mu \Z_{>0}
\right\}.
\nonumber
\end{align}
\end{theorem}

\noindent {\textbf{Remarks.}}
A knot $K \subset S^3$ is called a {\em{positive}} (respectively
{\em{negative}}) {\em{L-space knot}} 
if $K$ admits an L-space surgery for some finite
$S^3$-slope $m \mkern-2.5mu>\mkern-2.5mu 0$ (respectively 
$m \mkern-2.5mu<\mkern-2.5mu 0$).
Since $\mathcal{L}_{S^3}(Y^{(np,nq)}) = -\mathcal{L}_{S^3}(\bar{Y}^{(np,-nq)})$
for $\bar{Y}^{(np,-nq)}$, the $(np,-nq)$-torus-link
satellite of the mirror knot $\bar{K} \subset S^3$,
the above theorem and 
Theorem~\ref{thm: torus link satellite in s3, lambda action}
are easily adapted to satellites of negative L-space knots
or to negative torus links.
Any $p=0$ satellite is just the $n$-component unlink, with
$\mathcal{L}_{S^3} = \prod_{i=1}^n \left[-\infty, 0\right> \cup \left<0, +\infty\right]$.
Note that while
Theorem~\ref{thm: torus link satellite in s3, lambda action}
excludes the case of torus links proper
(satellites of the unknot) which are ``degenerate,'' {\em{i.e.}}, which have $1 \in \{p,q\}$,
this case is treated in $(iii)$ above, setting $p=1$ without loss of generality.
If $q=0$ in this case, we again have the $n$-component unlink.
For any nontrivial degenerate torus link,
part $(iii)$ above implies that 
the boundary of $\mathcal{L}_{\textsc{sf}}$ follows a piece-wise-constant chaotic pattern,
similar to the boundary of the region of Seifert fibered L-spaces.
This is unsurprising, since the irreducible
surgeries on $T(n,n)$ consist of all oriented
Seifert fibered spaces over $S^2$ of $n$ or fewer exceptional fibers.
Lastly, if
$K^{(np,nq)} \subset S^3$ is any nontrivial-torus-link satellite 
of a {\em{non}}-L-space knot in $S^3$,
then the L-space gluing result conjectured in \cite{lslope} for arbitrary closed oriented 3-manifolds with single-torus boundary---which the authors of \cite{HRW} have announced they expect to prove in the near future---would imply that $\mathcal{L}(Y^{(np,nq)}) = \Lambda(Y^{(np,nq)})$.

\textit{Proof of }$\mkern-3mu(\psi)$.
Let
$Y_j: = Y^{(np,nq)}_{\textsc{sf}}
(0, \ldots, 0, 
\overset{j}{\hat{*}}
%\overset{j}{*}
,0, \ldots, 0)$
denote the partial Dehn filling of $Y^{(np,nq)}$
which fills in all $n$ boundary components except 
$\partial Y_j = \partial_j Y^{(np,nq)}$
with regular fiber neighborhoods, so that,
by the definitions of 
$Y^{(np,nq)}$ and
$\hat{Y}^n_{(p,q)}$ in
(\ref{eq: cable union construction}) and
(\ref{eq: def of hat Y n p q}), we have
\vspace{-.3cm}
%***  squashing?
\begin{equation}
Y_j 
\mkern2mu=\mkern2mu 
Y \cup_{{\bar{\varphi}}} \hat{Y}^n_{(p,q)}
\mkern2mu=\mkern2mu 
Y \cup_{{\bar{\varphi}}} 
M_{S^2}(\mkern1mu{\textstyle{
\mkern-2mu-\mkern-1mu\frac{q^*\mkern-6mu}{p}}},
\overset{0}{\hat{*}},
0, \ldots,0, \overset{j}{\hat{*}},0\ldots, 0\mkern1mu),
\;\;\;\;
\bar{\varphi}: \partial Y \to -\partial_0 \hat{Y}^n_{(p,q)},
\end{equation}
with $\bar{\varphi}$ as defined in
Proposition \ref{prop: make satellite},
where we recall from Definition~\ref{def: p*q*} that
$pp^* \mkern-2mu-\mkern-1mu qq^* \mkern-2mu=\mkern-2mu 1$ with
$0 \mkern-.6mu\le\mkern-.6mu q^* \mkern-3mu<\mkern-.6mu p$.
Observing that ${Y_j}$ has the Dehn filling $Y_j(-\tilde{h}_j) = S^3$,
we take $\mu_j := -\tilde{h}_j \in H_1(\partial Y_j; \Z)$
for the meridian in our $S^3$ surgery basis for $H_1(\partial Y_j; \Z)$.

As shown, for example, in \cite{lgraph},
any homology class $\lambda_j \in H_1(\partial Y_j;\Z)$ representing the
rational longitude of $Y_j$ has $\textsc{sf}$-slope
$\pi_{\textsc{sf}}(\lambda_j)$
given by the negative sum of the $\textsc{sf}$-slope images 
of the rational
longitudes of the manifolds glued into the
boundary components of 
$\hat{Y}^n_{(p,q)}$,
plus the negative sum of Seifert-data slopes
(which are just the $\textsc{sf}$-slope images of the
rational longitudes 
of the corresponding fiber neighborhoods),
as follows:
\begin{equation}
\label{eq: calculation of longitude}
\pi_{\textsc{sf}}(\lambda_j)
=
-\left(-\mfrac{q^*}{p} + {{\bar{\varphi}}}_*^{{\mathbb{P}}}(\pi_{S^3\mkern-1mu}(\lambda))
+ {\textstyle{\sum_{i\neq j} 0}} \right)
=
-\left(-\mfrac{q^*}{p} + \mfrac{p^*}{q}\right)
= -\mfrac{1}{pq}.
\end{equation}
This uses the definition in
(\ref{eq: phi lin fract map on slopes for cable glue})
of the induced map
${{\bar{\varphi}}}_*^{{\mathbb{P}}} : \P(H_1(\partial Y; \Z))_{S^3} 
\to \P(H_1(\partial_0 \hat{Y}^n_{(p,q)} \Z))_{\textsc{sf}}$
on slopes, to calculate the slope
${{\bar{\varphi}}}^{\P}_*(\pi_{S^3\mkern-1mu}(\lambda)) 
=
{{\bar{\varphi}}}^{\P}_*(\pi_{S^3\mkern-1mu}(0)) =
\frac{p^*\mkern-4mu}{q}
\in \P(H_1(\partial_0 \hat{Y}^n_{(p,q)} \Z))_{\textsc{sf}}$.

To obtain $\mu_j  \cdot \lambda_j=1$,
we are constrained by the choice
$\mu_j \mkern-2mu=\mkern-2mu -h_j$ to select the representative
$\lambda_j \mkern-3mu:=\mkern-2.5mu \tilde{f}_j \mkern1mu+\mkern2mu pq\tilde{h}_j$
for the $\textsc{sf}$-slope $\pi_{\textsc{sf}}(\lambda_j) = -\frac{1}{pq}$.
The resulting homology change of basis
\begin{equation}
\mu_j \mapsto 0\tilde{f}_j -1 \tilde{h}_j,\;\; 
\lambda_j \mapsto 1\tilde{f}_j + pq \tilde{h}_j
\end{equation}
%Proposition \ref{prop: make satellite},
for $H_1(\partial_j Y^{(np,nq)}; \Z)$
then induces a map on slopes with inverse
\begin{equation}
\psi_j : 
\P(H_1(\partial_j Y^{(np,nq)}; \Z))_{\textsc{sf}} 
\to 
\P(H_1(\partial_j Y^{(np,nq)}; \Z))_{S^3},
\;\;\;\;
y_j \mapsto pq + {\textstyle{\frac{1}{y_j}}}.
\end{equation}

\smallskip
{\textit{Setup for}} $(i)$ {\textit{and}} $(ii)$.
We begin with the case in which
$K \subset S^3$ is nontrivial, so that its exterior 
$Y = S^3 \setminus \overset{\circ}{\nu}(K)$ is boundary incompressible.
It is easy to show (see ``example'' in \cite[Section 4]{lslope})
that such $Y$ has L-space interval
\begin{equation}
\mathcal{L}_{S^3}(Y) = [N, +\infty],\;\;\;\;
N:= \deg(\Delta(K)) - 1 = 2g(K) - 1,
\end{equation}
where $\Delta(K)$ and $g(K)$ are the Alexander polynomial and genus of $K$.
Writing 
\begin{equation}
{{\bar{\varphi}}}_*^{{\mathbb{P}}}(\mathcal{L}(Y))_{\textsc{sf}} 
= [[y_{0-}, y_{0+}]],
\end{equation}
we then use
(\ref{eq: phi lin fract map on slopes for cable glue})
to compute that
\begin{equation}
y_{0-}:= \mfrac{Nq^* - p^*}{Np - q}
= \mfrac{q^*}{p} + \mfrac{1}{p(q-Np)},
\;\;\;\;\;
y_{0+}:= \mfrac{q^*}{p}.
\end{equation}

For a given $\textsc{sf}$-slope
$\boldsymbol{y}:=(y_1, \ldots, y_n) \in (\Q \cup \mkern-2mu\{\infty\})^n_{\textsc{sf}}$,
we verify whether the Dehn filling
$Y^{(np,nq)}_{\textsc{sf}}(\boldsymbol{y})$
is an L-space
by examining the L-space interval, computed via
Theorem~\ref{thm: l-space interval for seifert jsj},
of a regular fiber exterior in $Y^{(np,nq)}_{\textsc{sf}}(\boldsymbol{y})$.
%$0 \in \P(H_1(\hat{Y}^{(np,nq)}; \Z))_{\textsc{sf}}$ lies 
%in the L-space interval of a regular fiber exterior in 
%$Y^{(np,nq)}_{\textsc{sf}}(\boldsymbol{y})$.
That is, if we let $\hat{Y}^{(np,nq)}$
denote the regular fiber exterior
\begin{equation}
\hat{Y}^{(np,nq)} := Y^{(np,nq)} \setminus \overset{\circ}{\nu}(f)
\end{equation}
for a regular fiber $f \subset \hat{Y}_{(p,q)}^n$,
then $Y^{(np,nq)}(\boldsymbol{y})$
is an L-space if and only if 
the meridional slope 
$0 \in \P(H_1(\hat{Y}^{(np,nq)}_{\textsc{sf}}(\boldsymbol{y}); \Z))_{\textsc{sf}}$ satisfies
$0 \in \mathcal{L}_{\textsc{sf}}(\hat{Y}^{(np,nq)}_{\textsc{sf}}(\boldsymbol{y}))$.
Since $Y$ is Floer simple and boundary incompressible,
Theorem~\ref{thm: l-space interval for seifert jsj}
tells us that if $\mathcal{L}(\hat{Y}^{(np,nq)}_{\textsc{sf}}(\boldsymbol{y}))$
is nonempty, then it is determined by 
$y_-, y_+\in 
\P(H_1(\hat{Y}^{(np,nq)}_{\textsc{sf}}(\boldsymbol{y}); \Z))_{\textsc{sf}}$,
where
\begin{align}
\label{eq: cabling def for y- and y_+}
&\;
y_- := \max_{k>0} y_-(k),\;\;\;\;\; y_+ := \min_{k>0} y_+(k),
\;\;\;\;\;\;\;
   \\
\label{eq: cabling def for y-}
y_-(k)
:=& -\mfrac{1}{k}\mkern-3mu\left(1 + 
\left\lfloor
\mkern-2mu-\mfrac{\mkern1.5muq^*\mkern-3.5mu}{p}k\mkern-1.5mu
\right\rfloor
+
{\textstyle{\sum_{i=1}^n \lfloor y_i k \rfloor}} +
(\left\lceil y_{0+}k\mkern-1.5mu
\right\rceil -1)
\right)
   \\
=& \mfrac{1}{k}\mkern-3mu\left( 
\left\lceil
\mkern-2mu\mfrac{\mkern1.5muq^*\mkern-3.5mu}{p}k\mkern-1.5mu
\right\rceil
-\left\lceil y_{0+}k\mkern-1.5mu
\right\rceil
-
{\textstyle{\sum_{i=1}^n \lfloor y_i k \rfloor}}
\right),  \nonumber
   \\
\label{eq: cabling def for y+}
y_+(k)
:=& -\mfrac{1}{k}\mkern-3mu\left(-1 + 
\left\lceil
\mkern-2mu-\mfrac{\mkern1.5muq^*\mkern-3.5mu}{p}k\mkern-1.5mu
\right\rceil
+
{\textstyle{\sum_{i=1}^n \lceil y_i k \rceil}} +
(\left\lfloor y_{0-}k\mkern-1.5mu
\right\rfloor +1)
\right)
   \\
=& \mfrac{1}{k}\mkern-3mu\left( 
\left\lfloor
\mkern-2mu\mfrac{\mkern1.5muq^*\mkern-3.5mu}{p}k\mkern-1.5mu
\right\rfloor
-\left\lfloor y_{0-}k\mkern-1.5mu
\right\rfloor
-
{\textstyle{\sum_{i=1}^n \lceil y_i k \rceil}}
\right).  \nonumber
\end{align}

Thus, since $y_{0+} = \frac{q^*}{p}$,
we have
\begin{align}
y_-(k)
=& -\mfrac{1}{k}
{\textstyle{\sum_{i=1}^n}}
\left\lfloor(\lfloor y_i \rfloor  +  [y_i] )k \right\rfloor
    \\
&=
-\mfrac{1}{k}
{\textstyle{\sum_{i=1}^n}}\lfloor[y_i] k \rfloor
-{\textstyle{\sum_{i=1}^n}}
\lfloor y_i \rfloor
  \nonumber
  \\
&\le
-{\textstyle{\sum_{i=1}^n}}
\lfloor y_i \rfloor \nonumber
\end{align}
for all $k > 0$, which, since 
$y_-(1) = -{\textstyle{\sum_{i=1}^n}}
\lfloor y_i \rfloor$,
implies 
\begin{equation}
y_- = -{\textstyle{\sum_{i=1}^n}}
\lfloor y_i \rfloor.
\end{equation}
For $y_+$, there are multiple cases.

{{{\textit{Proof of}} $(i)${\textit{:}} $N= 2g(K)-1 > \frac{q}{p}$.}}
Since $q-Np<0$, we have $y_{0-} < \frac{q^*}{p} = y_{0+}$,
which, by 
Theorem~\ref{thm: gluing structure theorem},
implies $\mathcal{L}(\hat{Y}^{(np,nq)}(\boldsymbol{y})) \mkern-1mu\neq\mkern-1mu \emptyset$
if and only if
$\boldsymbol{y} \mkern-1mu\in\mkern-1mu \Q^n$ and
$y_- \mkern-2mu\le\mkern-1mu y_+$.

{\textit{Case (a):}} $p >1$.  Since $0 < y_{0-} < \frac{q^*}{p} < 1$, we have
\begin{equation}
y_+
\;
\le
y_+(1)
=-{\textstyle{\sum_{i=1}^n \lceil y_i\rceil}},
\end{equation}
so that $y_- \le y_+$ if and only if
$\sum_{i=1}^n(\lceil y_i\rceil - \lfloor y_i\rfloor) \le 0$,
which, for $\boldsymbol{y} \in \Q^n$, occurs if and only if
$\boldsymbol{y} \in \Z^n$.
If $\boldsymbol{y} \in \Z^n$,
then $y_+(k) \le y_+(1)$ for all $k>0$, implying
\begin{equation}
y_+ = y_+(1) = -{\textstyle{\sum_{i=1}^n  y_i}} = y_-,
\end{equation}
so that
Theorem~\ref{thm: gluing structure theorem}
tells us
\begin{equation}
\mathcal{L}_{\textsc{sf}}(\hat{Y}^{(np,nq)}(\boldsymbol{y}))
= \{y_-\} = \{y_+\}
= \{-{\textstyle{\sum_{i=1}^n y_i}}\}.
\end{equation}
Thus, since
$Y^{(np,nq)}(\boldsymbol{y})$ is an L-space
if and only if 
$0 \in \mathcal{L}_{\textsc{sf}}(\hat{Y}^{(np,nq)}(\boldsymbol{y}))$,
we have
\begin{equation}
\mathcal{L}_{\textsc{sf}}(Y^{(np,nq)})
=
\Lambda_{\textsc{sf}}(Y^{(np,nq)})
:= \{ \boldsymbol{y} \in \Z^n \subset (\Q \cup \{\infty\})^n |\;
{\textstyle{\sum_{i=1}^n y_i}} = 0\}.
\end{equation}

{\textit{Case (b):}} $p =1$.  
In this case, $\frac{q^*}{p} = 0$ and $y_{0-} \mkern-1mu= -\frac{1}{N-q}$,
so that
\begin{equation}
y_+(k)
= \mfrac{1}{k}\left(\left\lceil\mfrac{k}{N-q}\right\rceil +
{\textstyle{\sum_{i=1}^n}}\lfloor [-y_i]k \rfloor \right)
-{\textstyle{\sum_{i=1}^n \lceil y_i\rceil}}.
\end{equation}
In particular, since $y_+(1) = 1 -{\textstyle{\sum_{i=1}^n \lceil y_i\rceil}}$,
the condition $y_- \le y_+ \le y_+(1)$ 
implies
\begin{equation}
\label{eq: case ib bound for ceil minus floor}
{\textstyle{\sum_{i=1}^n}}(\lceil y_i\rceil - \lfloor y_i\rfloor)
\le y_+ +{\textstyle{\sum_{i=1}^n}}\lceil y_i\rceil \le  1,
\end{equation}
which, for $\boldsymbol{y} \in \Q^n$,
occurs only if for some $j \in \{1, \ldots, n\}$
we have $y_i \in \Z$ for all $i \neq j$.
For such a $\boldsymbol{y}$, we then have
\begin{equation}
y_+(k)
= \mfrac{1}{k}\left(\left\lceil\mfrac{k}{N-q}\right\rceil +
\lfloor [-y_j]k \rfloor \right)
-{\textstyle{\sum_{i=1}^n \lceil y_i\rceil}}.
\end{equation}
For $k>0$, set
$s := \left\lceil\frac{k}{N-q} \right\rceil$
and write $k = s(N-q) - t$ with 
$0 \le t < N-q$.

If $[-y_i] \ge \frac{N-q-1}{N-q}$, then
\begin{equation}
y_+(k) +{\textstyle{\sum_{i=1}^n}}\lceil y_i\rceil
\ge \frac{s(N-q) - \lceil [-y_j]t\rceil}{s(N-q) -t} \ge 1
\end{equation}
for all $k >0$, which, since $y_+(1)+{\textstyle{\sum_{i=1}^n}}\lceil y_i\rceil = 1$, 
implies 
\begin{equation}
y_+ = 1 - {\textstyle{\sum_{i=1}^n}}\lceil y_i\rceil
= - {\textstyle{\sum_{i=1}^n}}\lfloor y_i\rfloor = y_-,
\end{equation}
so that $Y^{(np,nq)}(\boldsymbol{y})$
is an L-space if and only if
${\textstyle{\sum_{i=1}^n}}\lfloor y_i\rfloor = 0$.

If $[-y_i] < \frac{N-q-1}{N-q}$, then
\begin{equation}
y_+ 
+ {\textstyle{\sum_{i=1}^n}}\lceil y_i\rceil
\le
y_+(N-q) +{\textstyle{\sum_{i=1}^n}}\lceil y_i\rceil
\le \mfrac{N-q-1}{N-q} 
< 
\mkern1mu
1.
\end{equation}
The left half of (\ref{eq: case ib bound for ceil minus floor})
then tells us that
${\textstyle{\sum_{i=1}^n}}(\lceil y_i\rceil - \lfloor y_i\rfloor) < 1$,
implying  $\boldsymbol{y} \in \Z^n$ and
${\textstyle{\sum_{i=1}^n}\lceil y_i\rceil}
={\textstyle{\sum_{i=1}^n}\lfloor y_i\rfloor}$.
Thus, since $\Z \ni y_- \le y_+ < y_- +1$,
we have $y_- \le 0 \le y_+$ if and only if 
${\textstyle{\sum_{i=1}^n}}\lfloor y_i\rfloor = 0$.

In total, we have learned that 
$\boldsymbol{y} \in \mathcal{L}_{\textsc{sf}}(Y^{(np,nq)})$
if and only if ${\textstyle{\sum_{i=1}^n}}\lfloor y_i\rfloor = 0$
and there exists $j \in \{1, \ldots, n\}$
such that $y_i \in \Z$ for all $i \neq j$
and $[-y_j] \in \left[\frac{N-q-1}{N-q},1\right> \cup \{0\}$,
or equivalently, $[y_j] \in \left[0, \frac{1}{N-q}\right]$.
In other words,
\begin{align}
\mathcal{L}_{\textsc{sf}}(Y^{(np,nq)})
\,&=\, \mathfrak{S}_n \cdot
\left\{\left.\left[l_1+0, l_1 + {\textstyle{\frac{1}{N-q}}}\right] 
\times (l_2, \ldots, l_n) \right|
\boldsymbol{l} \in \Z^n,
\mkern3mu
{\textstyle{\sum_{i=1}^n l_i}} = 0
\right\}
   \\
&=\,
\Lambda_{\textsc{sf}}(Y^{(np,nq)}) 
\,+\, 
\mathfrak{S}_n \left(
\left[0, {\textstyle{\frac{1}{N-q}}}\right] \times \{0\}^{n-1}
\right).   
\mkern100mu\qed
\nonumber   
\end{align}

{\textit{Proof of}} $(ii)${\textit{:}} $N= 2g(K)-1 \le \frac{q}{p}$ {\textit{and}} $q>0$.
We divide this section into three cases:
$N = \frac{q}{p}$, $N < \frac{q}{p}$ with $K \subset S^3$ nontrivial,
and $K \subset S^3$ the unknot $p,q > 1$.

{\textit{Case}} $N = \frac{q}{p}$.
Here, $N>0$ implies $K \subset S^3$ is nontrivial,
and $Np = q$ implies $p=1$, so that
$y_{0+} = \frac{q^*}{p} = 0$
and $y_{0-} = -\frac{1}{N-q} = \infty$.  
Theorem~\ref{thm: gluing structure theorem}.iv
then implies that
$Z^{\infty}_{\textsc{sf}}(Y^{(np,nq)})
\subset
\mathcal{NL}_{\textsc{sf}}(Y^{(np,nq)})$,
but that
\begin{equation}
\label{eq: n_infty = 1 case}
((\Q \cup \{\infty\})^n \setminus \Q^n) \setminus
\mathcal{Z}_{\textsc{sf}}(Y^{(np,nq)})
\,\subset\, \mathcal{L}_{\textsc{sf}}(Y^{(np,nq)}),
\end{equation}
since these are the slopes with $n^{\infty} = 1$, and since
$0 \in \left<-\infty,+\infty\right>$.
For $\boldsymbol{y} \in \Q^n$,
Theorem~\ref{thm: gluing structure theorem}.iv
tells us
$\mathcal{L}_{\textsc{sf}}(\hat{Y}^{(np,nq)}(\boldsymbol{y})) = 
[[y_-, y_+]] =
\left[ -{\textstyle{\sum_{i=1}^n \lfloor y_i \rfloor}}, +\infty\right]$,
so that 
\begin{equation}
\label{eq: NL for N=q}
\mathcal{NL}_{\textsc{sf}}(Y^{(np,nq)})
=
\mathcal{Z}_{\textsc{sf}}(Y^{(np,nq)})
\cup
\{\boldsymbol{y}\in\Q^n\mkern1mu |\,
-\infty < 0 < y_-\}.
\end{equation}
Since $p+q -2g(K)p = q-Np = 0$, 
the definition of $y_+$ in part $(ii)$
of Theorem~\ref{thm: s3 satellite in sf basis}
makes $y_+ = -\infty$, and so 
Theorem~\ref{thm: s3 satellite in sf basis}.ii holds.

{\textit{Case}} $N < \frac{q}{p}$.
Here, 
$q-Np > 0$ implies
$0 \le y_{0+} = \frac{q^*}{p} < y_{0-} \le 1$.
Thus, since
$0 \in \left<-\infty, +\infty\right>$,
the $n^{\infty} = 1$ case of Theorem~\ref{thm: gluing structure theorem}.ii
implies (\ref{eq: n_infty = 1 case}) holds, whereas 
Theorem~\ref{thm: gluing structure theorem}.i
implies
$\mathcal{Z}_{\textsc{sf}}(Y^{(np,nq)}) \subset 
\mathcal{NL}_{\textsc{sf}}(Y^{(np,nq)})$.
For $\boldsymbol{y} \in \Q^n$,
Theorem~\ref{thm: gluing structure theorem}.iii
tells us that
$\mathcal{L}_{\textsc{sf}}(\hat{Y}^{(np,nq)}(\boldsymbol{y}))
= [[y_-, y_+]]$ with $y_- \ge y_+$,
so that $Y^{(np,nq)}(\boldsymbol{y})$ is {\em{not}} an L-space
if and only if $y_+ < 0 < y_-$.
We therefore have
\begin{equation}
\mathcal{NL}_{\textsc{sf}}(Y^{(np,nq)})
=\mathcal{Z}_{\textsc{sf}}(Y^{(np,nq)})
\;\cup\;
\{\boldsymbol{y} \in \Q^n\,|\; y_+ < 0 < y_-\},
\end{equation}
where, the definitions of $y_-$ and $y_+$
are appropriately adjusted in the case that
$K \subset S^3$ is the unknot.

{\textit{Case}} $N < \frac{q}{p}$
{\textit{with}} $K \subset S^3$ {\textit{nontrivial}}.
We already know that 
$y_- = -{\textstyle{\sum_{i=1}^n \lfloor y_i \rfloor}}$
when $K \subset S^3$ is nontrivial and 
$\boldsymbol{y} \in \Q^n$.  Thus, it remains to compute $y_+$
for $\boldsymbol{y} \in \Q^n$.

Since $x = \lceil x \rceil - [-x]$ 
for all $x \in \R$, we have
\begin{align}
y_+(k)
&= \mfrac{1}{k}\mkern-3mu\left(
\left\lfloor
\mkern-2mu\mfrac{\mkern1.5muq^*\mkern-3.5mu}{p}k\mkern-1.5mu
\right\rfloor
-
\left\lfloor
y_{0-}k
%\left(\mfrac{q^*}{p} + \mfrac{1}{p(q-Np)}\right)k
\right\rfloor
+
{\textstyle{\sum_{i=1}^n \lfloor [-y_i] k \rfloor}}
\right)
-{\textstyle{\sum_{i=1}^n \lceil y_i \rceil}}
\end{align}
for all $k > 0$.
Write $k = s(q-Np) + t$ for $s, t \in \Z_{\ge 0}$
with $s := \left\lfloor\frac{k}{q-Np}\right\rfloor$ and $t < q-Np$.
Using the facts that 
$q^*(q-Np) = q^*q-Npq^* = p^*p-1-Npq^* = p(p^*-Nq^*)-1$
and that
$\lfloor w \rfloor - \lfloor x\rfloor \ge \lfloor w-x\rfloor$
and $\lfloor-x\rfloor = -\lceil x \rceil$
for all $w,x \in \R$
(in (\ref{eq: approximate ybar for s3 satellite})),
we obtain
\begin{align}
\bar{y}_+(k) 
:=&
\mfrac{1}{k}\mkern-3mu\left(
\left\lfloor
\mkern-2mu\mfrac{\mkern1.5muq^*\mkern-3.5mu}{p}k\mkern-1.5mu
\right\rfloor
-
\left\lfloor
y_{0-}k
\right\rfloor
\right) \nonumber
   \\
=&
\mfrac{1}{k}\!\left(
\left\lfloor s(p^*-Nq^*) + \mfrac{-s + tq^*}{p}\right\rfloor
-\left\lfloor s(p^* - Nq^*) 
+ \mfrac{tq^* + t/(q-Np)}{p}
\right\rfloor
\right)
\nonumber
    \\
=&
\mfrac{1}{k}\!\left(
\left\lfloor -\mfrac{s}{p} + \mfrac{tq^*\!}{p} \right\rfloor
- \left\lfloor\mfrac{tq^*\!}{p}\right\rfloor\right)
\label{eq: simplified ybar for s3 satellite}
    \\
\ge&
\mfrac{-1}{s(q-Np) +t}\left\lceil \mfrac{s}{p}\right\rceil
\label{eq: approximate ybar for s3 satellite}
    \\
=&
-\mfrac{1}{q-Np}
+\mfrac{(s-\left\lceil s/p\right\rceil)(q-Np) + t}
{k(q-Np)}
\nonumber
    \\
\ge& 
-\mfrac{1}{q-Np}.
\label{eq: -1 over q-Np bound for ybar} 
\end{align}

If $\sum_{i=1}^n\lfloor[-y_i](q-Np)\rfloor = 0$,
then, writing $k=1(q-Np) + 0$, we can use line
(\ref{eq: simplified ybar for s3 satellite})
to compute $\bar{y}_+(q-Np)$, so that we obtain
\begin{equation}
y_+(q-Np)
=\bar{y}_+(q-Np) + 0
- {\textstyle{\sum_{i=1}^n\lceil y_i\rceil}}
=-\mfrac{1}{q-Np}
- {\textstyle{\sum_{i=1}^n\lceil y_i\rceil}}.
\end{equation}
Thus, since 
(\ref{eq: -1 over q-Np bound for ybar})
implies 
$y_+(k) \ge -\frac{1}{q-Np} 
- {\textstyle{\sum_{i=1}^n\lceil y_i\rceil}}$
for all $k>0$,
we conclude that
\begin{equation}
y_+ = -\mfrac{1}{q-Np} 
- {\textstyle{\sum_{i=1}^n\lceil y_i\rceil}}.
\end{equation}

On the other hand, if $\sum_{i=1}^n\lfloor[-y_i](q-Np)\rfloor > 0$,
then we know there exists $i_* \in \{1, \ldots, n\}$ for which
$y_{i_*} \ge (q-Np)^{-1}$.
Thus, writing
$k = s(q-Np) + t$ and using line
(\ref{eq: approximate ybar for s3 satellite}),
we obtain the lower bound
\begin{align}
y_+(k)
&\ge \bar{y}_+(k) + \mfrac{1}{k}
\left\lfloor \mfrac{1}{q-Np}k \right\rfloor
- {\textstyle{\sum_{i=1}^n\lceil y_i\rceil}}
\nonumber
   \\
&\ge -\mfrac{1}{k}\left\lceil \mfrac{s}{p}\right\rceil
+ \mfrac{1}{k}s
- {\textstyle{\sum_{i=1}^n\lceil y_i\rceil}}
\nonumber
   \\
&\ge
- {\textstyle{\sum_{i=1}^n\lceil y_i\rceil}}.
\end{align}
Since this bound is realized by
$y_+(1) = - {\textstyle{\sum_{i=1}^n\lceil y_i\rceil}}$,
we deduce that 
$y_+ = 
- {\textstyle{\sum_{i=1}^n\lceil y_i\rceil}}$.
Thus, since $q-Np = p+q - 2g(K)p$, the last line of
Theorem~\ref{thm: s3 satellite in sf basis}.ii holds.

{\textit{Case}} $N < \frac{q}{p}$
{\textit{with}} $p,q>1$ {\textit{and}} $K \subset S^3$ {\textit{the unknot}}.
Since $Y: = S^3 \setminus \overset{\mkern2mu\circ}{\nu}(K)$
satisfies
\begin{equation}
\mathcal{L}_{S^3}(Y) = [[0,0]] = \Q \cup \{\infty\} \setminus \{0\},
\end{equation}
we use 
(\ref{eq: phi lin fract map on slopes for cable glue})
to compute, for 
${{\bar{\varphi}}}_*^{{\mathbb{P}}}(\mathcal{L}(Y))_{\textsc{sf}} 
= [[y_{0-}, y_{0+}]]$, that
\begin{equation}
y_{0-} = y_{0+} \,= \mfrac{0q^* - 1p^*}{0p - 1q}
=\, \mfrac{p^*}{q} = \mfrac{q^*}{p} + \mfrac{1}{pq}.
\end{equation}

Thus, applying 
Theorem~\ref{thm: l-space interval for seifert jsj}
and mildly simplifying, we obtain that
\begin{align}
&y_- = \sup_{k>0} \,y_-(k)
\;\;\text{ and }\;\;
y_+ = \inf_{k>0} \,y_+(k),\;\;
\text{ for}
           \\
y_-(k)
:=&\, 
\mfrac{1}{k}\mkern-3mu\left(-1 + 
\left\lceil
\mkern-2mu\mfrac{\mkern1.5mu q^*\mkern-3.5mu}{p}k\mkern-1.5mu
\right\rceil
-
\left\lfloor
\mkern-2mu\mfrac{\mkern1.5mu p^*\mkern-3.5mu}{q}k\mkern-1.5mu
\right\rfloor
-
{\textstyle{\sum_{i=1}^n \lfloor [y_i] k \rfloor}}
\right) 
- {\textstyle{\sum_{i=1}^n \lfloor y_i \rfloor}},
        \\
y_+(k)
:=&\,
\mfrac{1}{k}\left(1 + \left\lfloor \mfrac{q^*}{p}k \right\rfloor
-\left\lceil\mfrac{p^*}{q}k\right\rceil
+
{\textstyle{\sum_{i=1}^n \lfloor [-y_i] k \rfloor}}
\right) 
-{\textstyle{\sum_{i=1}^n}} \lceil y_i \rceil.
\end{align}

For $y_-(k)$, we (again) obtain the bound
\begin{align}
y_{-\mkern-1mu}(k)
&= \mfrac{1}{k}\mkern-3mu\left( 
\left\lceil
\mkern-2mu\mfrac{\mkern1.5muq^*\mkern-3.5mu}{p}k\mkern-1.5mu
\right\rceil
\mkern1mu-\mkern1mu
\left(\left\lfloor \mkern-2.5mu\left(
\mkern-2mu\mfrac{\mkern1.5muq^*\mkern-3.5mu}{p}
\mkern-1.5mu+\mkern-1.5mu
\mfrac{1}{pq}
\right)\mkern-3mu
k\mkern-1.5mu
\right\rfloor 
\mkern-1mu+\mkern-1mu
1\right)
\mkern1mu-\mkern1mu
{\textstyle{\sum_{i=1}^n \lfloor \mkern-1.5mu[y_i] k\mkern-1.5mu \rfloor}}
\right) 
\mkern1mu-\mkern1mu
{\textstyle{\sum_{i=1}^n \lfloor y_i \rfloor}}  
\nonumber
     \\ 
&\le -{\textstyle{\sum_{i=1}^n \lfloor y_i \rfloor}}
\;\;\;\text{for all }k \in \Z_{>0},
\end{align}
which, for $p,q>\mkern-1mu1$ is realized by
$y_-(1) = -{\textstyle{\sum_{i=1}^n \lfloor y_i \rfloor}}$,
so that
$y_- =
-{\textstyle{\sum_{i=1}^n \lfloor y_i \rfloor}}
$.

To compute $y_+$,
we note that since $p, q \mkern-2mu>\mkern-2mu 1$, we can
invoke
Lemma~\ref{lemma: p+q bound} (below), so that
\begin{equation}
1 + \left\lfloor \mfrac{q^*}{p}k \right\rfloor
+ \left\lfloor \mfrac{k}{p+q} \right\rfloor
 \ge \left\lceil\mfrac{p^*}{q}k\right\rceil
\;\; \text{for all } k \in \Z_{>0}.
\end{equation}
(Here, we multiplied the original inequality by $k$
and then observed that the integer on the left hand side
must be bounded by an integer.)
In particular,
\begin{equation}
\label{eq: p+q bound for generic torus knot}
\mfrac{1}{k}\left(1 + \left\lfloor \mfrac{q^*}{p}k \right\rfloor
-\left\lceil\mfrac{p^*}{q}k\right\rceil
+ \left\lfloor \mfrac{k}{p+q} \right\rfloor
\right)
-{\textstyle{\sum_{i=1}^n}} \lceil y_i \rceil
\ge 
-{\textstyle{\sum_{i=1}^n}} \lceil y_i \rceil
\;\; \text{for all } k \in \Z_{>0}.
\end{equation}
Thus, when ${\textstyle{\sum_{i=1}^n}} \lfloor[-y_i](p+q)\rfloor > 0$,
so that at least one $y_i$ satisfies $[-y_i] \ge \frac{1}{p+q}$,
line (\ref{eq: p+q bound for generic torus knot})
tells us that 
$y_+(k) \ge 
-{\textstyle{\sum_{i=1}^n}} \lceil y_i \rceil$, a bound which is
realized by $y_+(1)$ when $p, q>1$.
On the other hand, if ${\textstyle{\sum_{i=1}^n}} \lfloor[-y_i](p+q)\rfloor = 0$,
then 
(\ref{eq: p+q bound for generic torus knot}) implies that
\begin{equation}
y_+(k) \ge 
-\mfrac{1}{k}\left\lfloor \mfrac{k}{p+q}\right\rfloor 
-{\textstyle{\sum_{i=1}^n}} \lceil y_i \rceil
\ge -\mfrac{1}{p+q}
-{\textstyle{\sum_{i=1}^n}} \lceil y_i \rceil
\;\;\text{for all }k \in \Z_{k>0},
\end{equation}
a bound which is realized by $y_+(p+q)$.
We therefore have
\begin{equation}
y_+ =
-{\textstyle{\sum_{i=1}^n}} \lceil y_i \rceil
-
\begin{cases}
\frac{1}{p+q}
&
{\textstyle{\sum_{i=1}^n}} \lfloor[-y_i](p+q)\rfloor =0
   \\
0
&
{\textstyle{\sum_{i=1}^n}} \lfloor[-y_i](p+q)\rfloor >0
\end{cases},
\end{equation}
completing the proof of part $(ii)$.  \qed

{\textit{Proof of }}$(iii)${\textit{:}}
$K \subset S^3$, $p=1$, $q>0$.
Here, we have the same case as above, but with $p=1$ and $q>0$,
implying $\frac{q^*}{p} = 0$ and $\frac{p^*}{q} = \frac{1}{q}$.
Thus, the Dehn filling $Y^{(n,nq)}_{\textsc{sf}}(\boldsymbol{y})$
is the Seifert fibered space $M_{S^2}(\frac{1}{q}, \boldsymbol{y})$,
and we have

\begin{align}
\mathcal{NL}_{\textsc{sf}}(Y^{(n,nq)}) 
&=
\mathcal{Z}_{\textsc{sf}}(Y^{(n,nq)}) \mkern4mu\cup
\left\{ \left.
%\vphantom{1-\left\lceil {\textstyle{\frac{k}{q}}} \right\rceil}
\boldsymbol{y} \in \Q^n \mkern.5mu\right|\mkern1mu 
y_+ < 0 < y_-
\right\}
    \\
&=
\mathcal{Z}_{\textsc{sf}}(Y^{(n,nq)}) \mkern7mu\cup
\nonumber
  \\
\hphantom{12.}
\left\{ 
\vphantom{1-\left\lceil {\textstyle{\frac{k}{q}}} \right\rceil}
\boldsymbol{y} \in \Q^n \mkern.5mu\right|\mkern1mu 
1\mkern-1mu-\mkern-3mu\left\lceil 
\mkern-1mu{\textstyle{\frac{k}{q}}}\mkern-1mu \right\rceil\mkern-2.6mu
-\mkern-3mu{\textstyle{\sum_{i=1}^n}} 
&\lceil 
 y_i k \rceil
<0<
\left.
-1\mkern-1mu -\mkern-3.6mu\left\lfloor 
\mkern-1mu{\textstyle{\frac{k}{q}}}\mkern-1mu \right\rfloor\mkern-2.6mu
-\mkern-3mu{\textstyle{\sum_{i=1}^n}} \lfloor y_i k \rfloor
\mkern15mu\forall\mkern3mu
k \mkern-2.5mu\in\mkern-2.5mu \Z_{>0}
\right\}.  
\mkern28mu
\qed
\nonumber
\end{align}

The above theorem leads to the following
\begin{cor}
Theorem~\ref{thm: torus link satellite in s3, lambda action}
holds.
\end{cor}
\begin{proof}
For the $p>1$ case of part $(i)$,
we simply replace $\Lambda_{\textsc{sf}}$ with $\Lambda_{S^3}$,
which contains the $S^3$-slope $(\infty, \ldots, \infty)$
in its orbit.
For part $(ii)$ and for the $p\mkern-2mu=\mkern-2mu 1$ case of part $(i)$,
it is straightforward to show that both the
expression in the bottom line of part $(ii)$ of
Theorem~\ref{thm: torus link satellite in s3, lambda action}
and the expression
$\mathfrak{S}_n\cdot\left([N, +\infty] \times \{\infty\}^{n-1}\right)$
in part $(i)$
contain fundamental domains (under the action of $\Lambda$)
of the respective L-space regions specified above.
\end{proof}

We now return to the lemma cited in the proof of 
Theorem~\ref{thm: s3 satellite in sf basis}.$ii$.
\begin{lemma}
\label{lemma: p+q bound}
If $p,q >1$, then 
$$
\mfrac{1}{k}
\left(
1 + \left\lfloor \mfrac{q^*}{p}k \right\rfloor
+ \left\lfloor \mfrac{k}{p+q} \right\rfloor
\right) \ge \mfrac{p^*}{q}
\;\;\text{for all } k \in \Z_{>0}.
$$
\end{lemma}
\begin{proof}
Using the notation
\begin{equation}
\label{eq: bracket p notation}
[x]_m := x - |m|\mkern-3.5mu\left\lfloor \mfrac{x}{|m|} \right\rfloor
\;\;\;
\text{for any }\;
x \in \R,\, m \in \Z_{\neq 0},
\end{equation}
we define $z(k) \in \Q$ for all $k \in \Z_{>0}$, as follows:
\begin{align}
z(k)
:=&\; 
k\left(
\mfrac{1}{k}\mkern-3mu
\left(
1 + \left\lfloor \mfrac{q^*}{p}k \right\rfloor
+ \left\lfloor \mfrac{k}{p+q} \right\rfloor
\right) - \mfrac{p^*}{q}
\right)
\nonumber
     \\
=&\;
1 + \mfrac{kqq^* - kpp^*}{pq} - \mfrac{[kq^*]_p}{p}
+
\left\lfloor \mfrac{k}{p+q} \right\rfloor
\nonumber
     \\
=&\;
\mfrac{[kq^{-1}]_p}{p} - \mfrac{k}{pq}
+
\left\lfloor \mfrac{k}{p+q} \right\rfloor
+
\begin{cases}
1 & [k]_p = 0
\\
0 & [k]_p \neq 0
\end{cases}
\label{eq: q inverse step}
     \\
=&\;
\mfrac{q[kq^{-1}]_p
\mkern1.5mu-\mkern1.5mu 
[k]_{p+q}}{pq}
\mkern4mu+\mkern3mu 
\left(\mfrac{pq-(p+q)}{pq} \right)\mkern-6mu
\left\lfloor \mfrac{k}{p+q} \right\rfloor 
+
\begin{cases}
1 & [k]_p = 0
\\
0 & [k]_p \neq 0,
\end{cases}
\label{eq: add and subtract p+q term}
\end{align}
where we note that 
\begin{equation}
pq-(p+q) 
\,=\, (p-1)(q-1)-1 
\,\ge\, 0
\;\;\text{ for  }\; p,q >1.
\end{equation}

We next claim that $z(k)\ge 0$ if $z(k-(p+q))\ge 0$.  First, for 
$[kq^{-1}]_p \notin \{0,1\}$, we have
\begin{align}
z(k)-z(k-(p+q)) 
&= \mfrac{q\cdot 1}{pq} + \mfrac{pq-(p+q)}{pq} >0,
\end{align}
so that $z(k) > z(k-(p+q)) \ge 0$.
If $[kq^{-1}]_p = 0$, implying $[k]_p = 0$, then 
line (\ref{eq: q inverse step})
gives
\begin{equation}
z(k)
=
- \mfrac{k}{pq}
\mkern2mu+\mkern1mu
\left\lfloor \mkern-1mu \mfrac{k}{p+q} \mkern-1mu \right\rfloor + 1
\ge \mfrac{k}{p+q} - \mfrac{k}{pq} \ge 0.
\end{equation}
This leaves us with the case in which $[kq^{-1}]_p = 1$, so that
line (\ref{eq: q inverse step})
yields
\begin{equation}
z(k) = 
\left\lfloor \mkern-1mu \mfrac{k}{p+q} \mkern-1mu \right\rfloor
- \mfrac{k-q}{pq}.
\end{equation}
Since $[kq^{-1}]_p = 1$ implies $k \equiv q \mkern2mu(\mod p)$,
we can write
$k = (sq + t)p + q$, with $s = \left\lfloor \mfrac{k-q}{pq} \right\rfloor$
and $t \in \{0 \ldots, q-1\}$.  When $t=0$, we obtain
\begin{equation}
z(k) = z(spq + q) = 
\left\lfloor \mkern-1mu \mfrac{spq + q}{p+q} \mkern-1mu \right\rfloor - s
=
\left\lfloor \mkern-1mu \mfrac{spq + q}{p+q} - s \mkern-1mu \right\rfloor
\ge 0\;\;\;
\text{ for all } p,q > 1.
\end{equation}
On the other hand, when $t \ge 1$, we have
\begin{equation}
z(k) \ge
\lfloor z(k) \rfloor 
=
\left\lfloor \mkern-1mu \mfrac{spq + tp + q}{p+q} \mkern-1mu \right\rfloor - (s+1)
=
\left\lfloor \mkern-1mu \mfrac{s(pq 
- (p \mkern-1.5mu + \mkern-1.5mu q))
\mkern2mu + \mkern2mu
(t-1)p}{p+q} \mkern-1mu \right\rfloor
\ge 0,
\end{equation}
completing the proof of our claim.
Since the case $k = p+q$ is subsumed in the case $[kq^{-1}]_p = 1$,
we also have $z(p+q) \ge 0$, and so by the induction,
it suffices to prove the lemma for $k< p+q$.

Suppose that $0 < k < p+q$ and $k \notin p\Z$ (since 
$z(k) \ge 0$ for $[k]_p = 0$),
so that we now have
\begin{equation}
 z(k) = \mfrac{1}{pq}\left(q[kq^{-1}]_p - k\right).
\end{equation}
Since $z(aq) =\frac{1}{pq}( q \cdot a - aq) = 0$ for $a \in \Z$,
we may also assume $k \notin q \Z$.
Now, the 
Chinese Remainder Theorem tells us that
\begin{equation}
k = \left[ q[kq^{-1}]_p + p[kp^{-1}]_q\right]_{pq},
\end{equation}
but since $0<k<p+q$ and $k \notin p\Z \cup q\Z$, we also have
\begin{equation}
k < p+q \le q[kq^{-1}]_p + p[kp^{-1}]_q < 2pq,
\end{equation}
%which requires that

\vspace{-.05cm}
%*** squash?

\begin{equation}
\text{requiring that}\;\;\;\;\;
q[kq^{-1}]_p + p[kp^{-1}]_q = k+ pq,
\hphantom{\text{requiring that}\;\;\;\;\;}
\end{equation}
%so that, at last, we have

\vspace{-.1cm}
%*** squash?

\begin{equation}
\mkern115mu
\text{so that}\;\;\;\;\;
 z(k) = \mfrac{1}{pq}\left(pq - p[kp^{-1}]_q\right) > 0.
\hphantom{\text{so that}\;\;\;\;\;}
\mkern130mu
\square\mkern-30mu
\end{equation}

\vspace{.1cm}
%*** space?

\section{L-space region topology}
\label{s: L-space region topology}

\subsection{Topologizing L-space regions}
To clarify the sense in which we interpret topological properties of L-space and non-L-space regions, we introduce the following notion.
\begin{definition}
For any subset 
$A \mkern-1mu\subset\mkern-2mu (\Q\cup\{\infty\})^n \mkern-.7mu\hookrightarrow \mkern-.5mu
(\R \cup \{\infty\})^n$ with complement 
$A^c := (\Q\cup\{\infty\})^n \setminus A$
and real closure $\overline{A} \subset (\R \cup \{\infty\})^n$,
we define the {\em{$\Q$-corrected $\R$-closure}}
%$A^{\R} \subset (\R \cup \{\infty\})^n$ of $A$ by
$$
A^{\lowR} \,:=\, \overline{A} \setminus A^c \;\subset\; (\lowR \cup \{\infty\})^n
$$
of $A$. Note that this implies $A^{\lowR} \cap (\Q\cup\{\infty\})^n=A$
\end{definition}
The $\Q$-corrected $\R$-closure is a particularly natural construction
for L-space regions, due to the following fact.

\begin{prop}
\label{prop: complements and closure}
If $\mathcal{L} \subset (\Q \cup \{\infty\})^n$ and 
$\mathcal{N}\mkern-2mu\mathcal{L} \subset (\Q \cup \{\infty\})^n$ are
the respective L-space and non-L-space regions for some compact oriented 
3-manifold $Y$ with $\partial Y = \coprod_{i=1}^n \mathbb{T}^2_i$, then
\begin{equation}
\mathcal{L}^{\lowR} \amalg \mathcal{N}\mkern-2mu\mathcal{L}^{\lowR} = (\R \cup \{\infty\})^n.
\end{equation}
\end{prop}
\begin{proof}
This is mostly due to the structure of L-space intervals (L-space regions for $n = 1$)
described in Section~\ref{s: L-space intervals and gluing}.
In particular, when $n=1$,
Proposition~\ref{prop: structure of L-space intervals} implies that
the pair
$(\mathcal{L}^{\lowR},\mathcal{N}\mkern-2mu\mathcal{L}^{\lowR})$ takes precisely one
of the following forms:
\begin{itemize}
\item[$(i)$] 
$\mathcal{L}^{\lowR} = \emptyset,
\;
\mathcal{N}\mkern-2mu\mathcal{L}^{\lowR}
= (\R \cup \mkern-2mu \{\infty\})^1$;

\item[$(ii)$]
$\mathcal{L}^{\lowR} = \{y\},
\;
\mathcal{N}\mkern-2mu\mathcal{L}^{\lowR}
= (\R \cup \mkern-2mu \{\infty\})^1 \setminus \{y\}$, for some 
$y \in (\Q \cup \mkern-2mu \{\infty\})^1$;

\item[$(iii)$]
$\mathcal{L}^{\lowR}
= (\R \cup \mkern-2mu \{\infty\})^1 \setminus \{l\},
\;
\mathcal{N}\mkern-2mu\mathcal{L}^{\lowR}
=\{l\}$,
for
$l \in (\Q \cup \mkern-2mu \{\infty\})^1$ the rational longitude;

\item[$(iv)$]
$\mathcal{L}^{\lowR} = 
I_{(y_-,y_+)},
\;
\mathcal{N}\mkern-2mu\mathcal{L}^{\lowR}
={I}^{\circ}_{(y_+,y_-)}$,
for some
$y_-, y_+ \in (\Q \cup \mkern-2mu \{\infty\})^1$ with $y_- \neq y_+$,
\end{itemize}

\vspace{-.08cm}

\noindent where $I_{(y_-,y_+)} \mkern-4mu\subset\mkern-2mu (\R \cup \mkern-2mu \{\infty\})^1$ denotes the
real closed interval with left-hand endpoint $y_-$ and right-hand endpoint $y_+$,
and $I^{\circ}_{(y_+,y_-)} \mkern-4mu\subset\mkern-2mu (\R \cup \mkern-2mu \{\infty\})^1$
is the interior in $(\R \cup \mkern-2mu \{\infty\})^1$ of $I_{(y_+,y_-)}$. In particular,
we always have 
$\mathcal{L}^{\lowR} \amalg 
\mathcal{N}\mkern-2mu\mathcal{L}^{\lowR}
= (\R \cup \mkern-2mu \{\infty\})^1$,
and each of $\mathcal{L}^{\lowR}$ and $\mathcal{N}\mkern-2mu\mathcal{L}^{\lowR}$
is either a single rational point, a single interval with rational endpoints,
empty, or the whole set.

Since intervals form a basis for the topology on $(\R \cup \{\infty\})^1$,
and since the 
product topology on $(\R \cup \{\infty\})^{k-1} \times (\R \cup \{\infty\})^1$
coincides with the usual topology on $(\R \cup \{\infty\})^k$ for any
$k\in \Z_{\ge 0}$, the proposition follows from 
induction on $n$.
%Section~\ref{s: L-space intervals and gluing}
%Definition~\ref{def: def of [[,]] notation}
%Proposition~\ref{prop: structure of L-space intervals}
%Definition~\ref{def: of interior of interval}
\end{proof}

%***unsquash?

\vspace{.02cm}

\subsection{L-space region topology for torus links}

There are five qualitatively different topologies possible for the L-space
region of a torus-link-satellite  of a knot in $S^3$.

\begin{theorem}
\label{thm: topology of torus link exterior L-space region}
For $n, p, q > \Z_{>0}$,
let	$K^{(np,nq)} \mkern-3mu\subset\mkern-3mu S^3$,
with exterior
$Y^{(np,nq)} := S^3 \setminus \overset{\circ}{\nu}(K^{(np,nq)})$, 
be the $T(np,nq)$-satellite of a positive L-space knot $K \subset S^3$.
Associate
$\mathcal{L}$,
$\mathcal{N}\mkern-2mu\mathcal{L}$,
$\Lambda$, and
$\mathcal{B}$
to $Y^{(np,nq)}$ as usual, with 
$\mathcal{B}$ the set of rational longitudes of $\mkern1muY^{(np,nq)}$
as discussed in Section~\ref{ss: rational longitudes B}.
%Set $\mathcal{L} := \mathcal{L}(Y^{(np,nq)})$,
%$\mathcal{N}\mkern-2mu\mathcal{L} := \mathcal{N}\mkern-2mu\mathcal{L}(Y^{(np,nq)})$,
%$\Lambda := \Lambda(Y^{(np,nq)})$,
%$\mathcal{B} := \mathcal{B}(Y^{(np,nq)})$,

\begin{itemize}
\item[$(i.a)$]
If $2g(K)\mkern-2mu-\mkern-2mu 1 \mkern-2mu>\mkern-2mu \frac{q}{p}$
and $p\mkern-2mu>\mkern-2mu 1$, 
or if 
$2g(K)\mkern-2mu-\mkern-2mu 1 \mkern-2mu>\mkern-2mu \frac{q}{p} + 1$
and $p\mkern-2mu=\mkern-2mu 1$, 
\\
then $\mathcal{L}^{\lowR}$ deformation retracts onto $\Lambda$.

\item[$(i.b)$]
If $2g(K)\mkern-2mu-\mkern-2mu 1 \mkern-2mu=\mkern-2mu \frac{q}{p}+1$ and
$n \mkern-2mu>\mkern-2mu 2$, 

\vspace{-.16cm}

\noindent then 
$\mathcal{L}^{\lowR}$ is connected,
$\pi_1(\mathcal{L}_{\textsc{sf}}^{\lowR}) 
\mkern-1.5mu\simeq\mkern-1.5mu 
\mathrm{ker}(\delta)$
%\mkern1mu\delta\mkern-1mu:\mkern-1mu 
%\mathfrak{F}^{\binom{n}{2}} \mkern-2mu\to\mkern-1mu 
%\Lambda_{\textsc{sf}})\mkern-1mu$
as in {\em{(\ref{eq: def of delta})}}, and
$\mkern2mu\mathrm{rank}\, H_1(\mathcal{L}^{\lowR}) 
\mkern-2mu=\mkern-2mu \textstyle{\binom{n}{2}}-1$.

\item[($i.c)$]
If $2g(K)\mkern-2mu-\mkern-2mu 1 \mkern-2mu=\mkern-2mu \frac{q}{p}+1$ and
$n \in \{1,2\}$,
 \\
then $\dim(\mathcal{L}^{\lowR})=1$ and
$\mathcal{L}^{\lowR}$
is contractible.
%deformation retracts onto 
%$\{\mkern-1mu \boldsymbol{\infty}\mkern-1mu\} 
%\mkern-2mu\subset\mkern-2mu (\Q\cup\mkern-2mu \{\infty\})_{S^3}^n$.

\item[($ii.a)$]
If $2g(K)\mkern-2mu-\mkern-2mu 1 \mkern-2mu=\mkern-2mu \frac{q}{p}$,
 \\
then $\dim(\mathcal{L}^{\lowR})=n$ and
$\mathcal{L}^{\lowR}$
is contractible.
%deformation retracts onto 
%$\{\mkern-1mu \boldsymbol{\infty}\mkern-1mu\} 
%\mkern-2mu\subset\mkern-2mu (\Q\cup\mkern-2mu \{\infty\})_{S^3}^n$.

\item[$(ii.b)$]
If $2g(K)\mkern-2mu-\mkern-2mu 1 \mkern-2mu<\mkern-2mu \frac{q}{p}$,
including the case of $K^{(np,nq)} \mkern-1mu=\mkern-.5mu T(np,nq)$,
\\
then $\mathcal{N}\mkern-2.5mu\mathcal{L}^{\lowR}\mkern-1mu$ 
deformation retract onto the $(n\!-\!1)$-torus
$\mkern1mu\mathcal{B}^{\mkern.6mu\lowR} = 
\mathbb{T}^{n-1} \mkern-3mu\subset\mkern-2mu (\R \cup\{\infty\})^n 
= \mathbb{T}^n\mkern-1mu$, and
$\mathcal{L}^{\lowR}$
deformation retracts onto a $\mathbb{T}^{n-1}$ parallel to 
$\mathcal{B}^{\mkern.6mu\lowR}\mkern-2mu$.
\end{itemize}
\end{theorem}

{\em{Proof of $(i.a)$.}}
Since we already have $\mathcal{L} = \Lambda$ for $p>1$, assume that $p=1$.
Theorem~\ref{thm: s3 satellite in sf basis} then tells us that
$\mathcal{L}_{\textsc{sf}}^{\lowR} = \Lambda_{\textsc{sf}}
+ \mathcal{P}(N\mkern-1mu,\mkern1mu q)$, where
\begin{equation}
\label{eq: p=1 prong}
\mathcal{P}(N\mkern-1mu,\mkern1mu q)
:=
{\textstyle{\coprod_{i=1}^n}}\mkern-2mu
\left(\{0\}^{i-1}
\times
\left[0,{\textstyle{\frac{1}{N-q}}}\right] \times \{0\}^{n-i}
\right)^{\lowR} \subset (\R \cup \{\infty\})^n.
\end{equation}
Clearly $\mathcal{P}(N\mkern-1mu,\mkern1mu q)$ deformation retracts onto 
$\boldsymbol{0} \in (\Q \cup \{\infty\})^n$.
Since $2g(K)-1 > \frac{q}{p} + 1$, we have  
$\mathcal{P}(N\mkern-1mu,\mkern1mu q) \subset 
\coprod_{i=1}^n \mkern-4mu\left[0,\frac{1}{2}\right]^{\R}_i$.
Thus all of the translates
$\left\{\boldsymbol{l}
+ \mathcal{P}(N\mkern-1mu,\mkern1mu q)\right\}_{\boldsymbol{l} \in \Lambda_{\textsc{sf}}}$ are
pairwise disjoint, and 
$\mathcal{L}_{\textsc{sf}}^{\lowR}$ deformatioin retracts onto $\Lambda_{\textsc{sf}}$.
\qed

\vspace{.3cm}
{\em{Proof of $(i.b)$.}}
When $2g(K)-1 \mkern-2mu=\mkern-2mu \frac{q}{p} \mkern-2mu+\mkern-1mu 1$, and 
$n\mkern-2mu>\mkern-2mu 1$, 
line (\ref{eq: p=1 prong}) still holds, but~this~time~with
$\left[0,\frac{1}{N-q}\right] \mkern-2mu=\mkern-2mu [0,1]$.
To see that
$\pi_0(\mathcal{L}_{\textsc{sf}}^{\lowR})\mkern-2mu=\mkern-2mu0$,
 first note that
$\Lambda_{\textsc{sf}}$ is generated by the elements
\vspace{-.3cm}
\begin{equation}
\label{eq: def of epsilon ij and i}
{\boldsymbol{\epsilon}}_{ij}:=
{\boldsymbol{\epsilon}}_i - {\boldsymbol{\epsilon}}_j,\;\;
i,j\in\{1, \ldots, n\},
\;\;\;\text{with}\;\;\;
{\boldsymbol{\epsilon}}_i := (0, \ldots, 0, \overset{i}{1},0\ldots, 0) \in \Z^n
\end{equation}
the standard
basis element for $\Z^n$.
Then for any such ${\boldsymbol{\epsilon}}_{ij}$ and any 
$\boldsymbol{l} \in \Lambda_{\textsc{sf}}$,
the origin ${\boldsymbol{l}}$ of the translate
$\mathcal{P}_{\boldsymbol{l}} := \boldsymbol{l}+ \mathcal{P}(N\mkern-1mu,\mkern1mu q)$
is path-connected to the origin
${\boldsymbol{\epsilon}}_{ij} + \boldsymbol{l}$
of the translate
$\mathcal{P}_{{\boldsymbol{\epsilon}}_{ij}+\boldsymbol{l}}$,
via the path
$\gamma^{\mkern1mu\boldsymbol{l}}_{ij}(t): [0,1] \to 
\mathcal{L}_{\textsc{sf}}^{\lowR}$,
\vspace{-.15cm}
\begin{equation}
\gamma^{\mkern2mu\boldsymbol{l}}_{ij}\mkern-3mu
\left|\vphantom{A}_{\left[0,\frac{1}{2}\right]}\right. 
\mkern-6mu(t) = \boldsymbol{l} + 2t{\boldsymbol{\epsilon}}_{i},
\;\;\;\;
\gamma^{\mkern2mu\boldsymbol{l}}_{ij}\mkern-3mu
\left|\vphantom{A}_{\left[\frac{1}{2},1\right]}\right. 
\mkern-6mu(t) = ({\boldsymbol{\epsilon}}_{ij} + \boldsymbol{l}) + 2(1-t){\boldsymbol{\epsilon}}_j,
\end{equation}

\vspace{-.17cm}

\begin{equation}
\boldsymbol{l}
\mkern20mu\leadsto\mkern20mu
\{(\boldsymbol{l}) + {\boldsymbol{\epsilon}}_i\} 
=
\mathcal{P}_{\boldsymbol{l}}
\cap
\mathcal{P}_{{\boldsymbol{\epsilon}}_{ij}+\boldsymbol{l}}
= 
\{({\boldsymbol{\epsilon}}_{ij} + \boldsymbol{l}) + {\boldsymbol{\epsilon}}_j\}
\mkern20mu\leadsto\mkern20mu
{\boldsymbol{\epsilon}}_{ij} + \boldsymbol{l}.
\end{equation}
Thus 
$\mathcal{L}_{\textsc{sf}}^{\lowR}$
is path connected (hence connected),
and in fact, these basic paths
$\gamma^{\boldsymbol{l}}_{ij}$
from
${\boldsymbol{l}}\in\mathcal{P}_{\boldsymbol{l}}$
to
${\boldsymbol{\epsilon}}_{ij}+\boldsymbol{l} \in \mathcal{P}_{{\boldsymbol{\epsilon}}_{ij}+\boldsymbol{l}}$
%freely generate the groupoid of homotopy classes of paths in 
generate the groupoid $G$ of homotopy classes of paths in
$\mathcal{L}_{\textsc{sf}}^{\lowR}$ between elements of $\Lambda_{\textsc{sf}}$.
Let $G_0 \subset G$ denote the subset of homotopy clases of paths starting 
at $\boldsymbol{0}$, so that elements of $G_0$ are uniquely represented by
reduced words
\begin{equation}
g =
\left(\gamma^{\mkern2mu\boldsymbol{l}_0(i_0,j_0,e_0)}_{i_0 j_0}\right)^{\mkern-4mu e_0}\mkern-6mu
\left(\gamma^{\mkern2mu\boldsymbol{l}_1}_{i_1 j_1}\right)^{\mkern-3mu e_1} \mkern-9mu\ldots 
\left(\gamma^{\mkern2mu\boldsymbol{l}_m}_{i_m j_m}\right)^{\mkern-3mu e_m} \mkern-5mu\in G,
\;\;\;\;
i_k \mkern-2mu<\mkern-2mu j_k,\;\,
e_k \mkern-2mu\in\mkern-2mu \{\pm1\}, \;\,
l_{k+1} \mkern-2mu=\mkern-2mu l_k \mkern-1mu+\mkern-1mu {\boldsymbol{\epsilon}}_{i_k j_k} \;\,\forall\;k,
\mkern-10mu
\end{equation}
with right-multiplication corresponding to concatenation of paths.
Note that we have replaced
\begin{equation}
\label{eq: l_0}
\gamma^{\mkern2mu\boldsymbol{l}}_{ji}
\mapsto \left(\gamma^{\mkern2mu\boldsymbol{l}-{\boldsymbol{\epsilon}}_{ij}}\right)_{ij}^{\mkern-2mu-1}
\;\forall\; i<j,
\;\;\;
\text{and defined}\;\;
\boldsymbol{l}_0(i,j,e) := 
\begin{cases}
0
& 
e = +1
 \\
-{\boldsymbol{\epsilon}}_{i j}
& 
e = -1
\end{cases},
\end{equation}
recalling that ${\boldsymbol{\epsilon}}_{ji} = -{\boldsymbol{\epsilon}}_{ij}$.

If we introduce the free group  
$\mathfrak{F}^{\binom{n}{2}}$ 
%on $\binom{n}{2}$ generators,
%generated on the symbols ${x}_{ij}$, 
and epimorphism
$\delta: \mathfrak{F}^{\binom{n}{2}} \to \Lambda_{\textsc{sf}}$,
\begin{equation}
\label{eq: def of delta}
\mathfrak{F}^{\binom{n}{2}}:=\left<{x}_{ij}\right>_{1\le i\mkern1mu<\mkern1mu j\le n}
%^{(\textsc{free})},
\;\;\;\;\;\;
\delta: \mathfrak{F}^{\binom{n}{2}} \to \Lambda_{\textsc{sf}},\;\;\;
\delta:{x}_{ij} \mapsto {\boldsymbol{\epsilon}}_{ij},
\end{equation}
then a straightforward inductive argument on word length shows that the forgetful map
\begin{equation}
\rho : G_0 \to \mathfrak{F}^{\binom{n}{2}},
\;\;\;
\gamma^{\mkern2mu\boldsymbol{l}}_{i j} \mapsto {x}_{ij},
\end{equation}
on words is invertible.
In particular, starting with $\rho^{-1}(1) = 1$, we can use the inductive rule
\begin{equation}
\rho^{-1}\left({w} \cdot ({x}_{ij})^e\right) = 
\rho^{-1}({w}) \cdot
\left(\gamma^{\mkern2mu \delta({w}) \,+\, 
\boldsymbol{l}_0(i,\mkern1mu j, \mkern1mu e)}_{i j}
\right)^{\mkern-3mu e}
\end{equation}
for any $i \mkern-2mu<\mkern-2mu j$, $e \mkern-2mu\in\mkern-2mu \{\pm 1\}$,
and word ${w} \mkern-2mu\in\mkern-2mu \left<{x}_{ij}\right>_{i<j}$
with known $\rho^{-1}({w})$, to reconstruct the map $\rho^{-1}$.
Thus, $G_0$ inherits the structure of a free group on $\binom{n}{2}$ generators.
Since $\delta\circ \rho(g)$ is the endpoint of any path $g \in G_0$, we then have

\begin{equation}
\mkern95mu
\pi_1(\mathcal{L}_{\textsc{sf}}^{\lowR}) \simeq 
\mathrm{ker}\left(\delta: \mathfrak{F}^{\binom{n}{2}} \to \Lambda_{\textsc{sf}}\right),\;\;\;
\dim H_1(\mathcal{L}_{\textsc{sf}}^{\lowR}) = \textstyle{\binom{n}{2}}-1.
\mkern95mu
\square\mkern-30mu
\end{equation}

\vspace{.18cm}

{\em{Proof of $(i.c)$.}}
When $n\mkern-2mu=\mkern-2mu 2$, the discussion in $(i.b)$ still holds, so 
$\pi_0(\mathcal{L}) \mkern-2mu=\mkern-2mu 0$ and
$\pi_1(\mathcal{L}) \mkern-2mu=\mkern-2mu \ker \delta \mkern-1mu=\mkern-1mu 1$. 
Thus, just as for $n=1$, we have 
$\mathcal{L}$ contractible and of dimension 1.
\qed

\vspace{.3cm}

{\em{Proof of $(ii.a)$.}}
Case
$N = \frac{q}{p}$.
of part $(ii)$ of the proof of
Theorem~\ref{thm: s3 satellite in sf basis} tells us in
(\ref{eq: NL for N=q})
that
\begin{equation}
\mathcal{NL}_{\textsc{sf}}
=
\mathcal{Z}_{\textsc{sf}}
\cup
\{\boldsymbol{y}\in\Q^n_{\textsc{sf}}\mkern1mu |\,
-\infty < 0 < y_-(\boldsymbol{y})\},
\end{equation}
which, given the definition of $y_-(\boldsymbol{y})$ in the theorem statement, implies that
\begin{equation}
\mathcal{L}_{\textsc{sf}}^{\lowR}
=
\mathcal{L}^{\lowR +}_{\textsc{sf}}
\;\cup\;
\mathcal{R}_{\textsc{sf}}^{\lowR} \setminus
\mathcal{Z}_{\textsc{sf}}^{\lowR};
\;\;\;
\mathcal{L}^{\lowR +}_{\textsc{sf}}
:=
\{ \boldsymbol{y} \in \R^n_{\textsc{sf}}
|\mkern1mu {\textstyle{\sum_{i=1}^n \lfloor y_i \rfloor }}  \ge 0\}.
\end{equation}
Since
$\mathcal{L}^{\lowR +}_{\textsc{sf}}\mkern-2mu$, of dimension $n$,
is contractible, with
$\mkern1mu\mathcal{R}_{\textsc{sf}}^{\lowR} \setminus \mathcal{Z}_{\textsc{sf}}^{\lowR}\mkern1mu$
inside its boundary, we are done.
\qed

\vspace{.25cm}

{\em{Proof of $(ii.b)$.}}
Since the result is trivial for $n=1$, we henceforth assume $n>1$.
Moreover, $2g(K)-1 < \frac{q}{p}$ implies $q>0$
unless $K$ is the unknot, in which case we can take the mirror of $K$
if $q<0$. Thus we also assume $q >0$ without
loss of generality.
The statement and proof of 
Theorem~\ref{thm: s3 satellite in sf basis} then tell us that
\begin{equation}
\label{eq: NL for top}
\mathcal{NL}_{\textsc{sf}}
=
\mathcal{Z}_{\textsc{sf}}
\cup
\mathcal{N},
\;\;\;
\mathcal{N}:=
\{\boldsymbol{y}\in\Q^n_{\textsc{sf}}\mkern1mu |\,
y_+(\boldsymbol{y}) < 0 < y_-(\boldsymbol{y})\},\;\;\text{with}
\end{equation}
\begin{align}
y_-(\boldsymbol{y}) := \max_{k>0} y_-(\boldsymbol{y},k),
\;\;&\;\;\;
y_-(\boldsymbol{y},k)
:= -\mfrac{1}{k}{\textstyle{\sum_{i=1}^n \lfloor y_i k \rfloor}} -c_-(k)
    \\
y_+(\boldsymbol{y}) := \min_{k>0} y_+(\boldsymbol{y},k),
\;\;&\;\;\;
y_+(\boldsymbol{y},k)
:= -\mfrac{1}{k}{\textstyle{\sum_{i=1}^n \lceil y_i k \rceil}} -c_+(k)
\end{align}
for all $\boldsymbol{y} \in \Q^n_{\textsc{sf}}$ and
for certain 
$c_-(k), c_+(k) \in \frac{1}{k}\Z$ bounded above and below 
by linear functions in $k$, and
determined by $p$, $q$, and $2g(K)-1$,
and on whether $K \subset S^3$ is trivial.
In particular, each of $c_-(k)$ and $c_+(k)$ 
are independent of $\boldsymbol{y} \in \Q^n_{\textsc{sf}}$.

%*** space?

\vspace{.03cm}

Since
$\mathcal{Z}_{\textsc{sf}} 
= \mathcal{B}_{\textsc{sf}} \setminus
(\mathcal{B}_{\textsc{sf}} \cap \Q^n_{\textsc{sf}})$,
implying
$\mathcal{Z}^{\lowR}_{\textsc{sf}} 
= \mathcal{B}^{\R}_{\textsc{sf}} \setminus
(\mathcal{B}^{\R}_{\textsc{sf}} \cap \R^n_{\textsc{sf}})$,
it remains to construct a deformation retraction from 
$\mathcal{N}^{\lowR}$ to 
$\mathcal{B}^{\lowR}_{\textsc{sf}} \cap \R^n_{\textsc{sf}} 
\mkern1mu\subset\mkern1mu \mathcal{N}^{\lowR}$.
Toward that end, we define
\begin{equation}
\boldsymbol{1}:= (1,\ldots, 1) \in \Z^n_{\textsc{sf}},
\;\;\;\;
l(\boldsymbol{y}) := 
{\textstyle{ -\frac{1}{pq} - \sum_{i=1}^n y_i }} \,\in \R^n_{\textsc{sf}},
\end{equation}
so that for
$\boldsymbol{y} \mkern-3mu\in\mkern-3mu \Q^n_{\textsc{sf}},\mkern-.5mu$ 
$l(\boldsymbol{y})$ is the rational longitude of
the exterior 
$\hat{Y}^{(np,nq)}(\boldsymbol{y}) 
\mkern-2.5mu:=\mkern-2.5mu 
Y^{(np,nq)}(\boldsymbol{y}) \mkern-1.5mu\setminus\mkern-.5mu
\overset{\circ}{\nu}(f)$ 
of a regular fiber $f$ in $Y^{(np,nq)}(\boldsymbol{y})$.
We then claim that the homotopy
\begin{equation}
\boldsymbol{z} : [0,1] \times  \mathcal{N}^{\lowR} \to \R^n_{\textsc{sf}},
\;\;\;
\boldsymbol{z}_t(\boldsymbol{y}) := 
\boldsymbol{y} + t \cdot {\textstyle{\frac{1}{n}}}l(\boldsymbol{y})\boldsymbol{1}
\end{equation}
provides a deformation retraction from 
$\mathcal{N}^{\lowR}$ to 
$\mathcal{B}^{\lowR}_{\textsc{sf}} \cap \R^n_{\textsc{sf}} 
\mkern1mu\subset\mkern1mu \mathcal{N}^{\lowR}$.

First, note that
(\ref{eq: NL for top}) also implies that
$\mathcal{NL}_{\textsc{sf}}(\hat{Y}^{(np,nq)}(\boldsymbol{z})) = 
\left<y_+(\boldsymbol{z}), y_+(\boldsymbol{z})\right>$ for all
$\boldsymbol{z} \in \Q^n_{\textsc{sf}}$.
Thus, for all $\boldsymbol{z} \in \Q^n_{\textsc{sf}}$, we have
$l(\boldsymbol{z}) \in \mathcal{NL}_{\textsc{sf}}(\hat{Y}^{(np,nq)}(\boldsymbol{z}))$,
so that
$l(\boldsymbol{z}) \in \left<y_+(\boldsymbol{z}), y_-(\boldsymbol{z})\right>$. 
Thus,
%and since $l$, $y_+$, and $y_+$ are continuous on $\R^n_{\textsc{sf}}$, we have
%$l(\boldsymbol{z}) \in \left<y_+(\boldsymbol{z}), y_-(\boldsymbol{z})\right>^{\mkern-1.5mu\lowR}$ for all
%$\boldsymbol{z} \in \R^n_{\textsc{sf}}$.
%In particular, we have
\begin{align}
0 <  (1-t)l(\boldsymbol{y})=l(\boldsymbol{z}_t(\boldsymbol{y})) 
< y_-(\boldsymbol{z}_t(\boldsymbol{y})) 
\;\;\text{for all } t\in \left[0,1\right>\cap\Q
\;\;\;\;\;\;\;\;
\text{if}\;\, 0 < l(\boldsymbol{y});
  \\
0 >  (1-t)l(\boldsymbol{y})=l(\boldsymbol{z}_t(\boldsymbol{y})) 
> y_+(\boldsymbol{z}_t(\boldsymbol{y})) 
\;\;\text{for all } t\in \left[0,1\right>\cap\Q
\;\;\;\;\;\;\;\;
\text{if}\;\, 0 > l(\boldsymbol{y})\hphantom{;}
\end{align}
for all $\boldsymbol{y} \in \Q^n_{\textsc{sf}}$,
where the equivalence
$(1-t)l(\boldsymbol{y})=l(\boldsymbol{z}_t(\boldsymbol{y}))$
follows quickly from the definitions of $l$ and $\boldsymbol{z}$.
Now, either by 
Proposition \ref{prop: complements and closure} and the structure
of the structure of L-space intervals, or by 
Calegari and Walker's studies of ``ziggurats'' \cite{ziggurat},
we know that as functions on $\R^n_{\textsc{sf}}$,
$y_-$ and $y_+$ are piecewise constant, with rational endpoints,
in each coordinate direction.
Thus, since $l$ and $\boldsymbol{z}$ are linear
and the above inequalities are strict, we have
\begin{align}
0 < y_-(\boldsymbol{z}_t(\boldsymbol{y})) 
\;\;\text{if}\;\, 0 < l(\boldsymbol{y}),
       \;\;\;\;
0 > y_+(\boldsymbol{z}_t(\boldsymbol{y})) 
\;\;\text{if}\;\, 0 > l(\boldsymbol{y})
\end{align}
for all $\boldsymbol{y} \in \R^n_{\textsc{sf}}$
and $t\in \left[0,1\right>^{\lowR}$.

On the other hand, our definitions of $y_{\pm}$ and $\boldsymbol{z}$
imply that for any $\boldsymbol{y} \in \mathcal{N}^{\lowR}$, we have 
\begin{align}
y_+(\boldsymbol{z}_t(\boldsymbol{y}))
\le y_+(\boldsymbol{y}) < 0
\;\;\text{for all } t\in [0,1]^{\lowR}
\;\;\;\;\;\;\;\;
\text{if}\;\, 0 < l(\boldsymbol{y});
  \\
y_-(\boldsymbol{z}_t(\boldsymbol{y}))
\ge y_-(\boldsymbol{y}) > 0 
\;\;\text{for all } t\in [0,1]^{\lowR}
\;\;\;\;\;\;\;\;
\text{if}\;\, 0 > l(\boldsymbol{y}).
\end{align}
Combining these three lines of inequalities tells us that for any
$\boldsymbol{y} \in \mathcal{N}^{\lowR}$, we have 
$\boldsymbol{z}_t(\boldsymbol{y}) \in \mathcal{N}^{\lowR}$
for all $t \in [0,1]^{\lowR}$. Thus $\boldsymbol{z}$
provides a deformation retraction from 
$\mathcal{N}^{\lowR}$ to 
$\mathcal{B}^{\lowR}_{\textsc{sf}} \cap \R^n_{\textsc{sf}} 
\mkern1mu\subset\mkern1mu \mathcal{N}^{\lowR}$.
\qed

\subsection{Topology of monotone strata}
Sections
\ref{s: iterated torus-link-satellites}
and
\ref{s: algebraic link satellites}
analyze the L-space surgery regions for
satellites by iterated torus-links and by algebraic links, respectively.
While these sections primarily focus on approximation tools,
Section \ref{ss: monotone strata} returns to the question of
exact L-space regions for such satellites, and describes 
how to decompose these L-space regions into strata according to
monotonicity criteria, which govern where the endpoints of local L-space
intervals lie, relative to asymptotes of maps on slopes induced by gluing maps.

Local monotonicity criteria also help determine the {\em{topology}}
of these strata, a phenomenon we illustrate with
Theorem \ref{thm: topology of monotone stratum}
in Section \ref{s: algebraic link satellites},
where we show that the $\Q$-corrected $\R$-closure of the
monotone stratum of the L-space surgery
region of an appropriate satellite link
admits a deformation retraction onto an embedded torus
analogous to that in 
Theorem \ref{thm: topology of torus link exterior L-space region}$.ii.b$ above.

\section{Iterated torus-link satellites}
\label{s: iterated torus-link-satellites}

Just as one can construct a torus-link-satellite exterior
from a knot exterior by gluing an appropriate Seifert fibered
space to the knot exterior (as described in 
Proposition~\ref{prop: make satellite}),
one constructs an {\textit{iterated}} torus-link-satellite
exterior by gluing an appropriate rooted (tree-)graph manifold
to the knot exterior, where this graph manifold is formed
by iteratively performing the Seifert-fibered-gluing operations
associated to individual torus-link-satellite operations.

\subsection{Construction of iterated torus-link-satellite exteriors}
An iterated torus-link-satellite of a knot exterior
$Y = M \setminus \overset{\mkern2mu\circ}{\nu}(K)$
is specified by a weighted, rooted tree $\Gamma$,
corresponding to the minimal JSJ decomposition of the
graph manifold glued to the knot exterior to form the satellite.
We weight each vertex $v \in \text{Vert}(\Gamma)$
by the 3-tuple $(p_v, q_v, n_v) \in \Z^3$
corresponding to the pattern link
$T_{v} := T(n_v p_v,n_v q_v)$.
As usual, we demand that $p_v, n_v > 0$ and that $q_v \neq 0$. 
(If any vertex had $p_v = 0$ or $q_v = 0$,
then our satellite-link exterior would be a nontrivial connected sum,
in which case we might as well have considered the
irreducible components of the exterior separately.
Moreover, the links
of complex surface singularities are irreducible,
so the algebraic links we consider later on will
necessarily be irreducible.)

The weight $(p_v, q_v, n_v)$ also specifies the JSJ component $Y_v$ as the
Seifert fibered exterior
\vspace{-.1cm}
%*** squashed?
\begin{equation}
Y_v 
\mkern2.5mu:=\mkern2.5mu 
\hat{Y}^{\mkern1mu n_v}_{\mkern-1mu(p_v,\mkern1mu  q_v)}
\mkern1mu=\mkern2mu
M_{S^2}(\mkern1mu{\textstyle{
\mkern-2mu-\mkern-1mu\frac{q^*_v}{p_v}}},\mkern-.5mu
\overset{0}{\hat{*}},
\overset{1}{\hat{*}}, \ldots, \overset{n}{\hat{*}}\mkern1mu)
\mkern2.5mu=\mkern2.5mu
\hat{Y}_{(p_v\mkern-1mu,\mkern1mu q_v)} 
\mkern-3mu\setminus \overset{\mkern1mu\circ}{\nu}\mkern.5mu
(T_v)
\end{equation}
of $T_v \subset 
\hat{Y}_{(p_v\mkern-1mu,\mkern1mu q_v)}$ as a link in the solid torus
\begin{equation}
\label{eq: iterated torus link solid torus}
\hat{Y}_{(p_v\mkern-1mu,\mkern1mu q_v)} 
\mkern1mu:=\mkern1mu 
M_{S^2}(\mkern1mu{\textstyle{
\mkern-2mu-\mkern-1mu\frac{q^*_v}{p_v},
\mkern-1mu{\textstyle{\frac{p^*_v}{q_v}}}}}\mkern1mu)
\mkern-1.2mu\setminus\mkern-1mu \overset{\circ}{\nu}(\lambda_0)
\mkern2mu=\mkern2mu 
S^3 
\mkern-1.2mu\setminus\mkern-1mu \overset{\circ}{\nu}(\lambda_0)
\mkern1mu=\mkern2mu 
\nu(\lambda_{-1}),
\end{equation}
for $\lambda_{-1}$ the multiplicity-$p_v$ fiber of merdional
$\textsc{sf}$-slope $y_{-1}^v \mkern-2mu=\mkern-2mu-\frac{q^*_v}{p_v}$,
and $\lambda_0$ the multiplicity-$q_v$ fiber
of meridional $\textsc{sf}$-slope $y_0^v\mkern-2mu=\mkern-2mu\frac{p^*_v}{q_v}$, as in 
Section
\ref{ss: seifert structures on exterior stuff}.
As~usual, 
$(p_v^*, q_v^*) \in \Z^2$
denotes the unique pair of integers satisfying
$\mkern1muvp_vp_v^* - q_vq_v^* = 1$ with
$q^*_v \in \{0, \ldots, p_v\mkern-2mu-\mkern-2mu 1\}$.

Specifying a root $r$ for the tree $\Gamma$ determines an orientation on edges, up to over-all sign. We choose to direct edges {\em{towards}} the root~$r$,
and write $E_{\mathrm{in}}(v)$ for the set of edges
terminating on a vertex $v$.
On the other hand, each non-root vertex $v$ has a unique 
edge emanating from it, and we call this outgoing edge $e_v$.
We additionally declare one edge $e_{r}$ to emanate from the
the root vertex $r$ towards a null vertex $\textsc{null} \notin \mathrm{Vert}(\Gamma)$,
which we morally associate (with no hat) to our original knot exterior,
$Y_{\textsc{null}} := Y = M \setminus \overset{\mkern2mu\circ}{\nu}(K)$.
For any (directed) edge $e \in \text{Edge}(\Gamma)$, we write $v(e)$ for the destination
vertex of $e$, so that $v = v(-e_v)$ for all $v \in \text{Vert}(\Gamma)$.

For notational convenience, 
we also associate an ``index'' $j(e) \in \{1, \ldots, n_{v(e)}\}$
to each edge~$e$, specifying the boundary component 
$\partial_{j(e)} \mkern-1mu Y_{v(e)}$ of $Y_{v(e)}$
to which $\partial_0 \mkern-1mu Y_{v(-e)}$ is glued when we embed the 
pattern link $T_{v(-e)}$ in a neighborhood of $\partial_{j(e)} Y_{v(e)}$.
As such, each edge $e \in \text{Edge}(\Gamma)$ corresponds 
to a gluing map
\begin{equation}
{\varphi}_e: \partial_0 Y_{v(-e)} \to -\partial_{j(e)}\mkern-1mu Y_{v(e)},
\end{equation}
along the incompressible torus joining
$Y_{v(-e)}$ to  $Y_{v(e)}$.
This
${\varphi}_e$ is the {\text{inverse}} of the map $\bar{\varphi}$ used in the
satellite construction of 
Proposition~\ref{prop: make satellite}.
We express its induced map on slopes
\begin{equation}
{\varphi}_{\mkern-2.05mu e\mkern.95mu *}^{{\mathbb{P}}}  :
\P(H_1(\partial_0 Y_{v(-e)}; \Z))_{\textsc{sf}} 
\,\longrightarrow\,
\P(H_1(\partial_{j(e)}\mkern-1mu Y_{v(e)}; \Z))_{\textsc{sf}},
\;\;\;\;\;
y \,\mapsto\,
\frac{yp_{v(-e)} - q^*_{v(-e)}}{y q_{v(-e)} - p^*_{v(-e)}},
\end{equation}
in terms of $\textsc{sf}$-slopes on both sides.
Thus ${\varphi}_{\mkern-2.05mu e\mkern.95mu *}^{{\mathbb{P}}}$
is orientation reversing.
Note that for any $v \in \text{Vert}(\Gamma)$,
the map
${\varphi}_{\mkern-2.05mu e_v\mkern.95mu *}^{{\mathbb{P}}}$
is determined by $(p_v,q_v,n_v)$ and $j(e_v)$.
We additionally define
\begin{equation}
J_v := \{j(e) |\mkern1mu e \in E_{\mathrm{in}}(v)\},
\;\;\text{and}\;\;
I_v:= \{1, \ldots, n_v\} \setminus J_v
\end{equation}
for each $v \in \text{Vert}(\Gamma)$, so that
the space of Dehn filling slopes of $Y^{\Gamma}$ is given by
\begin{equation}
\prod_{v \in \text{Vert}(\Gamma)} \prod_{i \in I_v} \P(H_1(\partial_i Y_v ; \Z)).
\end{equation}

Writing
$\Gamma_v$ for the subtree of $\Gamma$ of which $v$ is the root,
let $Y_{\Gamma_v}$ denote the graph manifold
with JSJ decomposition given by the Seifert fibered spaces
$Y_u$ and gluing maps ${\varphi}_{e(u)}$
for $u \in \text{Vert}(\Gamma_v)$, so that $Y_{\Gamma_v}$ is constructed recursively as
\begin{equation}
Y_{\Gamma_v} \mkern3mu=\mkern7mu Y_v 
\mkern10mu\cup_{\{\varphi_e\}} \mkern-8mu \coprod_{\mkern8mu e \in E_{\mathrm{in}(v)}}
\mkern-12mu Y_{\Gamma_{v(-e)}}.
\end{equation}
The exterior $Y^{\Gamma} := M \setminus \overset{\mkern2mu \circ}{\nu}(K^{\Gamma})$
of the iterated torus-link-satellite $K^{\Gamma} \subset M$ of $K \subset M$
specified by $\Gamma$ is then given by 
$Y^{\Gamma} = Y \mkern1mu\cup_{\varphi_{e_r}} \mkern-4mu Y_{\Gamma}$.

\subsection{Dehn fillings of $\boldsymbol{Y^{\Gamma}}\mkern-2mu$}
For $v \in \text{Vert}(\Gamma)$,
any ${\textsc{sf}}_{\Gamma_v}$-slope 
\begin{equation}
{\boldsymbol{y}}^{\Gamma_v} := 
\prod_{v\in \text{Vert}(\Gamma_v)} 
{\boldsymbol{y}}^v,
\mkern30mu
{\boldsymbol{y}}^v
\in
(\Q \cup \mkern-2mu \{\infty\})^{\mkern-1mu|I_v|}_{\textsc{sf}_v}
:=
\prod_{i \in I_v} \P(H_1(\partial_i Y_v; \Z))_{{\textsc{sf}}_v},
\end{equation}
determines an L-space interval
$\mathcal{L}(Y_{\Gamma_v}\mkern-1mu({\boldsymbol{y}}^{\Gamma_v}))$
for the $\textsc{sf}_{\Gamma_v}$-Dehn-filling 
$Y_{\Gamma_v}\mkern-1mu({\boldsymbol{y}}^{\Gamma_v})$.
If $Y_{\Gamma_v}\mkern-1mu({\boldsymbol{y}}^{\Gamma_v})$ is Floer simple,
then we write
\begin{equation}
[[y^v_{0-}, y^v_{0+}]] 
:=\mathcal{L}_{\textsc{sf}_v}(Y_{\Gamma_v}\mkern-1mu({\boldsymbol{y}}^{\Gamma_v})),
\mkern30mu
[[y^{v(e_v)}_{j(e_v)-}, y^{v(e_v)}_{j(e_v)+}]] =  
{\varphi}^{\P}_{\mkern -3mu e_v *}
(\mathcal{L}_{\textsc{sf}_v}(Y_{\Gamma_v}\mkern-1mu({\boldsymbol{y}}^{\Gamma_v})))
\end{equation}
to express 
$\mathcal{L}_{\textsc{sf}_v}(Y_{\Gamma_v}\mkern-1mu({\boldsymbol{y}}^{\Gamma_v}))$
in terms of 
$\textsc{sf}_v$-slopes and 
$\textsc{sf}_{v(e_v)}$-slopes, respectively.
Note that since
${\varphi}^{\P}_{\mkern -3mu e_v *}$
is orientation-reversing, we have
\begin{equation}
y^{v(e_v)}_{j(e_v)-} = {\varphi}^{\P}_{\mkern -3mu e_v *}
(y^v_{0+}),
\mkern30mu
y^{v(e_v)}_{j(e_v)+} = {\varphi}^{\P}_{\mkern -3mu e_v *}
(y^v_{0-}).
\end{equation}

If we focus instead on the incoming edges of $v$, then
any Dehn filling of the 
boundary components $\partial Y_{\Gamma_v} \setminus \partial Y_v$
of $Y_{\Gamma_v}$
allows us to partition the graph manifolds incident to $v$,
labeled by $J_v = \{j(e) | \mkern1mu E_{\mathrm{in}}(v)\}$,
according to whether they are boundary compressible (\textsc{bc})---a solid
torus or connected sum thereof---or boundary incompressible (\textsc{bi}): 
\begin{align}
\label{eq: def of J BC}
J_v^{\textsc{bc}}
&:= \{ j(e) |\mkern2mu
Y_{\Gamma_{v(-e)}}
\text{ is }\textsc{bc},\;
e \in E_{\mathrm{in}}(v)\},
     \\
\label{eq: def of J BI}
J_v^{\textsc{bi}}
&:= J_v \setminus J_v^{\textsc{bc}}.
\end{align}
Setting $y^v_j \mkern-2mu:=\mkern-2mu y^v_{j\pm}$ when $y^v_{j+} \mkern-4mu=\mkern-2mu y^v_{j+}$,
we additionally define the sets
$J_{v\Z}^{\textsc{bi}+}$,
$J_{v\Z}^{\textsc{bi}-}$ 
%$\subset J_v^{\textsc{bi}}$
and
$J_{v\Z}^{\textsc{bc}}$:
%\subset J_v^{\textsc{bc}}$, 
\begin{equation}
\label{eq: def of J Z}
J_{v\Z}^{\textsc{bi}\pm}
:= \{j \in J_v^{\textsc{bi}}|\mkern2mu
y^v_{j\pm} \in \Z\},\;\;
J_{v\Z}^{\textsc{bc}}
:= \{j \in J_v^{\textsc{bc}}|\mkern2mu
y^v_{j} \in \Z\}.
\end{equation}

For $v \mkern-2mu\in\mkern-2mu \text{Vert}(\Gamma)$, $k \mkern-2mu\in\mkern-2mu \Z_{> 0}$, define
$\bar{y}^v_{0\mp\Sigma}(k) \mkern-2mu:=\mkern-2mu 0$
if $\infty \mkern-2mu\in \mkern-2mu
\{y^v_{j\pm}\}_{j\in J_v} 
\mkern-4mu\cup\mkern-1mu \{y^v_i\}_{i\in I_v}$, and otherwise set
\begin{align}
\label{eq: sigma thing for yv0-}
\bar{y}^v_{0-\Sigma}(k)
&:= 
{\textstyle{\sum\limits_{j \in J^{\textsc{bi}}_v}}}
%{\textstyle{\sum\limits_{j \in J^{\textsc{bi}}_v \setminus J^{\textsc{bi}+}_{v\Z}}}}
\mkern-12mu
\left( \left\lceil [y_{j+}^v]k\right\rceil 
\mkern-2mu-\mkern-2mu 1\right)
\mkern12mu+\mkern1mu
{\textstyle{\sum\limits_{i \in I_v \cup  J^{\textsc{bc}}_v}}}
\mkern-8mu
\left\lfloor [y_i^v]k\right\rfloor,
     \\
\label{eq: sigma thing for yv0+}
\bar{y}^v_{0+\Sigma}(k)
&:= 
{\textstyle{\sum\limits_{j \in J^{\textsc{bi}}_v}}}
%{\textstyle{\sum\limits_{j \in J^{\textsc{bi}}_v \setminus J^{\textsc{bi}-}_{v\Z}}}}
\mkern-12mu
\left( \left\lceil [-y_{j-}^v]k\right\rceil 
\mkern-2mu-\mkern-2mu 1\right)
\mkern12mu+\mkern1mu
{\textstyle{\sum\limits_{i \in I_v \cup  J^{\textsc{bc}}_v}}}
\mkern-8mu
\left\lfloor [-y_i^v]k\right\rfloor.
\end{align}
In addition, define $\bar{y}^v_{0-} \mkern-3mu:= \sup_{k >0} \bar{y}^v_{0-}(k)$,
$\bar{y}^v_{0+} \mkern-3mu:= \inf_{k >0} \bar{y}^v_{0+}(k)$,
where
\begin{align}
\bar{y}^v_{0-}(k)
&:=
\mfrac{1}{k}\mkern-3.5mu\left(
-1 + \left\lceil \mfrac{q^*_v}{p_v}k \right\rceil
-\bar{y}^v_{0-\Sigma}(k)
\right)
\mkern2mu-\mkern2mu 
|J_{v\Z}^{\textsc{bi}+}|,
\label{eq: def of y_0-v-bar for iterated}
   \\
\bar{y}^v_{0+}(k)
&:=
\mfrac{1}{k}\mkern-3.5mu\left(
\mkern3mu 1 + \left\lfloor \mfrac{q^*_v}{p_v}k \right\rfloor
+\bar{y}^v_{0+\Sigma}(k)
\right)
\mkern2mu+\mkern2mu 
|J_{v\Z}^{\textsc{bi}-}|.
\label{eq: def of y_0+v-bar for iterated}
\end{align}
The ``$\sup$'' and ``$\inf$'' account for cases
in which $\bar{y}^v_{0\pm}(k) = 0$.
The above notation provides a convenient way
to repackage our computation of L-space interval endpoints.
\begin{prop}
\label{prop: basic defs of y_0+-}
If $y_{0-}^v, y_{0+}^v \in \P(H_1(\partial_0 Y_v; \Z))_{\textsc{sf}_v}$
are the (potential) L-space interval endpoints for 
$Y_{\Gamma_v}\mkern-2mu({\boldsymbol{y}}^{\Gamma_v})$
as defined in
Theorem~\ref{thm: l-space interval for seifert jsj}, then
\begin{align*}
y_{0-}^v 
&= \bar{y}_{0-}^v
\mkern.5mu-\mkern-3.5mu
{\textstyle{\sum\limits_{\hphantom{1\mkern-2mu} 
j \in J_v^{\textsc{bi}}}}}\mkern-9mu
\left(\mkern-1mu\lceil y^v_{j+} \mkern-2mu\rceil \mkern-3mu-\mkern-2.5mu 1\right)
\mkern2mu-\mkern-7mu
{\textstyle{\sum\limits_{\hphantom{1\mkern-2mu} 
j \in J_v^{\textsc{bc}}}}}\mkern-8mu
\lfloor y^v_{j} \rfloor
\mkern10mu-\mkern7mu
{\textstyle{\sum\limits_{i \in I_v}}}\mkern-6mu
\lfloor y^v_{i} \rfloor,
    \\
y_{0+}^v 
&= \bar{y}_{0+}^v
\mkern.5mu-\mkern-3.5mu
{\textstyle{\sum\limits_{\hphantom{1\mkern-2mu} 
j \in J_v^{\textsc{bi}}}}}\mkern-9mu
\left(\mkern-1mu\lfloor y^v_{j-} \mkern-2mu\rfloor \mkern-3mu+\mkern-2.5mu 1\right)
\mkern2mu-\mkern-7mu
{\textstyle{\sum\limits_{\hphantom{1\mkern-2mu} 
j \in J_v^{\textsc{bc}}}}}\mkern-8mu
\lceil y^v_{j} \rceil
\mkern10mu-\mkern7mu
{\textstyle{\sum\limits_{i \in I_v}}}\mkern-6mu
\lceil y^v_{i} \rceil.
\end{align*}
Moreover,
$\mkern1mu\bar{y}^v_{0\mp} \mkern-3mu= \frac{q^*_v}{p_v}$
when
$\mkern1mu\infty \mkern-1mu\in \mkern-1mu
\{y^v_{j\pm}\}_{j\in J_v} 
\mkern-4mu\cup\mkern-1mu \{y^v_i\}_{i\in I_v}$, but
\begin{align}
\bar{y}^v_{0-} 
\mkern-2mu&= 
\left\lceil \mfrac{q^*_v}{p_v} \right\rceil - 1
\;\;\;\;\text{if}\;\;\;\;
J_{v\Z}^{\textsc{bi}+} \mkern-1mu\neq\mkern-0mu \emptyset
\text{ and }
\mkern1mu\infty \mkern-1mu\notin \mkern-1mu
\{y^v_{j+}\}_{j\in J_v} 
\mkern-2mu\cup\mkern.5mu \{y^v_i\}_{i\in I_v};
    \\
\bar{y}^v_{0+} 
\mkern-2mu&= 1
\mkern60mu
\;\;\;\;\text{if}\;\;\;\;
J_{v\Z}^{\textsc{bi}-} \mkern-1mu\neq\mkern-0mu \emptyset
\text{ and }
\mkern1mu\infty \mkern-1mu\notin \mkern-1mu
\{y^v_{j-}\}_{j\in J_v} 
\mkern-2mu\cup\mkern.5mu \{y^v_i\}_{i\in I_v}.
\end{align}

\end{prop}

\begin{proof}
The displayed equations in
Proposition \ref{prop: basic defs of y_0+-}
come directly from the definitions
of $y_{0\pm}^v$ specified by 
Theorem~\ref{thm: l-space interval for seifert jsj},
but subjected to some mild manipulation of terms
using the facts that
$x = \lfloor x \rfloor + [x]$ and $x = \lceil x \rceil - [-x]$
for all $x \in \R$, and that
\begin{equation}
|J_{v\Z}^{\textsc{bi}+}| =
{\textstyle{\sum\limits_{j \in J^{\textsc{bi}}_v}}}
\mkern-8mu
\left( \lfloor y_{j+}^v \rfloor -  (\lceil y_{j+}^v \rceil 
\mkern-2mu-\mkern-2mu 1)\right),
\mkern30mu
-|J_{v\Z}^{\textsc{bi}-}| =
{\textstyle{\sum\limits_{j \in J^{\textsc{bi}}_v}}}
\mkern-8mu
\left( \lceil y_{j-}^v \rceil -  (\lfloor y_{j-}^v \rfloor 
\mkern-2mu+\mkern-2mu 1)\right).
\end{equation}

For the second half of the proposition,
first note that the $\mkern1mu\bar{y}^v_{0\mp} \mkern-3mu= \frac{q^*_v}{p_v}$
result follows directly from taking the $k \to \infty$ limit.
In the case of
$\mkern1mu\infty \mkern-1mu\notin \mkern-1mu
\{y^v_{j+}, y^v_{j-}, y^v_i|\,
j \in J_v, i\in I_v\}$
we temporarily set 
$\hat{y}_{0\mp}^v(k)
:= \bar{y}_{0\mp}^v(k) 
\pm
|J_{v\Z}^{\textsc{bi}\pm}|$, so that
\begin{align}
\hat{y}^v_{0-}\mkern-1mu(k)
\mkern2mu&\le\mkern4mu
\mfrac{1}{k}\mkern-5mu\left(
\left\lceil \mkern-2mu\mfrac{q^*_v\mkern-1mu}{p_v\mkern-1mu}k\mkern-2.5mu \right\rceil
+(|J_{v\Z}^{\textsc{bi}+}| -1) \right)
\mkern4mu\le\mkern4mu
\left\lceil \mkern-2mu\mfrac{q^*_v\mkern-1mu}{p_v\mkern-1mu}
\mkern-2.2mu \right\rceil
+ \mfrac{1}{k}(|J_{v\Z}^{\textsc{bi}+}| -1),
     \\
\hat{y}^v_{0+}\mkern-1mu(k)
\mkern2mu&\ge\mkern4mu
\mfrac{1}{k}\mkern-5mu\left(
\left\lfloor \mkern-2mu\mfrac{q^*_v\mkern-1mu}{p_v\mkern-1mu}k\mkern-2.5mu \right\rfloor
-(|J_{v\Z}^{\textsc{bi}-}| -1) \right)
\mkern4mu\ge\mkern4mu
-\mfrac{1}{k}(|J_{v\Z}^{\textsc{bi}-}| -1)
\end{align}
for all $k \in \Z_{>0}$.
If $J_{v\Z}^{\textsc{bi}+} \mkern-4mu \neq \mkern-1.5mu \emptyset$
(respectively $J_{v\Z}^{\textsc{bi}-} \mkern-4mu \neq \mkern-1.5mu \emptyset$),
then the above bound for $\hat{y}^v_{0-}\mkern-1mu(k)$
(respectively $\hat{y}^v_{0+}\mkern-1mu(k)$)
is nonincreasing (respectively nondecreasing) in $k$,
so that
\begin{align}
\hat{y}^v_{0-}\mkern-1mu(k)
\le 
\left\lceil \mkern-2mu\mfrac{q^*_v\mkern-1mu}{p_v\mkern-1mu}
\mkern-2.2mu \right\rceil
-1+
|J_{v\Z}^{\textsc{bi}+}|,
    \mkern30mu
\hat{y}^v_{0+}\mkern-1mu(k)
\ge 
1
-
|J_{v\Z}^{\textsc{bi}-}|
\end{align}
for all $k \in \Z_{>0}$.
Since these bounds are each realized when $k=1$,
this completes the proof
of the bottom line of the proposition.
\end{proof}

The above method of computation for
L-space interval endpoints
helps us to prove some useful bounds for these endpoints.
\begin{prop}
\label{prop: main bounds for y0- and y0+}
Suppose
$\bar{y}^v_{0+\Sigma}(k)$,
$\mkern1.5mu\bar{y}_{0-}^v\mkern-2mu$, and
$\mkern2mu\bar{y}_{0+}^v\mkern-3mu$
are as defined in
(\ref{eq: sigma thing for yv0+}),
(\ref{eq: def of y_0-v-bar for iterated}), and
(\ref{eq: def of y_0+v-bar for iterated}),
and that
$\mkern1mu\infty \mkern-1mu\notin \mkern-1mu
\{y^v_{j+}, y^v_{j-}, y^v_i|\,
j \in J_v, i\in I_v\}$.
Then $\mkern1.5mu\bar{y}_{0-}^v\mkern-2mu$ and
$\mkern2mu\bar{y}_{0+}^v\mkern-3mu$
satisfy the following properties.

\begin{align*}
\tag{$=$}
\mkern2mu\bar{y}_{0-}^v\mkern-3mu = \mfrac{q^*_v}{p_v}
\mkern5mu
\Longleftrightarrow 
\mkern5mu
\mkern2mu\bar{y}_{0+}^v\mkern-3mu = \mfrac{q^*_v}{p_v}
\mkern15mu 
&\Longleftrightarrow 
\mkern15mu
J^{\textsc{bi}}_v \mkern-3mu = \mkern-1.5mu \emptyset
\text{ and }
\{y^v_{j+}, y^v_{j-}, y^v_i|\,j \in J_v, i\in I_v\}
\mkern-2.5mu\subset\mkern-1mu \Z
  \\
%\Longleftrightarrow
&\mkern4.5mu \Longrightarrow 
\mkern15mu
Y_{\Gamma_v}\mkern-2mu({\boldsymbol{y}}^{\Gamma_v})
\text{ is }\, \textsc{bc}.
\end{align*}

{\noindent{$(-)\;$}}
If 
$Y_{\Gamma_v}\mkern-2mu({\boldsymbol{y}}^{\Gamma_v})$ is \textsc{bi},
%$J^{\textsc{bi}}_v \mkern-3mu \neq \mkern-1.5mu \emptyset$
%or
%$\{y^v_i |\mkern2mu i \mkern-2.5mu\in\mkern-3mu I_v \mkern-2mu\cup\mkern-2mu J_v \mkern-.2mu\}
%\mkern-2.5mu\not\subset\mkern-1mu \Z$, 
then
$\mkern1mu\bar{y}_{0-}^v \in
\left[
\left\lceil\mkern-2mu\frac{q^*_v}{p_v}\mkern-2mu\right\rceil \mkern-3mu-\mkern-3mu1, 
\frac{q^*_v}{p_v}
\right>
=
\begin{cases}
\left[\mkern-4mu\vphantom{\frac{a}{c}} \right.
0, 
\frac{q^*_v}{p_v} 
\left.\vphantom{\frac{a}{c}}\mkern-4mu\right>
  &
  \mkern10mu
  p_v \neq 1
      \\
\left[\mkern-4mu\vphantom{\frac{a}{c}} \right.
\mkern-2.2mu-\mkern-5mu 1, 0 
\left.\vphantom{\frac{a}{c}}\mkern-4mu\right>
  &
  \mkern10mu
  p_v = 1
\end{cases}.
$

{\noindent{$(+)\;$}}
If 
$Y_{\Gamma_v}\mkern-2mu({\boldsymbol{y}}^{\Gamma_v})$ is \textsc{bi},
%$J^{\textsc{bi}}_v \mkern-3mu \neq \mkern-1.5mu \emptyset$
%or
%$\{y^v_i |\mkern2mu i \mkern-2.5mu\in\mkern-3mu I_v \mkern-.2mu\}
%\mkern-2.5mu\not\subset\mkern-1mu \Z$,
then
$\bar{y}_{0+}^v \in 
\left<
\frac{q^*_v}{p_v},
\left\lfloor\mkern-2mu\frac{q^*_v}{p_v}\mkern-2mu\right\rfloor \mkern-3mu+\mkern-3mu1
\right]
=
\left<
\frac{q^*_v}{p_v},
1
\right]$.
If in addition,
$J_{v\Z}^{\textsc{bi}-} \mkern-4.5mu=\mkern-1.5mu \emptyset$, 

$\mkern20mu$then for $q_v > 0$ and $m \in \Z$ (possibly negative or zero)
with $m < \frac{q_v}{p_v}$, $\mkern1.5mu\bar{y}_{0+}^v\mkern-1mu$ satisfies

\begin{align*}
(i)\mkern102mu
\bar{y}_{0+}^v \in
\begin{cases}
%\left<\mkern-4mu\vphantom{\frac{a}{c}} \right.
\left<
\frac{q^*_v}{p_v},
\frac{p_v^*- \mkern1mu mq_v^*}{q_v - \mkern1mu mp_v} 
\right]
%\left.\vphantom{\frac{a}{c}}\mkern-4mu\right]
&
\bar{y}^v_{0+\Sigma}(q_v\mkern-1.8mu - \mkern-1.5mu m p_v) \mkern-2mu=\mkern-2mu 0
\vphantom{\left<\frac{A^{A^A}}{A^{A^A}}\right>}
    \\
%\left<\mkern-4mu\vphantom{\frac{a}{c}} \right.
\left<
\frac{p_v^*- \mkern1mu mq_v^*}{q_v - \mkern1mu mp_v} ,1
\right]
%\left.\vphantom{\frac{a}{c}}\mkern-4mu\right]
&
\bar{y}^v_{0+\Sigma}(q_v \mkern-1.8mu-\mkern-1.5mu m p_v) \mkern-2mu>\mkern-2mu 0
\end{cases},
\mkern90mu\hphantom{(ii)}
\end{align*}

{\noindent{$\hphantom{(+)\;}$}}
where we note that for $a,b \in\Z$ with $a, b < \frac{q_v}{p_v}$, one has
\begin{equation*}
(ii)\mkern150mu
\mfrac{p_v^*- aq_v^*}{q_v - ap_v} 
<
\mfrac{p_v^*- bq_v^*}{q_v - bp_v}
\iff
a < b.
\mkern135mu\hphantom{(ii)}
\end{equation*}

{\noindent{$\hphantom{(+)\;}$}}
Lastly, if $q_v > p_v > 1$, then 
\begin{equation*}
(iii)\mkern32mu
\bar{y}_{0+}^v = \mfrac{p^*_v}{q_v}
\;\;\;\iff\;\;\;
\bar{y}^v_{0+\Sigma}(q_v) \mkern-2mu=\mkern-2mu 0,\;
\bar{y}^v_{0+\Sigma}(q_v\mkern-2.5mu + \mkern-1.5mu p_v) \mkern-2mu>\mkern-2mu 0,\mkern-2mu
\text{ and }
\mkern1.5muJ_{v\Z}^{\textsc{bi}-} \mkern-4.5mu=\mkern-1.5mu \emptyset.
\end{equation*}

\end{prop}

To aid in the proof of $(=)$, we first prove the following 
\begin{claim*} If the hypotheses of
Proposition~\ref{prop: main bounds for y0- and y0+}
hold, then
$$
J^{\textsc{bi}}_v \mkern-3mu = \mkern-1.5mu \emptyset
\text{ and }
\{y^v_{j+}, y^v_{j-}, y^v_i|\,j \in J_v, i\in I_v\}
\mkern-2.5mu\subset\mkern-1mu \Z
\mkern5mu\iff\mkern5mu
Y_{\Gamma_v}(\boldsymbol{y}^{\Gamma_v})
\text{ is }
\textsc{bc}
\text{ and }
\mkern2mu\bar{y}_{0-}^v
\mkern-3mu = 
\mkern2mu\bar{y}_{0-}^v
\mkern-3mu = 
\mfrac{q^*_v}{p_v}.
$$
\end{claim*}
{\em{Proof of Claim.}} For the $\Rightarrow$ direction,
the lefthand side implies the
irreducible component of $Y_{\Gamma_v}(\boldsymbol{y}^{\Gamma_v})$
containing 
$\partial Y_{\Gamma_v}(\boldsymbol{y}^{\Gamma_v})$
is Seifert fibered over the disk
with one or fewer exceptional fibers, hence is $\textsc{bc}$,
and direct computation shows that
$\bar{y}_{0-}^v
\mkern-3mu = 
\mkern2mu\bar{y}_{0-}^v
\mkern-3mu = 
\frac{q^*_v}{p_v}$.

For the $\Leftarrow$ direction, suppose the righthand side holds.
Then $J_v^{\textsc{bi}} = \emptyset$, and
the irreducible component of $Y_{\Gamma_v}(\boldsymbol{y}^{\Gamma_v})$
containing  
$\partial Y_{\Gamma_v}(\boldsymbol{y}^{\Gamma_v})$
is Seifert fibered over the disk
with one or fewer exceptional fibers, implying
$\left|\left\{\vphantom{A^A}\right.\right.\mkern-6mu\frac{q^*_v}{p_v}, y^v_{j+}, y^v_{j-}, y^v_i
\mkern-2.5mu\left|\vphantom{A^A}\right.\mkern-1mu
j \in J_v, i\in I_v\left.\left.\vphantom{A^A}\mkern-7mu\right\} 
\mkern1mu\cap\mkern2mu \Z\right|
\le 1$.
If $\frac{q^*_v}{p_v} \notin \Z$, then we are done, 
but if $\frac{q^*_v}{p_v} \in \Z$ and 
$\left|\left\{\vphantom{A^A}\right.\right.\mkern-6mu\frac{q^*_v}{p_v}, y^v_{j+}, y^v_{j-}, y^v_i
\mkern-2.5mu\left|\vphantom{A^A}\right.\mkern-1mu
j \in J_v, i\in I_v\left.\left.\vphantom{A^A}\mkern-7mu\right\} 
\mkern1mu\cap\mkern2mu \Z\right|
= 1$, then the longitude $l$ satisfies
$l=-\sum_{i \in I_v \cup J_v} y^v_i \notin \Z$, contradicting the fact that
$l = \bar{y}_{0-}^v
\mkern-3mu = 
\mkern2mu\bar{y}_{0-}^v
\mkern-3mu = 
\frac{q^*_v}{p_v} = 0$. \qed

{{\textit{Proof of }$(-)${\textit{ and part of }}$(=).$}}
When $J^{\textsc{bi}+}_{v\Z} \mkern-3.2mu\neq\mkern-2mu \emptyset$,
$(-)$ follows from 
Proposition~\ref{prop: basic defs of y_0+-}, which tells us
$\bar{y}_{0-}^v \mkern-1mu= \bar{y}_{0-}^v(1)
=\left\lceil\mkern-2mu\frac{q^*_v}{p_v}\mkern-2mu\right\rceil - 1$.
When $J^{\textsc{bi}+}_{v\Z} \mkern-3.2mu=\mkern-1.7mu \emptyset$,
we have the bounds
\begin{equation}
\bar{y}_{0-}^v 
\mkern-2mu\ge\mkern1mu 
\bar{y}_{0-}^v(1)
= 
\begin{cases}
\mkern11.6mu 0
& p_v 
\mkern-3mu\neq\mkern-2mu 1
    \\
\mkern-2mu-1
& p_v
\mkern-3mu=\mkern-2mu 1
\end{cases},
\mkern35mu
\bar{y}_{0-}^v(k) \le
\mfrac{1}{k}\mkern-3mu
\left(
\left\lfloor \mfrac{q^*_v}{p_v}k\right\rfloor \mkern-3mu+\mkern-3mu 1 \mkern-.8mu
\right) < \mfrac{q^*_v}{p_v}
\mkern10mu \forall\; k \mkern-2.5mu\in\mkern-2.5mu \Z_{>0}.
\end{equation}
Thus, either $\mkern1mu\bar{y}_{0-}^v \in
\left[
\left\lceil\mkern-2mu\frac{q^*_v}{p_v}\mkern-2mu\right\rceil \mkern-3mu-\mkern-3mu1, 
\frac{q^*_v}{p_v}
\right>$ or $\bar{y}^v_{0-} \mkern-3mu= \frac{q^*_v}{p_v}$. The latter case implies
$\sup_{k \to +\infty} \bar{y}_{0-}^v(k)$ is not attained for finite $k$,
and so
Theorem~\ref{thm: l-space interval for seifert jsj} tells us that
$Y_{\Gamma_{\mkern-3mu v}}\mkern-2mu({\boldsymbol{y}}^{\Gamma_{\mkern-4mu v }})$
is \textsc{bc} and $\bar{y}^v_{0-} \mkern-3mu= \bar{y}^v_{0+}$.
\qed

\vspace{.2cm}
{{\textit{Proof of }$(+)${\textit{ and remainder of }}$(=).$}}
Proposition~\ref{prop: basic defs of y_0+-} tells us that 
$\bar{y}_{0+}^v = 1$
when $J^{\textsc{bi}-}_{v\Z} \mkern-3.2mu\neq\mkern-2mu \emptyset$,
so we henceforth assume $J^{\textsc{bi}-}_{v\Z} \mkern-3.2mu =\mkern-2mu \emptyset$.
In this case, we have
\begin{equation}
\bar{y}_{0+}^v 
\mkern-2mu\le\mkern1mu 
\bar{y}_{0+}^v(1)
= 
1,
\mkern35mu
\bar{y}_{0+}^v(k) \ge
\mfrac{1}{k}\mkern-3mu
\left(
\left\lceil \mfrac{q^*_v}{p_v}k\right\rceil \mkern-3mu-\mkern-3mu 1 \mkern-.8mu
\right) > \mfrac{q^*_v}{p_v}
\mkern10mu \forall\; k \mkern-2.5mu\in\mkern-2.5mu \Z_{>0}.
\end{equation}
Thus, either 
$\bar{y}_{0+}^v \in
\left<\mkern-4mu\vphantom{\frac{a}{c}} \right.
\frac{q^*_v}{p_v}, 1 
\left.\vphantom{\frac{a}{c}}\mkern-4mu\right]$
or
$\bar{y}^v_{0+} \mkern-3mu= \frac{q^*_v}{p_v}$. The latter case implies
$\inf_{k \to +\infty} \bar{y}_{0+}^v(k)$ is not attained for finite $k$,
which
Theorem~\ref{thm: l-space interval for seifert jsj} tells us implies that
$Y_{\Gamma_{\mkern-3mu v}}\mkern-2mu({\boldsymbol{y}}^{\Gamma_{\mkern-4mu v }})$
is \textsc{bc} and $\bar{y}^v_{0+} \mkern-3mu= \bar{y}^v_{0-}$.

For $(+.i)$, fix some $m \in \Z$ with $m < \frac{q_v}{p_v}$.
If 
$\bar{y}^v_{0+\Sigma}(q_v\mkern-1.8mu - \mkern-1.5mu m p_v) \mkern-2mu=\mkern-2mu 0$, then
\begin{equation}
\bar{y}^v_{0+} 
\le\,
\bar{y}^v_{0+}(q_v\mkern-1.8mu - \mkern-1.5mu m p_v) 
\,=\, \mfrac{1}{q_v\mkern-1.8mu - \mkern-1.5mu m p_v}
\mkern-3.5mu\left(1 + 
\left\lfloor \mfrac{q^*_v}{p_v}
(q_v\mkern-1.8mu - \mkern-1.5mu m p_v) \right\rfloor + 0 \right)
\,=\;\mfrac{p_v^* - mq^*_v}{q_v-mp_v}.
\end{equation}
Next, suppose that
$\bar{y}^v_{0+\Sigma}(q_v\mkern-1.8mu - \mkern-1.5mu m p_v) \mkern-2mu>\mkern-2mu 0$,
so that either $[-y_{j-}^v] > (q_v-mp_v)^{-1}$ for some $j \in J_v^{\textsc{bi}}$,
or $[-y_i^v] \ge (q_v-mp_v)^{-1}$ for some $i \in I_v \cup J_v^{\textsc{bi}}$
(or both occur).  Since for any rational $x > (q_v-mp_v)^{-1}$, we have
$
\left\lceil xk \right\rceil - 1 \ge \left\lfloor \mfrac{k}{q_v - mp_v} \right\rfloor
$
for all $k \in \Z_{>0}$,
the condition
$\bar{y}^v_{0+\Sigma}(q_v\mkern-1.8mu - \mkern-1.5mu m p_v) \mkern-2mu>\mkern-2mu 0$
therefore implies 
$\bar{y}^v_{0+\Sigma}(k) \ge \left\lfloor \frac{k}{q_v-\mkern1mu mp_v} \right\rfloor$
for all $k \in \Z_{>0}$.
We then have
\begin{align}
k\mkern-2mu\left(\bar{y}^v_{0+}(k) - \mfrac{p_v^* - mq^*_v}{q_v-mp_v}\right)
&\ge
k\left(
\mfrac{1}{k}\mkern-2mu\left(1+\left\lfloor \mfrac{q^*_v}{p_v}k\right\rfloor
+\left\lfloor \mfrac{k}{q_v - mp_v}\right\rfloor\right)
-\mfrac{p_v^* - mq^*_v}{q_v-mp_v}\right)
   \\
&=
1
+\left\lfloor \mfrac{k}{q_v - mp_v}\right\rfloor
-\mfrac{[q^*_v k]_{p_v}}{p_v}
+\mfrac{q^*_v}{p_v}k
-\mfrac{p_v^* - mq^*_v}{q_v-mp_v}k
\nonumber
   \\
&=
1
-\mfrac{[q^*_v k]_{p_v}}{p_v}
+\left\lfloor \mfrac{k}{q_v - mp_v}\right\rfloor
-\mfrac{k}{p_v(q_v-mp_v)}
\nonumber
  \\
&=
1
-\mfrac{[q^*_v k]_{p_v} \mkern0mu+\mkern2.5mu [k]_{q_v\mkern-1mu-mp_{\mkern-.2mu v}}
/(q_v\mkern-2mu-\mkern-1.5mump_v)}{p_v}
\mkern1mu+
\left(\mfrac{p_v\mkern-2mu-1}{p_v}\right)\mkern-4mu\left\lfloor \mfrac{k}{q_v - mp_v}\right\rfloor
\nonumber
  \\
&> 0
\nonumber
\end{align}
for all 
$k \in \Z_{>0}$, and so
$\bar{y}^v_{0+} > 
\frac{p_v^* - mq^*_v}{q_v-mp_v}$.

Statement $(+.ii)$ is a simple consequence of the fact that
\begin{equation}
\mfrac{p_v^*- bq_v^*}{q_v - bp_v}
-\mfrac{p_v^*- aq_v^*}{q_v - ap_v} 
\,=\, 
\mfrac{b-a}{(q_v - bp_v)(q_v - ap_v)}.
\end{equation}

This leaves us with $(+.iii)$.
Since $q_v > p_v > 1$ implies $\frac{p^*_v}{q_v} <1$,
but Proposition~\ref{prop: basic defs of y_0+-} tells us
$\bar{y}_{0+}^v = 1$ if
$\mkern1.5muJ_{v\Z}^{\textsc{bi}-} \mkern-4.5mu\neq\mkern-1.5mu \emptyset$,
we henceforth assume 
$\mkern1.5muJ_{v\Z}^{\textsc{bi}-} \mkern-4.5mu=\mkern-1.5mu \emptyset$.
Setting $m=0$ in $(+.i)$ then gives us
\begin{equation}
\bar{y}_{0+}^v \le \mfrac{p^*_v}{q_v}
\;\;\iff\;\;
\bar{y}^v_{0+\Sigma}(q_v) \mkern-2mu=\mkern-2mu 0.
\end{equation}
Similarly, setting $m=-1$ in $(+.i)$ yields the relation
\begin{equation}
\bar{y}_{0+}^v \le \mfrac{p^*_v+q^*_v}{q_v+p_v}
<\mfrac{p^*_v}{q_v}
\;\;\iff\;\;
\bar{y}^v_{0+\Sigma}(q_v+p_v) \mkern-2mu=\mkern-2mu 0,
\end{equation}
where we used $(+.ii)$ for the right-hand inequality.
Lastly, suppose that $\bar{y}^v_{0+\Sigma}(q_v+p_v) >0$.
By reasoning similar to that used in the proof of $(+.i)$,
this implies that
$\bar{y}^v_{0+\Sigma}(k) \ge \left\lfloor \frac{k}{p_v+q_v}\right\rfloor$
for all $k \in \Z_{>0}$.  We then have
\begin{equation}
\bar{y}^v_{0+}(k) 
\,\ge\,
\mfrac{1}{k}
\left(
1 + \left\lfloor \mfrac{q^*_v}{p_v}k \right\rfloor
+ \left\lfloor \mfrac{k}{p_v+q_v} \right\rfloor
\right) 
\,\ge\, \mfrac{p^*_v}{q_v}
\;\;\;\;\;
\forall \;k \in \Z_{>0},
\end{equation}
with the right-hand inequality
coming from
Lemma~\ref{lemma: p+q bound}, and so 
$\bar{y}^v_{0+} \ge
\frac{p^*_v}{q_v}$.
\qed

There is one more collection of estimates that will be
particularly useful in the case of general iterated
torus-link satellites.
\begin{prop}
\label{prop: bounds good for iterated case}
The following bounds hold.
\begin{align*}
(i)\hphantom{ii}\;\;
\text{If } \,q_v > 0 \text{ and }\mkern1.5mu
\hphantom{-}
\bar{y}_{0-\Sigma}^v([p_v]_{q_v})>0,
\,\;&\text{then}\;\;
\bar{y}^v_{0-}\mkern-2mu < \mfrac{p^*_v}{q_v} - \mfrac{1}{[p_v]_{q_v} (q_v)},
\mkern8mu\hphantom{-}\mkern-1mu
\varphi_{e_v *}^{\P}(\bar{y}_{0-}^v) > \left\lceil \mfrac{p_v}{q_v} \right\rceil - 1;
     \\
(ii)\hphantom{i}\;\;
\text{If } \,q_v < 0 \text{ and }\mkern1.5mu
\bar{y}_{0-\Sigma}^v([-p_v]_{q_v})>0,
\,\;&\text{then}\;\;
\bar{y}^v_{0-}\mkern-2mu < \mfrac{p^*_v}{q_v} + \mfrac{1}{[-p_v]_{q_v} (q_v)},
\mkern8mu
\varphi_{e_v *}^{\P}(\bar{y}_{0-}^v) > \left\lceil \mfrac{p_v}{q_v} \right\rceil - 1;
     \\
(iii)\;\;
\text{If } \,q_v > 0 \text{ and }\mkern1.5mu
\bar{y}_{0+\Sigma}^v([-p_v]_{q_v})>0,
\,\;&\text{then}\;\;
\bar{y}^v_{0+}\mkern-2mu > \mfrac{p^*_v}{q_v} + \mfrac{1}{[-p_v]_{q_v} (q_v)},
\mkern8mu
\varphi_{e_v *}^{\P}(\bar{y}_{0+}^v) < \left\lfloor \mfrac{p_v}{q_v} \right\rfloor + 1;
     \\
(iv)\mkern1mu\;\;
\text{If } \,q_v < 0 \text{ and }\mkern1.5mu
\hphantom{-}
\bar{y}_{0+\Sigma}^v([p_v]_{q_v})>0,
\,\;&\text{then}\;\;
\bar{y}^v_{0+}\mkern-2mu > \mfrac{p^*_v}{q_v} - \mfrac{1}{[p_v]_{q_v} (q_v)},
\mkern8mu
\hphantom{-}\mkern-1mu
\varphi_{e_v *}^{\P}(\bar{y}_{0+}^v) < \left\lfloor \mfrac{p_v}{q_v} \right\rfloor + 1.
\end{align*}
\end{prop}

{\noindent{\textit{Proof of }$(i)$.}}
If $[p_v]_{q_v} \in \{0,1\}$, then
$\bar{y}_{0-\Sigma}^v([p_v]_{q_v})\le 0$ and the claim holds
vacuously, so we assume $[p_v]_{q_v} \ge 2$, implying $p_v \ge 2$ and $q_v \ge 3$,
so that
\begin{equation}
\label{eq: part i bound les than p/q - 1}
\varphi_{e_v *}^{\P}\left(\mfrac{p^*_v}{q_v} - \mfrac{1}{[p_v]_{q_v} (q_v)}\right)
= \left\lceil \mfrac{p_v}{q_v} \right\rceil - 1.
\end{equation}
By reasoning similar to that used in the proof of $(+.i)$ above,
the hypothesis $\bar{y}^v_{0-\Sigma}([p_v]_{q_v}) > 0$
implies that $\bar{y}^v_{0-\Sigma}(k) \ge 
\left\lfloor \mfrac{k}{[p_v]_{q_v}}\right\rfloor$ for all $k \in \Z_{>0}$.
Thus, if we set
\begin{equation}
m:= \left\lfloor \mfrac{p_v}{q_v} \right\rfloor,
\;\text{ so that }\;
[p_v]_{q_v} = p_v - m q_v
\,\text{ and }\,
\mfrac{p^*_v}{q_v} - \mfrac{1}{[p_v]_{q_v} (q_v)}
= \mfrac{q^*_v - mp^*_v}{p_v - m q_v},
\end{equation}
then it suffices to prove negativity, for all $k \in \Z_{>0}$, of the difference
\begin{align}
\nonumber
k\left(\bar{y}^v_{0-}(k) - \mfrac{q^*_v - mp^*_v}{p_v - m q_v}\right)
&\le\mkern2mu
k \mkern-1mu\left(\mfrac{1}{k}\mkern-2mu\left(-1+\left\lceil\mfrac{q^*_v}{p}k\right\rceil 
-\left\lfloor \mfrac{k}{p_v-m q_v} \right\rfloor \right)
-\mfrac{q^*_v - mp^*_v}{p_v - m q_v}\right)
    \\
\nonumber
&=\mkern2mu
-1 + \mfrac{[-q^*_v k]_{p_v}}{p_v} 
- \left\lfloor \mfrac{k}{p_v-m q_v} \right\rfloor
+ \mfrac{q^*_v}{p}k
-\mfrac{q^*_v - mp^*_v}{p_v - m q_v}k
    \\
\label{eq: prop iterated bound, y- q>, 3rd line}
&=\mkern2mu
\mfrac{-p_v q_v + q_v [q^{-1}_v k]_{p_v}}{p_v q_v} 
- \left\lfloor \mfrac{k}{[p_v]_{q_v}} \right\rfloor
+\left(\mfrac{m}{p_v}\right)\mkern-4mu\mfrac{k}{[p_v]_{q_v}}.
\end{align}

Now, $\frac{m}{p_v}$ already satisfies the bound
\begin{equation}
\mfrac{m}{p_v}
= \mfrac{\lfloor {p_v}/{q_v}\rfloor}{p_v}
= \mfrac{p_v - [p_v]_{q_v}}{p_v q_v} 
\le \mfrac{p_v - 2}{p_v q_v} < \mfrac{1}{q_v} \le \mfrac{1}{3}.
\end{equation}
Thus, if $\left\lfloor\frac{k}{[p_v]_{q_v}} \right\rfloor \ge 1$, then
\begin{equation}
\left\lfloor \mfrac{k}{[p_v]_{q_v}}\right\rfloor
\ge 
\mfrac{1}{2}\left(\left\lfloor \mfrac{k}{[p_v]_{q_v}}\right\rfloor + 1\right)
>\mfrac{1}{2}\left( \mfrac{k}{[p_v]_{q_v}}\right)
> \left(\mfrac{m}{p_v}\right)\mkern-4mu\mfrac{k}{[p_v]_{q_v}},
\end{equation}
making the right-hand side of (\ref{eq: prop iterated bound, y- q>, 3rd line}) negative.

We therefore henceforth assume that 
$\left\lfloor\frac{k}{[p_v]_{q_v}} \right\rfloor = 0$, implying
$k < [p_v]_{q_v}$.
Thus, since 
\begin{equation}
q_v [q_v^{-1} k]_{p_v} + p_v[p_v^{-1} k]_{q_v} = p_v q_v + k,
%\;\;\text{implying}\;\;
%-p_v q_v + q_v [q_v^{-1} k]_{p_v} \le k p_v[p_v^{-1} k]_{q_v},
\end{equation}
and since $mq_v = p_v - [p_v]_{q_v}$,
we obtain
\begin{align}
k\left(\bar{y}^v_{0-}(k) - \mfrac{q^*_v - mp^*_v}{p_v - m q_v}\right)
&\le\mkern2mu
\mfrac{k - p_v [p^{-1}_v k]_{q_v} + (p_v - [p_v]_{q_v})\frac{k}{[p_v]_{q_v}} }{p_v q_v} 
\,=\,
\mfrac{- [p^{-1}_v k]_{q_v} + \frac{k}{[p_v]_{q_v}} }{q_v} < 0.
\end{align}
\qed

{\noindent{\textit{Proof of }$(ii)$.}}
The claim holds vacuously for $[-p_v]_{q_v} \in \{0,1\}$, so we
assume $[-p_v]_{q_v} \ge 2$ and $q_v \le -3$, in which case
\begin{equation}
\varphi_{e_v *}^{\P}\left(\mfrac{p^*_v}{q_v} + \mfrac{1}{[-p_v]_{q_v} (q_v)}\right)
= \left\lceil \mfrac{p_v}{q_v} \right\rceil - 1.
\end{equation}
Since 
$\bar{y}^v_{0-\Sigma}([-p_v]_{q_v}) > 0$
implies $\bar{y}^v_{0-\Sigma}(k) \ge 
\left\lfloor \mfrac{k}{[-p_v]_{q_v}}\right\rfloor$ for all $k \in \Z_{>0}$,
we set
\begin{equation}
m:= -\left\lfloor \mfrac{p_v}{q_v} \right\rfloor
=\left\lceil \mfrac{p_v}{|q_v|} \right\rceil,
\;\text{ with }\;
[-p_v]_{q_v} = -p_v - m q_v
\,\text{ and }\,
\mfrac{p^*_v}{q_v} + \mfrac{1}{[-p_v]_{q_v} (q_v)}
= \mfrac{-q^*_v - mp^*_v}{-p_v - m q_v}.
\end{equation}
Using arguments similar to those in part $(i)$, it is straightforward to derive the bound
\begin{align}
\nonumber
k\left(\bar{y}^v_{0-}(k) - \mfrac{-q^*_v - mp^*_v}{-p_v - m q_v}\right)
&\le\mkern2mu
\mfrac{-p_v |q_v| + |q_v| [|q|^{-1}_v(-k)]_{p_v}}{p_v |q_v|} 
- \left\lfloor \mfrac{k}{[-p_v]_{q_v}} \right\rfloor
+\left(\mfrac{m}{p_v}\right)\mkern-4mu\mfrac{k}{[-p_v]_{q_v}}
\end{align}
for all $k \in \Z_{>0}$, and to show that the right-hand side is negative
if $\left\lfloor \frac{k}{[-p_v]_{q_v}} \right\rfloor \ge 1$,
allowing us to assume that 
$\left\lfloor \frac{k}{[-p_v]_{q_v}} \right\rfloor = 0$
and $k < [-p_v]_{q_v}$.  Thus, since
\begin{equation}
|q_v| [|q_v|^{-1} (-k)]_{p_v} + p_v[p_v^{-1} (-k)]_{q_v} = p_v q_v - k,
%\;\;\text{implying}\;\;
%-p_v q_v + q_v [q_v^{-1} k]_{p_v} \le k p_v[p_v^{-1} k]_{q_v},
\end{equation}
and since $mq_v = p_v + [-p_v]_{q_v}$,
we obtain
\begin{align}
k\left(\bar{y}^v_{0-}(k) - \mfrac{-q^*_v - mp^*_v}{-p_v - m q_v}\right)
&\le\mkern2mu
\mfrac{-k - p_v [p^{-1}_v (-k)]_{q_v} + (p_v + [-p_v]_{q_v})\frac{k}{[-p_v]_{q_v}} }{p_v |q_v|} 
  \\
&=\mkern2mu
\mfrac{- [p^{-1}_v (-k)]_{q_v} + \frac{k}{[-p_v]_{q_v}} }{|q_v|} 
\;<\; 0.
\end{align}
\qed

{\noindent{\textit{Proofs of }$(iii)$\textit{ and }$(iv)$.}}
Respectively similar to proofs of $(ii)$ and $(i)$.
\qed

\subsection{L-space surgery regions for iterated satellites: Proof of Theorem
\ref{thm: iterated torus links satellites}}
We have finally done enough preparation to prove
Theorem \ref{thm: iterated torus links satellites}
from the introduction.

\vspace{.15cm}
{\noindent{\textit{Proof of Theorem \ref{thm: iterated torus links satellites}.}}}
The bulk of part $(i)$ is proven in
``Claim 1'' in the proof of 
Theorem~
\ref{thm: bgw and juhasz for general satellites}.
Since the right-hand condition of 
(\ref{eq: iff for s3 and lspace conds})
is equivalent to the condition that
$Y^{\Gamma}\mkern-2mu(\boldsymbol{y}^{\Gamma})$
be an L-space, Claim 1 proves that
$Y^{\Gamma}\mkern-2mu(\boldsymbol{y}^{\Gamma})$ is an L-space
if and only if 
$Y^{\Gamma}\mkern-2mu(\boldsymbol{y}^{\Gamma}) = S^3$.
Thus, if we define $\Lambda_{\Gamma}$ as in 
(\ref{eq: lambda_gamma def}),
then the statement
$\mathcal{L}(Y^{\Gamma}) = \Lambda_{\Gamma}$
holds tautologically.

The proof of part $(ii)$ begins similarly to the proof of
Theorem~\ref{thm: s3 satellite in sf basis}.$(i.b)$,
except that instead of deducing that 
$\sum_{i \in I_r} (\lceil y^r_i \rceil - \lfloor y^r_i \rfloor) \le 1$,
we deduce that
\begin{equation}
{\textstyle{\sum\limits_{i \in I_r}}} (\lceil y^r_i \rceil - \lfloor y^r_i \rfloor)
\mkern2mu+\mkern-4mu
{\textstyle{\sum\limits_{\hphantom{1\mkern-2mu} 
j \in J_r^{\textsc{bi}}}}}\mkern-9mu
\left(\left(\mkern-1mu\lfloor y^r_{j-} \mkern-2mu\rfloor \mkern-3mu+\mkern-2.5mu 1\right)
\mkern-1.5mu-\mkern-1.5mu
\left(\mkern-1mu\lceil y^r_{j+} \mkern-2mu\rceil \mkern-3mu-\mkern-2.5mu 1\right)\right)
\mkern2mu-\mkern-7mu
{\textstyle{\sum\limits_{\hphantom{1\mkern-2mu} 
j \in J_r^{\textsc{bc}}}}}\mkern-8mu
(\lceil y^r_{j} \rceil - 
\lfloor y^r_{j} \rfloor)
\le 1,
\end{equation}
with $y_{j-}^r \ge y_{j+}^r$ for all $j \in J_r^{\textsc{bc}}$.
In the case that
$\sum_{i \in I_r} (\lceil y^r_i \rceil - \lfloor y^r_i \rfloor) = 1$
and the other sums vanish, we are reduced to the original case
of 
Theorem~\ref{thm: s3 satellite in sf basis}.$(i.b)$,
obtaining the component
\begin{equation}
\Lambda_r \mkern-3.7mu\cdot\mkern-2.5mu
\mathfrak{S}_{|I_r|}( [N,+\infty] \mkern-4mu\times\mkern-4mu \{\infty\}^{\mkern-1mu|I_r|-1})
\mkern8mu
\times
{{\prod\limits_{e\in E_{\mathrm{in}}(r)}}} \Lambda_{\Gamma_{v(-e)}}
\;\;
\subset\;
\mathcal{L}_{S^3}(Y^{\Gamma}).
\end{equation}

In the case that 
$\sum_{i \in I_r} (\lceil y^r_i \rceil - \lfloor y^r_i \rfloor) = 0$,
we have that $\boldsymbol{y}^r \in \Z^{|I_r|}$,
and all but one incoming edge of $r$, say $e$, descend from trees with trivial fillings.
Performing these trivial fillings reduces $Y^{\Gamma}$ to the exterior of
a $\Gamma_{v(-e)}$-satellite of the $(1, q_r)$-cable of $K \subset S^3$,
but the $(1, q_r)$-cable is just the identity operation, so we are left with
the exterior $Y^{\Gamma_{v(-e)}}$ of the
$\Gamma_{v(-e)}$-satellite of $K \subset S^3$.
Considering this for all edges $e \in E_{\mathrm{in}}(r)$
then gives the remaining component
\begin{equation}
\coprod_{e \in E_{\mathrm{in}}(r)}\mkern-10mu
\left(
\mathcal{L}_{\mkern-1muS^3}\mkern-1mu(Y^{\Gamma_{v(-e)}}) \times
\Lambda_{\Gamma  \mkern2mu\setminus\mkern2mu \Gamma_{\mkern-1muv(-e)}}
\right)
\;\; \subset \;
\mathcal{L}_{S^3}(Y^{\Gamma}).
\end{equation}

\vspace{.1cm}

{\textit{Part $(iii)$.}} 
Before proceeding with the main inductive argument in this section,
we attend to some bookkeeping issues.
In particular, our inductive proof requires each $I_v$ to be nonempty.
For any $w \mkern-2mu\in\mkern-2mu \mathrm{Vert}(\Gamma)$ with
$I_{w} \mkern-3mu=\mkern-1mu \emptyset$,
we repair this situation artificially, as follows.
First, redefine $I_{w} \mkern-2.5mu:=\mkern-2.5mu \{1\}$.
Next, if $0 \ge m_{w}^+$, then set 
$\mathcal{L}^{\min +}_{\textsc{sf}_{w}} \mkern-3.5mu=\mkern-2mu \{0\}$,
and declare
$\mathcal{R}_{\textsc{sf}_{w}} \mkern-4mu\setminus\mkern-1mu \mathcal{Z}_{\textsc{sf}_{w}}
\mkern-3.5mu=\mkern-1.5mu
\mathcal{L}^{\min-}_{\textsc{sf}_{w}} 
\mkern-3.5mu=\mkern-2mu \emptyset$.
Finally, if 
$0 \mkern-.5mu<\mkern-.5mu m_w^+$, then 
$0 \mkern-.5mu<\mkern-.5mu m_w^+ \mkern-2.5mu<\mkern-.5mu m_w^-$, so set
$\mathcal{L}^{\min -}_{\textsc{sf}_w} \mkern-3.5mu=\mkern-2mu \{0\}$,
and declare
$\mathcal{R}_{\textsc{sf}_w} \mkern-4mu\setminus\mkern-1mu \mathcal{Z}_{\textsc{sf}_w}
\mkern-3.5mu=\mkern-1.5mu
\mathcal{L}^{\min+}_{\textsc{sf}_w} 
\mkern-3.5mu=\mkern-2mu \emptyset$.
%This ensures that
%${\textstyle{\sum\limits_{i \in I_v}}}\mkern-3mu
%\lfloor y^v_{i} \rfloor \ge m_v^+$ for any
%$\boldsymbol{y}^v \in \mathcal{L}^{\min +}_{\textsc{sf}_v}$
%and that
%${\textstyle{\sum\limits_{i \in I_v}}}\mkern-3mu
%\lceil y^v_{i} \rceil \le m_v^-$ for any
%$\boldsymbol{y}^v \in \mathcal{L}^{\min -}_{\textsc{sf}_v}$.

For a  vertex $v \in \text{Vert}(\Gamma)$,
inductively assume, for each incoming edge $e \in E_{\mathrm{in}}(v)$, that
for any
$\boldsymbol{y}^{\Gamma_{v(-e)}}
\in
\prod_{u \in \text{Vert}(\Gamma_{v(-e)})}
\left(\mathcal{L}^{\min+}_{\textsc{sf}_u} 
\cup
\mathcal{R}_{\textsc{sf}_u} \mkern-4mu\setminus\mkern-1mu \mathcal{Z}_{\textsc{sf}_u}
\cup
\mathcal{L}^{\min-}_{\textsc{sf}_u}\right)$, we have
\begin{align}
\label{eq: induct one interated}
\left\lceil \mfrac{p_{v(-e)}}{q_{v(-e)}}\right\rceil -1
\;\le\; 
y^v_{j(e)+} 
\,&\le\; 
y^v_{j(e)-} 
\;\le\; 
\left\lfloor \mfrac{p_{v(-e)}}{q_{v(-e)}}\right\rfloor +1,
\hphantom{\;\;\;\;\;
\text{if }\,
Y_{\Gamma_{v(-e)}}(\boldsymbol{y}^{\Gamma_{v(-e)}})
\,\;\text{is}\;\,\textsc{bi}}
    \\
\label{eq: induct two interated}
\left\lceil \mfrac{p_{v(-e)}}{q_{v(-e)}}\right\rceil -1
\;<\; 
y^v_{j(e)+},\mkern-5mu
\,&\mkern24.1mu\; 
y^v_{j(e)-} 
\;<\; 
\left\lfloor \mfrac{p_{v(-e)}}{q_{v(-e)}}\right\rfloor +1
\;\;\;\;\;
\text{if }\,
Y_{\Gamma_{v(-e)}}(\boldsymbol{y}^{\Gamma_{v(-e)}})
\,\;\text{is}\;\,\textsc{bi},
\end{align}
where, again, $y^v_{j(e)\pm}:= \varphi^{\P}_{e*}(y^{v(-e)}_{0\mp})$,
with $\mathcal{L}_{\textsc{sf}}(Y_{\Gamma_{v(-e)}}(\boldsymbol{y}^{\Gamma_{v(-e)}})) = 
[[y^{v(-e)}_{0-}, y^{v(-e)}_{0+}]]$.
Note that this inductive assumption already holds vacuously if $v$ is a leaf.

\vspace{.08cm}

If
$\boldsymbol{y}^v \in
\mathcal{R}_{\textsc{sf}_v} \mkern-4mu\setminus\mkern-1mu \mathcal{Z}_{\textsc{sf}_v}$,
then $y_{0-}^v = y_{0+}^v = \infty$, implying that
\begin{equation}
y^{v(e_v)}_{j(e_v)\pm} = \varphi^{\P}_{e_v *}(y_{0\mp}^v) = \mfrac{p_v}{q_v}
\in 
\left<
\left\lceil \mfrac{p_v}{q_v}\right\rceil -1,
\left\lfloor \mfrac{p_v}{q_v}\right\rfloor +1
\right>.
\end{equation}
%and that $Y_{\Gamma_{v(-e)}}(\boldsymbol{y}^{\Gamma_v})$ is \textsc{bc}.

If
$\boldsymbol{y}^v \in 
\mathcal{L}^{\min-}_{\textsc{sf}_v}  
\mkern-1.5mu\cup\mkern1.5mu\mathcal{L}^{\min+}_{\textsc{sf}_v}
\mkern-3mu\subset\mkern-0mu
\Q^{|I_v|}$,
then applying
(\ref{eq: induct one interated})
%and (\ref{eq: induct two interated})
to
Proposition 
\ref{prop: basic defs of y_0+-}
yields
\begin{align}
\label{eq: y0- bound from prop 5.1}
y_{0-}^v 
\;&\le\,\;
\bar{y}_{0-}^v\;
- 
{\sum\limits_{e\mkern1mu\in\mkern1mu 
E_{\mathrm{in}}\mkern-1mu(v)}}
\mkern-12mu
\left(\left\lceil
\mfrac{p_{v(-e)}}{q_{v(-e)}}
\right\rceil - 1\right)
\mkern10mu-\mkern7mu
{\textstyle{\sum\limits_{i \in I_v}}}\mkern-3mu
\lfloor y^v_{i} \rfloor,
    \\
\label{eq: y0+ bound from prop 5.1}
y_{0+}^v 
\;&\ge\,\;
\bar{y}_{0+}^v\;
\mkern.5mu-\mkern-3.5mu
{\sum\limits_{e\mkern1mu\in\mkern1mu 
E_{\mathrm{in}}\mkern-1mu(v)}}
\mkern-12mu
\left(\left\lfloor
\mfrac{p_{v(-e)}}{q_{v(-e)}}
\right\rfloor + 1\right)
\mkern10mu-\mkern7mu
{\textstyle{\sum\limits_{i \in I_v}}}\mkern-3mu
\lceil y^v_{i} \rceil,
\end{align}
and assuming
(\ref{eq: induct one interated})
%and (\ref{eq: induct two interated})
for
Theorem~\ref{thm: gluing structure theorem}
%and
%Proposition~\ref{prop: main bounds for y0- and y0+}
implies that
\begin{equation}
\label{eq: y0+ le y0- in iterated torus proof}
y_{0+}^v \le y_{0-}^v.
\end{equation}

Suppose $\boldsymbol{y}^v \in \mathcal{L}^{\min +}_{\textsc{sf}_v}$,
so that
the bound
$\textstyle{\sum\limits_{i\in I_v}} \lfloor y_i^v \rfloor \ge m^+_v$,
together with 
(\ref{eq: y0- bound from prop 5.1}), implies that
\begin{align}
\label{eq: + case bound iterated}
y_{0-}^v 
\;&\le\,\;
\bar{y}_{0-}^v\;
- 
\begin{cases}
2 
& q_v = -1
    \\
1
& J_v \neq \emptyset;\, q_v < -1 \text{ or } 
\frac{p_v}{q_v} > 1
    \\
0
& \text{otherwise}
\end{cases}.
\end{align}
Since 
Proposition~\ref{prop: main bounds for y0- and y0+} tells us
$\bar{y}_{0-}^v \le \frac{q^*_v}{p_v}$, we then have
\begin{equation}
\label{eq: y0s less than vertical asymptote iterated torus}
y_{0+}^v \le y_{0-} \le \mfrac{q^*_v}{p_v} < \mfrac{p^*_v}{q_v} = 
(\phi_{e_v*}^{\P})^{-1}(\infty).
\end{equation}
Since $\varphi_{e_v *}^{\P}$ is locally monotonically decreasing
in the complement of its vertical asymptote at
$(\phi_{e_v*}^{\P})^{-1}(\infty) = \frac{p^*_v}{q_v}$,
line (\ref{eq: y0s less than vertical asymptote iterated torus}) implies
\begin{equation}
\label{eq: bound on j+}
y^{v(e_v)}_{j(e_v)+} := \varphi_{e_v *}^{\P}(y^v_{0-})
\,\le\,
y^{v(e_v)}_{j(e_v)-} := \varphi_{e_v *}^{\P}(y^v_{0+})
< \phi_{e_v*}^{\P}(\infty) = \mfrac{p_v}{q_v}
<
\left\lfloor \mfrac{p_{v}}{q_{v}}\right\rfloor +1,
\end{equation}
where 
$\phi_{e_v*}^{\P}(\infty) 
=\frac{p_v}{q_v}$ is the location of the horizontal asymptote of $\varphi_{e_v *}^{\P}$.
Thus, since
$\left\lceil \frac{p_v}{q_v}\right\rceil -1 < 
\frac{p_v}{q_v} = \phi_{e_v*}^{\P}(\infty)$,
we deduce that to finish establishing our inductive hypotheses for $e_v$
in the $\boldsymbol{y}^v \in \mathcal{L}^{\text{min}\,+}_{\textsc{sf}_v}$ case,
%(\ref{eq: induct one interated}) and
%(\ref{eq: induct two interated}) in this case, 
it suffices to show that
\begin{equation}
\label{eq: + case iterated, what we want to show}
y^v_{0-} \;\le\;
(\varphi_{e_v *}^{\P})^{-1}\mkern-4mu\left(\left\lceil\mfrac{p_v}{q_v} \right\rceil - 1\right),
\;\;\text{
with equality only if }\,
Y_{\Gamma_v}\mkern-1mu(\boldsymbol{y}^{\Gamma_v}) \text{ is }\textsc{bc}.
\end{equation}

If $|q_v| = 1$, it is straightforward to compute that
\begin{equation}
\label{eq: + and |q|=1 case}
(\varphi_{e_v *}^{\P})^{-1}\left(\left\lceil\mfrac{p_v}{q_v} \right\rceil - 1\right)
=
\mfrac{p^*_v}{q_v} - 1
=
\begin{cases}
-2
&
(p_v, q_v) = (1,-1)
  \\
-1
&
(p_v, q_v) \neq (1,-1),
q_v = -1
  \\
\hphantom{-}0
&
q_v = 1
\end{cases}.
\end{equation}
Proposition
\ref{prop: main bounds for y0- and y0+}
tells us $\bar{y}_{0-} \le \frac{q^*_v}{p_v}$,
with equality only if
$Y_{\Gamma_v}(\boldsymbol{y}^{\Gamma_v})$ is boundary compressible.
Thus, since $\frac{q^*_v}{p_v} < 1$, with $\frac{q^*_v}{p_v} = 0$
when $p_v = 1$, it follows from 
(\ref{eq: + case bound iterated}) and
(\ref{eq: + and |q|=1 case})  that
(\ref{eq: + case iterated, what we want to show}) holds.

Next suppose that $|q_v| > 1$, so that
\begin{equation}
(\varphi_{e_v *}^{\P})^{-1}\mkern-3.5mu\left(\left\lceil\mfrac{p_v}{q_v} \right\rceil - 1\right)
=
\mfrac{p^*_v}{q_v} -
\pm \mfrac{1}{[\pm p_v]_{q_v} (q_v)}
\;\;\;
\text{ for }
\pm q_v > 1.
\end{equation}
If $0< \frac{p_v}{q_v} < 1$, this makes
$(\varphi_{e_v *}^{\P})^{-1}\mkern-4mu
\left(\left\lceil\frac{p_v}{q_v} \right\rceil - 1\right) = \frac{q^*_v}{p_v}$,
so that
(\ref{eq: + case iterated, what we want to show})
follows from 
(\ref{eq: + case bound iterated})
and the fact that 
$\bar{y}_{0-} \le \frac{q^*_v}{p_v}$,
with equality only if
$Y_{\Gamma_v}(\boldsymbol{y}^{\Gamma_v})$ is boundary compressible.
This leaves the cases in which
$\frac{p_v}{q_v} > 1$ or
$q_v < -1$.
If $J_v = \emptyset$, then
$\mathcal{L}^{\min +}_{\textsc{sf}_v}$ respectively excludes
$\left\{ \sum \lfloor y_i^v\rfloor = \sum \lfloor [y_i^v][-p_v]_{q_v} \rfloor = 0 \right\}$
or
$\left\{ \sum \lfloor y_i^v\rfloor = \sum \lfloor [y_i^v][p_v]_{q_v} \rfloor = 0 \right\}$,
and so
(\ref{eq: + case iterated, what we want to show}) follows from
part $(i)$ or $(ii)$, respectively, of 
Proposition \ref{prop: bounds good for iterated case}.
If $J_v \neq \emptyset$, then
$y_{0-}^v \le \mfrac{q^*_v}{p_v} - 1$, and it is easy to show that
\begin{equation}
\mfrac{q^*_v}{p_v} - 1 < \mfrac{p^*_v}{q_v} -
\pm \mfrac{1}{[\pm p_v]_{q_v} (q_v)}
\;\;\;
\text{ for }
\pm q_v > 1,
\end{equation}
completing our
inductive step
for the 
case of $\boldsymbol{y}^v \in \mathcal{L}^{\min +}_{\textsc{sf}_v}$.

The proof of our inductive step for the case of 
$\boldsymbol{y}^v \in \mathcal{L}^{\min -}_{\textsc{sf}_v}$
follows from symmetry under orientation reversal.

\vspace{.2cm}
Recall that we regard the root vertex $r$ at the bottom of the tree 
$\Gamma = \Gamma_r$ as having an outgoing edge $e_r$ pointing to
the empty vertex $v(e_r) := \textsc{null}$,
where this null vertex $v(e_r)$ 
corresponds to the exterior 
$Y := S^3 \setminus \overset{\mkern2mu\circ}{\nu}(K)$ 
of the companion knot $K \subset S^3$,
from which the satellite exterior 
$Y^{\Gamma} := Y_{\Gamma} \cup Y$ is formed.
Recursing down to this null vertex $v(e_r)$,
our above induction shows that
for any
$\boldsymbol{y}^{\Gamma}
\in
\prod_{v \in \text{Vert}(\Gamma_{r})}
\left(\mathcal{L}^{\min+}_{\textsc{sf}_v} 
\cup
\mathcal{R}_{\textsc{sf}_v} \mkern-4mu\setminus\mkern-1mu \mathcal{Z}_{\textsc{sf}_v}
\cup
\mathcal{L}^{\min-}_{\textsc{sf}_v}\right)$, we have
\begin{equation}
\left\lceil \mfrac{p_r}{q_r} \right\rceil - 1 
\;\le\; 
y^{v(e_r)}_{j(e_r)+} 
\;\le\;  
y^{v(e_r)}_{j(e_r)-} 
\;\le\;
\left\lfloor \mfrac{p_r}{q_r} \right\rfloor + 1.
\end{equation}
This final L-space interval
$\varphi_{e_r *}^{\P}(
\mathcal{L}_{\textsc{sf}}(Y_{\Gamma_r}(\boldsymbol{y}^{\Gamma})))
=[[y^{v(e_r)}_{j(e_r)-}, y^{v(e_r)}_{j(e_r)+}]]_{\overline{S}^3}$
expresses slopes in terms of the {\em{reversed}} $S^3$-slope basis, ``$\overline{S}^3$.''
That is, the meridian and longitude of $K\subset S^3$
have respective $\overline{S}^3$-slopes $0 = \frac{1}{\infty}$ and
$\infty = \frac{1}{0}$.
If $K \subset S^3$ is the unknot, then its $\overline{S}^3$-longitude $\infty$ satisfies 
$\infty \in [[y^{v(e_r)}_{j(e_r)-}, y^{v(e_r)}_{j(e_r)+}]]$,
so we deduce that 
$\boldsymbol{y}^{\Gamma} \in \mathcal{L}_{\textsc{sf}_{\Gamma}}(Y^{\Gamma})$ in this case.

If $K \subset S^3$ is nontrivial, then its exterior $Y$
has L-space interval $\mathcal{L}_{\overline{S}^3}(Y) = [0, \frac{1}{N}]_{\overline{S}^3}$.
Assume that $\Gamma$ satisfies hypothesis $(iii)$ of the theorem.
Then since $\frac{q_r}{p_r} \ge N := 2g(K) -1$ implies
$0< \frac{p_r}{q_r} \le 1$, 
(\ref{eq: induct one interated}) and
(\ref{eq: induct two interated})
tell us that
$0 \le y^{v(e_r)}_{j(e_r)+} \le y^{v(e_r)}_{j(e_r)-}$, with 
$y^{v(e_r)}_{j(e_r)+} = 0$ only if $Y_{\Gamma_r}(\boldsymbol{y}^{\Gamma})$
is \textsc{bc}.  Thus, it remains to show that
either $y^{v(e_r)}_{j(e_r)-} < \frac{1}{N}$,
or $y^{v(e_r)}_{j(e_r)-} \le \frac{1}{N}$ and 
$Y_{\Gamma}(\boldsymbol{y}^{\Gamma})$ is \textsc{bc}.

If $\boldsymbol{y}^{\Gamma}|_r \in \mathcal{L}^{\min +}_r$,
then (\ref{eq: bound on j+}) tells us
$y^{v(e_r)}_{j(e_r)-} < \frac{p_r}{q_r} < \frac{1}{N}$.
If $\boldsymbol{y}^{\Gamma}|_r \in \mathcal{R}_r\setminus \mathcal{Z}_r$,
then $Y_{\Gamma}(\boldsymbol{y}^{\Gamma})$ is \textsc{bc},
and since $y_{0-}^r = y_{0+}^r = \infty$, we have
$y^{v(e_r)}_{j(e_r)-} = y^{v(e_r)}_{j(e_r)+} = 
\frac{p_r}{q_r} \in \left<0, \frac{1}{N} \right]$.
This leaves us with the case of
$\boldsymbol{y}^{\Gamma}|_r \in \mathcal{L}^{\min -}_r$,
for which we have
\begin{equation}
\label{eq: last step of satelliting bit}
y_{0+}^r \ge\,
\bar{y}^r_{0+}
+
\begin{cases}
1
&
p\neq 1
 \\
2
&
p=1
\end{cases}
\;
\ge\,
\mfrac{q^*}{p}
+
\begin{cases}
1
&
p\neq 1
 \\
2
&
p=1
\end{cases}
\;
>\,
\mfrac{q^*}{p}
+
\mfrac{1}{p(q-pN)}
\,=\,
(\varphi_{e_r *}^{\P})^{-1}\mkern-5mu\left(\mkern-.5mu\mfrac{1}{N}\mkern-.5mu\right),
\end{equation}
implying 
$y^{v(e_r)}_{j(e_r)-} < \frac{1}{N}$,
and thereby completing the proof of the theorem.
\end{proof}

\noindent {\textbf{Remark.}}
It is only in
(\ref{eq: last step of satelliting bit})
that we use the hypothesis of part $(iii)$
that $q_r > 2g(K) - 1 =: N$. In the case that we do have
$p_r = 1$ and $q_r = N$, this implies
$(\varphi_{e_r *}^{\P})^{-1}\mkern-5mu\left(\mkern-.5mu\mfrac{1}{N}\mkern-.5mu\right) = \infty$,
requiring  $\mathcal{L}_{\textsc{sf}_r}^{\min -}$ to be empty, but we still have the modified result that
% *** squash?
\vspace{-.1cm}
\begin{equation}
\label{eq: special case of p=1 and q=N}
\mathcal{L}_{\textsc{sf}_{\Gamma}}(Y^\Gamma) \supset
\left(\prod_{v \in \mathrm{Vert}(\Gamma)}
\left(
\mathcal{L}^{\min -}_{\textsc{sf}_v}
\;\cup\;
\mathcal{R}_v \setminus \mathcal{Z}_v
\;\cup\;
\mathcal{L}^{\min +}_{\textsc{sf}_v}
\right) \right) \setminus \mathcal{L}_{\textsc{sf}_r}^{\min -}.
\end{equation}

\section{Satellites by algebraic links}
\label{s: algebraic link satellites}

\subsection{Smooth and Exceptional splices}
As mentioned in the introduction,
the class of algbraic link satellites is slightly more general
than the class of iterated torus link satellites,
in that the JSJ decomposition graph of the exterior 
must allow one extra type of edge.

To describe how this new type of edge is different, 
we first address the notion of {\em{splice}} maps.
Suppose $K_1 \mkern-2mu\subset\mkern-2mu M_1$ and 
$K_2 \mkern-2mu\subset\mkern-2mu  M_2$ are knots in
compact oriented 3-manifolds $M_1$ and $M_2$,
with each $\partial M_i$ a possibly-empty disjoint union of tori.
Let $Y_i \mkern-2mu:=\mkern-2mu M_i \setminus \overset{\mkern2mu\circ}{\nu}(K)$
denote the exterior of each knot $K_i \mkern-2mu\subset\mkern-2mu M_i$,
with $\partial_0 Y_i \mkern-2mu:=\mkern-2mu -\partial \overset{\mkern2mu\circ}{\nu}(K_i)$,
and choose a surgery basis $(\mu_i, \lambda_i) \in H_1(\partial_0 Y_i; \Z)$
for each exterior, with $\lambda_i$ the Seifert longitude if 
$M_i$ is either an integer homology sphere or a link
exterior in a specified integer homology sphere.

A gluing map $\phi : \partial_0 Y_1 \to -\partial_0 Y_2$
is then called a {\em{splice}} if the induced map on homology
sends $\mu_1 \mapsto \lambda_2$ and $\mu_2 \mapsto \lambda_1$.
Gluings via splice maps are minimally disruptive to homology.
%preserves more structure than a generic gluing map.
For instance, if $M_2$ is an integer homology sphere, then
$H_1(Y_1 \cup_{\phi} \mkern-2mu Y_2;\Z) \cong H_1(M_1;\Z)$.
If $M_2 \mkern-2mu=\mkern-2mu S^3$ and 
$K_2$ is an unknot, then we in fact have
$Y_1 \cup_{\phi} \mkern-2mu Y_2 = M_1$.  In particular, if $M_2$ is the
exterior $M_2 \mkern-2mu=\mkern-2mu S^3 \setminus \overset{\mkern2mu{\circ}}{\nu}(L_2)$
of some link $L_2 \mkern-2mu\subset\mkern-2mu S^3$, and if
$K_2 \mkern-2mu\subset\mkern-2mu M_2$ is an unknot in the composition
$K_2 \mkern2.5mu\into\mkern2.5mu M_2 \mkern1.5mu\into\mkern1.5mu S^3$,
then $Y_1 \cup_{\phi} \mkern-2mu Y_2$ is the exterior of the
satellite link
of the companion knot $K_1 \mkern-3.5mu\subset\mkern-3mu M_1$
by the pattern link
$L_2 \mkern-2.5mu\subset\mkern-2mu 
(S^3 \mkern-1mu\setminus\mkern-1mu \overset{\mkern2mu\circ}{\nu}(K_2))$.
In particular, for satellites by $T(np,nq)$,
this unknot $K_2 \subset M_2$
is the multiplicity-$q$ fiber $\lambda_0 \subset Y^n_{(p,q)}$
in the $T(np,nq)$-exterior 
\vspace{-.1cm}
%**** squashing?
\begin{equation}
Y^n_{(p,q)} := S^3 \mkern-1.2mu\setminus\mkern-1mu 
\overset{\circ}{\nu}(T(np,nq)
=M_{S^2}(\mkern1mu{\textstyle{
\mkern-2mu-\mkern-1mu\frac{q}{p},
\mkern-1mu{\textstyle{\frac{p}{q}}}}}\mkern1mu)
\mkern-1.2mu\setminus\mkern-1mu \overset{\circ}{\nu}
{\textstyle{\left(\coprod_{i=1}^n f_n\right)}},
\;\;\;
\textit{c.f.}\;
(\ref{eq: T(np,nq) as reg fibers in Seifert fibered space}).
\end{equation}

\vspace{-.06cm}
%**** squashing?

% $T_v \subset \hat{Y}_{(p_v\mkern-1mu, \mkern1mu q_v)}$
%in the solid torus
%$\hat{Y}_{(p_v\mkern-1mu,\mkern1mu q_v)} 
%\mkern1mu:=\mkern1mu 
%M_{S^2}(\mkern1mu{\textstyle{
%\mkern-2mu-\mkern-1mu\frac{q^*_v}{p_v},
%\mkern-1mu{\textstyle{\frac{p^*_v}{q_v}}}}}\mkern1mu)
%\mkern-1.2mu\setminus\mkern-1mu \overset{\circ}{\nu}(\lambda_0)
%=\nu(\lambda_{-1})$
%as in (\ref{eq: iterated torus link solid torus}),
%our unknot $K_2 \mkern-2mu\subset\mkern-2mu M_2$
%is the multiplicity-$q_v$ Seifert fiber $\lambda_0$.

In an iterated torus-link satellite, 
we only perform satellites on components of the companion link we are building.
That is, for an edge $e \in \text{Edge}(\Gamma)$ from $v(-e)$ to $v(e)$,
we always form a $T_{v(e)}$-torus-link satellite that splices 
the multiplicity-$q_{v(-e)}$-fiber $\lambda_0^{v(-e)}$,
with exterior
\vspace{-.08cm}
%**** squashing?
\begin{equation}
Y_{v(-e)} := Y^{n_{v(-e)}}_{(p_{v(-e)},q_{v(-e)})} \setminus
\overset{\circ}{\nu}(\lambda_0^{v(-e)}),
\end{equation}

\vspace{-.08cm}
%**** squashing?

\noindent to the $j(e)^{\text{th}}$ component of the $T_{v(e)}$ torus link.
Since this $j(e)^{\text{th}}$ link component is represented by 
the {\em{smooth}} fiber 
$f_{j(e)} \subset M_{S^2}(\mkern1mu{\textstyle{
\mkern-2mu-\mkern-1mu\frac{q^*_v}{p_v},
\mkern-1mu{\textstyle{\frac{p^*_v}{q_v}}}}}\mkern1mu)$,
we call this operation a 
{\em{smooth splice}}.

In an algebraic link exterior, however, an edge $e$ can also specify an
{\em{exceptional splice}} map, in which we splice
%the exterior of $Y_{v(-e)}$ of $T_{v(-e)}$ to the exterior 
%$Y_{v(e)}$ of $T_{v(e)}$, but {\em{not along a link component of $T_{v(e)}$}}.
%Instead, we splice
the $q_{v(-e)}$-fiber $\lambda_0^{v(-e)}$
%with exterior
%$Y_{v(-e)} := Y^{n_{v(-e)}}_{(p_{v(-e)},q_{v(-e)})} \setminus
%\overset{\circ}{\nu}(\lambda_0^{v(-e)})$,
to the 
exceptional
$p_{v(e)}$-fiber $\lambda_{-1}^{v(e)} \subset Y_{v(e)}$.
This multiplicity-$p_{v(e)}$ fiber $\lambda_{-1}^{v(e)}$ is not a component  of our
original companion link or its iterated satellites, but is rather
the {\em{core of the solid torus
$\nu(\lambda_{-1}^{v(e)})$ hosting
$T_{v(e)} \subset \nu(\lambda_{-1}^{v(e)})$, of which
$Y_{v(e)} = \nu(\lambda_{-1}^{v(e)}) \setminus \overset{\circ}{\nu}(T_{v(e)})$
is the exterior}}.
Thus an exceptional-splice satellite embeds the
solid torus hosting $T_{v(-e)}$
inside the solid torus hosting $T_{v(e)}$.
Since an exceptional splice at $e$ takes the satellite of 
the $p_{v(e)}$-fiber $\lambda_{-1}^{v(e)} \subset Y_{v(e)}$, we set $j(e) = -1$
in this case.
%\vspace{-.3cm}
%**** squashing?

\subsection{Slope maps induced by splices}
For the induced maps on slopes, we have
%**** squashing?
%\vspace{-.3cm}
\begin{align}
\text{Smooth splice}&\;\; \varphi_e : \partial_0 Y_{v(-e)} \to -\partial_{j(e)\mkern-2mu} Y_{v(e)},
 \\
\nonumber
[\varphi_{e*}^{\P}] 
\left(
\mkern-6mu
\begin{array}{cc}
p^*_{v(-e)} &  -q^*_{v(-e)}
  \\
q_{v(-e)} &  -p_{v(-e)}
\end{array}
\mkern-6mu
\right)
=&
\left(
\mkern-5mu
\begin{array}{cc}
1 & 0
  \\
0 & 1
\end{array}
\mkern-5mu
\right)
\implies 
\varphi_{e*}^{\P}(y) = \mfrac{p_{v(-e)}y - q^*_{v(-e)}}{q_{v(-e)}y - p^*_{v(-e)}}
\end{align}
for an edge $e$ corresponding to a smooth splice, and
\begin{align}
\label{eq: def of sigma}
\text{Exceptional splice}\;\; &\sigma_e : \partial_0 Y_{v(-e)} \to -\partial_{-1\mkern-2mu} Y_{v(e)},
 \\
\nonumber
[\sigma_{e*}^{\P}] \mkern-2mu
\left(
\mkern-6mu
\begin{array}{rl}
p_{v(-e)}^* &  \mkern-5mu -q^*_{v(-e)}
  \\
q_{v(-e)} &  \mkern-5mu  -p_{v(-e)}
\end{array}
\mkern-7mu
\right)
\mkern-3mu=\mkern-3mu
\left(
\mkern-7mu
\begin{array}{rr}
p^*_{v(e)} & \mkern-5mu   -q^*_{v(e)}
  \\
-q_{v(e)} & \mkern-5mu p_{v(e)}
\end{array}
\mkern-6mu
\right)
&\implies 
[\sigma_{e*}^{\P}] 
\mkern-1mu
=
\mkern-3mu
\left(
\mkern-6mu
\begin{array}{rr}
p^*_{v(e)} & \mkern-5mu   -q^*_{v(e)}
  \\
-q_{v(e)} & \mkern-5mu p_{v(e)}
\end{array}
\mkern-6mu
\right)
\mkern-6mu
\left(
\mkern-6mu
\begin{array}{rl}
p_{v(-e)} &  \mkern-5mu -q^*_{v(-e)}
  \\
q_{v(-e)} &  \mkern-5mu  -p^*_{v(-e)}
\end{array}
\mkern-7mu
\right)
\end{align}
for an edge $e$ corresponding to an exceptional splice.
To accommodate our notation
to these two different types of maps, we define
\begin{equation}
\phi_e:=
\begin{cases}
\sigma_e
&
j(e) = -1
 \\
\varphi_e
&
j(e) \neq -1.
\end{cases}
\end{equation}
In addition, since
the exceptional fiber at $\partial_{-1} Y_v$
is only exceptional if $p_v \!>\! 1$,
we adopt the convention that
exceptional splice edges only terminate
on vertices $v$ with $p_v > 1$.

\subsection{Algebraic links}

Eisenbud and Neumann show in \cite{EisenbudNeumann}
that a graph $\Gamma$ with such edges and vertices
specifies an algebraic link exterior
if and only if
$\Gamma$ satisfies the algebraicity conditions
\begin{align}
\nonumber
(i)&\;\; p_v, q_v, n_v > 0 \;\;\text{for all } v\in \text{Vert}(\Gamma),
    \\
(ii)&\;\;\mkern46mu \Delta_e > 0 \;\;\text{for all } e\in \text{Edge}(\Gamma),
\label{eq: alg cond}
    \\
\nonumber
\Delta_e :=\mkern20mu&\mkern-20mu
\begin{cases}
\;p_{v(e)}q_{v(-e)} - p_{v(-e)}q_{v(e)}
 & \;\;j(e) = -1
   \\
\;q_{v(-e)} - p_{v(e)} p_{v(-e)}q_{v(e)}
 & \;\;j(e) \neq -1
\end{cases}.
\end{align}
ensuring negative definiteness.
Eisenbud and Neumann also prove
that any algebraic link exterior can be realized by
such a graph.
Note that the above algebraicity conditions imply
\begin{equation}
0 < \mfrac{p_{v(-e)}}{q_{v(-e)}} < 1
\;\;\;\;\text{if}\;\;\;\; 
j(e) \neq -1.
\end{equation}

For notational convenience, we adopt the convention that
$J_v$ remains the same, only indexing incoming edges corresponding to {\em{smooth}} splices.
That is, we define
\begin{equation}
J_v := \{j(e)|e\in E_{\mathrm{in}}(v)\} \cap \{1, \ldots, n_v\},
\end{equation}
and its complement $I_v$ still indexes the remaining boundary components,
\begin{equation}
I_v := 
\{1,\ldots, n_v\} \cap J_v.
\end{equation}

Lastly, since $\phi_e$ is always orientation-reversing,
the induced map $\phi^{\P}_{e*}$ is still decreasing with respect to the circular order on
$\textsc{sf}$-slopes, and the impact of $\phi^{\P}_{e*}$ on the {\em{linear}} order of finite \textsc{sf}-slopes still depends on the
positions of the horizontal and vertical asymptotes
\begin{equation}
\label{eq: asymptote def}
\xi_{v(e)} := 
\phi_{e*}^{\P}(\infty) 
\in (\Q \cup \{\infty\})_{\textsc{sf}_{v(e)}}
\;\;\;\text{and}\;\;\;
\eta_{v(-e)} := 
(\phi_{e*}^{\P})^{-1}(\infty)
\in (\Q \cup \{\infty\})_{\textsc{sf}_{v(-e)}},
\end{equation}
respectively,
of the graph of $\phi^{\P}_{e*}$. More explicitly, we have

\begin{equation}
\label{eq: asymptote value}
\xi_{v(e)}
= 
\begin{cases}
\mfrac{p_{v(-e)}}{q_{v(-e)}}
&
j(e) \neq -1
\vspace{.1cm}
\\
-\mfrac{q^*_{v(e)}}{p_{v(e)}} + 
\mfrac{p_{v(-e_)}}{p_{v(e)}\Delta_{e}}
&
j(e) = -1
\end{cases},
\;\;\;\;
\eta_{v(-e)}
=
\begin{cases}
\mfrac{p^*_{v(-e)}}{q_{v(-e)}}
&
j(e) \neq -1
\vspace{.1cm}
  \\
\mfrac{q^*_{v(-e)}}{p_{v(-e)}} + 
\mfrac{p_{v(e)}}{p_{v(-e)}\Delta_{e}}
&
j(e) = -1
\end{cases}.
\end{equation}

\subsection{Adapting Propositions \ref{prop: basic defs of y_0+-} and
\ref{prop: main bounds for y0- and y0+} for algebraic link exteriors}

%as in the proof of
%Theorem~\ref{thm: iterated torus links satellites},
In the case of algebraic link exteriors, 
we must incorporate the possiblity of exceptional splices into
the expressions
$\bar{y}^v_{0-}(k)$ and 
$\bar{y}^v_{0+}(k)$ originally defined in 
(\ref{eq: def of y_0-v-bar for iterated}) and
(\ref{eq: def of y_0+v-bar for iterated}),
from which
\begin{equation}
\label{eq: def of y0}
\bar{y}^v_{0-} \mkern-3mu:= \sup_{k >0} \bar{y}^v_{0-}(k)
\;\;\;\text{and}\;\;\;
\bar{y}^v_{0+} \mkern-3mu:= \inf_{k >0} \bar{y}^v_{0+}(k)
\end{equation}
are defined.
The only changes that arise are localized to the summands
$\left\lceil
\mkern-2.5mu\frac{q^*_v}{p_v} k \mkern-2mu
\right\rceil$
and
$\left\lfloor
\mkern-2.5mu\frac{q^*_v}{p_v} k \mkern-2mu
\right\rfloor$
in $\bar{y}^v_{0-}(k)$ and 
$\bar{y}^v_{0+}(k)$, respectively.
We perform such modifications as follows.

First, for $v \in \mathrm{Vert}(\Gamma)$, with $\Gamma$ specifying an algebraic link exterior,
set
\begin{equation}
y^v_{-1\pm} 
:=
\begin{cases}
-\frac{q^*_v}{p_v}
&
-1 \notin j(E_{\mathrm{in}}(v))
  \\
y^v_{j(e')\pm} \mkern-3mu= 
\phi_{e'*}^{\P}\mkern-5mu\left( y^{v(-e')}_{0\mp}\right)
&
-1 = j(e'), e' \mkern-2mu\in\mkern-2mu E_{\mathrm{in}}(v),
\end{cases}
\end{equation}
and for  
$k \in \Z_{>0}$, 
define 
$y^{v\,\prime}_{-1+}(k)$ and $y^{v\,\prime}_{-1-}(k)$
by setting
\begin{align}
y^{v\,\prime}_{-1+}(k)
&:=
\begin{cases}
-\left(\left\lceil y^v_{-1 +} k\right\rceil - 1\right)
  &
Y_{\Gamma_{v(-e')}}(\boldsymbol{y}^{\Gamma_{v(-e')}}) \;\textsc{bi}
%-1 = j(e'), e' \mkern-2mu\in\mkern-2mu E_{\mathrm{in}}(v),
%Y_{\Gamma_{v(-e')}}(\boldsymbol{y}^{\Gamma_{v(-e')}}) \text{ is}\;\textsc{bi}
   \\
-\left\lfloor y^v_{-1+} k\right\rfloor
   &
Y_{\Gamma_{v(-e')}}(\boldsymbol{y}^{\Gamma_{v(-e')}}) \;\textsc{bc}
%-1 = j(e'), e' \mkern-2mu\in\mkern-2mu E_{\mathrm{in}}(v),
%Y_{\Gamma_{v(-e')}}(\boldsymbol{y}^{\Gamma_{v(-e')}}) \text{ is}\;\textsc{bc}
%\text{ or }
%-1 \notin j(E_{\mathrm{in}}(v))
\end{cases},
%\;\;\;\;\;\le\;\;
%\left\lceil \mfrac{q^*_v}{p_v} k \right\rceil,
    \\
y^{v\,\prime}_{-1-}(k)
&:=
\begin{cases}
-\left(\left\lfloor y^v_{-1 -} k\right\rfloor + 1\right)
  &
Y_{\Gamma_{v(-e')}}(\boldsymbol{y}^{\Gamma_{v(-e')}}) \;\textsc{bi}
%-1 = j(e'), e' \mkern-2mu\in\mkern-2mu E_{\mathrm{in}}(v),
%Y_{\Gamma_{v(-e')}}(\boldsymbol{y}^{\Gamma_{v(-e')}}) \text{ is}\;\textsc{bi}
   \\
-\left\lceil y^v_{-1-} k\right\rceil
   &
Y_{\Gamma_{v(-e')}}(\boldsymbol{y}^{\Gamma_{v(-e')}}) \;\textsc{bc}
%-1 = j(e'), e' \mkern-2mu\in\mkern-2mu E_{\mathrm{in}}(v),
%Y_{\Gamma_{v(-e')}}(\boldsymbol{y}^{\Gamma_{v(-e')}}) \text{ is}\;\textsc{bc}
%\text{ or }
%-1 \notin j(E_{\mathrm{in}}(v))
\end{cases},
%\;\;\;\;\;\ge\;\;
%\left\lfloor \mfrac{q^*_v}{p_v} k \right\rfloor + m_{-1}^v
\end{align}
where again, $e'\in E_{\mathrm{in}}(v)$ is the unique incoming edge with $j(e') = -1$,
if such $e'$ exists. If $-1 \notin j(E_{\mathrm{in}}(v))$,
then we take
$Y_{\Gamma_{v(-e')}}(\boldsymbol{y}^{\Gamma_{v(-e')}})$
to be boundary-compressible.

Next, we define 
$\bar{y}^{v\,\prime}_{0-}(k)$ and $\bar{y}^{v\,\prime}_{0+}(k)$
to be respective results of replacing the summand
$\left\lceil
\mkern-2.5mu\frac{q^*_v}{p_v} k \mkern-2mu
\right\rceil$
with 
$y^{v\,\prime}_{-1+}(k)$ in 
the definition of $\bar{y}^v_{0-}(k)$ in (\ref{eq: def of y_0-v-bar for iterated}),
and replacing the summand
$\left\lfloor
\mkern-2.5mu\frac{q^*_v}{p_v} k \mkern-2mu
\right\rfloor$
with 
$y^{v\,\prime}_{-1-}(k)$ in 
the definition of $\bar{y}^v_{0+}(k)$ in (\ref{eq: def of y_0+v-bar for iterated}).
That is, we set
\begin{align}
\label{eq: y'0-(k)}
\bar{y}^{v\,\prime}_{0-}(k)
&:=
\bar{y}^v_{0-}(k) +
\mfrac{1}{k}\mkern-3.5mu\left(
y^{v\,\prime}_{-1+}(k)
-\left\lceil \mfrac{q^*_v}{p_v}k \right\rceil \right),
   \\
\label{eq: y'0+(k)}
\bar{y}^{v\,\prime}_{0+}(k)
&:=
\bar{y}^v_{0-}(k) +
\mfrac{1}{k}\mkern-3.5mu\left(
y^{v\,\prime}_{-1-}(k)
- \left\lfloor \mfrac{q^*_v}{p_v}k \right\rfloor
\right),
\end{align}
and by analogy with the definition of
$\bar{y}_{0\pm}^v$
in (\ref{eq: def of y0}),
we define
\begin{equation}
\label{eq: defs of y0'}
\bar{y}^{v\,\prime}_{0-} \mkern-3mu:= \sup_{k >0} \bar{y}^{v\,\prime}_{0-}(k)
\;\;\;\text{and}\;\;\;
\bar{y}^{v\,\prime}_{0+} \mkern-3mu:= \inf_{k >0} \bar{y}^{v\,\prime}_{0+}(k).
\end{equation}
We are now ready to state and prove an analog of 
Proposition~\ref{prop: basic defs of y_0+-} and
a supplement to
%an extension to
Proposition~\ref{prop: main bounds for y0- and y0+}.

\begin{prop}
\label{prop: compute y0 for alg links}
Suppose $v \mkern-3mu\in\mkern-3mu \mathrm{Vert}(\Gamma)$ for a graph 
$\Gamma$ specifying the exterior of an algebraic link.
If $y_{0-}^v, y_{0+}^v \mkern-3mu\in\mkern-3mu \P(H_1(\partial_0 Y_v; \Z))_{\textsc{sf}_v}$
are the (potential) L-space interval endpoints for 
$Y_{\Gamma_v}\mkern-2mu({\boldsymbol{y}}^{\Gamma_v})$
as defined in
Theorem~\ref{thm: l-space interval for seifert jsj}, then
\begin{align*}
y_{0-}^v 
&= \bar{y}_{0-}^{v \, \prime}
\mkern.5mu-\mkern-3.5mu
{\textstyle{\sum\limits_{\hphantom{1\mkern-2mu} 
j \in J_v^{\textsc{bi}}}}}\mkern-9mu
\left(\mkern-1mu\lceil y^v_{j+} \mkern-2mu\rceil \mkern-3mu-\mkern-2.5mu 1\right)
\mkern2mu-\mkern-7mu
{\textstyle{\sum\limits_{\hphantom{1\mkern-2mu} 
j \in J_v^{\textsc{bc}}}}}\mkern-8mu
\lfloor y^v_{j} \rfloor
\mkern10mu-\mkern7mu
{\textstyle{\sum\limits_{i \in I_v}}}\mkern-6mu
\lfloor y^v_{i} \rfloor,
    \\
y_{0+}^v 
&= \bar{y}_{0+}^{v \, \prime}
\mkern.5mu-\mkern-3.5mu
{\textstyle{\sum\limits_{\hphantom{1\mkern-2mu} 
j \in J_v^{\textsc{bi}}}}}\mkern-9mu
\left(\mkern-1mu\lfloor y^v_{j-} \mkern-2mu\rfloor \mkern-3mu+\mkern-2.5mu 1\right)
\mkern2mu-\mkern-7mu
{\textstyle{\sum\limits_{\hphantom{1\mkern-2mu} 
j \in J_v^{\textsc{bc}}}}}\mkern-8mu
\lceil y^v_{j} \rceil
\mkern10mu-\mkern7mu
{\textstyle{\sum\limits_{i \in I_v}}}\mkern-6mu
\lceil y^v_{i} \rceil,
\end{align*}
for 
$J_v^{\textsc{bc}}$ and $J_v^{\textsc{bi}}$
as defined in 
(\ref{eq: def of J BC}) and
(\ref{eq: def of J BI}).
%Moreover, $\mkern1.5mu \bar{y}^{v \, \prime}_{0-} \mkern-2mu\le 
%\left\lceil \frac{q^*_v}{p_v} \right\rceil - 1$
%if $J_{v\Z}^{\textsc{bi}+} \neq \emptyset$, and
%$\mkern1.5mu\bar{y}^{v \, \prime}_{0+} \mkern-4mu\ge\mkern-1mu 1$ if 
%$J_{v\Z}^{\textsc{bi}-} \mkern-3mu\neq\mkern-1mu \emptyset$,
%where $J_v^{\textsc{bc}}$, $J_v^{\textsc{bi}}$, and $J_{v\Z}^{\textsc{bi}\pm}$ are defined in 
%(\ref{eq: def of J Z}).
\end{prop}
\begin{proof}
This follows directly from
Theorem~\ref{thm: l-space interval for seifert jsj}.

\end{proof}

\begin{prop}
\label{prop: y0bar y0prime compare}
Suppose that
$\Gamma$ specifies the exterior of an algebraic link,
and that $v \mkern-3mu\in\mkern-3mu \mathrm{Vert}(\Gamma)$ 
has an incoming edge $e'$ with 
$j(e') = -1$.  If
\begin{equation}
\label{eq: fake inductive statement}
-\mfrac{q^*_v}{p_v}
\;\le\; 
y^v_{j(e')+} 
\;\le\; 
y^v_{j(e')-} 
\;\le\; 
-\mfrac{q^*_v}{p_v}
\mkern1mu+\mkern1mu m^{e'-},
\end{equation}
%with equality only if
%$Y_{\Gamma_{v(-e)}}(\boldsymbol{y}^{\Gamma_{v(-e)}})$
%is \textsc{bc},
for some
$m^{e'-} \in \Z_{>0}$,
then $\bar{y}_{0-}^{v \, \prime}$ and $\bar{y}_{0+}^{v \, \prime}$ satisfy
\begin{equation}
\bar{y}_{0-}^{v \, \prime} \le \bar{y}_{0-}^v,
\;\;\;\;\;\;\;\;
%\;\;\text{and}\;\;
\bar{y}_{0+}^{v \, \prime} \ge \bar{y}_{0+}^v + m^{e'-}.
\end{equation}

\end{prop}

\begin{proof}
%Since $j(e') = -1$, 
It is straightforward to show that the bounds in
(\ref{eq: fake inductive statement}) imply that
\begin{equation}
y^{v\,\prime}_{-1+}(k)
\le 
\left\lceil
\mfrac{q^*_v}{p_v} k
\right\rceil,
\;\;\;\;\;\;\;
y^{v\,\prime}_{-1-}(k)
\ge
\left\lfloor
\mfrac{q^*_v}{p_v} k
\right\rfloor
+ m^{e'-}
\end{equation}
for all $k \in \Z_{>0}$.
Thus, since $m^{e'-} \in \Z$ implies that
$\frac{1}{k}\lfloor m^{e'-} k \rfloor = \frac{1}{k}\lceil m^{e'-} k \rceil =  m^{e'-}$,
the desired result follows directly from 
(\ref{eq: y'0-(k)}),
(\ref{eq: y'0+(k)}),
and the definitions of 
$\bar{y}_{0\mp}^{v \, \prime}$ in 
(\ref{eq: defs of y0'}).

\end{proof}

\subsection{L-space surgery regions for algbraic link satellites: Proof of Theorem \ref{thm: algebraic link satellites}}
If $\Gamma$ specifies a one-component algebraic link, {\em{i.e.}}, a knot,
then the L-space region is just an interval, determined
by iteratively computing the genus of successive cables.
For multi-component links, we can bound the L-space region 
as described in 
Theorem~\ref{thm: algebraic link satellites}
in the introduction.

\vspace{.2cm}
{\noindent{\textit{Proof of Theorem \ref{thm: algebraic link satellites}}}}.

The proofs of parts $(i)$ and $(ii)$ are the same as
those in the iterated torus link satellite case,
if one keeps in mind that the $p_r=1$ condition for
$(ii.b)$ and an explicit hypothesis for $(ii.a)$ 
each rule out the possibility of an incoming exceptional splice
at the root vertex.

The proof of part $(iii)$ also adapts the proof used for iterated torus satellites,
but we provide more details in this case.
Again, for bookkeeping convenience, we redefine
$I_{w} \mkern-2.5mu:=\mkern-2.5mu \{1\}$
and set
$\mathcal{L}^{\min +}_{\textsc{sf}_{w}} \mkern-3.5mu:=\mkern-2mu \{0\}$
and $\mathcal{R}_{\textsc{sf}_{w}} \mkern-4mu\setminus\mkern-1mu \mathcal{Z}_{\textsc{sf}_{w}}
\mkern-3.5mu:=\mkern-1.5mu
\mathcal{L}^{\min-}_{\textsc{sf}_{w}} 
\mkern-3.5mu:=\mkern-2mu \emptyset$,
for any 
$w \mkern-2mu\in\mkern-2mu \mathrm{Vert}(\Gamma)$ with
$I_{w} \mkern-3mu=\mkern-1mu \emptyset$.

For a vertex
$v \in \text{Vert}(\Gamma)$,
we inductively assume, for each incoming edge $e \in E_{\mathrm{in}}(v)$, that
for any
$\boldsymbol{y}^{\Gamma_{v(-e)}}
\in
\prod_{u \in \text{Vert}(\Gamma_{v(-e)})}
\left(\mathcal{L}^{\min+}_{\textsc{sf}_u} 
\cup
\mathcal{R}_{\textsc{sf}_u} \mkern-4mu\setminus\mkern-1mu \mathcal{Z}_{\textsc{sf}_u}
\cup
\mathcal{L}^{\min-}_{\textsc{sf}_u}\right)$, we have
\begin{align}
\label{eq: induct bi algebraic}
\;\;\;\;
\mu_e \mkern1mu+\mkern1mu m^{e+}
\;\le\; 
y^v_{j(e)+} 
\,&\le\; 
y^v_{j(e)-} 
\;\le\; 
\mu_e \mkern1mu+\mkern1mu m^{e-},
\hphantom{\;\;\;\;\;
\text{if }\,
Y_{\Gamma_{v(-e)}}(\boldsymbol{y}^{\Gamma_{v(-e)}})
\,\;\text{is}\;\,\textsc{bi}}
    \\
\label{eq: induct bc algebraic}
\;\;\;\;
\mu_e \mkern1mu+\mkern1mu m^{e+}
\;<\; 
y^v_{j(e)+},\mkern-5mu
\,&\mkern24.1mu\; 
y^v_{j(e)-} 
\;<\; 
\mu_e \mkern1mu+\mkern1mu m^{e-}
\;\;\;\;\;
\text{if }\,
Y_{\Gamma_{v(-e)}}(\boldsymbol{y}^{\Gamma_{v(-e)}})
\,\;\text{is}\;\,\textsc{bi},
\end{align}
\vspace{-.7cm}
% **** squashed?
\begin{equation}
\text{where}\;\;
\mu_e  := \phi_{\mkern-1mue*}^{\P}\mkern-6mu
\left(\mkern-3mu\mfrac{-q^*_{v(-e)}\mkern-1mu}{-p_{v(-e)}\mkern-1mu}\mkern-3mu\right)
= 
\begin{cases}
0
 &
j(e) \neq -1
   \\
\frac{-q^*_v}{p_v}
 &
j(e) = -1
\end{cases}
\end{equation}
is the meridian slope of the fiber of $Y_v$ to which $Y_{v(-e)}$
is spliced along $e$ (so the image of the longitude
of slope $\frac{-q^*_{v(-e)}\mkern-1mu}{-p_{v(-e)}\mkern-1mu}$
paired with the meridian of slope
$\frac{p^*_{v(-e)}\mkern-1mu}{q_{v(-e)}\mkern-1mu}$), and where
\begin{align}
m^{e+}
&:=
\begin{cases}
\left\lceil
\mkern-2mu
\frac{p_{v(-e)}}{q_{v(-e)}}
\mkern-2mu
\right\rceil
\mkern-3mu-\mkern-2mu 1
   &
j(e) \mkern-3mu\neq\mkern-3mu -1
   \\
0
&
j(e) \mkern-3mu=\mkern-3mu -1
\end{cases}
\;\;\;
=\; 0,
     \\
m^{e-}
&:=
\begin{cases}
\left\lfloor
\mkern-2mu
\frac{p_{v(-e)}}{q_{v(-e)}}
\mkern-2mu
\right\rfloor
\mkern-3mu+\mkern-2mu 1
   &
j(e) \mkern-3mu\neq\mkern-3mu -1
\vspace{.1cm}
   \\
\left\lceil \mkern-2mu\frac{p_{v(-e)}}{{p_v}\Delta_{e}}\mkern-3mu \right\rceil
\mkern-3mu+\mkern-2mu 1
\vphantom{\mfrac{A^B}{A^B}}
&
j(e) \mkern-3mu=\mkern-3mu -1
\end{cases}.
%\;\;\;
%=\;
%\begin{cases}
%1   &
%j(e) \mkern-3mu\neq\mkern-3mu -1
%   \\
%\left\lceil \mkern-2mu\frac{p_{v(-e)}}{p_{v}\Delta_{e}}\mkern-3mu \right\rceil
%\mkern-3mu+\mkern-2mu 1
%\vphantom{\mfrac{A^B}{A^B}}
%&
%j(e) \mkern-3mu=\mkern-3mu -1
%\end{cases}.
\end{align}
We then set
\begin{align}
m_v^+
&:=
\mkern4mu-\mkern-6mu
{\sum\limits_{e\mkern1mu\in\mkern1mu 
E_{\mathrm{in}}\mkern-1mu(v)}}
\mkern-12mu
m^{e+}
\,+\,
0
\;\;\;
=\;0,
     \\
\label{eq: def of mv}
m_v^-
&:=
\mkern4mu-\mkern-6mu
{\sum\limits_{e\mkern1mu\in\mkern1mu 
E_{\mathrm{in}}\mkern-1mu(v)}}
\mkern-12mu
m^{e-}
\,-\,
\begin{cases}
1
& J_v \neq \emptyset;\, j(e_v) \neq -1
    \\
\left\lceil \mkern-2mu\frac{p_{v(e_v)}}{p_v\Delta_{e_v}}\mkern-3mu \right\rceil 
\mkern-3mu+\mkern-2mu 1
&
j(e_v) \mkern-3mu=\mkern-3mu -1
    \\
0
& \text{otherwise}
\end{cases}.
\end{align}

The statement of
Theorem~\ref{thm: algebraic link satellites}
makes the substitution
$\left\lfloor\mkern-2mu\frac{p_{v(-e)}}{q_{v(-e)}}\mkern-2mu\right\rfloor
\mkern-3mu+\mkern-2mu 1 \,\mapsto\, 1$ 
for the $j(e) \neq -1$ case of $m^{e-}$, in its role as a summand of 
of $m_v^-$.
However, this substitution is an equivalence for all $v \in \mathrm{Vert}(\Gamma)$
and $e \in E_{\mathrm{in}}(v)$,
since the algebraicity condition
(\ref{eq: alg cond})
implies
$0 < \frac{p_{v(-e)}}{q_{v(-e)}} < 1$ for $j(e) \neq -1$.
Note that this does {\em{not}} imply 
$0 < \frac{p_r}{q_r} < 1$ for the root vertex $r = v(-e_r)$,
because of our declared convention that $v(e_r) = \textsc{null} \notin \mathrm{Vert}(\Gamma)$.
The algebraicity condition 
(\ref{eq: alg cond})
also implies that
the conditions $q = +1$, $q < 0$, and $\frac{p_v}{q_v} > +1$
are never met for $j(e_v) \neq -1$ and $v \neq r$.
Thus, the $j(e) \neq -1$ cases of our definitions of
$\mathcal{L}^{\min -}_{\textsc{sf}_v}$
and 
$\mathcal{L}^{\min +}_{\textsc{sf}_v}$ in 
Theorem~\ref{thm: algebraic link satellites}.$iii$
coincide with the respective definitions of
$\mathcal{L}^{\min -}_{\textsc{sf}_v}$
and 
$\mathcal{L}^{\min +}_{\textsc{sf}_v}$ in 
Theorem~\ref{thm: iterated torus links satellites}.$iii$.

If
$\boldsymbol{y}^v \in
\mathcal{R}_{\textsc{sf}_u} \mkern-4mu\setminus\mkern-1mu \mathcal{Z}_{\textsc{sf}_u}$,
then $y_{0-}^v = y_{0+}^v = \infty$. Thus,
$y^{v(e_v)}_{j(e_v)\pm} = \phi^{\P}_{e_v *}(y_{0\mp}^v) = \phi^{\P}_{e_v *}(\infty) 
=: \xi_{v(e_v)}$,
and referring to
%(\ref{eq: asymptote def})
(\ref{eq: asymptote value}) for the computation of $\xi_{v(e_v)}$, we have
\begin{equation}
\label{eq: horiz asymptote bound}
y^{v(e_v)}_{j(e_v)\pm} = 
\xi_{v(e_v)}
=
\mu_{e_v}
+
\begin{cases}
\mfrac{p_v}{q_v}
& j(e_v) \neq -1
 \\
\mfrac{p_v}{p_{v(e_v)}\Delta_{e_v}}
& j(e_v) = -1
\end{cases}
\;\;\in\;\; 
\left<\mu_{e_v} + m^{e_v+},
\mu_{e_v} + m^{e_v-}\right>.
\end{equation}

We assume
$\boldsymbol{y}^v \in 
\mathcal{L}^{\min -}_{\textsc{sf}_v} \cup \mathcal{L}^{\min +}_{\textsc{sf}_v}$
for the remainder.
This assumption, together with our inductive assumptions,
makes 
Theorem~\ref{thm: gluing structure theorem}
yield
\begin{equation}
y^v_{0+} \le y^v_{0-},
%\;\text{with equality only if}\;\; Y_{\Gamma_v}(\boldsymbol{y}^{\Gamma_v})
%\text{ is }
%\textsc{bc},
\end{equation}
%where the restriction on equality comes from 
%Proposition \ref{prop: main bounds for y0- and y0+}.
and Proposition \ref{prop: main bounds for y0- and y0+} tells us
\begin{equation}
\label{eq: alg link from prop for q*/p and y0- y0+}
\bar{y}^v_{0-} \le \mfrac{q^*_v}{p_v} \le \bar{y}^v_{0+},
\;\text{with equality only if}\;\; Y_{\Gamma_v}(\boldsymbol{y}^{\Gamma_v})
\text{ is }
\textsc{bc}.
\end{equation}
We furthermore already know that $\bar{y}_{0\pm}^v = \bar{y}_{0\pm}^{v \, \prime}$
when $-1 \notin j(E_{\mathrm{in}}(v))$. Combining this fact with
Proposition~\ref{prop: y0bar y0prime compare},
given our inductive assumptions, yields
\begin{equation}
\label{eq: y' bound for Lmin- alg link}
\bar{y}_{0-}^{v \, \prime} \le \bar{y}_{0-}^v,
\;\;\;\;\;\;\;\;
\bar{y}_{0+}^{v \, \prime} \ge \bar{y}_{0+}^v
+
\begin{cases}
m^{e'-}
& \exists\, e' \mkern-3.5mu\in\mkern-3.5mu E_\mathrm{in}(v), j(e') = -1
  \\
0
& -1 \notin j(E_\mathrm{in}(v))
\end{cases}.
\end{equation}

Suppose
$\boldsymbol{y}^v \in 
\mathcal{L}^{\min +}_{\textsc{sf}_v}$.
Then from Proposition~\ref{prop: compute y0 for alg links}, we have
\begin{align}
y_{0-}^v 
\,&\le\;
\bar{y}_{0-}^{v \, \prime}
-
{\textstyle{\sum\limits_{\hphantom{1\mkern-2mu} 
j \in J_v}}}\mkern-5mu
y^v_{j+}
\mkern3mu-\mkern3mu
{\textstyle{\sum\limits_{\hphantom{1\mkern-2mu} 
i \in I_v}}}\mkern-6mu
\lfloor y^v_{i} \rfloor
    \\
\,&\le\; 
\bar{y}_{0-}^{v \, \prime}
-
{\textstyle{\sum\limits_{\hphantom{1\mkern-2mu} 
j \in J_v}}}
\mkern-6.5mu m^{e+}
%\mkern-8mu(\mu_e \mkern-1.5mu+\mkern-1.5mu m^{e+})
\mkern-.5mu-\mkern15mu 
m_v^+
  \\
\,&=\;
\bar{y}_{0-}^{v \, \prime}
\mkern25mu\le\mkern25mu
\bar{y}^v_{0-} 
\mkern25mu\le\mkern25mu
\mfrac{q^*_v}{p_v}.
\end{align}
Thus, altogether we have
\begin{equation}
\label{eq: alg link L+ vert asymptote bound}
%y^v_{0+} \le y^v_{0-} \le \bar{y}^{v\,\prime}_{0-} \le \bar{y}^v_{0-} \le \mfrac{q^*_v}{p_v}
y^v_{0+} \le y^v_{0-}  \le \mfrac{q^*_v}{p_v}
<
\eta_v, \;
\text{ with equality only if }\,
Y_{\Gamma_v}(\boldsymbol{y}^{\Gamma_v})
\text{ is }
\textsc{bc}.
\end{equation}
Here, $\eta_v := (\phi_{e_v*}^{\P})^{-1}(\infty)$
is the location of the vertical asymptote
of $\phi_{e_v*}^{\P}$. 
The inequality
$\frac{q^*_v}{p_v} < \eta_v$ follows directly from the computation of
$\eta$ in (\ref{eq: asymptote value}), plus the fact that
$\frac{q^*_v}{p_v} < \frac{p^*_v}{q_v}$.
Since $\phi_{e_v*}^{\P}$ is locally monotonically decreasing on the
complement of $\eta_v$, this implies that
all the expressions on the left-hand side of
(\ref{eq: alg link L+ vert asymptote bound}) have 
$\phi_{e_v*}^{\P}$-images below the horizontal
asymptote at $\xi_{v(e_v)}$, but in reverse order. That is, we have
\begin{equation}
\phi_{e_v*}^{\P}\mkern-4mu\left(\mfrac{q^*}{p_v}\mkern-1mu\right)
\mkern-2mu:=\mkern-1mu
\mu_{e_v}
\mkern10mu\le\mkern10mu
y^{v(e_v)}_{j(e_v)+} 
\mkern-2mu:=\mkern-1mu
\phi_{e_v*}^{\P}(y^v_{0-})
\mkern10mu\le\mkern10mu
y^{v(e_v)}_{j(e_v)-}
\mkern-2mu:=\mkern-1mu
\phi_{e_v*}^{\P}(y^v_{0+}) 
\mkern10mu<\mkern10mu
\xi_{v(e_v)}.
\end{equation}
Thus, since $m^{e_v+}=0$
and since
(\ref{eq: horiz asymptote bound}) shows that
$\xi_{v(e_v)} < \mu_{e_v} + m^{e_v-}$, we obtain
\begin{equation}
\mu_{e_v} + m^{e_v+}
\mkern3mu\le\mkern4mu 
y^{v(e_v)}_{j(e_v)+}
\mkern4mu\le\mkern5mu
y^{v(e_v)}_{j(e_v)-}
\mkern3mu<\mkern4mu
\mu_{e_v} + m^{e_v-},
\end{equation}
with equality only if $Y_{\Gamma_v}(\boldsymbol{y}^{\Gamma_v})$
is \textsc{bc}.

%\begin{equation}
%\xi_{v(e_v)}
%= -\mfrac{q^*_{v(e_v)}}{p_{v(e_v)}} + \mfrac{p_v}{p_{v(e_v)}\Delta_{e_v}}
%\;\;\;
%\eta_v
%=\mfrac{q^*_v}{p_v} + \mfrac{p_{v(e_v)}}{p_{v}\Delta_{e_v}},
%\end{equation}

Lastly, suppose $\boldsymbol{y}^v \in \mathcal{L}^{\min -}_{\textsc{sf}_v}$.
Then combining
Proposition~\ref{prop: compute y0 for alg links}
(for line (\ref{eq: line 1 of Lmin- alg link}))
with
the righthand inequality of 
(\ref{eq: y' bound for Lmin- alg link}),
the inductive upper bounds on $y^v_{j}$ for $j \in J_v := j(E_{\mathrm{in}}(v))|_{>0}$,
and the upper bound 
$\sum_{i \in I_v} \lceil y_i^v \rceil \le m_v^-$ for
$\boldsymbol{y}^v \in \mathcal{L}^{\min -}_{\textsc{sf}_v}$
(for line (\ref{eq: line 2 of Lmin- alg link})),
we obtain
\begin{align}
\label{eq: line 1 of Lmin- alg link} 
y_{0+}^v 
&\ge\; \bar{y}_{0+}^{v \, \prime}
\mkern10mu-\mkern7mu
{\textstyle{\sum\limits_{\hphantom{1\mkern-2mu} 
j \in J_v}}}\mkern-8mu
%\left(\lfloor y^v_{j-} \rfloor + 1\right)
y^v_{j-}
\mkern18mu-\mkern7mu
{\textstyle{\sum\limits_{i \in I_v}}}\mkern-6mu
\lceil y^v_{i} \rceil,
    \\
\label{eq: line 2 of Lmin- alg link} 
&\ge\; \bar{y}_{0+}^v
\mkern10mu-
{\textstyle{\sum\limits_{e \in E_{\mathrm{in}}(v)}}}\mkern-8mu
m^{e-}
\mkern1mu-\mkern20mu
m_v^-.
\end{align}
Combining 
(\ref{eq: alg link from prop for q*/p and y0- y0+})
from
Proposition \ref{prop: main bounds for y0- and y0+}
with the definition 
(\ref{eq: def of mv})
of $m_v^-$ then gives
\begin{align}
y_{0+}^v
\;&\ge\,\;
\mfrac{q^*_v}{p_v}
+
\begin{cases}
1
& J_v \neq \emptyset;\, j(e_v) \neq -1
    \\
\left\lceil \mkern-2mu\frac{p_{v(e_v)}}{p_v\Delta_{e_v}}\mkern-3mu \right\rceil 
\mkern-3mu+\mkern-2mu 1
&
j(e_v) \mkern-3mu=\mkern-3mu -1
    \\
0
& \text{otherwise}
\end{cases},
\end{align}
with equality only if $Y_{\Gamma_v}\mkern-2mu(\boldsymbol{y}^{\Gamma_v})$ is \textsc{bc}.
When $j(e_v) \!\neq\! -1$, the desired inductive result~is~established in the 
$\boldsymbol{y}^v \in \mathcal{L}^{\min -}_{\textsc{sf}_v}$
case of the proof of
Theorem~\ref{thm: iterated torus links satellites}.
We henceforth assume $j(e_v) = -1$.

Since $j(e_v) = -1$ implies
$\frac{q^*_v}{p_v}
+ \left\lceil \mkern-2mu\mfrac{p_{v(e_v)}}{p_v\Delta_{e_v}}\mkern-3mu \right\rceil 
\mkern-3mu+\mkern-2mu 1 > \eta_v$, we have
$y_{0-}^v \ge y_{0+}^v >\eta_v$, {\em{i.e.}}, to the right of the
vertical asymptote of $\phi_{e_v*}^{\P}$ at $\eta_v$. The respective
$\phi_{e_v*}^{\P}$-images
$y_{j(e_v)+}^{v(e_v)}$ and $y_{j(e_v)-}^{v(e_v)}$ 
of $y^v_{0-}$ and $y^v_{0+}$ therefore
lie above the horizontal asymptote at $\xi_{v(e_v)}$, but with reversed order:
\begin{equation}
\mu_{e_v} + m^{e_v+} = -\mfrac{q^*_{v(e_v)}}{p_{v(e_v)}} < \xi_{v(e_v)} <
y_{j(e_v)+}^{v(e_v)}
\le
y_{j(e_v)-}^{v(e_v)}.
\end{equation}
It remains to show that
$\mu_{e_v} + m^{e_v-} \ge y_{j(e_v)-}^{v(e_v)} := \phi^{\P}_{e_v*}(y^v_{0+})$
(with equality only if $Y_{\Gamma_v}(\boldsymbol{y}^{\Gamma_v})$ is \textsc{bc}),
for which it suffices to show that 
$(\mu_{e_v} + m^{e_v-}) 
-
\phi_{e_v*}^{\P}\mkern-3mu\left(
\frac{q^*_v}{p_v}
+ \left\lceil \mkern-2mu\mfrac{p_{v(e_v)}}{p_v\Delta_{e_v}}\mkern-3mu \right\rceil 
\mkern-3mu+\mkern-2mu 1
\right) \ge 0$.
%Recall that $j(e_v)=-1$ implies $\phi^{\P}_{e_v*} = \sigma^{\P}_{e_v*}$. If we write
%$\left[\sigma^{\P}_{e_v*}\right] = $

Recall that any edge $e$ with $j(e)\mkern-2mu=\mkern-2mu-1$ has
$\mkern1.5mu\phi^{\P}_{e*} \mkern-2.5mu=\mkern-1.5mu \sigma^{\P}_{e*}\mkern-.5mu$. 
If we write
$\left[\sigma^{\P}_{e*}\right] \mkern-2mu=:\mkern-2mu
\left(
\begin{array}{cc}
\alpha_e 
& \Delta^{\prime}_e
  \\
\Delta_e
& 
\beta_e
\end{array}\right)$ for the entries of the matrix 
$\left[\sigma^{\P}_{e*}\right]$ as computed in
(\ref{eq: def of sigma}), then the relations 
$p_u^* p_u \mkern.5mu-\mkern1mu q_u^* q_u \mkern-1.5mu=\mkern-1.5mu 1$ for each
$u \mkern-2mu\in\mkern-2mu \mathrm{Vert}(\Gamma)$, particularly for 
$u \mkern-2mu\in\mkern-2.5mu \{v, v(e_v)\}$,
produce simplifications, incuding~the~identities
\begin{equation}
p_{v(e_v)} \alpha_{e_v} + q_{v(e_v)}^* \Delta_{e_v} = p_v,
\mkern20mu
p_v \beta_{e_v} + q_v^* \Delta_{e_v} = -p_{v(e_v)},
\mkern20mu
q_v^* \alpha_{e_v} + p_v \Delta^{\prime}_{e_v} = q^*_{v(e_v)},
\end{equation}
used in the intermediate steps suppressed in the following calculation.
We compute that
\begin{align}
\nonumber
(\mu_{e_v}\mkern-3mu
&+
m^{e_v-}) 
\mkern10mu-\mkern10mu
\phi_{e_v*}^{\P}\mkern-3mu\left(
\mfrac{q^*_v}{p_v}
+ \left\lceil \mkern-2mu\mfrac{p_{v(e_v)}}{p_v\Delta_{e_v}}\mkern-3mu \right\rceil 
\mkern-3mu+\mkern-2mu 1
\right) 
   \\
&=\,
\left(-\frac{q^*_{v(e_v)}}{p_{v(e_v)}}
+ \left\lceil \mkern-2mu\frac{p_v}{p_{v(e_v)}\Delta_{e_v}}\mkern-3mu \right\rceil 
\mkern-3mu+\mkern-2mu 1\right)
\mkern10mu-\mkern10mu
\frac{\mkern7mu\alpha_{e_v}\mkern-4.5mu
\left(p_v \Delta_{e_v} + \left[-p_{v(e_v)}\right]_{p_v \Delta_{e_v}}\right) 
+ p_v}{\Delta_{e_v}\mkern-4.5mu
\left(p_v \Delta_{e_v} + \left[-p_{v(e_v)}\right]_{p_v \Delta_{e_v}}\right) 
\mkern2mu \hphantom{+ p_v}}
    \\
&= 
\label{eq: final answer Lmin- alg link}
1
\mkern5mu+\mkern5mu 
\left[ \mkern-2mu\frac{p_v}{p_{v(e_v)}\Delta_{e_v}}\mkern-3mu \right]
\mkern8mu-\mkern8mu
\Delta_{e_v}^{-2}\mkern-4.5mu
\left(1 + \left[\mkern-2mu\frac{\!-p_{v(e_v)}}{p_v \Delta_{e_v}\!\!}\right]\right)^{\!-1}
    \\
&\ge 0,
\end{align}
since $[x]:= x-\lfloor x\rfloor$ implies $0 \le [x] < 1$ for $x \in \Q$,
and this completes our inductive argument.

% In fact, this is in some sense optimal, since
% (\ref{eq: final answer Lmin- alg link}) $= 0$ on the nose if
% $ p_v = p_{v(e_v)} = \Delta_{e_v} = 1$.
% Moreover, if
% $\phi_{e_v*}^{\P}\mkern-3mu\left(\mfrac{q^*_v}{p_v}
% + \left\lceil \mkern-2mu\mfrac{p_{v(e_v)}}{p_v\Delta_{e_v}}\mkern-3mu \right\rceil 
% \mkern-3mu+\mkern-2mu 1 \right)$
% is replaced with
% $\phi_{e_v*}^{\P}\mkern-3mu\left(\mfrac{q^*_v}{p_v}
% + \left\lceil \mkern-2mu\mfrac{p_{v(e_v)}}{p_v\Delta_{e_v}}\mkern-3mu \right\rceil 
% \mkern-3mu+\mkern-2mu m \right)$, then (\ref{eq: final answer Lmin- alg link})
% is replaced with
% $$1\mkern5mu+\mkern5mu 
% \left[ \mkern-2mu\frac{p_v}{p_{v(e_v)}\Delta_{e_v}}\mkern-3mu \right]
% \mkern8mu-\mkern8mu\Delta_{e_v}^{-2}\mkern-4.5mu
% \left(m + \left[\mkern-2mu\frac{\!-p_{v(e_v)}}{p_v \Delta_{e_v}\!\!}\right]\right)^{\!-1},$$
% which still requires the $1$ summand of $\mu_{e_v}\mkern-3mu + m^{e_v-}$
% in order to have $(\ref{eq: final answer Lmin- alg link}) \ge 0$ when 
% $p_v = p_{v(e_v)} = \Delta_{e_v} = 1$.

\vspace{.2cm}

Since $j(e_r) \neq -1$, the proof that these inductive bounds cause
the Dehn-filled satellite exterior 
$Y^{\Gamma}(\boldsymbol{y}^{\Gamma})
:= Y_{\Gamma}(\boldsymbol{y}^{\Gamma}) \cup (S^3 \setminus \overset{\mkern2mu{\circ}}{\nu}(K))$
to form an L-space whenever
\begin{equation}
\boldsymbol{y}^{\Gamma}
\in
\prod_{v \in \text{Vert}(\Gamma_{r})}
\left(\mathcal{L}^{\min+}_{\textsc{sf}_v} 
\cup
\mathcal{R}_{\textsc{sf}_v} \mkern-4mu\setminus\mkern-1mu \mathcal{Z}_{\textsc{sf}_v}
\cup
\mathcal{L}^{\min-}_{\textsc{sf}_v}\right)
\end{equation}

\vspace{-.15cm}

\noindent is the same as the corresponding argument in the proof of
Theorem~\ref{thm: iterated torus links satellites}.

\qed

\subsection{Monotone strata}
\label{ss: monotone strata}
Whether for iterated torus satellites and for algebraic link satellites,
our inner approximations
$\mathcal{L}^{\min}_{\textsc{sf}_{\Gamma}}(Y^\Gamma) :=
\prod_{v \in \mathrm{Vert}(\Gamma)}
\left(
\mathcal{L}^{\min -}_{\textsc{sf}_v}
\;\cup\;
\mathcal{R}_v \setminus \mathcal{Z}_v
\;\cup\;
\mathcal{L}^{\min +}_{\textsc{sf}_v}
\right)
$
each involve inductive bounds,
namely, (\ref{eq: induct one interated}) and (\ref{eq: induct bi algebraic}),
respectively, which
make $y^{v(e)}_{j(e)+} \le y^{v(e)}_{j(e)-}$,
as a function of $\boldsymbol{y}^{\Gamma}|_{\Gamma_{v(-e)}}$,
for each edge $e \in \mathrm{Edge}(\Gamma)$ and slope 
$\boldsymbol{y}^{\Gamma} \in 
\mathcal{L}^{\min}_{\textsc{sf}_{\Gamma}}(Y^\Gamma)$.
This implies that 
$\mathcal{L}^{\min}_{\textsc{sf}_{\Gamma}}(Y^\Gamma)$
is confined to a particular substratum
of $\mathcal{L}_{\textsc{sf}_{\Gamma}}(Y^\Gamma)$,
called the {\em{monotone stratum}}.

\begin{definition}
For any 
%slope
$\mkern1.5mu\boldsymbol{y}^{\Gamma} \mkern-1mu\in 
(\Q\cup\{\infty\})^{|I_{\Gamma}|}_{\textsc{sf}_{\Gamma}}$
and 
%vertex
$v \in \mathrm{Vert}(\Gamma)$,
we call
$\boldsymbol{y}^{\Gamma}$ {\em{monotone at}} $v$ if
\begin{equation}
\label{eq: monotonicity at v}
\infty \in 
\phi_{\mkern-1mue*\mkern1mu}^{\P}\mathcal{L}^{\circ}_{v(-e)}(\boldsymbol{y})\;\,
\forall\,e \in E_{\mathrm{in}}(v)
\;\;\;\text{and}\;\;\;
\infty\in
\phi_{\mkern-1mu{e_v}*\mkern1mu}^{\P}\mathcal{L}_v^{\circ}(\boldsymbol{y}).
\end{equation}
The {\em{monotone stratum}}
$\mathcal{L}^{\mathrm{mono}}_{\textsc{sf}_{\Gamma}}(Y^\Gamma)$
of $\mathcal{L}_{\textsc{sf}_{\Gamma}}(Y^\Gamma)$ is then the set of slopes 
$\boldsymbol{y}^{\Gamma} \in \mathcal{L}_{\textsc{sf}_{\Gamma}}(Y^\Gamma)$
such that $\boldsymbol{y}^{\Gamma}$ is monotone at all $v \in \mathrm{Vert}(\Gamma)$.
\end{definition}

In the above, for brevity, we have adopted the following
\begin{notation}
For any 
slope
$\mkern1.5mu\boldsymbol{y}^{\Gamma} \mkern-1mu\in 
(\Q\cup\{\infty\})^{|I_{\Gamma}|}_{\textsc{sf}_{\Gamma}}$
and 
vertex
$v \in \mathrm{Vert}(\Gamma)$,
we shall write
\begin{equation}
\mathcal{L}_v(\boldsymbol{y}):=
\mathcal{L}_{\textsc{sf}_v\mkern-4mu}
(Y_{\Gamma_v}\mkern-1.5mu(\boldsymbol{y}^{\Gamma}|_{\Gamma_v})),
\;\;\;\;
\mathcal{L}_v^{\circ}(\boldsymbol{y}):=
\mathcal{L}_{\textsc{sf}_v\mkern-4mu}^{\circ}
(Y_{\Gamma_v}\mkern-1.5mu(\boldsymbol{y}^{\Gamma}|_{\Gamma_v})).
\end{equation}
\end{notation}

\noindent \textbf{Remark.}
Note that if 
$\mathcal{L}^{\circ}_{v(-e)}(\boldsymbol{y}) \neq \emptyset$
for all $e \in E_{\mathrm{in}}(v)$
and 
$\mathcal{L}_v^{\circ}(\boldsymbol{y}) \neq \emptyset$, then
\begin{equation*}
\phi_{\mkern-1mue*\mkern1mu}^{\P}\mathcal{L}_{v(-e)}(\boldsymbol{y})
= [[y^v_{j(e)-}, y^v_{j(e)+}]] \;\;\forall\;e \in E_{\mathrm{in}}(v),
\;\;\;
\phi_{\mkern-1mu{e_v}*\mkern1mu}^{\P}\mathcal{L}_v(\boldsymbol{y})
=[[y^{v(e_v)}_{j(e_v)-},y^{v(e_v)}_{j(e_v)+}]],
\end{equation*}
and the monotonicity condition
(\ref{eq: monotonicity at v})
at $v$ is equivalent to the condition that
\begin{equation*}
y^v_{j(e)+} \le y^v_{j(e)-} \;\forall\;e \in E_{\mathrm{in}}(v),
\;\;\;
y^{v(e_v)}_{j(e_v)+} \le y^{v(e_v)}_{j(e_v)-},
\end{equation*}
corresponding to the endpoint-ordering consistent with that for generic Seifert fibered L-space intervals. We call this condition ``monotonicity'' because of its preservation of this ordering.

The tools developed in 
Sections \ref{s: iterated torus-link-satellites} and
\ref{s: algebraic link satellites}
can be used in much more general settings than that of the
inner approximation theorems we proved, so long as one first
decomposes 
$\mathcal{L}_{\textsc{sf}_{\Gamma}}(Y^\Gamma)$
into strata according to monotonicity conditions,
similar to how torus link satellites must first be classified according to whether
$2g(K)-1 \le \frac{q}{p}$.
Monotonicity conditions also impact the topology of strata.

\begin{theorem}
\label{thm: topology of monotone stratum}
Suppose that $K^{\Gamma} \subset S^3$ is an algebraic link satellite,
specified by $\Gamma$, of a positive L-space knot $K\subset S^3$,
where either $K$ is trivial, or
$K$ is nontrivial with
$\frac{q_r}{p_r} > \mkern2mu 2g(K) -1$.
Let $V \subset \mathrm{Vert}(\Gamma)$ denote the subset of vertices $v \in V$
for which $|I_v| > 0$.

Then the $\Q$-corrected $\R$-closure
$\mathcal{L}^{\mathrm{mono}}_{\textsc{sf}_{\Gamma}}(Y^\Gamma)^{\R}$
of the monotone stratum of 
$\mathcal{L}_{\textsc{sf}_{\Gamma}}(Y^\Gamma)$
is of dimension $|I_{\Gamma}|$ and 
%in $(\R \cup \{\infty\})^{|I_{\Gamma}|}_{\textsc{sf}_{\Gamma}}$,
deformation retracts onto an $(|I_{\Gamma}| - |V|)$-dimensional embedded torus,
\begin{equation}
\label{eq: torus}
\mathcal{L}^{\mathrm{mono}}_{\textsc{sf}_{\Gamma}}(Y^\Gamma)^{\R}
\;\;\to\;\;
%_{\text{def. retracts to}}\;
{\textstyle{\prod_{v\in V}}} \mathbb{T}^{|I_v|-1}
\;\into\; 
{\textstyle{\prod_{v\in V}}} 
(\R\cup\{\infty\})^{|I_v|}_{\textsc{sf}_v},
\end{equation}
projecting to an embedded torus
$\mathbb{T}^{|I_v|-1}
\;\into\; 
(\R\cup\{\infty\})^{|I_v|}_{\textsc{sf}_v}$
parallel to 
$\mathcal{B}_{\textsc{sf}_v} \subset (\R\cup\{\infty\})^{|I_v|}_{\textsc{sf}_v}$
at each $v \in V$.
\end{theorem}

\begin{proof}
We argue by induction, recursing downward from the leaves of $\Gamma$ towards its root.
Observe that for any $v \in \mathrm{Vert}(\Gamma)$, we have the fibration
\begin{equation}
\label{eq: mono fibration}
\mathcal{L}^{\mathrm{mono}}_{\textsc{sf}_{\Gamma_v}}\mkern-1mu(Y^{\Gamma_v})
\;\;\longrightarrow\;\;
{\textstyle{\prod_{e \in E_{\mathrm{in}}(v)}}}
\mathcal{L}^{\mathrm{mono}}_{\textsc{sf}_{\Gamma_{v(-e)}}}\mkern-4mu(Y^{\Gamma_{v(-e)}})
\end{equation}
with fiber
\begin{equation}
\mathcal{T}_{\boldsymbol{y_*}}^v
:= \left\{ \boldsymbol{y}^{\Gamma_v} \in 
\mathcal{L}^{\mathrm{mono}}_{\textsc{sf}_{\Gamma_v}}\mkern-1mu(Y^{\Gamma_v})
\left|\vphantom{y^{\Gamma_v}}\right.\;\,
\boldsymbol{y}^{\Gamma_v}\mkern-4mu
\left.\vphantom{\frac{a}{a}}\right|_{\prod_{e\in E_{\mathrm{in}}(v)} \Gamma_{v(-e)}}
= \boldsymbol{y_*}
\right\}
\end{equation}
over 
$\boldsymbol{y_*}
\in
{\textstyle{\prod_{e \in E_{\mathrm{in}}(v)}}}
\mathcal{L}^{\mathrm{mono}}_{\textsc{sf}_{\Gamma_{v(-e)}}}\mkern-4mu(Y^{\Gamma_{v(-e)}})$
for $I_v \neq \emptyset$, with $\mathcal{T}_{\boldsymbol{y_*}}^v$
regarded as a point when $I_v = \emptyset$.

For $v \in \mathrm{Vert}(\Gamma)$, inductively assume the theorem holds for
$\Gamma_{v(-e)}$ for all $e \in E_{\mathrm{in}}(v)$.
(Note that this holds vacuously when $v$ is a leaf, in which case we declare
$\mathcal{T}_{\emptyset}^v
:= \mathcal{L}^{\mathrm{mono}}_{\textsc{sf}_{\Gamma_v}}\mkern-1mu(Y^{\Gamma_v})$.)

If $I_v = \emptyset$, then the fibration in (\ref{eq: mono fibration}) is the identity map,
making the theorem additionally hold for $\Gamma_v$. Next assuming $I_v \neq \emptyset$,
we claim the $\Q$-corrected $\R$-closure 
$(\mathcal{T}_{\boldsymbol{y_*}}^v)^{\R}$ of $\mathcal{T}_{\boldsymbol{y_*}}^v$
is of dimension $|I_v|$ and deformation retracts onto an embedded torus 
$\mathbb{T}^{|I_v|-1}
\;\into\; 
(\R\cup\{\infty\})^{|I_v|}_{\textsc{sf}_v}$
parallel to 
$\mathcal{B}_{\textsc{sf}_v} \subset (\R\cup\{\infty\})^{|I_v|}_{\textsc{sf}_v}$.
In fact, the proof of this statement is nearly identical to the proof of 
Theorem \ref{thm: topology of torus link exterior L-space region}$.ii.b$
in Section \ref{s: L-space region topology},
but with the replacement 
\begin{equation}
\mathcal{N}
\mkern-2mu:=\mkern-2mu
\{\boldsymbol{y}\in\Q^n_{\textsc{sf}}\mkern1mu |\,
y_+(\boldsymbol{y}) < 0 < y_-(\boldsymbol{y})\}
\;\,
\longrightarrow
\;\,
\mathcal{N}_v
\mkern-2mu:=\mkern-2mu
\{\boldsymbol{y}^v\in\Q^{|I_v|}_{\textsc{sf}_v}\mkern1mu |\;
y_{0+}^v(\boldsymbol{y^v}) < \eta_v
 < y_{0-}^v(\boldsymbol{y}^v)\}\mkern-10mu
\end{equation}
in line (\ref{eq: NL for top})
(where $\eta_v$, computed in (\ref{eq: asymptote value}),
is the position of the vertical asymptote of $\phi^{\P}_{e_v*}$),
along with a few minor analogous adjustments corresponding to this change.

It remains to show that the fibration in 
(\ref{eq: mono fibration}) is trivial, but this follows from the fact that
\begin{equation}
\mathcal{L}^{\min}_{\textsc{sf}_{\Gamma_v}}(Y^{\Gamma_v}) :=
\prod_{u \in V \cap \mathrm{Vert}(\Gamma_v)}
\left(
\mathcal{L}^{\min -}_{\textsc{sf}_u}
\;\cup\;
\mathcal{R}_u \setminus \mathcal{Z}_u
\;\cup\;
\mathcal{L}^{\min +}_{\textsc{sf}_u}
\right)
\end{equation}
is a {\em{product}} over $u \in V \cap \mathrm{Vert}(\Gamma_v)$ which embeds into
$\mathcal{L}^{\mathrm{mono}}_{\textsc{sf}_{\Gamma}}(Y^{\Gamma_v})$,
and for reasons again similar to the proof of 
Theorem \ref{thm: topology of torus link exterior L-space region}$.ii.b$,
each factor
$\mathcal{L}^{\min -}_{\textsc{sf}_u}
\;\cup\;
\mathcal{R}_u \setminus \mathcal{Z}_u
\;\cup\;
\mathcal{L}^{\min +}_{\textsc{sf}_u}$
also deformation retracts onto an embedded torus
$\mathbb{T}^{|I_u|-1}
\;\into\; 
(\R\cup\{\infty\})^{|I_u|}_{\textsc{sf}_u}$
parallel to 
$\mathcal{B}_{\textsc{sf}_u} \subset (\R\cup\{\infty\})^{|I_u|}_{\textsc{sf}_u}$,
completing the proof.
\end{proof}

\section{Extensions of L-space Conjecture Results}
\label{s: bgw conjectures}

As mentioned in the introduction,
Boyer-Gordon-Watson 
\cite{BGW}
conjectured several years ago that among prime, closed, oriented
3-manifolds, L-spaces are those 3-manifolds whose fundamental groups
do not admit a left orders (LO).  Similarly, 
Juh{\'a}sz 
\cite{JuhaszConj}
conjectured that prime, closed, oriented 3-manifold are L-spaces
if and only if they fail to admit a co-oriented taut foliation (CTF).
Procedures which generate new collections of L-spaces or non-L-spaces,
such as surgeries on satellites, provide new testing grounds for these
conjectures.

For $Y$ a compact oriented 3-manifold
with boundary a disjoint union of $n>0$ tori,
define the slope subsets 
$\mathcal{F}(Y), \mathcal{LO}(Y) \subset \prod_{i=1}^n \P(H_1(\partial_i Y;\Z))$
so that
\begin{align}
\mathcal{F}(Y)
&:= 
\left\{
\boldsymbol{\alpha} \in \prod_{i=1}^n \P(H_1(\partial_i Y;\Z))
\right.\left|
\begin{array}{c}
Y \text{ admits a CTF } \mathcal{F} \text{ such that }
\\
\mathcal{F}|_{\partial Y} \text{ is the product foliation of slope } 
\boldsymbol{\alpha}.
\end{array}
\vphantom{\prod_{i=1}^n \P(H_1(\partial_i Y;\Z))}\mkern-5mu
\right\},
    \\
\mathcal{LO}(Y) 
&:= 
\left\{
\boldsymbol{\alpha} \in \prod_{i=1}^n \P(H_1(\partial_i Y;\Z))
\right.\left|\;
\pi_1(Y(\boldsymbol{\alpha})) \text{ is LO.}
\vphantom{\prod_{i=1}^n \P(H_1(\partial_i Y;\Z))}
\right\}.
\end{align}
Note that $\boldsymbol{\alpha} \in \mathcal{F}(Y)$ implies
that $Y(\boldsymbol{\alpha})$ admits a CTF, but the converse,
while true for $Y$ a graph manifold, is not known in general.

\subsection{Proof of Theorem
\ref{thm: bgw and juhasz} and Generalizations}
The proof of 
Theorem~\ref{thm: bgw and juhasz}
relies on the related gluing behavior of 
co-oriented taut foliations, left orders on fundamental groups,
and the property of being an non-L-space,
for a pair $Y_1, Y_2$ of compact oriented 3-manifolds with torus boundary
glued together via a gluing map $\varphi : \partial Y_1  \to \partial Y_2$.

That is, the contrapositives of 
Theorems \ref{thm: gluing theorem for floer simple or graph manifold} and
\ref{thm: knot exterior gluing theorem} tell us that if
$Y_i$ have incompressible boundaries and are
both Floer simple manifolds, both graph manifolds,
or an L-space knot exterior and a graph manifold, then 
\begin{equation}
\label{eq: L-space gluing contrapositive}
\varphi_{*}^{\P}(\overline{\mathcal{NL}(Y_1)}) \cap \mkern1mu\overline{\mathcal{NL}(Y_2)} \neq \emptyset 
\;\;\iff\;\;
Y_1 \cup_{\varphi} \mkern-4mu Y_2 \mkern2mu
\text{ not an L-space.}
\end{equation}
The analogous statements for CTFs and LOs, while true for graph manifolds
(once an exception is made for reducible slopes in the case of CTFs),
are not established in general.  However, we still
have weak gluing statements in the general case.
Since product foliations of matching slope can always be glued together,
we have
\begin{equation}
\label{eq: foliation gluing}
\varphi_{*}^{\P}({\mathcal{F}(Y_1)}) \cap {\mathcal{F}(Y_2)} \neq \emptyset
\;\;\implies\;\;
Y_1 \cup_{\varphi} \mkern-4mu Y_2 \mkern2mu,
\text{ if prime, admits a CTF.}
\end{equation}
Moreover,
Clay, Lidman, and Watson \cite{ClayLidmanWatson} built on a result of 
Bludov and Glass
\cite{BludovGlassI}
to~show~that
\begin{equation}
\label{eq: left order gluing}
\varphi_{*}^{\P}({\mathcal{LO}(Y_1)}) \cap {\mathcal{LO}(Y_2)} \neq \emptyset
\;\;\implies\;\;
\pi_1(Y_1 \cup_{\varphi} \mkern-4mu Y_2)
\text{ is LO.}
\end{equation}

\begin{theorem}
\label{thm: bgw and juhasz for general satellites}
Suppose $Y^{\Gamma}$
is the exterior of an algebraic link satellite or 
(possibly-iterated) torus-link satellite
of a nontrivial positive L-space knot $K \mkern-2mu\subset\mkern-2mu S^3$ 
of genus $g(K)$ and exterior $Y$,
with $p_r > 1$ and $-1 \notin j(E_{\mathrm{in}}(r))$.

\noindent $\;\;(\textsc{LO})$
Suppose $\mathcal{LO}(Y) \supset \mathcal{NL}(Y)$.

$\;(\textsc{lo}.i)$
If $2g(K)-1 > \frac{q_r+1}{p_r}$, then 
$\mathcal{LO}(Y^{\Gamma}) = 
\mathcal{NL}(Y^{\Gamma}).$

$\;(\textsc{lo}.ii)$
If $2g(K)-1 < \frac{q_r}{p_r}$ and $\Gamma = r$ specifies a torus link satellite, then
$\mathcal{LO}(Y^{\Gamma}) \supset$
$$\;(\mathcal{NL}(Y^{\Gamma}) \mkern-1mu\setminus\mkern-1mu \mathcal{R}(Y^{\Gamma}))
\;\setminus\; \Lambda(Y^{\Gamma}) \mkern-2mu\cdot\mkern-2mu 
(\left[-\infty, N_{\Gamma}\right>^n \mkern-1mu\setminus\mkern-1mu \left[-\infty, N_{\Gamma}-p_r\right>^n),$$

\vspace{-.1cm}
where $N_{\Gamma} := p_r q_r - q_r - p_r + 2g(K)p_r$.
\vspace{.2cm}

\noindent $\;\;(\textsc{CTF})$
Suppose $\mathcal{F}(Y) = \mathcal{NL}(Y)$.

$\;(\textsc{ctf}.i)$ 
If $2g(K)-1 > \frac{q_r+1}{p_r}$, then 
$\mathcal{F}(Y^{\Gamma}) = 
\mathcal{NL}(Y^{\Gamma}) \setminus \mathcal{R}(Y^{\Gamma}).$

$\;(\textsc{ctf}.ii)$
If $2g(K)-1 < \frac{q_r}{p_r}$ and $\Gamma = r$ specifies a torus link satellite, then
$\mathcal{F}(Y^{\Gamma}) \supset$
$$\;(\mathcal{NL}(Y^{\Gamma}) \mkern-1mu\setminus\mkern-1mu \mathcal{R}(Y^{\Gamma}))
\;\setminus\; \Lambda(Y^{\Gamma}) \mkern-2mu\cdot\mkern-2mu 
(\left[-\infty, N_{\Gamma}\right>^n \mkern-1mu\setminus\mkern-1mu \left[-\infty, N_{\Gamma}-p_r\right>^n),$$

\end{theorem}
Note that the requirement that $K$ be nontrivial
is just to simplify the theorem statement.
If $K$ is trivial, then any surgery on $Y^{\Gamma}$
is a graph manifold or a connected sum thereof,
in which case
the L-space conjectures written down by Boyer-Gordon-Watson and Juh{\'a}sz
are already proven to hold,
through the work of Boyer and Clay \cite{BoyerClay}
and of Hanselman, J.~Rasmussen, Watson, and the author \cite{HRRW}.
In addition, the author explicitly shows in
\cite{lgraph} that any graph manifold $Y_{\Gamma}$
always satisfies
\begin{equation}
\mathcal{LO}(\Gamma) = 
\mathcal{NL}(Y_{\Gamma}),
\;\;\;\;
\mathcal{F}(Y_{\Gamma}) = \mathcal{NL}(Y_{\Gamma}) \setminus \mathcal{R}(Y_{\Gamma}).
\end{equation}

{\noindent{\textit{Proof of Theorem.}}}
Suppose that 
$\mathcal{LO}(Y) \supset \mathcal{NL}(Y)$
(respectively 
$\mathcal{F}(Y) = \mathcal{NL}(Y)$).
Since
\begin{equation}
Y^{\Gamma} = Y_{\Gamma} \cup_{\phi_{e_r}}\mkern-2mu Y
\;\;\;\text{and}\;\;\;
(\phi_{e^*}^{\P})^{-1}(\mathcal{L}(Y)) = 
\left[\mfrac{q^*_r}{p_r}-\mfrac{1}{p_r(p_rN-q_r)}, \mfrac{q^*_r}{p_r}\right]_{\textsc{sf}},
\end{equation}
where $N := 2g(K)-1$,
it follows from 
(\ref{eq: left order gluing})
(respectively
(\ref{eq: foliation gluing}))
that in order to prove that
$\boldsymbol{y}^{\Gamma} \in \mathcal{LO}(Y^{\Gamma})$
(respectively 
$\boldsymbol{y}^{\Gamma} \in \mathcal{F}(Y^{\Gamma}) \cup \mathcal{Z}(Y^{\Gamma})$),
it suffices to show that
\begin{equation}
\label{eq: bgwj condition}
\mathcal{N\mkern-1.5muL}_{\textsc{sf}}(Y_{\Gamma}(\boldsymbol{y}^{\Gamma})) \cap 
\left(\left[-\infty,  \mfrac{q^*_r}{p_r}-\mfrac{1}{p_r(p_rN-q_r)}\right>
\cup \left<\mfrac{q^*_r}{p_r},  +\infty\right]\right)
\neq \emptyset.
\end{equation}
On the other hand,
(\ref{eq: L-space gluing contrapositive}) implies that
$\boldsymbol{y}^{\Gamma} \in \mathcal{NL}(Y^{\Gamma})$ if and only if
\begin{equation}
\label{eq: nlspace condition}
\overline{\mathcal{N\mkern-1.5muL}_{\textsc{sf}}(Y_{\Gamma}(\boldsymbol{y}^{\Gamma}))} \cap 
\left(\left[-\infty,  \mfrac{q^*_r}{p_r}-\mfrac{1}{p_r(p_rN-q_r)}\right]
\cup \left[\mfrac{q^*_r}{p_r},  +\infty\right]\right)
\neq \emptyset
\end{equation}
when $Y_{\Gamma}(\boldsymbol{y}^{\Gamma})$ is $\textsc{bi}$,
and if and only if 
(\ref{eq: bgwj condition}) holds when 
$Y_{\Gamma}(\boldsymbol{y}^{\Gamma})$ is $\textsc{bc}$.

Fix some slope 
$\boldsymbol{y}^{\Gamma} \in (\Q \cup \{\infty\})^{\sum_{v\in \mathrm{Vert}(\Gamma)} |I_v|}$
and write
\begin{equation}
\mathcal{L}_{\textsc{sf}}(Y_{\Gamma}(\boldsymbol{y}^{\Gamma})) = [[y_{0-}, y_{0+}]],
\end{equation}
as usual.
It is straightforward to show that
(\ref{eq: nlspace condition}) fails to hold if and only if 
\begin{equation}
\label{eq: new lspace condition}
\mfrac{q^*_r}{p_r}-\mfrac{1}{p_r(p_rN-q_r)}
\;<\; y_{0+}
\;\le\; y_{0-}
\;<\;
\mfrac{q^*_r}{p_r},
\end{equation}
and that
(\ref{eq: bgwj condition}) fails to hold if and only if
\begin{equation}
\label{eq: new bgwj condition}
\mfrac{q^*_r}{p_r}-\mfrac{1}{p_r(p_rN-q_r)}
\;\le\; y_{0+}
\;\le\; y_{0-}
\;\le\;
\mfrac{q^*_r}{p_r}.
\end{equation}
Note that $p_r > 1$ implies
\begin{equation}
\label{eq: bounds on q*/p and interval}
0 \le 
\mfrac{q^*_r}{p_r}-\mfrac{1}{p_r(p_rN-q_r)}
<
\mfrac{q^*_r}{p_r} < 1,
\;\;\;\text{with}\;\;\;\,
\mfrac{q^*_r}{p_r}-\mfrac{1}{p_r(p_rN-q_r)} = 0
\iff
N = \mfrac{q_r+1}{p_r}.
\end{equation}

We begin by proving the following claim.

\begin{claim*}
If  
$N:= 2g(K)-1 > \frac{q_r}{p_r}$, $p_r > 1$, and
$-1 \notin j(E_{\mathrm{in}}(r))$,
then 
\begin{equation}
\label{eq: iff for s3 and lspace conds}
Y^{\Gamma}\mkern-2mu(\boldsymbol{y}^{\Gamma}) = S^3
\;\;\iff\;\;
\left\{
\mkern-8mu
\begin{array}{l}
\text{(\ref{eq: new lspace condition}) holds if }
Y_{\Gamma}(\boldsymbol{y}^{\Gamma})
\;\text{is}\; \textsc{bi},
  \\
\text{(\ref{eq: new bgwj condition}) holds if }
Y_{\Gamma}(\boldsymbol{y}^{\Gamma})
\;\text{is}\; \textsc{bc},
\end{array}
\right.
\end{equation}
If, in addition, 
$N:= 2g(K)-1 \neq \frac{q_r+1}{p_r}$, then
\begin{equation}
\label{eq: second part of first claim in bgw}
\text{(\ref{eq: new bgwj condition}) holds}
\;\;\iff\;\;
\left\{
\mkern-8mu
\begin{array}{l}
\text{(\ref{eq: new lspace condition}) holds if }
Y_{\Gamma}(\boldsymbol{y}^{\Gamma})
\;\text{is}\; \textsc{bi},
  \\
\text{(\ref{eq: new bgwj condition}) holds if }
Y_{\Gamma}(\boldsymbol{y}^{\Gamma})
\;\text{is}\; \textsc{bc},
\end{array}
\right.
\end{equation}
\end{claim*}

{\textit{Proof of Claim.}}
Suppose $N > \frac{q_r}{p_r}$, $p_r > 1$,
and $-1 \notin j(E_{\mathrm{in}}(r))$.
Then
Proposition \ref{prop: basic defs of y_0+-}
together with 
Proposition~\ref{prop: main bounds for y0- and y0+}.$(=)$
imply that
\begin{equation}
y_{0+} \in \mfrac{q_r^*}{p_r} + \Z 
\iff 
y_{0-} \in \mfrac{q_r^*}{p_r} + \Z 
\iff 
Y_{\Gamma}(\boldsymbol{y}^{\Gamma})
\;\text{is}\; \textsc{bc}
\implies y_{0-} = y_{0+},
\end{equation}
so that
\begin{equation}
\text{(\ref{eq: new bgwj condition}) holds and }
Y_{\Gamma}(\boldsymbol{y}^{\Gamma})
\;\text{is}\; \textsc{bc}
\;\iff\;
y_{0+} = \mfrac{q_r^*}{p_r}
\;\iff\;
y_{0-} = \mfrac{q_r^*}{p_r}
\;\implies\;
Y^{\Gamma}(\boldsymbol{y}^{\Gamma}) = S^3.
\end{equation}
Thus, since the fact that $S^3$ is an L-space makes the
$\implies$ implication of 
(\ref{eq: iff for s3 and lspace conds}) automatically hold,
this exhausts the case when 
$Y_{\Gamma}(\boldsymbol{y}^{\Gamma})$ is
$\textsc{bc}$.
Next suppose that 
$Y_{\Gamma}(\boldsymbol{y}^{\Gamma})
\;\text{is}\; \textsc{bi}$,
so that
Proposition~\ref{prop: main bounds for y0- and y0+}.(+)
implies
$y_{0+} \in \left<\frac{q^*_r}{p_r}, 1\right] + \Z$.
Then
(\ref{eq: bounds on q*/p and interval})
implies that
(\ref{eq: new lspace condition}) always fails to hold, and that
(\ref{eq: new bgwj condition}) fails to hold
if $N \neq \frac{q_r + 1}{p_r}$,
completing the proof of the claim.

\vspace{.1cm}

Continuing with the proof of the theorem,
since the right-hand condition of 
(\ref{eq: iff for s3 and lspace conds})
and
(\ref{eq: second part of first claim in bgw})
is equivalent to $Y^{\Gamma}(\boldsymbol{y}^{\Gamma})$ being an L-space,
and since 
(\ref{eq: new bgwj condition}) is the negation of
(\ref{eq: bgwj condition}), we have shown that
(\ref{eq: bgwj condition}) holds
if and only if 
$\boldsymbol{y}^{\Gamma} \in \mathcal{NL}(Y^{\Gamma})$,
proving that
$\mathcal{LO}(Y^{\Gamma}) \supset \mathcal{NL}(Y^{\Gamma})$
(respectively
$\mathcal{F}(Y^{\Gamma}) \supset \mathcal{NL}(Y^{\Gamma}) \setminus \mathcal{R}(Y^{\Gamma})$)
if
$\mathcal{LO}(Y) = \mathcal{NL}(Y)$
(respectively 
$\mathcal{F}(Y) = \mathcal{NL}(Y)$),
with
$N > \frac{q_r}{p_r}$, $p_r > 1$,
and $-1 \notin j(E_{\mathrm{in}}(r))$.
Since $S^3$ has no co-oriented taut foliations or left-orders
on its fundamental group, we then have
$\mathcal{LO}(Y^{\Gamma}) = \mathcal{NL}(Y^{\Gamma})$
(respectively $\mathcal{F}(Y^{\Gamma}) 
= \mathcal{NL}(Y^{\Gamma}) \setminus \mathcal{R}(Y^{\Gamma})$).

\vspace{.3cm}

This leaves the case in which 
we have a single $T(np,nq)$ torus-link satellite,
with $N:= 2g(K)-1 < \frac{q}{p}$.
Arguments similar to those above then show that if
$\mathcal{LO}(Y) = \mathcal{NL}(Y)$
(respectively 
$\mathcal{F}(Y) = \mathcal{NL}(Y)$),
then $\boldsymbol{y}^{\Gamma} \in \mathcal{NL}(Y^{\Gamma})$
implies 
that $\boldsymbol{y}^{\Gamma} \in \mathcal{LO}(Y^{\Gamma})$
(respectively $\boldsymbol{y}^{\Gamma} \in \mathcal{F}(Y^{\Gamma}) \cup \mathcal{Z}(Y^{\Gamma})$),
provided that 
\begin{equation}
y_{0+} 
\neq 
\mfrac{q^*}{p} + \mfrac{1}{p(q-pN)} \;\; \left(= \mfrac{p^*-q^*N}{q-pN}\right).
\end{equation}

Propositions \ref{prop: basic defs of y_0+-} and
\ref{prop: main bounds for y0- and y0+}.(+)
then tell us that
$y_{0+} = \frac{p^*-q^*N}{q-pN}$
implies
\begin{equation}
\sum_{i=1}^n \lceil y_i \rceil = 0,
\sum_{i=1}^n \lfloor [-y_i](q-Np) \rfloor = 0,
\sum_{i=1}^n \lfloor [-y_i](q-(N-1)p) \rfloor > 0,
\end{equation}
which, under change of basis to $S^3$-slopes, becomes
\begin{equation}
\Lambda(Y^{\Gamma}) \mkern-2mu\cdot\mkern-2mu 
(\left[-\infty, pq-q + pN\right>^n \mkern-1mu\setminus\mkern-1mu 
\left[-\infty, pq-q+pN-p\right>^n ).
\end{equation}
Since $pq-q+pN = pq - q - p + 2(K)p$, the theorem follows.

\qed

\subsection{Exceptional Symmetries}
\label{ss: exceptional symmetries}

As mentioned in the introduction,
there are instances, for exteriors of iterated torus-link satellites or
algebraic link satellites, in which the 
$\Lambda$-type symmetries for Seifert fibered components
have their influence extend across edges.
This phenomenon is more relevant in the context of exceptional splices.

\begin{prop}
Suppose $Y^{\Gamma}$ is the exterior of an algebraic link satellite
$K^{\Gamma} \subset S^3$
of a nontrivial positive L-space knot $K \subset S^3$ of genus $g(K)$,
with $N:= 2(g)K-1 > \frac{q_r}{p_r}$ and $-1 \in j(E_{\mathrm{in}}(r))$.
Then $\mathcal{L}(Y^{\Gamma}) = \bigcup_{e\in E_{\mathrm{in}}(r)} \mathcal{L}_e$, where
$$
\mathcal{L}_e :=
\left\{
\boldsymbol{y}^{\Gamma} 
\left.
\vphantom{\begin{array}{c}A \\ A \\ A\end{array}}
\right|
\begin{array}{l}
Y_{\Gamma_{v(-e')}}(\boldsymbol{y}^{\Gamma_{v(-e')}})
\;\text{ is }\;\textsc{bc},\,
\text{ with }\, y^r_{j(e')\pm} \in \Z,
\text{ for all } \;
e' \in E_{\mathrm{in}}(r),
  \\
\mfrac{q^*_r}{p_r}-\mfrac{1}{p_r(p_rN-q_r)}
\;<\; y_{0+}
\;\le\; y_{0-}
\;<\;
\mfrac{q^*_r}{p_r}
\;\;\text{ if }\;
Y_{\Gamma_{v(-e)}}(\boldsymbol{y}^{\Gamma_{v(-e)}})
\;\text{ is }\;\textsc{bi},\,
  \\
\mfrac{q^*_r}{p_r}-\mfrac{1}{p_r(p_rN-q_r)}
\;\le\; y_{0+}
\;\le\; y_{0-}
\;\le\;
\mfrac{q^*_r}{p_r}
\;\;\text{ if }\;
Y_{\Gamma_{v(-e)}}(\boldsymbol{y}^{\Gamma_{v(-e)}})
\;\text{ is }\;\textsc{bc}.
\end{array}
\right\}
$$

\end{prop}
\begin{proof}
This is established by a straightforward but tedious
adaptation of the arguments used to prove Claim 1
in the proof of 
Theorem~\ref{thm: bgw and juhasz for general satellites}.
\end{proof}
Note that the latter two conditions place strong constraints on
$\boldsymbol{y}^r$ as well.  In particular, we must have
$\boldsymbol{y}^r \in \Z^{|I_v|}$ unless 
$Y_{\Gamma_{v(-e)}}(\boldsymbol{y}^{\Gamma_{v(-e)}})
\;\text{ is }\;\textsc{bc}$ with $y_{j(e)\pm} \in \Z$,
in which case $y_i^r \in \Z$ all but at most one $i \in I_v$.

This phenomenon also affects 
$\Lambda_{\Gamma}$ as defined in
(\ref{eq: lambda_gamma def}).  While
$\Lambda_{\Gamma} \supset \prod_{v\in \text{Vert}(\Gamma)} \Lambda_v$,
this containment is proper if there is an edge
$e \in \text{Edge}(\Gamma)$ for which one can have
$Y_{\Gamma_{v(-e)}}(\boldsymbol{y}^{\Gamma_{v(-e)}})$ \textsc{bc}
with $0 \neq y^{v(e)}_{j(e)\pm} \in \Z$ if $j(e) \neq -1$,
or any integer value $y^{v(e)}_{j(e)\pm} \in \Z$ if $j(e) = -1$.
For a satellite without exceptional splices, the situation
is still relatively simple.
It is straightforward to show that the above edge condition
can hold only if $(p_{v(-e)}, q_{v(-e)}) = (1,2)$.
For such an edge, one must locally replace the product
$\Lambda_{v(-e)} \times \Lambda_{v(e)}$ with the product
$(\Lambda_{v(-e)} \times \Lambda_{v(e)})
\,\cup\,
(\Lambda_{v(-e)}^{(1)} \mkern-2mu\times \Lambda_{v(e)}^{(1)})$,
where $\Lambda_v^{(1)}:= \{\boldsymbol{y}^v \in \Z^{|I_v|} | \sum y^v_i = 1\}$.
If $\Gamma$ has no exceptional splice edges and no vertices with
$(p_v, q_v) = (1,2)$, then $\Lambda_{\Gamma} = \prod_{v \in \text{Vert}(\Gamma)} \Lambda_v$.

\bibliography{toruslinks}

\begin{thebibliography}{10}

\bibitem{BludovGlassI}
V.~V. Bludov and A.~M.~W. Glass.
\newblock Word problems, embeddings, and free products of right-ordered groups
  with amalgamated subgroup.
\newblock {\em Proc. Lond. Math. Soc. (3)}, 99(3):585--608, 2009.

\bibitem{Bowden}
Jonathan Bowden.
\newblock Approximating {$C^0$}-foliations by contact structures.
\newblock {\em Geom. Funct. Anal.}, 26(5):1255--1296, 2016.

\bibitem{BoyerClay}
Steven Boyer and Adam Clay.
\newblock Foliations, orders, representations, {L}-spaces and graph manifolds.
\newblock {\em Adv. Math.}, 310:159--234, 2017.

\bibitem{BGW}
Steven Boyer, Cameron~McA. Gordon, and Liam Watson.
\newblock On {L}-spaces and left-orderable fundamental groups.
\newblock {\em Math. Ann.}, 356(4):1213--1245, 2013.

\bibitem{ziggurat}
Danny Calegari and Alden Walker.
\newblock Ziggurats and rotation numbers.
\newblock {\em J. Mod. Dyn.}, 5(4):711--746, 2011.

\bibitem{ClayLidmanWatson}
Adam Clay, Tye Lidman, and Liam Watson.
\newblock Graph manifolds, left-orderability and amalgamation.
\newblock {\em Algebr. Geom. Topol.}, 13(4):2347--2368, 2013.

\bibitem{EisenbudNeumann}
David Eisenbud and Walter Neumann.
\newblock {\em Three-dimensional link theory and invariants of plane curve
  singularities}, volume 110 of {\em Annals of Mathematics Studies}.
\newblock Princeton University Press, Princeton, NJ, 1985.

\bibitem{EliashbergThurston}
Yakov~M. Eliashberg and William~P. Thurston.
\newblock {\em Confoliations}, volume~13 of {\em University Lecture Series}.
\newblock American Mathematical Society, Providence, RI, 1998.

\bibitem{gorskyhom}
Eugene Gorsky and Jennifer Hom.
\newblock Cable links and {L}-space surgeries.
\newblock {\em Quantum Topol.}, 8(4):629--666, 2017.

\bibitem{GNalglink}
Eugene Gorsky and Andr{\'a}s N{\'e}methi.
\newblock Links of plane curve singularities are {$L$}-space links.
\newblock {\em Algebr. Geom. Topol.}, 16(4):1905--1912, 2016.
\newblock arXiv:1403.3143.

\bibitem{GNLbounded}
Eugene Gorsky and Andr\'{a}s N\'{e}methi.
\newblock On the set of {L}-space surgeries for links.
\newblock {\em Adv. Math.}, 333:386--422, 2018.

\bibitem{HRRW}
Jonathan Hanselman, Jacob Rasmussen, Sarah~Dean Rasmussen, and Liam Watson.
\newblock {Taut foliations on graph manifolds}.
\newblock arXiv:1508.05911, 2015.

\bibitem{HRW}
Jonathan Hanselman, Jacob Rasmussen, and Liam Watson.
\newblock {Bordered Floer homology for manifolds with torus boundary via
  immersed curves}.
\newblock arXiv:1604.03466.

\bibitem{HanWat}
Jonathan Hanselman and Liam Watson.
\newblock {A calculus for bordered Floer homology}.
\newblock { arXiv:1508.05445}, 2015.

\bibitem{Heddencableii}
Matthew Hedden.
\newblock On knot {F}loer homology and cabling. {II}.
\newblock {\em Int. Math. Res. Not. IMRN}, (12):2248--2274, 2009.

\bibitem{Homcable}
Jennifer Hom.
\newblock A note on cabling and {$L$}-space surgeries.
\newblock {\em Algebr. Geom. Topol.}, 11(1):219--223, 2011.

\bibitem{JankinsNeumann}
Mark Jankins and Walter~D. Neumann.
\newblock Rotation numbers of products of circle homeomorphisms.
\newblock {\em Math. Ann.}, 271(3):381--400, 1985.

\bibitem{JuhaszConj}
Andr{\'a}s Juh{\'a}sz.
\newblock A survey of {H}eegaard {F}loer homology.
\newblock In {\em New ideas in low dimensional topology}, volume~56 of {\em
  Ser. Knots Everything}, pages 237--296. World Sci. Publ., Hackensack, NJ,
  2015.
\newblock arXiv:1310.3418.

\bibitem{KazezRobertsCzero}
William~H. Kazez and Rachel Roberts.
\newblock Approximating {$C^{1,0}$}-foliations.
\newblock In {\em Interactions between low-dimensional topology and mapping
  class groups}, volume~19 of {\em Geom. Topol. Monogr.}, pages 21--72. Geom.
  Topol. Publ., Coventry, 2015.

\bibitem{LiRoberts}
Tao Li and Rachel Roberts.
\newblock Taut foliations in knot complements.
\newblock {\em Pacific J. Math.}, 269(1):149--168, 2014.

\bibitem{liulspace}
Yajing Liu.
\newblock {$L$}-space surgeries on links.
\newblock {\em Quantum Topol.}, 8(3):505--570, 2017.

\bibitem{NemethiLO}
Andr\'as N\'emethi.
\newblock Links of rational singularities, {L}-spaces and {LO} fundamental
  groups.
\newblock {\em Invent. Math.}, 210(1):69--83, 2017.
\newblock {arXiv:1510.07128}.

\bibitem{NeumannEnd}
Walter~D. Neumann and Jonathan Wahl.
\newblock The end curve theorem for normal complex surface singularities.
\newblock {\em J. Eur. Math. Soc. (JEMS)}, 12(2):471--503, 2010.

\bibitem{OSGen}
Peter. Ozsv{\'a}th and Zolt{\'a}n Szab{\'o}.
\newblock Holomorphic disks and genus bounds.
\newblock {\em Geom. Topol.}, 8:311--334, 2004.
\newblock math.GT/0311496.

\bibitem{OSRat}
Peter~S. Ozsv{\'a}th and Zolt{\'a}n Szab{\'o}.
\newblock Knot {F}loer homology and rational surgeries.
\newblock {\em Algebr. Geom. Topol.}, 11(1):1--68, 2011.

\bibitem{lslope}
Jacob Rasmussen and Sarah~Dean Rasmussen.
\newblock Floer simple manifolds and {L}-space intervals.
\newblock {\em Adv. Math.}, 322:738--805, 2017.
\newblock {arXiv:1508.05900}.

\bibitem{lgraph}
Sarah~Dean Rasmussen.
\newblock L-space intervals for graph manifolds and cables.
\newblock {\em Compos. Math.}, 153(5):1008--1049, 2017.
\newblock {arXiv:1511.04413}.

\end{thebibliography}
\bibliographystyle{plain}

\end{document}